%% file: main.tex
\documentclass{article}
\pdfoutput=1	
\usepackage{fullpage}	
\usepackage[breaklinks,hidelinks]{hyperref}
\hypersetup{
  pdftitle={Computable Quantification in Reflective Grounded Arithmetic},
  pdfauthor={Bryan Ford}}

\usepackage{amsmath}	
\usepackage{amssymb}	
\usepackage{amsthm}

\newtheorem{thm}{Theorem}[section]

\input{setup}

\crefname{thm}{Theorem}{Theorems}
\crefname{cor}{Corollary}{Corollaries}
\crefname{lem}{Lemma}{Lemmas}
\crefname{prop}{Proposition}{Propositions}
\crefname{defi}{Definition}{Definitions}
\crefname{rem}{Remark}{Remarks}
\crefname{exa}{Example}{Examples}

\makeatletter
\@ifundefined{cor}{\declaretheorem[sibling=thm,name=Corollary,
  refname={Corollary,Corollaries},Refname={Corollary,Corollaries}]{cor}}{}
\@ifundefined{lem}{\declaretheorem[sibling=thm,name=Lemma,
  refname={Lemma,Lemmas},Refname={Lemma,Lemmas}]{lem}}{}
\@ifundefined{prop}{\declaretheorem[sibling=thm,name=Proposition,
  refname={Proposition,Propositions},
  Refname={Proposition,Propositions}]{prop}}{}
\@ifundefined{defi}{\declaretheorem[sibling=thm,name=Definition,
  style=definition,refname={Definition,Definitions},
  Refname={Definition,Definitions}]{defi}}{}
\@ifundefined{rem}{\declaretheorem[sibling=thm,name=Remark,
  style=remark,refname={Remark,Remarks},
  Refname={Remark,Remarks}]{rem}}{}
\@ifundefined{exa}{\declaretheorem[sibling=thm,name=Example,
  style=definition,refname={Example,Examples},
  Refname={Example,Examples}]{exa}}{}
\@ifundefined{thmc}{\declaretheorem[name=Theorem,numbered=no]{thmc}}{}
\makeatother

\newcommand{\cmp}{\mathit{cmp}}		
\newcommand{\den}[1]{\llbracket #1 \rrbracket}	
\newcommand{\cA}{\mathcal A}		
\newcommand{\hs}{\mathrel{\leadsto}}	
\newcommand{\lft}[1]{{\uparrow}#1}	
\newcommand{\vz}{v_0}			
\newcommand{\bev}[1]{\Downarrow_{#1}}	
\newcommand{\hev}[1]{\Downarrow^{\mathrm{hd}}_{#1}}	
\newcommand{\wbev}[1]{\Downarrow^{\mathrm{w}}_{#1}}	
\newcommand{\whev}[1]{\Downarrow^{\mathrm{w,hd}}_{#1}}	
\newcommand{\wf}{\mathrm{wf}}		
\newcommand{\fv}{\mathrm{fv}}		
\newcommand{\proves}{\mathrel{:}}	
\newcommand{\natbot}{\mathbb{N}_\bot}	
\newcommand{\tv}[1]{\lceil #1 \rceil}	
\newcommand{\tvT}{\mathsf{tt}}		
\newcommand{\tvF}{\mathsf{ff}}		
\newcommand{\omv}{\Downarrow^{\omega}}	
\newcommand{\artifacttag}{\texttt{rga-}\allowbreak\texttt{2026-10-20}\xspace}

\newcommand{\sap}{\mathbin{@}}		
\newcommand{\spl}[1]{\llparenthesis #1 \rrparenthesis}	
\newcommand{\cpair}[2]{\langle #1, #2 \rangle}	
\newcommand{\cns}{\mathbin{::}}		
\newcommand{\canon}[1]{\mathrm{can}(#1)}	
\newcommand{\hok}[2]{\mathrm{Tr}(#1,#2)}	
\newcommand{\ssub}[3]{[{#1}{\mapsto}{#2}]{#3}}	
\newcommand{\bsub}[2]{{#1}[{#2}]}	
\newcommand{\sof}[2]{\hat{#1}({#2})}	
\newcommand{\asub}[1]{\sigma_{#1}}	
\newcommand{\gtrue}[2]{{#1} \models {#2}}	
\newcommand{\wgtrue}[2]{{#1} \models^{\mathrm{w}} {#2}}	
\newcommand{\agtrue}[2]{{#1} \models_{\mathrm{A}} {#2}}	
\newcommand{\ogtrue}[2]{{#1} \models^{\mathrm{w}}_{\mathrm{o}} {#2}}	
\newcommand{\lftk}[2]{{\uparrow}^{#1}{#2}}	
\newcommand{\lftat}[3]{{\uparrow}^{#2}_{#1}{#3}}	
\newcommand{\vc}{\mathrel{\approx}}	
\newcommand{\vcd}[1]{\mathrel{\approx_{#1}}}	
\newcommand{\ost}[1]{\mathrel{\rightarrowtail_{#1}}}	
\newcommand{\osts}[1]{\mathrel{\rightarrowtail^{*}_{#1}}}	
\newcommand{\iter}[2]{\mathrm{it}^{#2}_{#1}}	
\newcommand{\reb}[3]{{#1}[{#2}\!:=\!{#3}]}	
\newcommand{\wfpart}[1]{{#1}^{\wf}}	

\newcommand{\iZ}{\mathsf{Init}_0}	
\newcommand{\iS}{\mathsf{Init}_S}	
\newcommand{\iL}{\mathsf{Init}_L}	
\newcommand{\iR}{\mathsf{Init}_R}	
\newcommand{\iSub}{\mathsf{Sub}}	
\newcommand{\iRec}{\mathsf{Rec}}	
\newcommand{\llam}{\lambda\lambda}	

\newcommand{\hred}{\mathrm{hred}}	
\newcommand{\Some}[1]{\mathsf{Some}\,#1}
\newcommand{\None}{\mathsf{None}}
\newcommand{\dst}[1]{\mathrm{dest}_{#1}}	
\newcommand{\sfdom}[1]{\mathrm{dom}\,#1}	

\newcommand{\rrep}{\mathit{rep}}	
\newcommand{\rmat}{\mathit{mat}}	

\newcommand{\tadd}{\mathit{add}}
\newcommand{\tmult}{\mathit{mult}}
\newcommand{\tack}{\mathit{ack}}
\newcommand{\tsucf}{\mathit{sucF}}
\newcommand{\tmulf}{\mathit{mulF}}

\newcommand{\sfapp}{\mathit{sfApp}}	

\begin{document}

\title{Computable Quantification in \\ Reflective Grounded Arithmetic}
\author{Bryan Ford \\ EPFL}
\date{}					
\maketitle
\input{abs}

%
%

\input{intro}
\input{bg}
\input{system}
\input{metatheory}
\input{power}
\input{omega}
\input{character}

\input{formal}
\input{related}
\input{concl}
\input{ack}

\clearpage
\appendix
\crefname{appendix}{Appendix}{Appendices}
\Crefname{appendix}{Appendix}{Appendices}
\crefalias{section}{appendix}
\crefalias{subsection}{appendix}
\crefalias{subsubsection}{appendix}
\input{app-roadmap}
\input{app-complete}
\input{app-sound}
\input{app-compile}
\input{app-re}
\input{app-represent}
\input{app-arith}
\input{app-omega}
\input{app-character}
\input{app-notation}
\input{app-endpoints}

\bibliographystyle{plain}
\bibliography{logic}

\end{document}

%% file: setup.tex
\usepackage{times}
\usepackage{xspace}
\usepackage{cite}		
\usepackage{bussproofs}
\usepackage{stmaryrd}
\usepackage[noabbrev,capitalize]{cleveref}
\usepackage{thmtools}
\usepackage{xstring}
\usepackage{graphicx}
\usepackage{subcaption}
\usepackage{nicematrix}		
\usepackage{longtable}
\usepackage{pifont}

\usepackage{isabelle}
\usepackage{isabellesym}
\usepackage{pdfpages}
\usepackage{ragged2e}
\newcommand{\SNIP}[2]{\expandafter\newcommand\csname snippet--#1\endcsname{#2}}
\IfFileExists{snips.tex}{\input{snips}}{}

\newcommand{\GetSnip}[1]{%
    \ifcsname snippet--#1\endcsname%
        \csname snippet--#1\endcsname%
    \else%
        \PackageWarning{snips}{Snippet ``#1'' is undefined.}%
        \emph{Warning: Snippet ``#1'' is undefined.}%
    \fi%
}

\newcommand{\Snippet}[1]{{%
  \newcount\i
  \i=0
  \loop
    \GetSnip{#1-\the\i}%
    \advance \i 1
  \ifcsname snippet--#1-\the\i\endcsname
  \repeat
}}

\newcommand{\SnippetPart}[3]{{%
  \newcount\i
  \i=#1
  \loop
    \ifnum \i=#2
      \renewcommand{\isanewline}{}%
    \fi
    \GetSnip{#3-\the\i}%
    \advance \i 1
    \ifnum \i>#2 {}
    \else \repeat
}}

\hypersetup{breaklinks=true}

\newcommand{\com}[1]{}

\long\def\note#1{}	

\newcommand{\eg}{e.g.,\xspace}

\newcommand{\gd}{GD\xspace}		
\newcommand{\bga}{BGA\xspace}		
\newcommand{\pga}{PGA\xspace}		
\newcommand{\rga}{RGA\xspace}		
\newcommand{\pra}{PRA\xspace}		
\newcommand{\re}{r.e.\xspace}		

\newcommand{\gdl}{G\"odel\xspace}

\newcommand{\limp}{\rightarrow}		
\newcommand{\liff}{\leftrightarrow}	
\newcommand{\ldef}{\equiv}		

\newcommand{\typestyle}[1]{\textsf{#1}\xspace}
\com{	
\newcommand{\tbool}{\typestyle{bool}}

\newcommand{\tnat}{\typestyle{nat}}

}
\newcommand{\tbool}{\typestyle{B}}	

\newcommand{\tnat}{\typestyle{N}}

\newcommand{\judgment}[1]{\textsf{ #1}}

\newcommand{\jbool}{\judgment{\tbool}}

\newcommand{\jnat}{\judgment{\tnat}}

\newcommand{\snat}{\mathbb{N}}		

\makeatletter
\newlength{\doublefracgap}
\DeclareRobustCommand{\doublefrac}[2]{%
  \mathinner{\mathpalette\doublefrac@{{#1}{#2}}}%
}
\newcommand{\doublefrac@}[2]{\doublefrac@@#1#2}
\newcommand{\doublefrac@@}[3]{%
  \ooalign{%
    \raisebox{\doublefracgap}{$\m@th#1\frac{#2}{\phantom{#3}}$}\cr
    \raisebox{-\doublefracgap}{$\m@th#1\frac{\phantom{#2}}{#3}$}\cr
  }%
}
\newcommand{\ddoublefrac}[2]{{\displaystyle\doublefrac{#1}{#2}}}

\makeatother

\newcommand{\irl}[1]{\ensuremath{\mathit{#1}}}		

\newcommand{\infrule}[3][]{\cfrac{#2}{#3}\IfStrEq{#1}{}{}{\ \irl{#1}}}
\newcommand{\infeqv}[3][]{\ddoublefrac{#2}{#3}\IfStrEq{#1}{}{}{\ \irl{#1}}}
\newcommand{\infceqv}[4][]{\cfrac{#2\qquad}{}\ddoublefrac{#3}{#4}
				\IfStrEq{#1}{}{}{\ \irl{#1}}}

\newcommand{\z}{\mathbf 0}

\newcommand{\suc}{\mathbf S}

\newcommand{\cR}{\mathcal R}

\newcommand{\quo}[1]{\ulcorner{#1}\urcorner}

\newcommand{\tofn}[1]{\underline{#1}}



%% file: abs.tex
\begin{abstract}
Informal statements of \gdl's incompleteness theorems often run:
``any consistent formal system with arithmetic
has true but unprovable statements.''
That reading assumes a classical notion of truth,
under which a universal statement is true
as soon as each of its numeral instances is.
This paper presents \emph{reflective grounded arithmetic} (\rga),
in which truth is \emph{grounded} in computation
and universal quantification is grounded \emph{reflectively}:
a universal statement is true when the system's own proof search
certifies its schematic instance,
and false when the system refutes one numeral instance.
\rga allows recursive definitions without restriction
and proves addition, multiplication and Ackermann's function total
as theorems quantified inside the system.
\rga is consistent,
and it is complete for its own, grounded notion of truth:
every well-formed statement that evaluates to true is provable.
\gdl's theorems still hold for \rga, as a companion paper shows.
The statements that classical reasoning calls true but unprovable
have every numeral instance provable but no certified schematic instance,
so they are \emph{ungrounded} in \rga;
read instance by instance, in an $\omega$-semantics
for which \rga is also sound, they are true.
All results are machine-checked in Isabelle/HOL:
soundness and consistency;
completeness for grounded truth, open statements included;
$\omega$-soundness and $\omega$-consistency;
that grounded truth is recursively enumerable
and \rga represents exactly the recursively enumerable sets;
and $\omega$-incompleteness.
The logic is neither classical nor intuitionistic:
double negation is eliminable as an entailment,
excluded middle fails wherever a statement has no boolean value,
\rga has the disjunction and witness properties,
and the deduction theorem holds exactly at provably decided hypotheses.
\end{abstract}

%% file: intro.tex
\section{Introduction}
\label{sec:intro}

\gdl's first incompleteness theorem is often described informally as
proving that ``any consistent formal system containing arithmetic
must contain true but unprovable statements.''
This typical summary identifies one key ingredient the theorem assumes,
namely enough arithmetic to express self-reference.
Another usually-unstated assumption, however,
is that ``true'' means classical truth,
where every formula is assumed bivalent (either true or false),
and a universally quantified statement is true
as soon as each of its numeral instances is true.
\gdl's theorems themselves concern provability,
or the computational search for theorems in a formal system,
and apply to both classical and intuitionistic (Heyting) arithmetic.
\gdl's theorems show in effect that
no classical or intuitionistic system with arithmetic power
can ever fully align truth with computational provability.

This paper presents \emph{reflective grounded arithmetic}, or \rga,
a formal system that meets the folklore ingredient list above,
but systematically aligns truth with provability and computation.
\rga has powerful arithmetic,
including provably-total addition and multiplication and more;
it has quantification over the natural numbers;
its theorems are recursively enumerable;
it has full recursively-enumerable self-reference;
and it has a computational operational semantics
for which it is sound and complete.
\rga in essence offers an existence proof
that semantic truth and syntactic provability \emph{can} be fully aligned
in a powerful system of logic with quantification and arithmetic,
contrary to the folklore interpretation of \gdl's theorems.
The result presents a distinct perspective on computational reasoning,
which often agrees with classical logic but sometimes differs strikingly,
in formally defensible ways,
as this paper explores.

\rga builds on \emph{grounded deduction}
(\gd)~\cite{ford24reasoning},
whose goal is to bring the freedom of unconstrained recursive
definition, which is routine in everyday programming, into
consistent formal reasoning.
\gd in turn continues the tradition of functional reasoning systems
established by Church and Curry,
whose original versions were
inconsistent~\cite{church32set,kleene35inconsistency},
but which have since produced strong and consistent systems,
some with unrestricted recursion
and induction over the natural numbers~\cite{barendregt93systems,czajka13higher,czajka13partiality,czajka15semantic}.
\rga extends this tradition's focus on computational expressiveness
with a grounded operational semantics for which it is sound and complete.
Further, \rga contains its own non-primitive,
recursively-definable truth predicate,
a capability that neither the classical, intuitionistic, nor illative precedents
have demonstrated (\cref{sec:related:fitch}).

In \gd,
a statement is true not by classical fiat
but by computational \emph{grounding}:
\eg two terms are equal only when both compute the same value.
A \emph{dynamic type check} --
the claim that a term denotes a natural number,
written `$t \jnat$' --
is true when the computation of $t$ actually terminates with a numeral.
\gd calls this discipline \emph{habeas quid}
(\cref{sec:bg:gd}):
an existence claim must produce the thing itself.
Nonterminating computations harmlessly denote no value ($\bot$),
and statements about them are typically neither true nor false.
Prior work showed that quantifier-free 
\emph{propositional grounded arithmetic} or \pga
is already powerful, sound, and complete
for its intended semantics of grounded truth~\cite{ford25have}.
General-purpose mathematical reasoning, however, demands quantifiers.

Extending grounded arithmetic beyond propositional logic
poses the question:
what \emph{grounds} a universally quantified statement?
The na\"ive answer -- the truth of all the statement's numeral instances --
is the classical $\omega$-rule,
an infinitary condition no effective proof system can decide.
\rga's answer is \emph{reflection}:
a universal statement $\forall x.\,\phi(x)$ is grounded-true
when a search through the system's \emph{own proofs},
built into the definition of evaluation,
finds a proof of the schematic instance $\phi(x)$
for a fresh variable $x$.
The universal $\forall x.\,\phi(x)$ is grounded-false
when that search refutes
some particular numeral instance $\phi(\bar n)$:
a refuted universal always carries
an explicit counterexample witness.
Quantifier truth is thereby grounded in computation
exactly as equational truth is,
and it stays recursively enumerable.
\rga thus embodies \emph{computable quantification},
a quantifier whose truth is recursively enumerable.
The quantifier is a combinator whose evaluation
runs a proof search.
The proof system is defined first,
by purely syntactic rules,
and the semantics then consults derivability
as a fixed relation (\cref{sec:system}).
The resulting system keeps \gd's freedoms:
unconstrained recursion via a general recursion combinator,
call-by-name application,
and functional abstraction ($\lambda$) as the sole binder.
To those freedoms \rga adds quantified reasoning,
including internal mathematical induction.

\rga is consistent and semantically sound.
On well-formed statements it is also complete
for its own notion of truth:
every provable statement is grounded-true,
and every grounded-true statement is provable.
This is semantic completeness for grounded truth,
not syntactic or negation completeness,
which demands that every sentence be either provable or refutable.
\rga does not decide every statement.
Further,
\rga is complete for \emph{its own intended semantics}, grounded truth,
not just complete relative to some arbitrary model of a consistent theory
that \gdl's completeness theorem promises --
a model that need not have any meaningful relationship
to the theory's intended semantics.

Classically, completeness for the intended model of arithmetic
is impossible, since classical true arithmetic is not recursively enumerable.
Grounded truth is recursively enumerable,
which is what makes possible
\rga's completeness for its intended semantics.
The result is not the triviality that every consistent effective theory
is complete for ``true means provable'':
only \rga's quantifier clauses consult provability,
and the rest of its semantics is ordinary strong Kleene evaluation,
which the provability of a classical theory does not satisfy
(\cref{sec:metatheory:iff}).
\gdl's first and second theorems hold for \rga,
as the companion paper~\cite{ford-conrga}
proves in Isabelle/HOL:
\rga neither proves nor refutes
its \gdl sentence or its consistency sentence.
What changes is the truth of a universal quantified statement.
\rga's universal quantifier is true only when a proof certifies
its schematic instance,
so the \gdl and consistency sentences,
each of whose numeral instances \rga proves,
are \emph{ungrounded} in \rga's semantics:
no finite evaluation establishes them,
and their unprovability is no gap in \rga's completeness.
Read instance by instance,
in a classical $\omega$-semantics for which \rga is also sound
(\cref{sec:metatheory:omega}),
quantified \rga formulas can be true but unprovable
as the \gdl folklore interpretation says.
\rga is paracomplete in the Kripke tradition~\cite{kripke75outline},
but where Kripke's sentences are \emph{classically grounded},
in a bivalent base arithmetic,
\rga's are \emph{computationally grounded},
arithmetic included (\cref{sec:bg:kripke}).

This paper's main formal results are as follows,
all of them machine-checked using Isabelle/HOL~\cite{nipkow02isabelle}:

\begin{itemize}

\item \textbf{Soundness and consistency.}
\rga is consistent: no statement is both provable and refutable,
and in particular \rga does not prove $\z = \tofn{1}$.
\rga is moreover sound for its operational semantics:
whatever \rga proves is grounded-true.
The \emph{habeas quid} discipline itself is expressible as a theorem:
every provable totality claim `$t \jnat$'
(``$t$ is a natural number'')
is backed by an actual terminating computation
(\emph{\tnat-soundness}).
\rga is also sound for a classical $\omega$-semantics,
in which a universal is true when all its numeral instances are,
and it is $\omega$-consistent.

\item \textbf{Open completeness.}
On well-formed statements, open ones included,
provability from no hypotheses coincides with grounded truth:
a statement is provable exactly when it evaluates to true
in its operational semantics.
This is completeness in the weak sense,
for derivations from no hypotheses;
under hypotheses it fails in general, for two reasons (\cref{sec:omega:weak-strong}).
The well-formed statements that \rga neither proves nor refutes
are exactly those that are neither grounded-true nor grounded-false.

\item \textbf{Expressive power.}
The unary predicates \rga represents are
exactly the recursively enumerable (\re) sets.
Every \re\ set is represented by a closed term,
via a verified compiler from primitive-recursive indices
into \rga's recursion combinator.
Conversely, \rga represents nothing beyond the \re\ sets,
because its own provability is \re
As internal showcases,
\rga proves the totality of addition and multiplication
as quantified theorems, by internal induction.
For each primitive-recursive function,
\rga proves the compiled function total
by one quantified theorem,
reaching the primitive-recursive functions
directly through its recursion combinator,
with no \gdl-$\beta$ coding detour.
\rga reaches beyond primitive recursion:
it proves Ackermann's function total by a quantified theorem
whose inner induction consumes a quantified induction hypothesis.
Value completeness also gives,
for each closed ground-type term of \gdl's System~T,
a derivation that the term denotes a natural number,
by normalization of System~T carried out in HOL.

\item \textbf{Limit: $\omega$-incompleteness.}
Because grounded truth is \re\
while instance-wise $\Pi_1$ truth is not,
from a classical metalogic perspective there is a family of statements
whose numeral instances are all provable
while their universal closure is unprovable and ungrounded.
The gap between instance-wise and universal provability
sits exactly at the boundary of recursive enumerability.
The classic halting diagonal in \rga's term language
forms such a family,
and the companion paper~\cite{ford-conrga} shows that
\rga's own consistency sentence is another.

\item \textbf{The logic's character.}
\rga occupies a corner of the logical landscape
distinct from both the classical and intuitionistic traditions.
Double negation is eliminable as an entailment,
by a primitive rule in each direction,
while excluded middle fails wherever a statement has no boolean value,
and assuming it at the Liar is explosive.
\rga has the disjunction property,
and a strong witness property:
a refuted universal always comes with
a refuted numeral instance,
because grounded falsity of a universal is
counterexample-certified by construction.
The deduction theorem holds exactly at
provably decided hypotheses.
Elsewhere, discharging a hypothesis would manufacture
a decided disjunction
that the hypothetical judgment never asserted,
and Curry's paradox shows that no conditional with modus ponens
can discharge every hypothesis.

\end{itemize}

Every theorem, corollary and proposition in this paper
is machine-checked;
the few arguments that are not carry a visible \emph{Remark} label.
The complete development --
syntax, proof system, operational semantics,
soundness, completeness, the representation theorems,
and the $\omega$-incompleteness construction --
is formalized in Isabelle/HOL with no unproven obligations.
The \rga-specific development comprised
37 theories and about 34{,}000 lines
when it was measured (\cref{sec:formal:avail}).
That development sits atop
the shared grounded-arithmetic infrastructure
(\cref{sec:formal}),
and it is released together with the companion paper's
development (\cref{sec:formal:avail}).

\paragraph{AI disclosure.}
The machine-checked proofs and the first drafts of this paper's text
were produced with substantial assistance from Claude models from
Anthropic, working under the direction and review of the author
(\cref{sec:formal:ai}).
The definitions that a reader must trust
are listed in \cref{sec:formal},
and the numbered results with their Isabelle names
in \cref{sec:app:endpoints};
they were reviewed by the author.

\paragraph{Relation to other papers.}
This paper extends
the quantifier-free \pga system~\cite{ford25have}.
The present work inherits that development's foundations:
the grounded reading of equality and dynamic type checks,
formulas as boolean-valued terms,
the propositional core of the rule set,
and the \gdl-coding and
primitive-recursion infrastructure
underlying the computability results.
Everything quantified is new here:
the quantifier-as-combinator syntax,
the reflective semantics that grounds
universal truth in proof search,
the quantifier proof rules and
internal induction,
and the quantified metatheory --
including the entanglement of
soundness, completeness, and determinism
that reflection forces
(\cref{sec:metatheory}),
and the $\omega$-incompleteness analysis
that quantifiers first make possible
(\cref{sec:omega}).
The companion paper~\cite{ford-conrga}
carries part of this metatheory inside \rga itself.
It compiles a decider of \rga's evaluation at each fuel bound
and checkers of \rga's proofs into \rga terms;
\rga then proves internally that the decider never affirms
$\z = \tofn{1}$,
proves soundness for each of its proofs
and completeness for each grounded-true well-formed term,
as schemata,
and proves \gdl's second theorem for itself
as an entailment and in a decided-antecedent form
(the forms of a conditional are defined in \cref{sec:character:dt}).
That paper also proves \gdl's first and second theorems for \rga
in Isabelle/HOL.

The rest of this paper is organized as follows.
\Cref{sec:bg} summarizes the background on truth-value gaps
and grounded deduction.
\Cref{sec:system} presents \rga:
syntax, proof system, and reflective operational semantics.
\Cref{sec:metatheory} develops the core metatheory --
completeness, soundness, determinism, consistency,
and the truth--provability coincidence
with its consequences.
\Cref{sec:power} characterizes \rga's expressive power
and presents the internal arithmetic.
\Cref{sec:omega} proves $\omega$-incompleteness,
presents the $\omega$-semantics,
and shows where completeness fails under hypotheses.
\Cref{sec:character} locates \rga on the logical map.
\Cref{sec:formal} describes the Isabelle/HOL formalization,
\cref{sec:related} discusses related work,
and \cref{sec:concl} concludes.
The appendices transcribe the mechanized proofs
into mathematical prose,
with a roadmap on page~\pageref{sec:app:roadmap}:
completeness (\cref{sec:app:complete}),
soundness, determinism, and consistency (\cref{sec:app:sound}),
the verified compiler (\cref{sec:app:compile}),
recursive enumerability of grounded truth (\cref{sec:app:re}),
the representation theorem (\cref{sec:app:represent}),
the internal arithmetic (\cref{sec:app:arith}),
$\omega$-incompleteness (\cref{sec:app:omega}),
and the logic's character (\cref{sec:app:character}).
\Cref{sec:app:notation} collects the notation,
and \cref{sec:app:endpoints} maps every numbered result
to its Isabelle fact, theory and session.

%% file: bg.tex
\section{Background}
\label{sec:bg}

This section summarizes what a reader needs
in order to follow the technical development:
truth-value gaps in Kripke's tradition (\cref{sec:bg:kripke}),
the grounded-deduction discipline built on that tradition
(\cref{sec:bg:gd}),
and the quantifier-free grounded systems
that \rga extends (\cref{sec:bg:pga}).
Readers familiar with this material may skip ahead to \cref{sec:system};
\cref{sec:related} covers the wider landscape,
and \cref{sec:app:notation} collects the paper's notation.

\subsection{Truth-Value Gaps and Kripke's Theory of Truth}
\label{sec:bg:kripke}

Classical logic expects every well-formed statement
to be either true or false.
A \emph{paracomplete} logic drops this expectation:
some statements may be neither.
The gap is not a refusal to say
which truth value a statement ``really'' has.
A statement like the Liar (``this statement is false'')
cannot consistently take either value,
since either one leads to a contradiction.
A paracomplete logic accepts the possibility of missing truth values
instead of ruling by decree that
every formula must ``be'' either true or false --
an edict that, by \gdl's theorems, \emph{provable} truth can never live up to.

Kripke~\cite{kripke75outline} made truth-value gaps
mathematically tractable.
Starting from the atomic facts,
compound statements are evaluated by
Kleene's strong three-valued rules~\cite{kleene52introduction}.
A disjunction is true when at least one disjunct is true --
even if the other disjunct lacks a value --
false when both disjuncts are false,
and otherwise valueless;
a negation swaps true and false
and preserves valuelessness.
These rules are \emph{monotone}:
once a statement gets a value at some evaluation stage,
later stages never take that value back or change it.
Repeating the valuation rules therefore converges to a fixed point.
A statement is \emph{grounded} when
some stage gives the statement a value,
and \emph{ungrounded} when no stage ever does.
The Liar is ungrounded --
as is the merely circular Truthteller
(``this statement is true'') --
and both are harmless:
paradoxical reasoning never gets
a truth value to work with.
Kripke's construction was about a truth predicate
added on top of classical arithmetic,
so its sentences are \emph{classically grounded}:
grounded in a bivalent base model,
with gaps confined to the truth predicate.
Grounded deduction instead moves
this monotone, fixed-point view of truth
into the logic itself,
arithmetic included.
Its sentences are \emph{computationally grounded},
grounded in computation all the way down,
so truth-value gaps can occur at the arithmetic level.
The key benefit of \rga's approach is to align logical truth with computation:
an \rga formula is true precisely when -- and indeed \emph{because} --
the formula represents a computation that terminates
with a ``true'' value.

\subsection{Grounded Deduction}
\label{sec:bg:gd}

Grounded deduction (\gd)~\cite{ford24reasoning}
applies the paracomplete view of truth
to a longstanding goal:
bringing \emph{unconstrained recursive definition},
routine in everyday programming,
into formal reasoning.
In \gd, any function may be defined
by recursion in terms of itself and others,
with no termination proof, totality check,
or stratification (layering in the Russell/Tarski tradition) required at the time of definition.
The price of this freedom is that
a term denotes only what the term's computation actually produces.
An equation `$s = t$' is grounded-true
when both sides compute to a common value,
grounded-false when the two sides compute to different values,
and valueless when either side diverges.
A nonterminating computation harmlessly denotes
no value at all ($\bot$).

Because a term may denote nothing,
\gd reasoning is disciplined by \emph{dynamic type checks}:
the judgment `$t \jnat$' claims that
the computation of $t$ terminates in a natural-number value,
and some inference rules require such a check
before a term's value may be used in further reasoning.
\gd calls this discipline \emph{habeas quid}:
to reason with a thing, one must first ``have a thing'' to reason with.
Proving a recursive function to be total (always-terminating)
amounts to symbolically executing the function's computation
through these rules.
Once totality is proved,
later reasoning about the function
reduces to familiar classical-style steps.

\subsection{Quantifier-free Grounded Arithmetic}
\label{sec:bg:pga}

Prior work built \gd into
machine-checked systems of
quantifier-free grounded arithmetic~\cite{ford25have}.
\emph{Basic grounded arithmetic} (\bga)
contains no logical operators at all,
only terms, equations, and dynamic type checks.
Yet \bga already expresses arbitrary general-recursive functions,
and proves each correct result
on any fixed arguments on which the function terminates.
\emph{Propositional grounded arithmetic} (\pga)
adds grounded logical operators
implemented as \gdl-style \emph{reflective} computations
over the \bga layer:
a propositional statement is itself a computation
that searches for the statement's own grounds.
\pga supports inductive reasoning about open formulas.
Its logical power is comparable to
Skolem's \pra~\cite{skolem23begrundung},
while its functional power is Turing-complete.
Both \bga and \pga are consistent and sound,
and complete for grounded truth:
a statement is provable exactly when
the statement's evaluation yields true.
This is not syntactic or negation completeness (\cref{sec:intro}).
For classical truth, \gdl's first theorem rules out that combination
in every consistent effective system containing enough arithmetic,
because classical arithmetic truth is not recursively enumerable.
It is available here because grounded truth,
like provability, is recursively enumerable.

What these quantifier-free systems lack
is general-purpose mathematical reasoning:
universally quantified theorems,
and induction that concludes them.
\rga closes this gap.
\rga extends \pga's reflective implementation
of the propositional operators
to the universal quantifier itself,
as \cref{sec:system} presents next.
The grounded existential quantifier then ``comes for free''
via the standard classical duality of quantifiers,
which \rga preserves.

%% file: system.tex
\section{The \rga Formal System}
\label{sec:system}

\rga consists of three components,
presented in dependency order:
a term syntax (\cref{sec:system:syntax}),
a proof system over that syntax (\cref{sec:system:proof}),
and an operational semantics defining grounded truth
(\cref{sec:system:sem}).
The order matters,
because the components are not independent layers.
The operational semantics of the universal quantifier
invokes the proof system:
evaluating a universal statement runs a proof search.
The proof system is therefore defined first,
by purely syntactic rules with no semantic side conditions,
and the semantics is an ordinary inductive definition
that consults derivability as a fixed relation.
It consults it only positively:
a certificate makes a statement true or false,
and the absence of a certificate gives no verdict.
That positivity keeps evaluation monotone in the step index
and grounded truth recursively enumerable (\cref{thm:truth-re}).
This section presents each component in turn;
the reflective knot is tied in \cref{sec:system:sem:all}.

\subsection{Syntax}
\label{sec:system:syntax}

\begin{figure}
\centering
\small
$
t \ldef v
  \mid \bot
  \mid \z
  \mid \suc t
  \mid t = t
  \mid \neg t
  \mid t \lor t
  \mid \lambda t
  \mid t \cdot t
  \mid \cR
  \mid \cA
$
\medskip

\begin{tabular}{r@{~$\ldef$~}l@{\qquad}r@{~$\ldef$~}l}
$t \ne u$ & $\neg(t = u)$ &
$t \jnat$ & $t = t$ \\
$t \jbool$ & $t \lor \neg t$ &
$t \land u$ & $\neg(\neg t \lor \neg u)$ \\
$t \limp u$ & $\neg t \lor u$ &
$t \liff u$ & $(t \limp u) \land (u \limp t)$ \\
$\forall.\,b$ & $\cA \cdot (\lambda b)$ &
$\exists.\,b$ & $\neg(\cA \cdot (\lambda \neg b))$ \\
\end{tabular}
\caption{Term syntax of \rga (top) and derived forms (bottom).
  Variables $v$ are de Bruijn indices and $\lambda$ is the only
  binder; the quantifiers bind nothing themselves, so the derived
  $\forall.\,b$ quantifies the de Bruijn variable $0$ of the
  body $b$.}
\label{fig:syntax}
\end{figure}

\Cref{fig:syntax} shows \rga's complete syntax,
with ten primitive operators.
Terms comprise
variables,
the divergent constant $\bot$,
the arithmetic constructors zero and successor,
the grounded connectives
(equality, negation, disjunction),
$\lambda$-abstraction with application,
and two nullary combinator constants:
the primitive recursor $\cR$
and the universal quantifier $\cA$.
Variables are de Bruijn indices
and $\lambda$ is the sole binding operator.
Application is call-by-name
and $\lambda$-terms are first-class,
so unconstrained recursion needs no primitive.
Self-application gives recursion directly.
For a term $f$, put
$\mathit{self}(f) \ldef \lambda\,(\lft f \cdot (\vz \cdot \vz))$
and $\mathit{fix}(f) \ldef \mathit{self}(f) \cdot \mathit{self}(f)$,
so that $\mathit{fix}(f)$ steps in one head step
(\cref{sec:system:proof})
to $f \cdot \mathit{fix}(f)$
(\cref{sec:app:character:fix}).
The standard example is the Liar
$L \ldef \mathit{fix}(\lambda\,\neg\vz)$,
which converts to its own negation in two head steps,
$L \hs (\lambda\,\neg\vz) \cdot L \hs \neg L$.
Even the primitive $\bot$ is kept
only as a convenient standard divergent term.

The syntax saves work in three ways.
First, there is no separate formula syntax:
formulas are boolean-valued terms,
where the boolean values are the numerals
$\tofn{1}$ (true) and $\tofn{0}$ (false) --
a design inherited from \pga.
Second, the dynamic type check
and the boolean type check
are themselves derived formulas.
The type check `$t \jnat$' abbreviates
the grounded self-equality `$t = t$',
which holds exactly when $t$ computes a natural number.
The type check `$t \jbool$' abbreviates the grounded excluded middle
`$t \lor \neg t$',
which holds exactly when $t$ computes a boolean value.
Third, the primitive universal quantifier is not a binder
but a combinator.
A universally-quantified formula is
the combinator constant $\cA$ applied to an ordinary function term,
usually a $\lambda$-abstraction.
A quantifier instance is then
the ordinary application of that function to a numeral.
All binding structure thereby lives in
lambda abstraction $\lambda$,
and the quantifier rules below
manipulate only function applications.

A term is \emph{well-formed} --
written $\wf(t)$ in the rules --
when the term is built from the ten operators
at their correct arities and binder shapes.
Well-formedness is the only syntactic side condition
appearing anywhere in \rga's rules.
There is no formula/term stratification or static typing
in the Russell tradition~\cite{whitehead11principia},
no positivity or predicativity restriction,
and no quantifier-complexity class distinguished
by any rule.

Free and bound variables obey a
\emph{variable-range invariant}
that runs through \rga's design.
Bound (quantified) variables range over
the natural numbers only,
enforced by explicit `$\vz \jnat$' hypotheses
in the rules and semantics below.
Free variables, in contrast, carry no such obligation:
they stand for arbitrary computations --
any well-formed term,
possibly divergent and possibly a function.
At the base type their values form
the flat Scott domain $\natbot$,
the naturals lifted to include divergence ($\bot$),
and the soundness theorem's assignments range over $\natbot$.
Substituting an arbitrary well-formed term for a free variable
with no \emph{habeas quid} obligation on the substituend
is sound because free variables may stand for
any closed well-formed term, $\lambda$-terms included
(\cref{thm:sound}, part~2, and \cref{sec:metatheory:sound}).

\subsection{Judgments and Proofs}
\label{sec:system:proof}

The proof system derives \emph{judgments}
of the form $\Gamma \vdash c$,
where the conclusion $c$ is a term (a formula)
and the hypothesis context $\Gamma$ is
a finite set of terms.
Because hypotheses are sets, not lists,
the structural operations of contraction and exchange are invisible.
Because the semantics must quantify over proofs,
\rga fixes a concrete, codable proof format.
A \emph{proof} is a finite list of judgments,
written in reverse order of derivation:
its first entry is its conclusion,
and each entry follows from later entries
by a single rule of \cref{fig:rules,fig:rules-reflect}.
$P \proves \Gamma \vdash c$ means that
$P$ is a valid proof whose conclusion is $\Gamma \vdash c$,
and $\Gamma \vdash c$ is \emph{derivable}
when some valid proof concludes it.
Proofs, like all the syntactic objects of this paper,
are \gdl-codable,
and the semantics of \cref{sec:system:sem:all}
quantifies over exactly these codes.

\begin{figure}
\centering
\small
\textit{structural rules}\\[0.5ex]
$\infrule[hyp]{\wf(c)}{\Gamma, c \vdash c}$
\qquad
$\infrule[wk]{\Gamma \vdash c}{\Gamma, a \vdash c}$
\qquad
$\infrule[cut]{\Gamma \vdash a \quad \Gamma, a \vdash c}
  {\Gamma \vdash c}$
\qquad
$\infrule[sub]{\Gamma \vdash c \quad \wf(\Gamma,c,s)}
  {[v{\mapsto}s]\Gamma \vdash [v{\mapsto}s]c}$
\medskip

\textit{truth values are numerals}\\[0.5ex]
$\infeqv[T{=}]{\Gamma \vdash a}{\Gamma \vdash a = \tofn{1}}$
\qquad
$\infeqv[F{=}]{\Gamma \vdash \neg a}{\Gamma \vdash a = \tofn{0}}$
\medskip

\textit{negation}\\[0.5ex]
$\infrule[{\neg}E]{\Gamma \vdash \neg a \quad \Gamma \vdash a
    \quad \wf(c)}
  {\Gamma \vdash c}$
\qquad
$\infeqv[{\neg}{\neg}]{\Gamma \vdash a}{\Gamma \vdash \neg\neg a}$
\medskip

\textit{disjunction}\\[0.5ex]
$\infrule[{\lor}I_1]{\Gamma \vdash a \quad \wf(b)}
  {\Gamma \vdash a \lor b}$
\qquad
$\infrule[{\lor}I_2]{\Gamma \vdash b \quad \wf(a)}
  {\Gamma \vdash a \lor b}$
\qquad
$\infrule[{\lor}I_3]{\Gamma \vdash \neg a \quad \Gamma \vdash \neg b}
  {\Gamma \vdash \neg(a \lor b)}$
\\[1ex]
$\infrule[{\lor}E_1]{\Gamma \vdash a \lor b \quad
    \Gamma, a \vdash c \quad \Gamma, b \vdash c}
  {\Gamma \vdash c}$
\qquad
$\infrule[{\lor}E_2]{\Gamma \vdash \neg(a \lor b)}
  {\Gamma \vdash \neg a}$
\qquad
$\infrule[{\lor}E_3]{\Gamma \vdash \neg(a \lor b)}
  {\Gamma \vdash \neg b}$
\medskip

\textit{grounded equality}\\[0.5ex]
$\infeqv[{=}refl]{\Gamma \vdash a \jnat}{\Gamma \vdash a = a}$~{\scriptsize (notational)}
\qquad
$\infrule[{=}sym]{\Gamma \vdash a = b}{\Gamma \vdash b = a}$
\qquad
$\infrule[{=}subs]{\Gamma \vdash a = b \quad
    \Gamma \vdash [v{\mapsto}a]p \quad \wf(p)}
  {\Gamma \vdash [v{\mapsto}b]p}$
\\[1ex]
$\infrule[{\ne}sym]{\Gamma \vdash a \ne b}{\Gamma \vdash b \ne a}$
\qquad
$\infrule[{=}TI]{\Gamma \vdash a \jnat \quad \Gamma \vdash b \jnat}
  {\Gamma \vdash (a = b) \jbool}$
\qquad
$\infrule[{=}TE]{\Gamma \vdash (a = b) \jbool}
  {\Gamma \vdash (a \jnat) \land (b \jnat)}$
\medskip

\textit{natural numbers and induction}\\[0.5ex]
$\infrule[\z{}I]{}{\Gamma \vdash \z \jnat}$
\qquad
$\infrule[\suc{}TI]{\Gamma \vdash a \jnat}
  {\Gamma \vdash \suc a \jnat}$
\qquad
$\infrule[\suc{}TE]{\Gamma \vdash \suc a \jnat}
  {\Gamma \vdash a \jnat}$
\qquad
$\infrule[\suc{}{\ne}\z]{\Gamma \vdash a \jnat}
  {\Gamma \vdash \suc a \ne \z}$
\\[1ex]
$\infrule[\suc{}{\ne}I]{\Gamma \vdash a \ne b}
  {\Gamma \vdash \suc a \ne \suc b}$
\qquad
$\infrule[\suc{}{=}E]{\Gamma \vdash \suc a = \suc b}
  {\Gamma \vdash a = b}$
\qquad
$\infrule[\suc{}{\ne}E]{\Gamma \vdash \suc a \ne \suc b}
  {\Gamma \vdash a \ne b}$
\qquad
$\infrule[\bot{}TE]{\Gamma \vdash \bot \jnat \quad \wf(p)}
  {\Gamma \vdash p}$
\\[1ex]
$\infrule[ind]{\Gamma \vdash [i{\mapsto}\z]c \quad
    v_i \jnat,\, c,\, \Gamma \vdash [i{\mapsto}\suc v_i]c \quad
    \Gamma \vdash a \jnat \quad
    \wf(c) \quad i \notin \fv(\Gamma)}
  {\Gamma \vdash [i{\mapsto}a]c}$
\caption{The \rga proof rules, part one:
  the structural and logical trunk shared with \pga.
  A double line means the rule holds in both directions
  (each direction is one primitive rule);
  $\wf(\cdot)$ marks well-formedness side conditions,
  which attach to terms not already
  occurring in a premise judgment,
  except that \irl{sub} and \irl{\cR\suc{}{=}} also require
  well-formedness of terms in their premises.
  The two directions of \irl{{=}refl} are identities,
  because `$a \jnat$' abbreviates `$a = a$'.}
\label{fig:rules}
\end{figure}

\begin{figure}
\centering
\small
\textit{syntactic head steps}\\[0.5ex]
$\infrule[\beta]{\wf(t)}{(\lambda b) \cdot t \hs b[t]}$
\qquad
$\infrule[\cR\z]{\wf(h)}{\cR \cdot g \cdot h \cdot \z \hs g}$
\qquad
$\infrule[\cR\suc]{\wf(g)}
  {\cR \cdot g \cdot h \cdot (\suc \tofn{n}) \hs
   h \cdot \tofn{n} \cdot (\cR \cdot g \cdot h \cdot \tofn{n})}$
\qquad
$\infrule[spine]{t \hs t'}{t \cdot u \hs t' \cdot u}$
\medskip

\textit{conversion along head steps}\\[0.5ex]
$\infrule[conv]{\Gamma \vdash c' \quad c \hs c'}
  {\Gamma \vdash c}$
\qquad
$\infrule[convE]{\Gamma \vdash c \quad c \hs c'}
  {\Gamma \vdash c'}$
\qquad
$\infrule[{\neg}convE]{\Gamma \vdash \neg c \quad c \hs c'}
  {\Gamma \vdash \neg c'}$
\qquad
$\infrule[conv{=}]{\Gamma \vdash c' \jnat \quad c \hs c'}
  {\Gamma \vdash c = c'}$
\medskip

\textit{schematic recursor unfolding}\\[0.5ex]
$\infrule[\cR\suc{}{=}]{\Gamma \vdash m \jnat \quad
    \Gamma \vdash (h \cdot m \cdot (\cR \cdot g \cdot h \cdot m)) \jnat
    \quad \wf(g)}
  {\Gamma \vdash \cR \cdot g \cdot h \cdot (\suc m) =
    h \cdot m \cdot (\cR \cdot g \cdot h \cdot m)}$
\medskip

\textit{the reflective quantifier}\\[0.5ex]
$\infrule[{\cA}I]{\vz \jnat,\, \lft\Gamma \vdash \lft f \cdot \vz}
  {\Gamma \vdash \cA \cdot f}$
\qquad
$\infrule[{\cA}E]{\Gamma \vdash \cA \cdot f \quad \Gamma \vdash a \jnat}
  {\Gamma \vdash f \cdot a}$
\\[1ex]
$\infrule[{\neg\cA}I]{\Gamma \vdash \neg(f \cdot a) \quad
    \Gamma \vdash a \jnat}
  {\Gamma \vdash \neg(\cA \cdot f)}$
\qquad
$\infrule[{\neg\cA}E]{\Gamma \vdash \neg(\cA \cdot f) \quad
    \vz \jnat,\, \neg(\lft f \cdot \vz),\, \lft\Gamma \vdash \lft c}
  {\Gamma \vdash c}$
\caption{The \rga proof rules, part two:
  computation and the universal quantifier.
  The head-step relation $t \hs t'$ is a
  primitive-recursively decidable syntactic relation,
  not a judgment;
  $b[t]$ is the $\beta$-instance substituting $t$ for
  the outermost bound variable of $b$,
  $\tofn{n}$ is a syntactic numeral,
  and $\lft{(\cdot)}$ shifts all free de Bruijn indices up by one,
  so that $\vz$ is fresh for $\lft\Gamma$, $\lft f$, and $\lft c$.}
\label{fig:rules-reflect}
\end{figure}

\Cref{fig:rules,fig:rules-reflect} show
\rga's complete set of 42 primitive rules.
Two of them, the directions of \irl{{=}refl}, are identities,
because `$a \jnat$' abbreviates `$a = a$'.
The count is small because the derived operators
do real work.
All rules for
$\land$, $\limp$, $\liff$, $\ne$, $\jnat$, $\jbool$,
transitivity of equality,
and the familiar classical-style connective algebra
are derived by unfolding the abbreviations
into the primitive connectives.
The rules for the derived operators are
mechanically checked as lemmas over the 42 primitives,
at no fundamental cost to the metatheory.

A few conventions apply throughout.
Side conditions $\wf(t)$ appear
at least wherever a rule introduces a term
not already present in some premise judgment.
Derivable judgments therefore always have
well-formed conclusions,
while hypotheses are unconstrained.
An ill-formed or ungrounded hypothesis is harmless:
a hypothesis can only be used
via rules whose own side conditions
limit what the hypothesis can conclude.
Substitution $[v{\mapsto}s]$ replaces
a free variable by an arbitrary well-formed term $s$ --
including in open terms, $\lambda$-abstractions,
and diverging terms --
with no groundedness premise;
\cref{sec:metatheory:sound} shows that the rule is sound.

Several rule groups deserve a closer look.

\paragraph{Grounded equality and \emph{habeas quid}.}
The notation `$a \jnat$' for `$a = a$' makes
the \emph{habeas quid} discipline concrete.
Reflexivity $a = a$ is not free as in classical arithmetic,
but holds exactly when $a$ denotes a natural number.
Grounded equality is thereby
an equality of values, not descriptions.
Diverging terms are not equal to anything,
and the truth of an equality implies
that both sides terminate (\cref{lem:app:eqD}).
Substitutivity (\irl{{=}subs}) is nonetheless
full congruence:
the context $p$ is an arbitrary well-formed term,
including positions under $\lambda$
and in unevaluated arguments.

\paragraph{Truth values and the logic.}
The \irl{T{=}}/\irl{F{=}} pair identifies
truth of a formula with the formula's
evaluation to a boolean numeral.
The pair thereby makes ``$a$ is provable'' equivalent to
``$a$ provably has value $\tofn{1}$''.
Negation is the boolean flip,
so double-negation elimination (\irl{{\neg}{\neg}})
is simply sound --
a first clear hint that \rga is not intuitionistic.
The disjunction rules give
the proof-level counterpart of
Kleene's strong disjunction.
\rga needs explicit negative-case (``not-or'')
introduction and elimination rules,
because its negation is the boolean flip
rather than implication of a falsum.
No rule introduces a negation by refuting a hypothesis,
as intuitionistic negation introduction does,
so $\neg(a \lor b)$ must be introduced and eliminated directly.

\paragraph{Induction without formula classes.}
The induction rule \irl{ind} places
no restriction on the motive $c$
beyond well-formedness:
the motive may contain quantifiers, recursion,
even ungrounded subformulas.
Fragments of classical arithmetic restrict induction by formula complexity
to calibrate strength; \rga's rule has no such restriction.
In \rga's semantics the arithmetical hierarchy collapses:
the grounded truth of every formula,
whatever its quantifiers, is $\Sigma_1$ (\cref{thm:truth-re}),
and one induction rule serves the whole language.
The premise `$a \jnat$' supplies
the number to induct up to:
induction, like everything else in \rga,
is grounded in an actual value.

\paragraph{Computation in the proof system.}
The head-step relation $\hs$
(\cref{fig:rules-reflect}, top)
is a deliberately restricted head reduction:
call-by-name $\beta$, which substitutes the argument unevaluated;
recursor dispatch at syntactic numerals only,
never reducing the scrutinee;
and reduction in function position of an application spine.
The conversion rules carry provability
along a head step in either direction,
\irl{{\neg}convE} carries \emph{refutation}
forward along a step,
and \irl{conv{=}} records a step as
a provable equation when the reduct has a value.
The head-step relation is deliberately syntactic and
primitive-recursively decidable.
Proof checking must remain effective,
so proofs convert along checkable steps,
and all semantic content stays in the premises.
The schematic rule \irl{\cR\suc{}{=}} complements
the syntactic steps.
The rule unfolds the recursor at $\suc m$ for an
arbitrary term $m$ proved to be a natural --
the form needed under induction,
where no syntactic numeral is in hand
(\cref{sec:power:arith}).
The \emph{habeas quid} profile of call-by-name shows
in the side conditions.
The $\beta$-rule demands no termination of the argument,
since a function may discard its argument unevaluated.
Recursor unfolding at $\suc m$, by contrast, demands `$m \jnat$',
because the recursor scrutinizes the value.
And \irl{conv{=}} demands a value for the reduct,
because grounded equality needs one.
Refutation transport (\irl{{\neg}convE}) needs no side condition
(\cref{sec:app:complete:prelim}).

\paragraph{The reflective quantifier.}
The four rules at the bottom of \cref{fig:rules-reflect}
are the heart of \rga.
Universal introduction (\irl{{\cA}I}) is \emph{schematic}.
To prove $\cA \cdot f$,
one must first prove the single open instance $\lft f \cdot \vz$
for the fresh variable $\vz$.
That variable is assumed only to be some natural number.
Universal elimination (\irl{{\cA}E}) is
plain function application:
from $\cA \cdot f$ and a witness `$a \jnat$',
conclude $f \cdot a$ --
no substitution, no lifting.
Instantiating a quantifier is
just applying a function,
and the $\beta$-step that unfolds
$f \cdot a$ into a formula instance
is ordinary conversion.
The two negative rules mirror the semantics of
refuted universals.
$\neg(\cA \cdot f)$ is introduced from
an explicit counterexample (\irl{{\neg\cA}I}).
That formula is eliminated (\irl{{\neg\cA}E}) by
reasoning under a fresh assumed counterexample:
a hypothetical witness $\vz$ with
$\neg(\lft f \cdot \vz)$.
The elimination is sound precisely because,
as the semantics will guarantee,
a false universal always carries a witness.

\subsection{Operational Semantics and Grounded Truth}
\label{sec:system:sem}

\begin{figure}
\centering
\small
\textit{head shapes} \quad
$w \ldef \lambda b
  \mid \cR \mid \cR\,g \mid \cR\,g\,h
  \mid \cA$
\medskip

\textit{value evaluation (selected clauses)}\\[0.5ex]
$\infrule[]{}{\z \bev{s+1} 0}$
\qquad
$\infrule[]{a \bev{s} n}{\suc a \bev{s+1} n+1}$
\qquad
$\infrule[]{a \bev{s} m \quad b \bev{s} n}
  {(a = b) \bev{s+1} \tv{m = n}}$
\qquad
$\infrule[]{a \bev{s} \tvF}{\neg a \bev{s+1} \tvT}$
\\[1ex]
$\infrule[]{a \bev{s} \tvT}{(a \lor b) \bev{s+1} \tvT}$
\qquad
$\infrule[]{b \bev{s} \tvT}{(a \lor b) \bev{s+1} \tvT}$
\qquad
$\infrule[]{a \bev{s} \tvF \quad b \bev{s} \tvF}
  {(a \lor b) \bev{s+1} \tvF}$
\\[1ex]
$\infrule[]{t \hev{s} \lambda b \quad b[u] \bev{s} r}
  {(t \cdot u) \bev{s+1} r}$
\qquad
$\infrule[]{t \hev{s} \cR\,g\,h \quad m \bev{s} 0 \quad g \bev{s} r}
  {(t \cdot m) \bev{s+1} r}$
\\[1ex]
$\infrule[]{t \hev{s} \cR\,g\,h \quad m \bev{s} k+1 \quad
    (h \cdot \tofn{k} \cdot (\cR \cdot g \cdot h \cdot \tofn{k}))
      \bev{s} r}
  {(t \cdot m) \bev{s+1} r}$
\medskip

\textit{the reflective clauses}\\[0.5ex]
$\infrule[allT]{t \hev{s} \cA \quad
    \quo{P} \le s \quad
    P \proves \{\vz \jnat\} \vdash \lft u \cdot \vz}
  {(t \cdot u) \bev{s+1} \tvT}$
\\[1ex]
$\infrule[allF]{t \hev{s} \cA \quad
    \quo{P} \le s \quad
    P \proves \emptyset \vdash \neg(u \cdot \tofn{i})}
  {(t \cdot u) \bev{s+1} \tvF}$
\caption{The \rga operational semantics (selected clauses).
  $t \bev{s} r$: closed term $t$ evaluates to natural number $r$
  within step index $s$;
  $t \hev{s} w$: $t$ head-evaluates to head shape $w$.
  $\tvT = 1$ and $\tvF = 0$;
  $\quo{P}$ is the \gdl code of the proof $P$.
  Omitted: the head-evaluation clauses that build head shapes
  (mirroring the application clauses shown), which let the
  recursor's selected branch itself reduce to a function --
  higher-order recursion.}
\label{fig:sem}
\end{figure}

Semantic evaluation in \rga is defined on closed terms.
An open formula is evaluated by first
\emph{closing} the formula under an assignment
$A : \mathrm{Var} \to \natbot$.
The closing substitution replaces a variable assigned
natural number $n$ by the numeral $\tofn{n}$,
and a variable assigned $\bot$ by the divergent term $\bot$.
Assignments thus give the base-type reading
of the variable-range invariant:
a free variable stands for a natural number
or for a divergent computation.

Two mutually inductive, step-indexed judgments
define evaluation (\cref{fig:sem}).
\emph{Head evaluation} $t \hev{s} w$ reduces a term
to one of five weak-head shapes:
a $\lambda$-abstraction,
the recursor applied to at most two arguments (three shapes),
or the bare quantifier combinator $\cA$.
Functions are never values, only head shapes:
the observable value domain is exactly
the natural numbers,
unchanged from \bga.
\emph{Value evaluation} $t \bev{s} r$ evaluates
a term to a natural number.
The arithmetic and connective clauses are
those of \bga:
equality evaluates both sides to naturals and compares,
negation flips genuine booleans,
and disjunction is Kleene-strong.
The application clauses dispatch on
the head shape of the function position,
substituting call-by-name $\beta$-arguments unevaluated.
Divergence is simply the absence of a derivation:
there is no clause for $\bot$,
no clause for a free variable
(closed terms have none),
and no explicit ``error'' value.

\subsubsection{The Reflective Quantifier Clauses}
\label{sec:system:sem:all}

The two clauses at the bottom of \cref{fig:sem}
tie the reflective knot.
A universal formula $\cA \cdot u$
evaluates to true when
some valid \rga proof $P$
derives the \emph{schematic instance}
$\{\vz \jnat\} \vdash \lft u \cdot \vz$.
That is the same judgment shape that
universal introduction \irl{{\cA}I} asks for.
The certificate $P$ is a proof in
exactly the system of \cref{sec:system:proof},
\gdl-coded and bounded by the step index.
The universal evaluates to false when
some proof refutes a particular numeral instance
$\emptyset \vdash \neg(u \cdot \tofn{i})$:
falsity of a universal is always
counterexample-certified.
Both clauses mention derivability only positively:
a certificate gives a verdict,
and the absence of one gives none.
Statements reachable by neither clause --
at any step index --
are simply ungrounded:
such statements have no truth value,
in the Kripkean sense of \cref{sec:bg:kripke}.

The step index plays the role of
Kripke's construction stages,
with one difference:
at index $s+1$ a quantifier clause consults proofs
of a fixed proof relation with codes at most $s$,
not the truth set of stage $s$:
\rga's language has no primitive truth predicate
for a stage to consult.
(The companion paper~\cite{ford-conrga} \emph{defines}
a truth predicate inside \rga from this semantics;
\cref{sec:related:truth} discusses in what sense it is one.)
Evaluation is monotone in the index:
a value once obtained is never taken back.
Ever-larger indices admit ever-larger proofs,
so the quantifier searches, in the limit,
the system's entire proof enumeration:
grounded truth of a universal is
schematic provability,
and grounded falsity is
refutability of an instance.
A quantifier search that finds neither remains valueless ($\bot$).
The step index is the index of step-indexed semantics
and of clocked interpreters%
~\cite{appel01indexed,dreyer11logical,owens16functional,amin17type},
and it is what the companion paper calls \emph{fuel}.
Each fuel slice of evaluation is decided by
a primitive-recursive function,
sound at each fuel and complete up to a change of fuel
(theory \texttt{RGA\_Decide});
the companion paper compiles that decider into \rga~\cite{ford-conrga}.

\emph{Grounded truth} then has
a single definition.
It is truth as establishment:
a statement is grounded-true when a finite evaluation establishes it,
with a proof certificate at each universal.
Formally,
a formula $t$ is \emph{grounded-true} under assignment $A$
when the closure of $t$ under $A$
evaluates to $\tvT$ at some step index,
and \emph{grounded-false} when the closure
evaluates to $\tvF$;
a closed formula needs no assignment.
The clauses of \cref{fig:sem} also apply, unchanged,
to open terms:
a free variable has no clause, so it is stuck, as $\bot$ is.
An open formula is \emph{grounded-true} when it evaluates to $\tvT$
in this sense.
It is then grounded-true under every assignment,
and under every closing substitution of closed well-formed terms
(\cref{sec:app:sound:capstone});
the converse is not proved.
Grounded truth is a constructive notion of truth
in a narrow sense, truth only through a construction
(\cref{sec:related:bhk}),
but \rga's logic is not intuitionistic logic
(\cref{sec:character}).
A formula $c$ is a \emph{grounded consequence} of $\Gamma$,
written $\agtrue{\Gamma}{c}$,
when every assignment under which all hypotheses in $\Gamma$
are grounded-true makes $c$ grounded-true.
\Cref{sec:metatheory} proves that every derivable judgment
is a grounded consequence,
and that every grounded-true well-formed formula
is derivable from no hypotheses.
Under hypotheses, completeness fails (\cref{thm:weak-strong}).

Note what the quantifier semantics does not do.
The clauses do not consult
the truth of all numeral instances --
that would be the $\omega$-rule,
and \cref{sec:omega} proves that
the reflective quantifier is strictly weaker
from a classical perspective
(\eg from the HOL metalogic).
The quantifier is a computation like any other in \gd:
one whose evaluation happens to run
a proof search over the very system
being defined.

%% file: metatheory.tex
\section{Core Metatheory}
\label{sec:metatheory}

This section establishes the core results about \rga:
completeness, soundness, determinism of evaluation,
consistency,
and the resulting exact coincidence of
provability with grounded truth.
In a conventional logic these would be
independent developments.
In \rga the reflective quantifier ties them together:
the semantics contains proof search,
so facts about evaluation rest on
facts about derivability, and conversely.
The metatheory consequently unfolds
in an order that may surprise a classical reader.
Completeness comes first and is self-contained,
soundness consumes completeness,
and determinism of evaluation consumes soundness.
Each theorem below names the corresponding
machine-checked result in the Isabelle/HOL development
(\cref{sec:formal}).
The coincidence also yields
the disjunction and witness properties
(\cref{sec:metatheory:iff}),
and the section closes with the formalization's witnessed semantics
(\cref{sec:metatheory:arch}).
\Cref{sec:app:complete,sec:app:sound} give the completeness and soundness
proofs in full.

\subsection{Completeness}
\label{sec:metatheory:complete}

\begin{thm}[value completeness; \textsf{rga\_complete\_val}]
\label{thm:complete-val}
If $a$ is well-formed and $a \bev{s} r$,
then $\emptyset \vdash a = \tofn{r}$.
\end{thm}

\begin{cor}[open completeness; \textsf{rga\_complete\_open}]
\label{cor:complete-open}
If $c$ is well-formed and $c \bev{s} \tvT$,
then $\emptyset \vdash c$.
\end{cor}

\begin{proof}[Proof sketch]
By mutual induction on the evaluation derivation,
with a companion invariant for head evaluation.
Each value clause maps to
the introduction rules of its connective,
and the truth-value bridge rules
(\irl{T{=}}/\irl{F{=}})
convert between bare truth and boolean values.
The conversion rules track head-evaluation steps.
The proof system converts along the head steps that evaluation takes,
with one transport at recursor dispatch,
where evaluation uses the scrutinee's value
(\cref{sec:app:complete:transport}).
This is why the conversion rules are primitive.
The quantifier cases are where reflection pays.
The \irl{allT} clause has
a valid proof of the schematic instance as its premise.
This is exactly the premise
universal introduction \irl{{\cA}I} needs.
The completeness proof therefore just reads
the certificate off the evaluation derivation.
The \irl{allF} case works the same way.
That case hands the refutation certificate
and the numeral witness
to \irl{{\neg\cA}I}.
Completeness at the quantifier,
classically the point where an $\omega$-rule
would be unavoidable,
is here true by construction:
grounded quantifier truth in \rga was
\emph{defined} as provability.
Every other case is a genuine induction step,
and \cref{sec:metatheory:iff} explains
why this does not make completeness definitional.
\Cref{sec:app:complete} gives the full proof.
\end{proof}

The theorem needs no prior consistency, soundness, or determinism:
completeness is a self-contained induction,
which the remaining metatheory builds on.
It holds at open terms,
with free variables stuck (\cref{sec:system:sem}).
This is completeness in the weak sense,
for derivations from no hypotheses;
under hypotheses it fails (\cref{thm:weak-strong}).

\subsection{Soundness}
\label{sec:metatheory:sound}

\begin{thm}[soundness; \textsf{rga\_sound}, \textsf{rga\_sound\_sjt}]
\label{thm:sound}
Suppose $\Gamma \vdash c$.
\begin{enumerate}
\item For every assignment $A$:
  if every hypothesis in $\Gamma$ is grounded-true under $A$,
  then $c$ is grounded-true under $A$.
\item If every hypothesis in $\Gamma$ is well-formed,
  then for every substitution $\sigma$ of closed well-formed terms
  for the free variables:
  if every $\sof{\sigma}{h}$ with $h \in \Gamma$ is grounded-true,
  then $\sof{\sigma}{c}$ is grounded-true.
\end{enumerate}
\end{thm}

Soundness holds in every context:
an entailment $\Gamma \vdash c$ preserves grounded truth
from the hypotheses to the conclusion,
under every assignment.
The proof is a rule induction over the derivation,
with two twists forced by \rga's design;
\cref{sec:app:sound} gives it in full.

The first twist concerns the induction motive.
\Cref{thm:sound}'s own statement is not inductive:
the rule \irl{sub} can substitute a $\lambda$-term
for a free variable in function position,
which no assignment can simulate,
and so turn a vacuously true judgment into a false one
(a machine-checked counterexample, \cref{sec:app:sound:fails}).
The proof instead quantifies over closing substitutions
of arbitrary closed well-formed terms, $\lambda$-terms included,
and obtains part~(1) of \cref{thm:sound}
as the special case of value terms
(\cref{sec:app:sound:repair}).

The second twist is the reflective one.
In the \irl{{\cA}I} case,
the semantics demands an actual certificate:
to make $\cA \cdot f$ grounded-true,
some valid proof of the schematic instance
$\{\vz \jnat\} \vdash \lft f \cdot \vz$ must exist,
with no residual hypotheses.
The induction hypothesis supplies only
semantic truth of the premise judgment
$\vz \jnat, \lft\Gamma \vdash \lft f \cdot \vz$.
Completeness closes the gap.
Under any closing substitution satisfying $\Gamma$,
each hypothesis is grounded-true,
hence derivable by \cref{cor:complete-open}.
Cut then discharges the hypotheses
to produce the certificate.
Soundness of a single rule thus consumes
completeness of the whole system.
This reasoning path is well-founded
only because completeness came first,
proved independently of soundness.

\subsection{Determinism and Consistency}
\label{sec:metatheory:det}

\begin{thm}[determinism;
  \textsf{osr\_unique}, \textsf{osr\_hev\_excl}]
\label{thm:det}
Evaluation is deterministic:
if $t \bev{s} r$ and $t \bev{s'} r'$ then $r = r'$,
and if $t \hev{s} w$ and $t \hev{s'} w'$ then $w = w'$.
Moreover the two judgments are exclusive:
no term both head-evaluates and value-evaluates.
\end{thm}

\begin{proof}[Proof sketch]
By mutual rule induction.
Every structural case is routine.
The one interesting collision is
a universal $\cA \cdot u$ carrying
both a truth certificate and a falsity certificate:
a valid proof of the schematic instance
$\{\vz \jnat\} \vdash \lft u \cdot \vz$
and a valid refutation
$\emptyset \vdash \neg(u \cdot \tofn{i})$.
Apply soundness (\cref{thm:sound}) to each certificate separately,
under the closing substitution that sends every variable
to the numeral $\tofn{i}$.
That substitution makes the lift invisible,
so both certificates speak about the same closed instance
of $u \cdot \tofn{i}$,
which then evaluates to both $\tvT$ and $\tvF$;
determinism on closed terms,
which the formalization proves first, on the witnessed relation
(\cref{sec:metatheory:arch}),
excludes that.
\Cref{sec:app:sound:collision} gives the details.
\end{proof}

Determinism of evaluation is thus
not a syntactic fact:
the semantics embeds a proof system,
and uniqueness of evaluation results rests on
the soundness of that proof system.
A reader accustomed to
proving determinism of a small-step relation
by inspection of the rules
may wish to pause here.
In \rga, the claim that ``the evaluator never disagrees with itself''
is a theorem \emph{about proofs}.

\begin{thm}[consistency; \textsf{rga\_consistent}]
\label{thm:consistent}
No statement is both provable and refutable:
there is no $c$ with
$\emptyset \vdash c$ and $\emptyset \vdash \neg c$.
\end{thm}

\begin{proof}[Proof sketch]
Soundness makes $c$ evaluate to both
$\tvT$ and (via the negation clause) $\tvF$
under any assignment,
contradicting determinism.
\end{proof}

The statement carries no well-formedness or
closedness hypothesis and quantifies over
all statements at once.
In particular \rga does not derive the false equation $\z = \tofn{1}$.

\begin{cor}[\textsf{rga\_no\_proof\_0eq1}]
\label{cor:no-0eq1}
$\emptyset \nvdash \z = \tofn{1}$.
\end{cor}

\Cref{thm:consistent} is consistency as a fact proved in HOL
about \rga.
Consistency can also be written as an \rga sentence,
which says that no number codes an \rga proof of $\z = \tofn{1}$.
\rga does not derive that sentence;
this is \gdl's second theorem for \rga,
proved in the companion paper~\cite{ford-conrga}.

The argument of \cref{thm:consistent} also works under hypotheses
that can all be made grounded-true at once.

\begin{thm}[\textsf{rga\_consistent\_sat}]
\label{thm:sat-consistent}
If some assignment makes every hypothesis in $\Gamma$ grounded-true,
then $\Gamma \vdash c$ and $\Gamma \vdash \neg c$
are not both derivable.
\end{thm}

Soundness would make $c$ evaluate to both values
under that assignment.

\subsection{Grounded Truth Is Provability}
\label{sec:metatheory:iff}

The pieces now fit together into
the paper's central characterization.

\begin{thm}[the truth--provability coincidence;
  \textsf{rga\_open\_iff}, \textsf{rga\_val\_iff}]
\label{thm:iff}
For well-formed $c$, open or closed:
\begin{enumerate}
\item $\emptyset \vdash c$
  \;iff\; $c \bev{s} \tvT$ for some $s$;
\item $\emptyset \vdash c = \tofn{r}$
  \;iff\; $c \bev{s} r$ for some $s$.
\end{enumerate}
\end{thm}

Provability is grounded truth;
provable equations are evaluations.
\rga is therefore an effective,
consistent proof system that is
\emph{complete for its own semantics},
in the sense of completeness for grounded truth.
Both directions of \cref{thm:iff} hold at
open terms, with free variables stuck,
so the coincidence extends to
schematic reasoning --
the form that quantifier introduction consumes.

\paragraph{The \gdl and consistency sentences.}
\rga's consistency sentence says that no number codes
an \rga proof of $\z = \tofn 1$,
and its \gdl sentence says of itself that no number codes
an \rga proof of it~\cite{ford-conrga}.
\gdl's theorems still apply to \rga's provability:
\rga neither proves nor refutes either sentence~\cite{ford-conrga}.
There is no conflict with completeness,
because neither sentence is grounded-true.
The classical argument concludes that the \gdl sentence is true
because each of its numeral instances is true.
Grounded truth makes a universal true only when a proof certifies
its schematic instance,
so the \gdl and consistency sentences,
each of whose numeral instances \rga proves,
are ungrounded rather than true.
Gaps alone would not explain this:
the $\omega$-semantics of \cref{sec:metatheory:omega} also has gaps
(\cref{prop:omv-liar}),
but it reads the universal instance by instance,
and in it both sentences are true and unprovable~\cite{ford-conrga}.
Neither sentence is paradoxical.
\rga's evaluation does not establish them,
although HOL establishes \rga's consistency from outside
(\cref{thm:consistent}).

\paragraph{Consequences.}
One further consequence states
the \emph{habeas quid} discipline as a theorem.

\begin{thm}[\tnat-soundness; \textsf{rga\_N\_sound}]
\label{thm:nsound}
If $\emptyset \vdash t \jnat$,
then $t \bev{s} k$ for some step index $s$
and natural number $k$.
\end{thm}

A provable totality claim is never empty:
behind every derivable `$t \jnat$' stands
an actual terminating computation
producing an actual natural number.
\Cref{sec:power} leans on this fact repeatedly.
As \rga proves functions such as addition, multiplication,
or Ackermann's function to be total,
\cref{thm:nsound} converts each such theorem
into a termination guarantee for
the underlying computation.

Two admissible rules on theorems,
the judgments with no hypotheses, open or closed,
make \rga's provability constructive in character.

\begin{thm}[disjunction property; \textsf{ir\_disj\_prop}]
\label{thm:disj}
If $\emptyset \vdash a \lor b$,
then $\emptyset \vdash a$ or $\emptyset \vdash b$.
\end{thm}

\begin{thm}[witness property; \textsf{rga\_not\_all\_witness}]
\label{thm:witness}
If $\emptyset \vdash \neg(\cA \cdot u)$,
then $\emptyset \vdash \neg(u \cdot \tofn{n})$
for some numeral $\tofn{n}$.
\end{thm}

\begin{proof}[Proof sketch]
Both run through soundness and completeness.
A provable disjunction is grounded-true,
and Kleene's disjunction evaluates to true
only through a true disjunct,
which \cref{thm:iff} makes provable.
A provable $\neg(\cA \cdot u)$ is
grounded-true,
so $\cA \cdot u$ evaluates to false.
Only one clause produces that value,
namely \irl{allF},
and it carries a refuted numeral instance.
The refutation certificate is
itself a derivation of
$\neg(u \cdot \tofn{n})$.
\end{proof}

Since the existential is
the dual $\exists.\,b \ldef
\neg(\cA \cdot \lambda\neg b)$,
\cref{thm:witness} says that
every provable existential is witnessed
by a provable numeral instance.
\Cref{sec:character} places both properties on the logical map.

\paragraph{Why completeness is not definitional.}
Since a universal is grounded-true exactly when \rga proves
its schematic instance,
the completeness theorem may look true by definition:
any consistent recursively enumerable theory is ``complete''
for the semantics ``true means provable''.
Two facts separate \cref{thm:iff} from that triviality.
First, grounded truth is a compositional strong Kleene evaluation,
and only the two quantifier clauses consult derivability.
At every other clause, completeness is the fact that the rules follow
evaluation, head-step conversion included,
and soundness is the fact that \rga's provability
satisfies the compositional clauses.
The degenerate semantics of a classical theory fails them:
Peano arithmetic proves the disjunction of its consistency sentence
and its negation but neither disjunct,
while \rga's provability has the disjunction and witness properties
(\cref{thm:disj,thm:witness})
and backs every totality claim with a value (\cref{thm:nsound}).
Second, soundness for a semantics that contains the proof system
is no formality: its proof consumes completeness
(\cref{sec:metatheory:sound}),
and determinism and consistency rest on it
(\cref{sec:metatheory:det}).

Grounded truth is also relative to the system.
The quantifier clauses consult \rga's proofs,
so an extension of \rga by a new rule has a grounded truth of its own:
adding the consistency sentence as an axiom,
for instance, makes its schematic instance derivable,
so the extension's quantifier clause makes the sentence true.
Grounded truth is therefore not a system-independent standard,
and ``complete for its own semantics'' is a claim about one system.
The folklore interpretation of \gdl's theorems concerns a system-independent,
instance-wise notion of truth.
For that notion, represented here by the $\omega$-semantics,
\rga is sound but incomplete (\cref{cor:omega-true}),
as the folklore says.

\subsection{The Witnessed Semantics}
\label{sec:metatheory:arch}

The formalization proves uniqueness, exclusivity and soundness
against a \emph{witnessed} variant of the evaluation pair
(judgments $\wbev{s}$ and $\whev{s}$).
It repeats \cref{fig:sem} clause for clause,
except that each quantifier clause also carries, as premises,
the semantic facts its certificate promises:
the truth of every numeral instance for \irl{allT},
the falsity of the refuted instance for \irl{allF}.
These premises are positive,
so the definition is still an ordinary inductive one,
and each case of those proofs finds the facts it needs
inside the derivation at hand,
with no induction on the step index
(\cref{sec:app:sound:witness}).

\begin{thm}[the witness collapse; \textsf{osr\_wosr}, \textsf{wosr\_osr}]
\label{thm:collapse}
For every closed term $t$, step index $s$ and value $r$:
$t \bev{s} r$ iff $t \wbev{s} r$.
The same holds for head evaluation.
\end{thm}

Dropping the witness premises maps every witnessed derivation
onto an official one;
the converse on closed terms is soundness
(\cref{sec:app:sound:collapse}).
The companion paper~\cite{ford-conrga} works with
the witnessed relation throughout,
and \cref{thm:collapse} makes it
the same semantics on closed terms.

%% file: power.tex
\section{Expressive Power}
\label{sec:power}

This section measures what \rga can express and prove.
Externally, the unary predicates \rga represents
are exactly the recursively enumerable sets
(\cref{sec:power:re}).
Internally, \rga proves substantial arithmetic
as ordinary quantified theorems:
totality of addition, of multiplication,
of each compiled primitive-recursive function,
and, beyond primitive recursion, of Ackermann's function.
Value completeness also gives, for each closed ground-type term
of \gdl's System~T, a derivation that the term denotes a natural
(\cref{sec:power:arith}).
Both halves rest on a verified compiler
from primitive-recursive function indices
into \rga terms (\cref{sec:power:compile}).
\Cref{sec:app:compile,sec:app:re,sec:app:represent,sec:app:arith}
give the proofs in full.

\subsection{A Verified Compiler from Primitive Recursion}
\label{sec:power:compile}

The computability backbone of the development is
a two-argument primitive-recursive function calculus.
This calculus follows Robinson's reduction of
general primitive recursion to
the two-argument form~\cite{robinson50essentially}.
A function index $f$ in this calculus denotes
a function $\den{f} : \mathbb{N}^2 \to \mathbb{N}$.
The same calculus underlies
the \gdl coding, the proof checker,
and the classical recursion theory
used throughout the formalization.

A compiler $\cmp$ maps each index $f$
to a closed well-formed \rga term
that contains no logical operator, no $\bot$,
and no quantifier combinator $\cA$:
its only function-forming operators are $\lambda$
and the recursor $\cR$.

\begin{thm}[compiler correctness; \textsf{cmp\_correct}]
\label{thm:cmp}
If $A \bev{s_a} a$ and $B \bev{s_b} b$
for closed $A$ and $B$,
then $(\cmp f) \cdot A \cdot B \bev{s} \den{f}(a,b)$
for some $s$.
\end{thm}

Primitive recursion maps onto $\cR$ directly:
composition becomes application,
and the recursion schema becomes one $\cR$ spine.
There is no \gdl-$\beta$ coding detour
through arithmetized sequence encodings.
The compiled term computes
call-by-name on arbitrary evaluating arguments,
so compiled functions compose freely
with all other \rga terms.
\Cref{sec:app:compile} gives the compiler and its correctness proof.

\subsection{Exactly the Recursively Enumerable Sets}
\label{sec:power:re}

Which unary predicates can \rga represent?
The formalization states recursive enumerability
computationally:
a predicate $p$ is r.e.\ when $p$ is
the termination domain of a general-recursive computation,
represented by a primitive-recursive step function.
The characterization then has two independent halves.

\begin{thm}[grounded truth is r.e.; \textsf{rga\_truth\_re}]
\label{thm:truth-re}
The set of (codes of) well-formed grounded-true formulas
is recursively enumerable.
\end{thm}

The claim covers all of grounded truth,
including quantifiers, unconstrained recursion,
and reflective certificates.
The proof runs through the truth--provability coincidence
(\cref{thm:iff}).
For well-formed formulas,
grounded truth is derivability from no hypotheses,
and derivability is a search for a proof certificate
whose validity is checkable primitive-recursively.
A single step function therefore enumerates the truth set,
by enumerating proofs rather than evaluation derivations
(\cref{sec:app:re}).
The companion paper~\cite{ford-conrga}
sharpens the theorem into a decider of evaluation
at each fuel bound,
compiled into an \rga term.

\begin{thm}[representation;
  \textsf{rga\_representable\_iff\_renp}]
\label{thm:represent}
A predicate $p$ is recursively enumerable
iff some closed well-formed term $t$ \emph{represents} $p$
(in the weak sense):
\[
  p(n) \iff \emptyset \vdash t \cdot \tofn{n}
  \quad\text{for all } n.
\]
\end{thm}

\begin{proof}[Proof sketch]
($\Leftarrow$)
Derivability is proof search.
Proofs are codable,
and proof checking is primitive-recursive.
The set $\{n \mid \emptyset \vdash t \cdot \tofn{n}\}$
is therefore the domain of a search computation.
($\Rightarrow$)
Take a step function for $p$
with index $f$,
so that $p(n)$ holds iff some step $i$ accepts.
The representing term is then
\[
  \mathit{rep}_f \;\ldef\;
  \lambda\, \neg\bigl(\cA \cdot
    \lambda\, ((\cmp f) \cdot v_1 \cdot \vz = \z)\bigr),
\]
read: ``not every step rejects $n$''.
The existential ``some step accepts''
is expressed by quantifier duality,
and \rga's semantics of refuted universals
makes the duality effective:
the universal inside $\mathit{rep}_f$ is
grounded-\emph{false} exactly when
some numeral instance is refutable.
The accepting step supplies
that counterexample,
and \cref{thm:complete-val} converts the counterexample into
a refutation certificate.
This direction transfers between
truth and provability by
the coincidence \cref{thm:iff}.
\Cref{sec:app:represent} gives both directions in full.
\end{proof}

\Cref{thm:represent} is the \rga counterpart of
the classical fact that the recursively enumerable sets
are exactly the sets weakly representable
in a consistent recursively enumerable extension of Robinson arithmetic.
\rga represents every recursively enumerable predicate
by a single application $t \cdot \tofn{n}$,
with no coding of the predicate as a formula schema.
And \rga represents nothing more:
provability, and with it
(by \cref{thm:iff}) grounded truth,
remains \re{}
The boundary has a sharp consequence
for the quantifier,
developed in \cref{sec:omega}:
from a classical metalogic perspective,
instance-wise truth of a $\Pi_1$ statement
is not \re, so the reflective universal
cannot coincide with the $\omega$-rule.

\subsection{Internal Arithmetic}
\label{sec:power:arith}

The external characterization says
which predicates \rga represents.
This subsection shows \rga
\emph{doing mathematics} --
in particular,
proving quantified totality theorems
by internal induction.
The arithmetic ingredients of
\cref{sec:intro} are exhibited here as actual objects.
\Cref{sec:app:arith} gives the derivations in full.

Addition and multiplication are
first-class $\lambda$-terms over the recursor:
\begin{align*}
\mathit{add} &\;\ldef\;
  \lambda x\, y.\;
  \cR \cdot x \cdot (\lambda k\, r.\; \suc r) \cdot y
\\
\mathit{mult} &\;\ldef\;
  \lambda x\, y.\;
  \cR \cdot \z \cdot (\lambda k\, r.\;
    \mathit{add} \cdot x \cdot r) \cdot y
\end{align*}
These examples are
shown with named variables for readability;
the development uses de Bruijn indices.

\begin{thm}[\textsf{rga\_add\_total}, \textsf{rga\_mult\_total}]
\label{thm:addmult}
$\emptyset \vdash
  \forall.\forall.\;
  (\mathit{add} \cdot v_1 \cdot \vz) \jnat$
\quad and \quad
$\emptyset \vdash
  \forall.\forall.\;
  (\mathit{mult} \cdot v_1 \cdot \vz) \jnat$.
\end{thm}

Each proof is an internal induction (\irl{ind})
on the recursion argument,
with the conversion rules computing
the $\beta$- and $\cR$-steps
inside the motive.
The multiplication proof applies the addition theorem:
the induction step needs totality of $\mathit{add} \cdot x \cdot r$
for the schematic $x$ and $r$ at hand,
and universal elimination (\irl{{\cA}E})
instantiates \cref{thm:addmult}'s first half there.
An open lemma $x \jnat, y \jnat \vdash (\mathit{add} \cdot x \cdot y) \jnat$
would serve as well, through \irl{sub} and \irl{cut}, as in \pra.
The quantifier becomes indispensable when
the induction hypothesis itself is a universal,
as in Ackermann's function below.

The two examples generalize to the whole primitive-recursive class,
uniformly in the index.

\begin{thm}[\textsf{cmp\_total}]
\label{thm:cmp-total}
For each primitive-recursive index $f$:
$\emptyset \vdash
  \forall.\forall.\;
  ((\cmp f) \cdot v_1 \cdot \vz) \jnat$.
\end{thm}

The theorem is a schema:
one \rga theorem per index,
not one \rga theorem quantified over indices.
Its recursion case uses universal introduction
to carry the result out of the induction's context
and universal elimination to apply it at the actual arguments
(\cref{sec:app:arith:indind}).

\rga's totality proofs do not
stop at primitive recursion.
Ackermann's classical
non-primitive-recursive
function~\cite{ackermann28hilbert}
is definable through
term-level self-application --
$\mathit{ack} \ldef \mathit{fix}(F)$
for a step term $F$ that
nests one recursor inside another
and recurses through
the fixpoint variable.
\rga proves the Ackermann function total.

\begin{thm}[\textsf{rga\_ack\_total}]
\label{thm:ack}
$\emptyset \vdash
  \forall.\forall.\;
  (\mathit{ack} \cdot v_1 \cdot \vz) \jnat$.
\end{thm}

The proof interlocks two internal inductions,
mirroring the textbook termination argument.
The outer induction (on the first argument)
proves the quantified statement
$\forall.\,(\mathit{ack} \cdot v_1 \cdot \vz) \jnat$.
The inner induction (on the second argument)
consumes that outer induction hypothesis --
itself a universally quantified formula --
through \irl{{\cA}E} at
the recursive call's arguments.
An induction whose \emph{hypothesis}
is a quantified theorem,
instantiated at values computed
during the induction,
is precisely the reasoning pattern
the reflective quantifier exists to support.

\gdl's System~T~\cite{godel58finiten},
the simply typed $\lambda$-calculus over the naturals
with primitive recursion at all finite types,
defines Ackermann's function and much more.
The development treats System~T as a fragment of \rga's own terms:
the simply typed terms,
with a fully applied recursor spine
$\cR \cdot g \cdot h \cdot m$ typed as \gdl's recursor
at any result type.
No separate syntax or translation is involved.

\begin{thm}[\textsf{fragment\_habeas\_quid}]
\label{thm:systemT}
If $t$ is a closed System~T term of ground type,
then $\emptyset \vdash t \jnat$.
\end{thm}

Every compiled function $\cmp f$ is itself a System~T term,
of type $\iota \to \iota \to \iota$ (\textsf{cmp\_hastype}).
The proof is Tait's reducibility argument~\cite{tait67intensional},
carried out in HOL:
its fundamental lemma gives every closed ground-type term
of the fragment a value,
and value completeness (\cref{thm:complete-val})
turns that value into a derivation.
\Cref{thm:systemT} thus transfers normalization of System~T,
proved in HOL, to \rga:
any closed well-formed term that evaluates is provably a natural.
It shows that \rga's derivations reach every closed ground-type
System~T term, one derivation per term;
it is not evidence of \rga's deductive strength,
which \cref{thm:addmult,thm:cmp-total,thm:ack} measure.
Whether \rga proves, for a System~T function $F$ of type $\iota \to \iota$,
the single quantified theorem $\forall.\,(F \cdot \vz) \jnat$
is open.
\rga's rules unfold a recursor at a non-numeral argument
only at the base type (\irl{\cR\suc{}{=}}, whose conclusion is an equation);
a recursor of function type unfolds only at syntactic numerals,
by head steps.
The proof of \cref{thm:ack} avoids recursion at function type
through the fixpoint construction.

Combined with \tnat-soundness
(\cref{thm:nsound}),
each totality theorem is
an actual termination guarantee.
From \cref{thm:ack} and \irl{{\cA}E},
every instance
$\mathit{ack} \cdot \tofn{m} \cdot \tofn{n}$
is provably a natural,
and behind each such theorem stands
a terminating evaluation
producing the value.
Nothing in these developments
steps outside the system:
the definitions use unconstrained recursion,
the proofs use internal induction,
and \emph{habeas quid} is
enforced -- and delivered -- throughout.

%% file: omega.tex
\section{Limits: \texorpdfstring{$\omega$}{omega}-Incompleteness and Weak Completeness}
\label{sec:omega}

The reflective quantifier grounds universal truth in
the system's own proof search.
A natural suspicion is that
this semantic move is the $\omega$-rule under another name,
so that ``grounded-true'' for a universal
quietly means ``all numeral instances are true.''
\Cref{thm:truth-re} already refutes
that suspicion in the abstract:
grounded truth is r.e.,
while instance-wise $\Pi_1$ truth is not.
This section makes the gap concrete,
exhibiting a particular family of statements whose
numeral instances are all provable
while their universal closure is unprovable --
and, more sharply, semantically ungrounded.
The witness is the classical halting diagonal,
compiled into \rga's own term language.
The section then presents an $\omega$-semantics
that reads the quantifier instance by instance,
for which \rga is also sound,
and shows that \rga's completeness holds
only for derivations from no hypotheses.
\Cref{sec:app:omega} gives the incompleteness proofs in full,
and \cref{sec:app:sound:omega} the $\omega$-semantics.

\subsection{The Diagonal Family}
\label{sec:omega:family}

The computability layer of the development provides
a universal interpreter as
a primitive-recursive step function.
Here $\mathit{sfApp}\,\langle x,x \rangle\,n$ executes
step $n$ of computation $x$
applied to input $x$.
The step function returns $0$ when that step does not halt
the computation.
The halting diagonal is
$K \ldef \{x \mid
  \text{computation } x \text{ halts on input } x\}$,
so that
$x \notin K$ iff
$\mathit{sfApp}\,\langle x,x \rangle\,n = 0$ for all $n$.
$K$ is r.e., and the complement of $K$ is not.
That is the classical diagonal argument,
which the development also formalizes.

Compiling the interpreter
(\cref{thm:cmp}) turns
the diagonal into \rga formulas:
\begin{align*}
u_x &\;\ldef\;
  \lambda\,\bigl(
    (\cmp\,\mathit{sfApp}) \cdot \tofn{\langle x,x \rangle}
      \cdot \vz = \z\bigr)
&
U_x &\;\ldef\; \cA \cdot u_x .
\end{align*}
The instance $u_x \cdot \tofn{n}$ says
``step $n$ does not halt computation $x$ on $x$''.
The universal $U_x$ says
``computation $x$ never halts on $x$'',
a $\Pi_1$ statement.
Each instance is a terminating
primitive-recursive computation,
so value completeness (\cref{thm:complete-val})
settles each instance:
provable when the step rejects,
refutable when the step halts.

The family behaves cleanly on
the halting side.
For $x \in K$,
some instance is refutable,
and the counterexample rule \irl{{\neg\cA}I}
makes $U_x$ provably false.
For $x \notin K$,
every numeral instance of $U_x$ is provable.
The question is whether
the universals $U_x$ themselves are provable --
and not all of them can be.

\subsection{The Theorem}
\label{sec:omega:thm}

\begin{thm}[$\omega$-incompleteness;
  \textsf{rga\_omega\_incomplete}]
\label{thm:omega}
There is a closed well-formed $u$ with
$\emptyset \vdash u \cdot \tofn{n}$
for every $n$,
while $\emptyset \nvdash \cA \cdot u$.
\end{thm}

\begin{proof}[Proof sketch]
Suppose instead that every closed well-formed $u$
with all instances provable had
a provable universal.
Then for the diagonal family,
$x \notin K$ iff $\emptyset \vdash U_x$.
Left to right holds by instance-wise completeness
plus the supposed $\omega$-completeness.
Right to left holds by soundness:
a provable $U_x$ is grounded-true,
so every instance is true,
so no step halts the computation.
But derivability is r.e.\
(the argument of \cref{thm:truth-re}, \cref{sec:app:re}),
so the equivalence would make
the complement of $K$ r.e.\ --
contradicting the diagonal argument.
\end{proof}

The proof is a counting argument
at the r.e.\ boundary,
and \cref{thm:omega} is stated existentially.
The development also locates the gap
inside the diagonal family.

\begin{cor}[\textsf{rga\_ouniv\_gap}]
\label{cor:gap}
Some $x \notin K$ has $\emptyset \nvdash U_x$.
\end{cor}

For many $x \notin K$ the universal $U_x$
is provable --
whenever the non-halting of $x$ has
a reason \rga can establish schematically.
What \cref{thm:omega} rules out is
any single sound, effective proof system catching
all the non-halting diagonals at once.
\rga, being sound and effective, is no exception.
The companion paper~\cite{ford-conrga}
gives natural witnesses outside the diagonal family.
\rga's consistency sentence,
which says that no number codes an \rga proof of $\z = \tofn{1}$,
has every numeral instance provable,
while \rga does not prove the sentence itself;
the same holds for \rga's \gdl sentence.

\subsection{The \texorpdfstring{$\omega$}{omega}-Semantics}
\label{sec:metatheory:omega}

Grounded truth is not the only semantics for which \rga is sound.
The \emph{$\omega$-semantics} reads the quantifier
instance by instance, as classical arithmetic does.
Its evaluation $t \omv r$ mirrors the evaluation pair of
\cref{fig:sem} without step indices,
and replaces the two certificate premises
by infinitary instance conditions:
$\cA \cdot u$ is \emph{$\omega$-true}
when every numeral instance $u \cdot \tofn{n}$ is $\omega$-true,
with no uniform ground demanded,
and \emph{$\omega$-false}
when some numeral instance is $\omega$-false,
with no refutation certificate demanded.
That is Kripke's $\omega$-style quantifier clause
(\cref{sec:related:truth}).
Every other clause is unchanged,
so the $\omega$-semantics still has gaps.
It is the least relation closed under these clauses,
and it is not recursively enumerable
(\cref{thm:omv-not-re,rem:omv-not-re}).
It is a partial classical valuation,
\emph{classically grounded} at the universal
in the sense of \cref{sec:bg:kripke}:
classical at the universal,
still partial elsewhere (\cref{prop:omv-liar}).
\Cref{sec:app:sound:omega} gives the definitions and proofs.

\begin{thm}[\textsf{omv\_ouniv\_iff}, \textsf{omega\_truth\_not\_rens}]
\label{thm:omv-not-re}
The universal $U_x$ of \cref{sec:omega:family} is $\omega$-true
exactly when $x \notin K$.
Hence the set of $x$ for which $U_x$ is $\omega$-true
is not recursively enumerable.
\end{thm}

Every numeral instance of $U_x$ evaluates,
so $U_x$ is $\omega$-true exactly when no step halts the computation,
and the complement of $K$ is not recursively enumerable.

\begin{rem}[not mechanized]
\label{rem:omv-not-re}
Since $x \mapsto \quo{U_x}$ is computable,
\cref{thm:omv-not-re} reduces the complement of $K$
to the set of codes of $\omega$-true sentences,
which is therefore not recursively enumerable either.
\end{rem}

\begin{prop}[\textsf{omv\_liar\_unvalued}]
\label{prop:omv-liar}
The Liar $L$ of \cref{sec:system:syntax} has no value
in the $\omega$-semantics.
\end{prop}

\begin{thm}[\textsf{omv\_unique}, \textsf{osr\_omv}]
\label{thm:omv-extends}
$\omega$-evaluation is deterministic,
and it extends evaluation on closed terms:
if $t$ is closed and $t \bev{s} r$, then $t \omv r$.
\end{thm}

\begin{thm}[$\omega$-soundness;
  \textsf{rga\_omega\_sound\_strip}, \textsf{rga\_omega\_true\_consistent}]
\label{thm:omega-sound}
If $\Gamma \vdash c$
and every well-formed hypothesis in $\Gamma$ is $\omega$-true,
then $c$ is $\omega$-true
(for an open $c$, with its free variables stuck).
Consequently no finite set of well-formed $\omega$-true hypotheses
derives $\z = \tofn{1}$.
\end{thm}

The proof is a rule induction of the same shape as
\cref{thm:sound}'s,
and the development proves it in every context,
under every well-formed closing substitution
that makes the hypotheses $\omega$-true.
\Cref{thm:omega-sound} adds something
that grounded soundness does not give:
\rga stays consistent when extended by hypotheses
that are true in the instance-wise reading of the quantifier
even where they are not grounded-true.
\rga is also $\omega$-consistent,
in the classical sense of the term.

\begin{thm}[$\omega$-consistency; \textsf{rga\_omega\_consistent}]
\label{thm:omega-consistent}
There is no $u$ such that
$\emptyset \vdash \neg(\cA \cdot u)$
and $\emptyset \vdash u \cdot \tofn{n}$ for every $n$.
\end{thm}

\begin{proof}[Proof sketch]
A refuted universal has a refuted numeral instance
(the witness property, \cref{thm:witness}),
and that instance would be both provable and refutable,
against \cref{thm:consistent}.
\end{proof}

The $\omega$-semantics strictly extends grounded truth:
\cref{sec:omega:ungrounded} exhibits closed statements
that are $\omega$-true but unprovable,
hence not grounded-true (\cref{cor:omega-true}).

\subsection{The Semantic Sharpening}
\label{sec:omega:ungrounded}

Classically, an unprovable $\Pi_1$ statement
with all instances true is
the standard example of a \emph{true but unprovable} sentence.
\rga relocates this phenomenon.

\begin{cor}[$\omega$-ungroundedness;
  \textsf{rga\_omega\_ungrounded}]
\label{cor:ungrounded}
There is a closed well-formed $u$ with
$\emptyset \vdash u \cdot \tofn{n}$ for every $n$,
while $\cA \cdot u$ evaluates to true
at no step index.
\end{cor}

The corollary is immediate from
\cref{thm:omega} and the coincidence
\cref{thm:iff}:
in \rga, unprovable is not-grounded-true.

The witness universal is not grounded-false either.

\begin{thm}[\textsf{rga\_all\_instances\_not\_false}, \textsf{rga\_omega\_witness\_ungrounded}]
\label{thm:not-false}
If $u$ is well-formed and $\emptyset \vdash u \cdot \tofn{n}$ for every $n$,
then $\cA \cdot u$ evaluates to $\tvF$ at no step index.
In particular, the $u$ of \cref{cor:ungrounded} can be chosen
so that $\cA \cdot u$ evaluates to neither $\tvT$ nor $\tvF$
at any step index.
\end{thm}

If $\cA \cdot u$ were grounded-false,
its negation would be grounded-true,
hence provable by \cref{thm:iff},
while every numeral instance is provable --
against $\omega$-consistency (\cref{thm:omega-consistent}).
So the witness universal is ungrounded outright:
the reflective search finds
neither a schematic proof
nor a counterexample,
ever.

The quantifier is thus not extensional in its body's values.
For a gap point $x$ of \cref{cor:gap},
the body $\lambda.(\z = \z)$ and the witness $u_x$
agree at every numeral, each instance being true,
yet $\cA \cdot \lambda.(\z = \z)$ is provable
and $\cA \cdot u_x$ is not grounded-true.
A universal depends on its body through \rga's proofs about the body;
evaluation respects value congruence of the body
and head-step conversion
(\cref{sec:app:sound:vcr}),
not pointwise agreement on the numerals.

In the $\omega$-semantics of \cref{sec:metatheory:omega},
by contrast, the witness universal is true.

\begin{cor}[\textsf{omv\_strictly\_extends\_grounded}]
\label{cor:omega-true}
Some closed well-formed $c$ is $\omega$-true
while $\emptyset \nvdash c$
and $c$ evaluates to $\tvT$ at no step index.
\end{cor}

Every sentence of that kind is ungrounded,
not only the witness universal.

\begin{thm}[\textsf{omega\_true\_unprovable\_ungrounded}]
\label{thm:omega-ungrounded}
Every closed well-formed term
that is $\omega$-true and unprovable
has no value at any step index.
In particular it is neither grounded-true nor grounded-false.
\end{thm}

It is not grounded-true by \cref{thm:iff}.
Any other value at a step index would be an $\omega$-value
(\cref{thm:omv-extends}),
and $\omega$-values are unique, so the term,
being $\omega$-true, has no other.

The relocation is the point.
The classical reader expects
a true-but-unprovable sentence,
and in the $\omega$-semantics,
which reads the quantifier classically,
the witness universal is exactly that.
In \rga's own semantics it is instead
an \emph{ungrounded} sentence.
Grounded truth declines to decide that sentence,
for exactly the same reason
the system's notion of proof does so.
Grounded truth and provability move together
(\cref{thm:iff}),
and what gives is
the classical conviction that
instance-wise truth must sum to
universal truth.
Summing infinitely many instance proofs
is the $\omega$-rule,
an infinitary ground.
The reflective quantifier accepts
only an effective ground --
one schematic proof,
uniform in the variable.

\begin{rem}[not mechanized]
\label{rem:no-effective}
The proof of \cref{thm:omega} uses only that
provability is r.e.\
and that a provable $U_x$ has true numeral instances.
It therefore applies to every effective proof system
that expresses the diagonal family and is sound for it:
no such system proves the universal for every $x \notin K$.
The gap is thus not an artifact of \rga's design
but the boundary of recursive enumerability itself.
\rga's contribution is to place the gap
inside the semantics,
rather than between
the semantics and the proof system.
\end{rem}

\subsection{Weak Completeness, Not Strong}
\label{sec:omega:weak-strong}

\Cref{cor:complete-open} is completeness in the weak sense:
for derivations from no hypotheses.
\Cref{thm:sound} gives soundness in the strong sense:
$\Gamma \vdash c$ implies $\agtrue{\Gamma}{c}$.
The converse, strong completeness, fails,
and \cref{thm:weak-strong} says where:
at an open hypothesis,
and at a closed hypothesis exactly when the hypothesis
is neither grounded-true nor explosive.

\begin{thm}[\textsf{rga\_strong\_complete\_fails\_open}, \textsf{rga\_strong\_complete\_closed\_iff}]
\label{thm:weak-strong}
\leavevmode
\begin{enumerate}
\item \emph{At an open hypothesis.}
  There is a closed well-formed $u$
  with $\emptyset \vdash u \cdot \tofn n$ for every $n$,
  such that $\agtrue{\{\vz \jnat\}}{u \cdot \vz}$
  but $\vz \jnat \nvdash u \cdot \vz$.
\item \emph{At a closed hypothesis that is never grounded-true.}
  For \rga's consistency sentence $\mathit{Con}$~\cite{ford-conrga},
  $\agtrue{\{\mathit{Con}\}}{\z = \tofn 1}$
  but $\mathit{Con} \nvdash \z = \tofn 1$.
\item \emph{The closed boundary.}
  For a closed hypothesis $a$:
  every closed well-formed grounded consequence $c$ of $a$
  satisfies $a \vdash c$
  if and only if $a$ is grounded-true or $a \vdash \z = \tofn 1$.
\end{enumerate}
\end{thm}

In part~1, every assignment that makes $\vz \jnat$ grounded-true
sends $\vz$ to a natural number $n$,
and $u \cdot \tofn n$ is provable, hence grounded-true.
But a derivation of $\vz \jnat \vdash u \cdot \vz$
would give $\emptyset \vdash \cA \cdot u$ by universal introduction,
against \cref{thm:omega}.
In part~2, the consistency sentence is not grounded-true~\cite{ford-conrga},
so $\z = \tofn 1$ is vacuously its grounded consequence;
\rga does not derive $\z = \tofn 1$ from it,
because the sentence is $\omega$-true~\cite{ford-conrga}
and \rga is sound for the $\omega$-semantics
(\cref{thm:omega-sound}).
In part~3, if $a$ is grounded-true,
a closed grounded consequence $c$ of $a$ is grounded-true,
and \cref{cor:complete-open} and weakening give $a \vdash c$.
If $a \vdash \z = \tofn 1$, then $a \vdash c$ for every
well-formed $c$, by \irl{{=}sym}, \irl{\suc{}{\ne}\z} and \irl{{\neg}E}.
Otherwise $a$ is never grounded-true,
so $\z = \tofn 1$ is vacuously its grounded consequence,
while $a \nvdash \z = \tofn 1$.
A hypothesis with a value other than $\tvT$ is explosive.

\begin{thm}[\textsf{value\_ne1\_explosive}]
\label{thm:value-explosive}
If $a$ is well-formed and $a \bev{s} k$ with $k \ne 1$,
then $a, \Gamma \vdash \z = \tofn 1$.
\end{thm}

\irl{T{=}} gives $a, \Gamma \vdash a = \tofn 1$,
\cref{thm:complete-val} gives $\emptyset \vdash a = \tofn k$,
and $\tofn k \ne \tofn 1$ is provable.
So strong completeness fails at a closed hypothesis
only when the hypothesis has no value and is not explosive.
The Liar is explosive as a hypothesis (\cref{thm:lem});
the consistency sentence and the \gdl sentence are not.

Grounded consequence is thus sound for derivations under hypotheses,
and complete for derivations from none,
but not in general for derivations under hypotheses.
Which semantics for derivations under hypotheses
would make \rga strongly complete is open (\cref{sec:concl}).

%% file: character.tex
\section{The Logic's Character}
\label{sec:character}

The preceding sections measured \rga's strength and limits;
this section locates \rga on the logical map.
The coordinates are unusual.
\rga's propositional rules are those of strong Kleene consequence
with only `true' designated;
whether they are complete for it is not claimed here.
Double negation is eliminable as an entailment,
so the logic is not intuitionistic.
Excluded middle fails wherever a statement has no boolean value,
and assuming it at the Liar is explosive,
so the logic is not classical.
\rga has the disjunction property,
and refuted universals yield explicit counterexamples
(\cref{sec:metatheory:iff}).
And the deduction theorem holds exactly at
provably decided hypotheses.
\rga keeps every structural rule
-- weakening, exchange, contraction, and cut --
so none of these traits comes from a missing structural rule.
Each traces back to
the same source:
truth is grounded evaluation,
and the rules commit to
nothing that evaluation does not deliver.
\Cref{sec:app:character} gives the proofs in full.
\rga's membership of the \gd family is itself machine-checked
(\cref{sec:formal:notes}).

\subsection{Between Classical and Intuitionistic}
\label{sec:character:between}

Double-negation elimination is
a primitive double rule (\irl{{\neg}{\neg}}),
so $a$ and $\neg\neg a$ entail each other in every context.
Negation flips boolean values,
so a doubly negated formula with a boolean value
has the value of the formula itself;
at a term with a non-boolean value such as $\tofn 2$,
$\neg\neg\tofn 2$ has no value,
so $a \dashv\vdash \neg\neg a$ does not make $a$ and $\neg\neg a$
intersubstitutable.
As an implication, double-negation elimination holds
only where \rga decides $a$:
$\emptyset \vdash \neg\neg a \limp a$ is the disjunction
$\neg\neg\neg a \lor a$,
so by the disjunction property (\cref{thm:disj})
it holds exactly when \rga proves $a$ or refutes it
(\cref{thm:dne-imp}).
The intuitionistic reading of negation --
a refutation is a construction,
and two refutations do not make a proof --
does not apply here.

Excluded middle fails in a stronger sense
than mere unprovability.
Recall the Liar $L \ldef \mathit{fix}(\lambda\, \neg \vz)$
of \cref{sec:system:syntax},
which converts to its own negation in two head steps.

\begin{thm}[\textsf{rga\_liar\_unprovable}, \textsf{rga\_liar\_irrefutable}]
\label{thm:liar}
$\emptyset \nvdash L$ and $\emptyset \nvdash \neg L$.
\end{thm}

$L$ has no value at any step index,
so by \cref{thm:iff}
\rga proves neither $L$ nor $\neg L$.
That does not make $L$ inert in derivations.
\rga proves $L \lor \z = \z$, for instance,
and the Liar is explosive as a hypothesis:

\begin{thm}[\textsf{liar\_hyp\_explosive}, \textsf{rga\_liar\_lem\_explosive}]
\label{thm:lem}
For every well-formed $c$:
$\;L,\, \Gamma \vdash c$
and
$\;L \lor \neg L,\, \Gamma \vdash c$.
\end{thm}

Each derivation is two conversion steps
per branch and explosion.
Assuming $L$, conversion yields $\neg L$,
and explosion fires.
Assuming $\neg L$, conversion yields $L$ again.
No \gdl coding, no diagonal lemma:
unconstrained recursion supplies
self-reference directly.
The companion paper~\cite{ford-conrga}
shows that \rga can also say this about itself:
with a truth predicate for \rga's own, grounded notion of truth
defined inside \rga (\cref{sec:related:truth}),
\rga proves that neither $L$ nor $\neg L$ is true.

Excluded middle at the Liar is explosive
in every context, not only as a hypothesis.

\begin{thm}[\textsf{rga\_liar\_lem\_explosive\_rule}]
\label{thm:lem-rule}
If $\Gamma \vdash L \lor \neg L$,
then $\Gamma \vdash c$ for every well-formed $c$.
\end{thm}

The proof cuts the derivation of $L \lor \neg L$
against the second half of \cref{thm:lem}.

\begin{rem}[not mechanized]
\label{rem:lem-scheme}
The same derivation uses only the rules
\irl{hyp}, \irl{conv}, \irl{convE}, \irl{{\neg}E}, \irl{{\lor}E_1}
and \irl{cut}.
So any extension of \rga closed under them
that derives $\Gamma \vdash a \lor \neg a$
for every well-formed $a$
derives $\Gamma \vdash \z = \tofn{1}$,
and no consistent extension validates excluded middle as a scheme.
\end{rem}

So excluded middle as a scheme is refuted,
because it would let reasoning use the explosive $L$.
Excluded middle is not, however, simply absent.
The boolean type check
$a \jbool \ldef a \lor \neg a$
\emph{is} excluded middle,
available for a formula once it has been proved
for that formula (\emph{habeas quid}).
\rga thus has excluded middle
as a \emph{judgment} --
earned formula by formula,
exactly on the decided formulas --
rather than as an axiom scheme.

\begin{thm}[\textsf{ir\_negTI}, \textsf{ir\_orTI}]
\label{thm:decided}
If $\Gamma \vdash a \jbool$ and $\Gamma \vdash b \jbool$,
then $\Gamma \vdash (\neg a) \jbool$
and $\Gamma \vdash (a \lor b) \jbool$,
hence also $\Gamma \vdash (a \limp b) \jbool$.
\end{thm}

Decidedness thus propagates through the connectives,
and the dynamic type check is the license to reason classically:
case analysis on a decided formula
is ordinary disjunction elimination,
and the deduction theorem holds at decided hypotheses
(\cref{sec:character:dt}).

Three notions of decidedness meet here.
\rga \emph{decides} $a$ when it proves $a$ or proves $\neg a$;
$a$ is \emph{provably decided} in $\Gamma$ when $\Gamma \vdash a \jbool$;
and $a$ is grounded-boolean when it evaluates to $\tvT$ or $\tvF$.
In the empty context the first two coincide,
by the disjunction property (\cref{thm:disj}),
and for closed well-formed $a$ all three coincide (\cref{thm:iff}).

\begin{table}[t]
\centering
\small
\renewcommand{\arraystretch}{1.25}
\begin{tabular}{@{}p{0.15\textwidth}p{0.26\textwidth}p{0.2\textwidth}p{0.18\textwidth}p{0.1\textwidth}@{}}
\textbf{principle} & \textbf{implication, $\emptyset \vdash \ldots$}
  & \textbf{decided-antecedent form} & \textbf{entailment} & \textbf{rule} \\
\hline
double-negation elimination
  & $\neg\neg a \limp a$ iff \rga decides $a$ (\cref{thm:dne-imp})
  & -- & $\neg\neg a \vdash a$ (\irl{{\neg}{\neg}}) & yes \\
Markov
  & $\neg\neg\exists.\,b \limp \exists.\,b$ iff \rga decides $\exists.\,b$
    (\cref{thm:dne-imp});
    fails at a decided matrix (\cref{thm:markov-fails})
  & -- & $\neg\neg\exists.\,b \vdash \exists.\,b$ & yes (Markov's rule) \\
excluded middle
  & $a \lor \neg a$ iff \rga decides $a$ (\cref{thm:disj});
    explosive at $L$ (\cref{thm:lem-rule});
    as a scheme explosive (\cref{rem:lem-scheme})
  & -- & -- & -- \\
self-implication
  & $a \limp a$ iff \rga decides $a$; fails at $L$ (\cref{thm:dt})
  & -- & $a \vdash a$ (\irl{hyp}) & yes \\
discharge
  & $\Gamma \vdash a \limp c$ from $a, \Gamma \vdash c$, for every $c$,
    iff $\Gamma \vdash a \jbool$ (\cref{thm:dt-exact})
  & $\exists.\,D, \Gamma \vdash c$ iff
    $\Gamma \vdash \forall.\,(D \limp \lft c)$, for decided $D$
    (\cref{thm:sigma-dt})
  & -- & -- \\
modus ponens
  & -- & -- & $a \limp c, a \vdash c$ (\cref{thm:dt-exact}(1)) & yes \\
\end{tabular}
\caption{The forms in which the principles of \cref{sec:character} hold
  in \rga: implication, the decided-antecedent form, entailment,
  and rule (\cref{sec:character:dt}).
  ``\rga decides $a$'' means that \rga proves $a$ or proves $\neg a$.}
\label{tab:forms}
\end{table}

\subsection{Markov's Principle}
\label{sec:character:markov}

The witness property (\cref{thm:witness}) says that
every provable existential $\exists.\,b$
is witnessed by a provable numeral instance.
Markov's rule follows:
if $\emptyset \vdash \neg\neg\exists.\,b$,
then $\emptyset \vdash \exists.\,b$
by double-negation elimination,
and the existential carries a witness.
Markov's principle, the unbounded-search postulate
of the Russian constructivist school,
thus holds in \rga as an entailment,
$\neg\neg\exists.\,b \vdash \exists.\,b$,
and as a rule, for every $b$.
As an implication it holds only where \rga decides $\exists.\,b$,
and a decided matrix does not suffice.

\begin{thm}[\textsf{dne\_imp\_iff\_decided}, \textsf{markov\_imp\_iff\_decided}]
\label{thm:dne-imp}
For well-formed $a$,
$\emptyset \vdash \neg\neg a \limp a$
if and only if \rga proves $a$ or proves $\neg a$.
In particular, for well-formed $b$,
$\emptyset \vdash \neg\neg\exists.\,b \limp \exists.\,b$
if and only if \rga decides $\exists.\,b$.
\end{thm}

If \rga proves $a$, disjunction introduction gives
$\neg\neg\neg a \lor a$, which is the implication;
if \rga proves $\neg a$, then it proves $\neg\neg\neg a$
by \irl{{\neg}{\neg}}, and disjunction introduction again gives the implication.
Conversely, the implication is the disjunction
$\neg\neg\neg a \lor a$,
and by \cref{thm:disj} one disjunct is provable,
so \rga refutes or proves $a$.
At $a \ldef \exists.\,b$ this is Markov's principle as an implication.
A decided matrix does not suffice.

\begin{thm}[\textsf{markov\_imp\_fails}]
\label{thm:markov-fails}
There is an $x \notin K$ such that
$\emptyset \vdash u_x \cdot \tofn{n}$ for every $n$,
while
$\emptyset \nvdash \neg\neg\neg U_x \limp \neg U_x$.
\end{thm}

Here $\neg U_x$ is the existential
``some step halts computation $x$ on $x$''
of \cref{sec:omega:family},
and every instance of its matrix is decided.
For the $x$ of \cref{cor:gap},
the existential is neither provable (it is false)
nor refutable (refuting it would prove $U_x$),
so by \cref{thm:dne-imp} the implication fails.

The witness of \cref{thm:witness} exists by soundness;
finding it from a proof is a search over proofs,
which \cref{thm:witness} guarantees to stop.
Heyting arithmetic has the disjunction and witness properties too,
as admissible rules,
so they do not separate \rga from intuitionism;
double-negation elimination as an entailment does.
Nor does \rga contain Heyting arithmetic:
$\emptyset \vdash a \limp a$ holds only where \rga decides $a$
(\cref{thm:dt} and below),
so it fails at the Liar and at \rga's consistency sentence~\cite{ford-conrga}.

\subsection{The Deduction Theorem, Exactly at Decided Hypotheses}
\label{sec:character:dt}

Implication in \rga is material by definition,
$a \limp c \ldef \neg a \lor c$.
A conditional ``if $a$ then $c$'' can therefore take
several forms in \rga, and they differ.
The \emph{implication} $\Gamma \vdash a \limp c$ is one formula,
true when the disjunction is grounded-true.
The \emph{entailment} $a, \Gamma \vdash c$ is one derivation
under a hypothesis:
evidence for $a$ yields evidence for $c$.
A \emph{rule}, ``if $\vdash a$ then $\vdash c$'',
is a closure property of derivability,
proved in the metatheory.
And for a hypothesis $\exists.\,D$ whose matrix $D$ is decided,
the \emph{decided-antecedent form}
$\Gamma \vdash \forall.\,(D \limp \lft c)$
replaces the undecided hypothesis
by a family of decided ones.
From an implication the entailment always follows.
The converse step, discharging the hypothesis,
is the deduction theorem,
and it holds exactly at provably decided hypotheses.

\begin{thm}[\textsf{ir\_dt\_use}, \textsf{deduction\_theorem\_exact}]
\label{thm:dt-exact}
\leavevmode
\begin{enumerate}
\item If $\Gamma \vdash a \limp c$, then $a, \Gamma \vdash c$.
\item For well-formed $a$:
  (for every $c$, if $a, \Gamma \vdash c$ then $\Gamma \vdash a \limp c$)
  if and only if $\Gamma \vdash a \lor \neg a$.
\end{enumerate}
\end{thm}

Part~1 is case analysis on the disjunction $\neg a \lor c$,
a derivable rule in every context.
In part~2,
the ``if'' direction is a derived rule:
case analysis on $a \lor \neg a$.
The ``only if'' direction is the instance $c \ldef a$,
whose implication $\neg a \lor a$ is decidedness itself.
At the Liar the boundary shows.

\begin{thm}[\textsf{rga\_deduction\_theorem\_fails}]
\label{thm:dt}
$L \vdash L$,
but $\emptyset \nvdash L \limp L$,
that is, $\emptyset \nvdash \neg L \lor L$.
\end{thm}

The entailment $L \vdash L$ is an instance of the hypothesis rule.
The implication $L \limp L$ would need
a grounded-true disjunction,
and that requires a true disjunct.
The Liar has none.
The diagnosis generalizes.
A hypothetical judgment
$a, \Gamma \vdash c$ promises only an entailment:
evidence for $a$ yields evidence for $c$.
The implication $a \limp c$, or $\neg a \lor c$, promises more:
a disjunction with at least one side grounded-true.
Where $a$ is not provably decided,
discharging the hypothesis would manufacture
a decided disjunction
that the hypothetical judgment never asserted.
Hypothetical reasoning in \rga is thus
strictly more permissive than implication --
hypothesizing is free
(\cref{sec:system:proof}),
while implications assert
decided alternatives.
The phenomenon is the familiar one of strong Kleene logic,
whose material conditional validates no $a \limp a$
at a gappy $a$~\cite{kleene52introduction};
Field's non-material conditional is
the standard repair~\cite{field08saving}.
\rga keeps the material conditional.
With the disjunction property (\cref{thm:disj}),
$\emptyset \vdash a \limp c$ then holds only where
\rga refutes $a$ or proves $c$,
so no implication between two sentences that \rga neither proves
nor refutes is a theorem, not even $a \limp a$.

\Cref{thm:dt-exact} says how far conditional proof reaches in \rga:
exactly to the decided antecedents.
It cannot reach further,
for the material conditional or for any other conditional
with modus ponens.
The obstacle is Curry's paradox.
The fixed point $C \ldef \mathit{fix}(\lambda\,(\vz \Rightarrow \z = \tofn 1))$
converts in two head steps to $C \Rightarrow \z = \tofn 1$.
Assuming $C$, conversion and modus ponens give $\z = \tofn 1$;
conditional proof then discharges $C$,
conversion turns the result back into $C$,
and modus ponens gives $\z = \tofn 1$ from no hypotheses.
Contraction costs nothing,
because contexts are sets:
$C$, used twice, is one hypothesis.

\begin{thm}[Curry; \textsf{curry\_explosion}, \textsf{ext\_curry\_explosion}]
\label{thm:curry}
Let $\Rightarrow$ be a term former that
preserves well-formedness and commutes with substitution.
\begin{enumerate}
\item Every entailment relation over \rga's terms
  that has the hypothesis rule,
  head conversion in both directions,
  modus ponens for $\Rightarrow$,
  and conditional proof for $\Rightarrow$
  (from $a, \Gamma \vdash b$ infer $\Gamma \vdash a \Rightarrow b$)
  at every well-formed antecedent $a$,
  derives $\z = \tofn 1$ from no hypotheses.
\item Consequently \rga has no such former:
  modus ponens and conditional proof at every well-formed antecedent
  are not both derived or admissible rules of \rga
  for any $\Rightarrow$.
\end{enumerate}
\end{thm}

Part~2 is part~1 at \rga's own entailment,
with \cref{thm:consistent}.
For \rga's material conditional $\neg a \lor b$,
modus ponens holds (\cref{thm:dt-exact}(1)),
so conditional proof at arbitrary antecedents is excluded;
the Liar already shows it,
since that rule would derive both $\neg L$ and $L$.

The decided-antecedent form makes the theorem usable
at undecided existential hypotheses.

\begin{thm}[\textsf{sigma\_deduction\_theorem}]
\label{thm:sigma-dt}
If $\vz \jnat, \lft\Gamma \vdash D \jbool$, then
$\exists.\,D, \Gamma \vdash c$
if and only if
$\Gamma \vdash \forall.\,(D \limp \lft c)$.
\end{thm}

The existential hypothesis $\exists.\,D$ itself
need not be decided;
only its matrix must be, at the eigenvariable.
The companion paper~\cite{ford-conrga} shows that
\gdl's proof of the second theorem survives in \rga
in exactly these forms:
its steps hold as entailments, rules,
and decided-antecedent implications,
but not as implications between undecided sentences.
The Liar here plays only
the role of a convenient
ungrounded hypothesis.
A systematic study of paradoxes
in grounded reasoning is
a separate line of work~\cite{srikanth26formalizing}.

\Cref{tab:forms} collects the forms in which
the principles of this section hold.
A cell that rests on a Remark says so.

%% file: formal.tex
\section{The Formalization}
\label{sec:formal}

Every theorem, corollary and proposition in this paper
is machine-checked in Isabelle/HOL~\cite{nipkow02isabelle}
(Isabelle2025-2);
arguments labelled \emph{Remark} are not mechanized.
The theories underlying this paper's results
contain no unproven obligations
and add no axioms to HOL.
A machine-checked theorem means what its definitions mean,
so a reader must also trust the definitions
the theorems are stated in:
the term syntax, well-formedness and substitution
(\texttt{RGA\_Syntax}, on the shared de~Bruijn kit);
the 42 rules (\textsf{derives}, \texttt{RGA\_Proof})
and the generic validity of proof lists;
the head-step relation \textsf{hstep};
the evaluation pair \textsf{hev}, \textsf{osr}
(\texttt{RGA\_Semantics}),
with the \gdl coding that enters through
the bound $\quo{P} \le s$;
grounded truth and consequence
(\textsf{rtrue}, \textsf{rjt}, \textsf{sjt});
the $\omega$-semantics \textsf{omv};
recursive enumerability (\textsf{renp});
the compiler \textsf{cmp};
and the System~T typing \textsf{hastype}.
Evaluation results and assignment values
share one HOL type, read as $\mathbb{N}$ and as $\natbot$
(\cref{rem:app:encoding}).
Each theorem header names the corresponding Isabelle fact,
and \cref{sec:app:endpoints} maps every numbered result
to its fact, theory and session.

Most of these facts lie in the Isabelle session \texttt{RGA}.
The facts behind \cref{thm:curry} lie in the session
\texttt{IGA}, in a theory that imports only session \texttt{RGA}.
Part~2 of \cref{thm:weak-strong} lies in the session
\texttt{RGA\_Goedel} of the companion paper's
development~\cite{ford-conrga},
because \rga's consistency sentence is defined there.
This section maps the paper onto the theories,
records the development's shape and size,
and notes the design facts that machine-checking settled.

\subsection{Theory Map and Statistics}
\label{sec:formal:map}

At the measurement point of \cref{sec:formal:avail},
session \texttt{RGA} comprised 37 theory files
and 34{,}042 lines (\texttt{wc -l}),
and the 20 theories of session \texttt{RGA} that then backed
this paper's results comprised 26{,}039 lines.
The session has grown since;
the table below lists the theories that back the results now.

\begin{center}
\begin{tabular}{ll}
\textbf{Paper section} & \textbf{Theories} \\
\hline
\cref{sec:system} (syntax, rules, semantics) &
  \texttt{RGA\_Syntax}, \texttt{RGA\_Proof},
  \texttt{RGA\_Semantics} \\
\cref{sec:system:sem:all} (fuel deciders) &
  \texttt{RGA\_Decide} \\
\cref{sec:metatheory} (metatheory) &
  \texttt{RGA\_Complete}, \texttt{RGA\_Sound} \\
\cref{sec:power} (compiler, representation) &
  \texttt{RGA\_Compile}, \texttt{RGA\_Primrec},
  \texttt{RGA\_RE}, \texttt{RGA\_Represent} \\
\cref{sec:power} (internal arithmetic) &
  \texttt{RGA\_Demo}, \texttt{RGA\_Total},
  \texttt{RGA\_Ackermann} \\
\cref{sec:power} (System~T) &
  \texttt{RGA\_SystemT}, \texttt{RGA\_SystemT\_PR} \\
\cref{sec:omega} ($\omega$-incompleteness) &
  \texttt{RGA\_Omega} \\
\cref{sec:metatheory:omega} ($\omega$-semantics) &
  \texttt{RGA\_Omega\_Model}, \texttt{RGA\_Omega\_Sound},
  \texttt{RGA\_Omega\_Gradient} \\
\cref{sec:character} (character) &
  \texttt{RGA\_Sound}, \texttt{RGA\_Paradox}, \texttt{RGA\_Omega} \\
\multicolumn{2}{l}{\emph{other sessions}} \\
\cref{sec:character:dt} (Curry) &
  \texttt{IGA\_Curry} (session \texttt{IGA}) \\
\cref{sec:omega:weak-strong} (consistency sentence) &
  \texttt{Entailment\_Goedel} (session \texttt{RGA\_Goedel}) \\
\end{tabular}
\end{center}

These theories build on
the project's shared infrastructure,
developed for \bga and \pga
and reused wholesale:
a generic de Bruijn syntax framework,
the \gdl-coding typeclass,
the two-argument primitive-recursive calculus
with its verified step-function computability layer
(universal interpreter, diagonal sets,
recursive-enumerability theory),
and the generic proof-list framework
that makes derivability codable by construction.

\subsection{Mechanization Notes}
\label{sec:formal:notes}

A few engineering choices shaped the development.

\paragraph{One syntax datatype, shared.}
Terms across the \gd systems inhabit
a single generic datatype --
a variable constructor plus
an operator constructor carrying
an operator code, a binder count,
and an argument list.
The substitution, lifting, and freshness kit
is proved once against the generic type;
\rga's ten operators are
pretty-constructor definitions over that kit,
and each \gd system instantiates
the same machinery.
Formulas-as-terms falls out:
nothing distinguishes a formula
but the rules that mention it.

\paragraph{Judgments and proofs are data.}
Hypothesis contexts are finite sets
(Isabelle's \textsf{fset}),
and proofs are lists of judgments
checked by a generic validity predicate.
The design is not incidental:
the reflective semantics quantifies over
\gdl codes of proofs,
so proofs must be first-order data
from the outset --
there is no separate
``arithmetization of syntax'' phase,
because syntax is born arithmetized
through the coding typeclass.

\paragraph{\rga in the \gd family.}
The development defines generic \gd rule interfaces
as Isabelle locales:
a propositional core covering negation, disjunction, explosion,
double negation and hypothesis discipline
(\textsf{gd\_prop\_}\allowbreak\textsf{minimal}),
and interfaces for the naturals,
for the universal quantifier and its negation,
and for the existential quantifier and its negation
(the \textsf{ga\_nat} family,
\textsf{ga\_all}, \textsf{ga\_not\_all},
\textsf{ga\_ex}, \textsf{ga\_not\_ex}).
\rga interprets all of them,
by \texttt{sublocale} proofs from its 42 primitive rules,
and \pga interprets the propositional core;
lemmas proved once against an interface
transfer to every system that interprets it.

\paragraph{Design choices forced by counterexamples.}
Three features of \rga's design are forced by counterexamples.
Two are machine-checked:
the substitutional soundness motive
(\cref{prop:app:motive-fails})
and the lifted certificate in the \irl{allT} clause
(\cref{exa:app:diag}).
The third concerns the conversion rules:
a closed formula that is machine-checked to be grounded-true
would not be derivable without them
(\cref{sec:app:complete:spinegap}).

\subsection{AI Contribution Statement}
\label{sec:formal:ai}

Substantial parts of this work were
carried out with Claude models from Anthropic,
working under the direction and review of
the human author.
The AI contributions include
the mechanization of
large parts of the \rga metatheory campaigns --
the soundness and completeness developments,
the verified compiler and
the representation theorem,
the $\omega$-incompleteness construction,
the $\omega$-semantics,
and internal totality proofs, System~T included --
as well as first drafts of this paper's text.
The human author directed the research
based on intuitions that had formed over a number of years,
established the designs and their revisions,
proved further results,
reviewed and edited throughout,
and bears sole responsibility for
the paper's claims, framing,
and treatment of related work.
The collaboration is documented
in the project repository:
a working notebook records
the campaigns and design decisions,
and individual commits carry
explicit AI co-author attribution.
The formal results themselves are
machine-checked by Isabelle,
so their correctness is
independent of their authorship.

\subsection{Artifact}
\label{sec:formal:avail}

The complete mechanically-checked development,
session \texttt{RGA} together with the sessions of
the companion paper~\cite{ford-conrga} and the session \texttt{IGA},
is released at the tag \artifacttag
of the project repository,
the same tag the companion paper cites.
The theorem names and theory files that this paper gives
are those of the development at that tag.
The size figures of \cref{sec:intro,sec:formal:map}
were measured at the tag
\texttt{conrga-}\allowbreak\texttt{paper-v3},
the measured endpoint of the companion paper's proof campaign
including its cleanup.
The README gives, for this paper,
the build command
(\texttt{isabelle build -d hol RGA IGA RGA\_Goedel}),
and notes that session \texttt{QGA}
and the file \texttt{Cantor.thy} lie outside this development.

%% file: related.tex
\section{Related Work}
\label{sec:related}


\rga sits at the intersection of several traditions --
theories of truth, combinatory logic,
proof theory, and constructive semantics --
and differs from each in a characteristic way.
One quantity organizes the comparison:
the complexity of the truth set.
Kripke's least fixed point is
$\Pi^1_1$-complete~\cite{burgess86truth};
\rga's grounded truth is
recursively enumerable (\cref{thm:truth-re}).
To the author's knowledge,
\rga is the first quantified arithmetic,
with internal induction and unrestricted recursion,
whose intended semantics leaves sentences of the arithmetic itself,
not only sentences about truth, possibly ungrounded,
while keeping the truth set $\Sigma_1$.
Each ingredient alone is old.
Kleene's strong three-valued logic was introduced
for partial recursive predicates~\cite[\S 64]{kleene52introduction},
and \rga's quantifier-free predecessors \bga and \pga
already have arithmetic statements without a value~\cite{ford25have}.
Fitch's system K keeps theoremhood recursively enumerable
by renouncing negation and the universal quantifier
(\cref{sec:related:fitch}).
Kripke--Feferman, its partial variant, and Field's theory
confine the gaps to a truth predicate over classical arithmetic,
and Reinhardt's significant sentences are recursively enumerable
as the theorems of a classical instrument
rather than as a semantics.
The alternative design choices surveyed below are,
in each tradition,
the price or the mechanism of getting there.

\subsection{Theories of Truth}
\label{sec:related:truth}

Kripke's fixed-point
construction~\cite{kripke75outline} is
\rga's direct ancestor
(\cref{sec:bg:kripke}),
and the quantifier clause is
exactly where \rga departs.
Kripke's universal is true when
all numeral instances are true --
the $\omega$-style clause that drives
the fixed point to
$\Pi^1_1$-completeness~\cite{burgess86truth}.
No r.e.\ truth predicate can satisfy
that clause;
\rga replaces the clause with
reflected proof search
and accepts $\omega$-incompleteness
(\cref{sec:omega}) as the exact cost.
\rga can thus be read as
an effective counterpart of
Kripke's construction:
what grounded truth becomes when
the $\omega$-rule is traded for
the system's own proof search.

The axiomatic-truth tradition
describes Kripke's semantics
proof-theoretically.
Kripke--Feferman
(KF)~\cite{feferman91reflecting,reinhardt86remarks,cantini90theory}
axiomatizes the fixed points in
classical logic,
producing the famous inner/outer mismatch:
KF proves the Liar's excluded middle;
with the consistency axiom it also proves
that neither disjunct is true~\cite{halbach11axiomatic}.
Reinhardt~\cite{reinhardt86remarks} proposed
repairing the mismatch instrumentally,
asserting only what KF proves true --
an r.e., gappy, provability-grounded
collection of significant sentences,
recently analyzed by Castaldo and
Stern~\cite{castaldo23reinhardt}.
\rga can be read as a native counterpart of that program:
its recursively enumerable, gappy,
provability-grounded set of truths
is not extracted from a classical instrument
but is its own semantics,
with soundness and consistency machine-checked.
The inner/outer mismatch
never arises,
because \rga is gappy throughout.

Partial Kripke--Feferman
(PKF)~\cite{halbach06axiomatizing} is
the closest axiomatic relative:
like \rga, PKF works inside
the paracomplete logic,
with no classical outer layer.
PKF retains the $\omega$-style
quantifier clause,
so PKF's intended semantics remains
$\Pi^1_1$.
PKF also shows what a partial logic can cost in strength:
it has the proof-theoretic strength of ramified analysis
up to $\omega^\omega$, against ramified analysis up to $\varepsilon_0$ for
KF~\cite{halbach06axiomatizing,feferman91reflecting}.
Halbach and Horsten,
and later Halbach and
Nicolai~\cite{halbach06axiomatizing,halbach18costs},
locate that cost in
induction over semantic vocabulary.
The diagnosis converges with
\rga's \emph{habeas quid} discipline from
the opposite direction:
what PKF diagnoses as a cost,
\gd adopts as the design principle
(\cref{sec:bg:gd}),
and \rga's boolean type check
`$a \jbool$' plays the role of a determinateness judgment.
Friedman--Sheard~\cite{friedman87axiomatic}
and Halbach's disquotational
PUTB~\cite{halbach09reducing}
chart nearby territory --
the former a warning that
classical compositional truth rules
force $\omega$-inconsistency,
the latter evidence that
thin, r.e., schema-based systems
can carry substantial strength.
The companion paper~\cite{ford-conrga}
defines a truth predicate $\mathit{True}$ inside \rga,
from a compiled decider of this paper's semantics,
and records which Friedman--Sheard, KF and PKF axioms
hold for it, and in which form.

Whether $\mathit{True}$ is a truth predicate
depends on the notion of truth it is measured against.
\rga's language has no \emph{primitive} truth predicate,
as KF, PKF, Field's theory
and the applicative truth theories (\cref{sec:related:fitch}) have:
$\mathit{True}$ is a defined term,
and \rga's gaps lie in the arithmetic itself,
not only where a truth predicate is applied.
Measured against classical truth,
in the $\omega$-semantics of \cref{sec:omega},
$\mathit{True}$ is not a truth predicate:
its extension is the set of \rga's theorems,
a $\Sigma_1$ provability predicate,
which misses the \gdl sentence and the consistency sentence,
both true there,
and which by Tarski's theorem
cannot be classically transparent.
Measured against \rga's own notion of truth,
grounded truth or truth as establishment,
$\mathit{True}$ is \rga's truth predicate
for its whole language, quantified statements included:
for well-formed $t$, \rga proves $t$
exactly when it proves $\mathit{True}(\quo{t})$
(the rules NEC and CONEC),
and \rga proves that neither the Liar nor its negation is true.
That $\mathit{True}$ coincides with provability
is then not a disguise but the grounded thesis,
which this paper's soundness and completeness theorems
(\cref{thm:sound,cor:complete-open}) prove:
grounded truth and provability coincide
for well-formed statements.
Neither perspective overrules the other;
they measure $\mathit{True}$ against different notions of truth.

McGee's theorem~\cite{mcgee85truthlike}
and L\"ob's theorem~\cite{lob55solution}
are the two rocks any such design
must steer past.
McGee shows that T-introduction,
consistency, T-distribution,
and the instances-to-universal inference
jointly force $\omega$-inconsistency.
\rga's language has no primitive truth predicate,
so in this paper only the last of McGee's principles
has a counterpart,
and \rga never asserts
the instances-to-universal direction:
universal introduction consumes
one schematic proof,
not $\omega$-many instance facts.
The companion paper's defined predicate $\mathit{True}$
behaves the same way:
it reaches a true universal
only from a proof of its open instance,
never by an $\omega$-rule~\cite{ford-conrga}.
\rga is $\omega$-consistent (\cref{thm:omega-consistent})
and sound for the instance-wise $\omega$-semantics
(\cref{thm:omega-sound}),
which together certify that the instances-to-universal inference
McGee needs is not derivable from true premises.
L\"ob's theorem forbids
an unrestricted internal
provability-implies-truth axiom;
\rga's reflective quantifier accordingly lives in
rule form,
in the semantics' certificate clauses
and the quantifier rules,
never as an internal implication.
For \rga's own compiled provability predicate,
the companion paper~\cite{ford-conrga} shows that
L\"ob's rule holds,
that formalized L\"ob fails as an implication at $\z = \tofn{1}$
but holds as an entailment,
and that local reflection at $\z = \tofn{1}$,
as an entailment,
holds in a context exactly when
that context derives \rga's consistency sentence,
so it fails in the empty context.

\subsection{Combinatory Logic and the R.E. Design Goal}
\label{sec:related:fitch}

\rga's syntactic architecture --
untyped combinators,
$\lambda$ as sole binder,
the quantifier as an applied constant --
descends from Church~\cite{church32set,church40formulation}
and Curry~\cite{curry30grundlagen}.
Church's simple type theory made
quantification a constant $\Pi$
applied to an abstraction,
the architecture of every HOL-family
proof assistant since --
including the Isabelle/HOL
this work is mechanized in.
The untyped ancestors were
inconsistency minefields:
Kleene and Rosser felled
Church's original
system~\cite{kleene35inconsistency},
and Curry's paradox~\cite{curry42inconsistency}
distilled the mechanism.
Curry's illative tradition survived,
and its recent systems are consistent and strong.
Barendregt, Bunder and Dekkers~\cite{barendregt93systems}
gave illative systems complete for first-order
propositional and predicate calculus,
and Czajka~\cite{czajka13higher,czajka15semantic}
proved higher-order illative systems consistent
by model constructions.
These systems keep the untyped term syntax unrestricted,
as \rga does,
and restrict the logical \emph{rules} instead,
with premises that a term is a proposition.
The correspondence with \gd is close:
the illative premise $\mathsf{H}\,p$ (``$p$ is a proposition'')
plays the role of \gd's dynamic type check `$p \jbool$',
and the illative implication introduction,
which needs $\mathsf{H}\,p$ for the discharged hypothesis,
is the counterpart of \rga's deduction theorem,
which holds exactly at provably decided hypotheses;
Czajka's natural-number predicate plays the role of `$t \jnat$'.
Czajka's system for partiality and recursion
in higher-order logic~\cite{czajka13partiality}
comes closest to \gd's goal:
it admits unrestricted recursive definitions,
has the natural numbers with unrestricted induction
and provably total addition and multiplication,
and is proved consistent.
\rga differs from this line of work in its semantics and in its quantifier.
Illative systems add logical constants
(such as $\Xi$, $\mathsf{H}$ and $\mathsf{P}$)
to an untyped calculus
and justify them by set-theoretic models,
which establish consistency.
\rga's operators are computable functions,
its semantics is grounded evaluation,
and it is sound and complete for that semantics
(\cref{thm:sound,cor:complete-open}),
with a recursively enumerable truth set.
\rga's universal quantifier is grounded
in \rga's own proof search,
which is what lets the companion paper~\cite{ford-conrga}
analyze \gdl's theorems and a truth predicate inside \rga;
to the author's knowledge,
no illative system has been studied in this direction.
Caie's analysis of Bunder's paradox~\cite{caie20bunder}
asks which rules for propositionhood itself are consistent;
the question has a direct counterpart in \gd,
where `$p \jbool$' is itself a computation
that may fail to terminate.
So \rga can be read, genealogically,
as Curry's illative combinatory logic
with grounded semantics.
Feferman's and Cantini's type-free theories of truth
over combinatory logic~\cite{feferman84toward,cantini96logical}
pair an untyped applicative base with a primitive truth predicate;
\rga's language has none,
and its gaps lie in the arithmetic itself.
(The truth predicate of the companion paper~\cite{ford-conrga}
is defined in \rga, not added to it;
see \cref{sec:related:truth}.)

Fitch's basic logic is
the tradition's most direct
anticipation of \rga's goal.
Fitch's system K~\cite{fitch42basic}
was designed so that theoremhood is
exactly r.e. --
``constructive definability is
recursive enumerability'' --
and to stay r.e.\ the system renounced
negation and the universal quantifier.
Fitch's extension with
negation and quantifiers~\cite{fitch48extension}
lost recursive enumerability outright.
Fitch's later system Q~\cite{fitch74elements,fitch81consistency}
is a consistent combinatory logic
with induction over the natural numbers,
whose restrictions differ from both
the modern illative systems and \gd's.
\rga keeps grounded negation and
a universal quantifier
while staying r.e.,
by grounding the quantifier in
reflected proof search rather than
in an infinitary closure condition.

\subsection{The \texorpdfstring{$\omega$}{omega}-Rule and Reflection}
\label{sec:related:reflection}

The premise shape of
\rga's universal introduction has
a classical ancestor in
the restricted $\omega$-rule of
Shoenfield and
Sch\"utte~\cite{shoenfield59omega}:
infer the universal from
a recursive enumeration of
instance proofs.
Iterated, the restricted rule is
complete for true arithmetic --
and therefore not r.e.:
recognizing a correct recursive
$\omega$-proof requires checking
well-foundedness of
an infinite coded tree,
a $\Pi^1_1$-complete task.
\rga's rule replaces
``an effective enumeration of
instance proofs exists''
with
``the system itself proves
the schematic instance'' --
one finite, primitive-recursively
checkable proof.
The demotion is deliberate:
\rga is provably incomplete for
classical arithmetic truth
(\cref{sec:omega}),
and complete for
grounded truth instead
(\cref{thm:iff}).

Reflection principles offer
the other classical comparison.
Turing's ordinal
logics~\cite{turing39ordinals} and
Feferman's transfinite
progressions~\cite{feferman62transfinite}
iterate reflection statements
along ordinal notations,
achieving completeness through
an external transfinite scaffolding;
Kreisel and L\'evy~\cite{kreisel68reflection}
calibrate reflection against
transfinite induction.
\rga's \irl{allT} clause is uniform reflection
turned into a truth clause:
a provable schematic instance makes the universal true.
The reflection lives in the semantics, not in a theorem of \rga,
and needs no ordinal parameter and no external hierarchy.
The strength \rga thereby attains,
measured against
the iterated-reflection and
transfinite-induction hierarchies,
is a natural question
this paper leaves open.
A lower bound is cheap:
Ackermann's function is provably total by one quantified theorem
(\cref{thm:ack}),
so the functions \rga proves total in quantified form
exceed the primitive-recursive ones.
No further bound is claimed,
and \rga does not contain Heyting arithmetic
(\cref{sec:character:markov}).

Willard's self-verifying systems~\cite{willard01self}
prove their own consistency
by weakening arithmetic:
they do not prove multiplication total,
and they keep classical logic.
\rga is not a system of that kind.
It proves multiplication total (\cref{thm:addmult}),
each compiled primitive-recursive function total
(\cref{thm:cmp-total}, one quantified theorem per index),
and Ackermann's function total (\cref{thm:ack}).
It does not prove its own consistency sentence~\cite{ford-conrga}:
\gdl's second theorem holds for \rga.
What \rga changes is the reading of the universal quantifier,
which a certificate settles rather than its instances,
and the logic's conditional:
the steps of \gdl's proof hold in \rga
as entailments and rules,
but not as implications between undecided sentences
(\cref{sec:character:dt}).
The folklore checklist of \cref{sec:intro}
is therefore right that \gdl's theorems apply.
It reads ``true'' instance by instance,
and in that reading, \rga's $\omega$-semantics,
the unprovable \gdl and consistency sentences are true;
under grounded truth they are ungrounded.

\subsection{Constructive Semantics of Quantification}
\label{sec:related:bhk}

\rga's grounded truth is a constructive notion of truth
in a narrow sense:
truth only through a construction,
in the tradition that identifies truth with proof,
from Brouwer and Heyting~\cite{brouwer75philosophy,heyting71intuitionism}
to Dummett, Prawitz and
Martin-L\"of~\cite{dummett78truth,prawitz77meaning,martin-lof80intuitionistic}.
It is not intuitionistic logic:
\rga's logic is strong Kleene,
$a$ and $\neg\neg a$ entail each other,
and \rga does not contain Heyting arithmetic
(\cref{sec:character:markov}).
The BHK reading --
a proof of a universal is
a method producing a proof of
each instance --
is \rga's semantic intuition.
The constructive tradition's
central objection to
fixing that method space is due to
Dummett~\cite{dummett63godel}:
identify proofs with derivations in
one formal system,
and \gdl incompleteness
manufactures a sentence that is
evidently true but unprovable,
refuting the identification.
\rga fixes the formal system anyway
and answers the objection through the quantifier:
a universal is established only by a proof of its schematic instance,
never by the truth of each instance.
The manufactured sentence comes out
ungrounded rather than true-but-unprovable.
\rga's evaluation does not establish it,
though it is true in the instance-wise $\omega$-semantics
and HOL's consistency proof (\cref{thm:consistent})
establishes it from outside;
it is not paradoxical
(\cref{sec:omega} and~\cite{ford-conrga}).
The reply has a price:
grounded truth is not closed under
the reasoning that justifies accepting \rga.
Kleene's realizability~\cite{kleene45interpretation}
makes the BHK method
a partial recursive function,
but the resulting notion of truth is
far from effective:
whether a number realizes a sentence
is arithmetical for each sentence
but not decidable, nor even recursively enumerable, in general,
whereas \rga's certificates are
finite proofs with
a primitive-recursive checker.
The witness and disjunction properties of
\cref{sec:metatheory:iff} are admissible rules
that Heyting arithmetic shares,
so they do not place \rga apart from intuitionism;
double-negation elimination as an entailment does.
\rga's witness property follows from the semantics:
an \irl{allF} certificate names the refuted instance.
Markov's principle, the Russian school's postulate,
holds in \rga as an entailment and as a rule,
and as an implication only where \rga decides the existential
(\cref{sec:character:markov}).

Finally, the logic-of-partial-terms
tradition supplies
the quantifier guards.
Beeson's treatment of
definedness~\cite{beeson85foundations}
has quantifiers range over
defined values only --
the direct precedent for
the `$a \jnat$' premise of
universal elimination --
and Aczel's Frege
structures~\cite{aczel80frege} guard
propositionhood of a universal by
propositionhood of the instances,
the ancestor of \rga's
boolean-typing question for
quantified formulas.
Feferman's definedness
logic~\cite{feferman95definedness}
is the same discipline in
axiomatic form.
\gd's contribution to this lineage is
making the guards computational:
definedness is termination,
dynamically checked by operational-semantic evaluation,
and proved by \emph{habeas quid} obligations.

\subsection{Mechanized Metatheory}
\label{sec:related:mech}

\gdl's incompleteness theorems have
been machine-checked in
classical settings
by Shankar~\cite{shankar94metamathematics},
O'Connor in Coq~\cite{oconnor05essential},
Harrison in HOL Light~\cite{harrison-hol-light-godel},
Paulson in Isabelle/HOL~\cite{paulson14machine,paulson15mechanised},
Popescu and Traytel~\cite{popescu19formally,popescu21distilling},
and Kirst and coauthors~\cite{kirst23godel,kirst23synthetic},
formalizing the limitative results of
systems studied from outside.
The companion paper~\cite{ford-conrga}
mechanizes \gdl's first and second theorems for \rga,
on the development this paper describes,
and surveys that literature.
The present development differs
in genre:
the mechanization is of
a nonclassical system's
complete positive metatheory --
semantics, proof system,
and the exact coincidence between them
(\cref{thm:iff}) --
with the limitative result
(\cref{thm:omega}) obtained
inside the same framework.

%% file: concl.tex
\section{Conclusion}
\label{sec:concl}

This paper has presented \rga,
a reflective grounded arithmetic
that satisfies the folklore ingredient list
under which \gdl's incompleteness theorems
are expected to apply.
That list asks for arithmetical reasoning,
including provably-total addition and multiplication;
quantification over the natural numbers;
recursively-enumerable theorems;
and full \re\ self-reference.
\rga is consistent and sound,
and complete for grounded truth,
while \gdl's theorems continue to hold for it~\cite{ford-conrga}.
The sentences that classical reasoning calls true but unprovable
are ungrounded in \rga,
because \rga's universal is true only when a proof certifies
its schematic instance;
read instance by instance, as the folklore reads them,
they are true.
One principle makes this possible,
and it is applied everywhere:
ground every commitment in computation.
Equality is grounded in evaluation,
and totality claims in termination.
The step this paper adds grounds
universal quantification in
the system's own reflected proof search.
The reflective quantifier keeps
grounded truth recursively enumerable,
and its price is paid at the r.e.\ boundary:
$\omega$-incompleteness.
The logic pays further prices of its own:
completeness can fail under hypotheses,
and the deduction theorem holds only at decided hypotheses,
a limit that Curry's paradox makes unavoidable.
Every theorem in this paper is machine-checked,
and three design features were forced by
counterexamples (\cref{sec:formal:notes}).

The grounded-deduction program
continues beyond this paper.
The companion paper~\cite{ford-conrga}
shows what \rga can do about itself:
it reasons about its own semantics and proofs inside itself,
through a compiled decider of evaluation at each fuel bound
and compiled proof checkers,
and it proves \gdl's second theorem for itself
as an entailment and in decided-antecedent form,
while \gdl's first and second theorems hold for it in HOL.
One further sequel concerns what \rga
looks like from outside.
That sequel treats semantic completions that decide
what reflection leaves ungrounded,
and the classical principles whose consistency with \rga
such completions establish,
omniscience-flavored principles among them.
Another concerns deliberately
unsound-but-consistent extensions
and the model theory that
tames them.
A third is proof-theoretic:
locating \rga and the \rga-like systems
on the ordinal-analysis map
that \cref{sec:related:reflection}
left open,
which the companion paper also lists as open.
A fourth is semantic:
\rga is complete for grounded truth in the weak sense,
for derivations from no hypotheses,
and under a closed hypothesis exactly when that hypothesis
is grounded-true or explosive
(\cref{thm:weak-strong}).
Strong completeness fails for two reasons.
An open hypothesis such as $x \jnat$
makes every instance of the $\omega$-incomplete family
a grounded consequence without a schematic proof,
and a closed hypothesis that is never grounded-true,
such as the consistency sentence,
has every conclusion as a grounded consequence.
Which semantics for derivations under hypotheses
would make \rga strongly complete is open;
evaluation relative to an oracle that makes the hypotheses true
would address the second reason but not the first.
And nearest at hand,
the grounded systems are converging on
a shared, machine-checked rule interface
(\cref{sec:formal:notes}).
That convergence points toward
a uniform metatheory of
the \gd family as a whole.

Several questions stay open.
Is evaluation monotone in information,
so that a formula true with $\bot$ for a variable
stays true with any numeral in its place?
Does \rga prove, for each System~T function, a single quantified totality theorem,
and would a recursor rule unfolding at function type stay sound
with a primitive-recursive checker?
Do soundness and completeness for their own grounded semantics
hold generically for extensions of \rga?
Are \rga's propositional rules complete
for strong Kleene consequence?
And can an $\omega$-sound extension of \rga prove \rga's consistency?
The companion paper~\cite{ford-conrga} shows that no extension
whose derivations are sound for the $\omega$-semantics
can do so by that paper's internal soundness induction,
because the induction's motive is not $\omega$-true at a hypothesis node
of an accepted proof with an ungrounded case hypothesis;
whether some other internal argument could succeed is open.

The broader claim of
the grounded-deduction program is that
unconstrained recursive definition --
the everyday freedom of programming --
can live inside consistent formal reasoning.
\rga extends that claim
from propositional to quantified arithmetic.
\gd's freedom of recursive definition survives induction,
totality theorems,
and completeness for grounded truth,
in an effective system that
knows its own limits.

%% file: ack.tex
\section*{Acknowledgments}
\addcontentsline{toc}{section}{Acknowledgments}

This work is part of the larger grounded deduction (\gd) project.
It was significantly inspired by conversations with
the following collaborators on that project,
and by their prior and ongoing work:
Dimitrios Alexopoulos,
Elliot Bobrow,
Fatemeh Boloukazari,
Yusuf Demir,
Mahdis Mirzaei,
Omar Muhammad,
and
Ananthajit Srikanth.
This project has also benefited from the generous support
of the AXA Research Fund.

%% file: app-roadmap.tex
\section*{Appendix Roadmap}
\label{sec:app:roadmap}
\addcontentsline{toc}{section}{Appendix Roadmap}

The appendices transcribe the machine-checked development
into ordinary mathematical prose,
at a level of detail intended to let a reader
follow each proof or re-mechanize it
in another proof assistant.
Each backs one section of the paper
and may be read against it;
none is a prerequisite for the main text.

\begin{center}
\begin{tabular}{@{}l>{\raggedright\arraybackslash}p{0.27\textwidth}>{\raggedright\arraybackslash}p{0.5\textwidth}@{}}
\textbf{Appendix} & \textbf{Backs} & \textbf{Transcribes} \\
\hline
\ref{sec:app:complete} & \cref{sec:metatheory:complete} &
  completeness (\texttt{RGA\_Complete}) \\
\ref{sec:app:sound} &
  \cref{sec:metatheory:sound,sec:metatheory:det,sec:metatheory:iff,sec:metatheory:arch,sec:metatheory:omega} &
  soundness, determinism, consistency, the
  truth--provability coincidence,
  the witnessed semantics (\texttt{RGA\_Sound}),
  and the $\omega$-semantics
  (\texttt{RGA\_Omega\_Model}, \texttt{RGA\_Omega\_Sound}) \\
\ref{sec:app:compile} & \cref{sec:power:compile} &
  the verified compiler (\texttt{RGA\_Compile}) \\
\ref{sec:app:re} & \cref{sec:power:re} &
  recursive enumerability of grounded truth
  (\texttt{RGA\_RE}, \texttt{RGA\_Primrec}) \\
\ref{sec:app:represent} & \cref{sec:power:re} &
  the representation theorem (\texttt{RGA\_Represent}) \\
\ref{sec:app:arith} & \cref{sec:power:arith} &
  internal arithmetic
  (\texttt{RGA\_Demo}, \texttt{RGA\_Total},
   \texttt{RGA\_Ackermann})
  and System~T (\texttt{RGA\_SystemT}, \texttt{RGA\_SystemT\_PR}) \\
\ref{sec:app:omega} & \cref{sec:omega} &
  $\omega$-incompleteness (\texttt{RGA\_Omega}) \\
\ref{sec:app:character} & \cref{sec:character},
  and the disjunction and witness properties
  of \cref{sec:metatheory:iff} &
  the logic's character
  (\texttt{RGA\_Paradox}, \texttt{RGA\_Sound}) \\
\ref{sec:app:notation} & the whole paper &
  a glossary of the notation \\
\ref{sec:app:endpoints} & \cref{sec:formal} &
  each numbered result's Isabelle fact, theory and session \\
\end{tabular}
\end{center}

Two dependencies among the appendices are worth knowing
before reading any of them.
\Cref{sec:app:complete} is self-contained:
its induction uses no soundness,
no consistency, and no determinism.
\Cref{sec:app:sound} consumes it,
at three points and for two different reasons,
and everything later rests on both.
Standard background --
the de Bruijn substitution algebra,
the primitive-recursive calculus,
the G\"odel-coding layer --
is cited rather than re-derived.

%% file: app-complete.tex
\section{Completeness in Detail}
\label{sec:app:complete}

This appendix gives the completeness argument
(\cref{thm:complete-val,cor:complete-open})
in full detail,
transcribed from the machine-checked development
(\texttt{RGA\_Complete}).
The presentation assumes \cref{sec:system}'s
syntax, rules, and semantics,
and cross-references those figures freely
rather than restating them,
but it assumes nothing from \cref{sec:metatheory} --
the argument below is self-contained,
which is what allows \cref{sec:app:sound}'s
soundness proof to consume it.

\subsection{What Is Proved, and in What Shape}
\label{sec:app:complete:plan}

The claim is that \rga proves
everything its own evaluator computes.
Stated for values rather than for truth:
whenever a well-formed term $a$
evaluates to the natural number $r$,
the equation $a = \tofn{r}$
is derivable from no hypotheses.
Truth is the special case $r = 1$,
because \rga's formulas are terms
and its truth values are the numerals $\tofn 1$ and $\tofn 0$.

The proof is a single mutual induction
over the two evaluation judgments of \cref{fig:sem},
carrying two motives:

\begin{itemize}
\item the \emph{value motive}, for $\bev{}$ --
  if $a \bev{s} r$ and $\wf(a)$ then $\emptyset \vdash a = \tofn{r}$;
\item the \emph{head motive}, for $\hev{}$ --
  if $t \hev{s} w$ and $\wf(t)$ then $t$ is interchangeable
  with the canonical term of the head shape $w$
  inside any provable numeral valuation,
  underneath any spine of arguments.
\end{itemize}

The head motive is the technically interesting half,
and \cref{sec:app:complete:xfer} builds it.
Three features of the setup deserve
statement before the case analysis,
because each is a place where
a reader reconstructing the proof
can go wrong.

\paragraph{Well-formedness lives inside the motive.}
Rule induction requires each motive to be
a predicate of the inducted relation's arguments alone,
so $\wf(\cdot)$ cannot be a standing assumption of the induction.
The predicate appears instead as the antecedent of
an implication within each motive.
Every case therefore opens by assuming
well-formedness of the term in the clause's \emph{conclusion},
and must then \emph{derive} well-formedness of
the terms in the clause's premises,
either by structural decomposition of the conclusion term
or -- in the cases where a premise mentions a term
that the conclusion does not --
from the head motive itself.
\Cref{sec:app:complete:xfer} returns to this point,
which is the subtlest thing in the appendix.

\paragraph{The step index does no work.}
No case below inspects the step index $s$,
and no case needs monotonicity of evaluation in the index.
In the two quantifier clauses in particular,
the bound $\quo P \le s$ on the certificate's G\"odel code
is simply discarded:
completeness cares that a certificate exists,
not how large the certificate is.
The point is worth stating plainly,
because the soundness argument of \cref{sec:app:sound}
had to be architected specifically
to avoid a descent on that index
(\cref{sec:app:sound:nodescent}).
The asymmetry is real,
and it originates here:
the certificate clauses mention derivability positively,
so completeness may consume a certificate as a black box,
while soundness must \emph{believe} one.

\paragraph{Nothing later is presupposed.}
The induction uses no consistency, no soundness,
and no determinism of evaluation.
None of the three is available at this point
in the development, and none is needed.
Completeness accordingly comes first
in \cref{sec:metatheory}'s ordering.

\subsection{Conventions}
\label{sec:app:complete:prelim}

\subsubsection*{Terms, well-formedness, numerals}

Terms are the ten-operator syntax of \cref{fig:syntax}.
Well-formedness $\wf(t)$ holds when
$t$ is built from those ten operators
at their correct arities and binder shapes;
every variable is well-formed.
The predicate is structural and compositional --
$\wf(t \cdot u)$ iff $\wf(t)$ and $\wf(u)$,
$\wf(\lambda b)$ iff $\wf(b)$,
$\wf(a = b)$ iff $\wf(a)$ and $\wf(b)$,
and so on for each constructor --
and these equivalences are used below without comment.
Numerals are $\tofn 0 \ldef \z$
and $\tofn{k{+}1} \ldef \suc \tofn{k}$;
every numeral is well-formed and closed,
both by induction on $k$.

One notational trap recurs.
The type check `$a \jnat$' is \emph{defined} as `$a = a$',
so a fact about $a \jnat$
and the corresponding fact about $a = a$
are literally the same fact.
The development moves between the two readings freely,
and so does this appendix.

\subsubsection*{Spines}

Much of the argument concerns
a term sitting at the head of an application spine.
For a term $t$ and a finite list
$\vec u = \spl{u_1, \dots, u_k}$ of terms,
write
\[
  t \sap \vec u \;\ldef\;
  ((t \cdot u_1) \cdot u_2) \cdots u_k ,
\]
so that $t \sap \spl{} = t$
and $t \sap (u \cns \vec u) = (t \cdot u) \sap \vec u$.
Because application spines are left-nested,
the last equation is
an \emph{identity of terms} rather than a lemma,
and re-bracketing between the two forms
is never counted as a step below.
Three further facts about spines
are one-line inductions on the list
and are likewise used silently:
$\wf(t \sap \vec u)$ holds iff $\wf(t)$ holds
and every element of $\vec u$ is well-formed;
a variable is free in $t \sap \vec u$ iff
it is free in $t$ or in some element of $\vec u$;
and head steps propagate through spines,
so $t \hs t'$ implies $t \sap \vec u \hs t' \sap \vec u$
(iterated congruence in function position, rule \irl{spine}).

\subsubsection*{Head steps and $\beta$-instances}

$\bsub{b}{u}$ is the call-by-name $\beta$-instance,
substituting $u$ for the $\lambda$-bound variable of $b$
and shifting the remaining free variables down;
the argument is substituted unevaluated.
The head-step relation $\hs$ is
the one displayed at the top of \cref{fig:rules-reflect}:
$\beta$ at a syntactic $\lambda$-redex,
recursor dispatch at a \emph{syntactic} numeral scrutinee,
and congruence in function position only.
Its side conditions sit exactly where
a term would otherwise vanish from the reduct --
the discarded recursor branch,
the $\beta$ argument that a body ignores --
so that a step never manufactures
well-formedness out of nothing.

The conversion rules (\cref{fig:rules-reflect}) carry derivability
along head steps.
Refutation transport (\irl{{\neg}convE})
carries no side condition at all:
a refuted redex evaluates through the head step,
so its falsity certificate already speaks about the reduct.
The backward refutation direction is not primitive;
it is derived from \irl{conv{=}},
since a refuted reduct has the value $\z$.
\label{sec:app:complete:spinegap}%
The conversion rules convert along the whole head-step relation,
congruence in function position included,
and not only at a top-level $\lambda$-redex.
The closed formula
$((\lambda\lambda(\z = \z)) \cdot \z) \cdot \z$,
a curried two-argument function applied to two arguments,
shows why:
it is grounded-true (\textsf{spine\_gap\_true}),
because head evaluation reduces the inner redex
and the outer application then evaluates the body,
but its outermost application has no syntactic $\lambda$ at its head,
so a $\beta$-conversion that fires only at a top-level redex
cannot reach it.
With the spine rule it is derivable,
by completeness (\cref{sec:app:complete:thms}).

\subsubsection*{Working at the level of derivability}

The formalization distinguishes three levels:
a relation saying that a rule may extend a given proof list;
the relation $P \proves \Gamma \vdash c$,
holding when $P$ is a \emph{valid} proof list
whose conclusion, its first entry, is $\Gamma \vdash c$;
and derivability $\Gamma \vdash c$,
which holds when some valid proof list yields the judgment.
Each primitive rule is stated at the first level
and lifted to derivability by a one-line lemma.
That lifting is where \cref{sec:system:proof}'s phrase
``each entry follows from later entries''
is cashed out,
and it is bookkeeping:
this appendix elides it
and states every rule directly at the $\vdash$ level.

The certificate clauses of \cref{fig:sem}
use $P \proves \Gamma \vdash c$
with $\Gamma \vdash c$ the conclusion of $P$,
as \cref{sec:system:proof} defines it.
The two quantifier cases below
are handed such a certificate
and immediately weaken it to derivability,
which is all they need.
\Cref{sec:app:sound} also uses a weaker relation,
containment of a judgment anywhere in a valid proof,
and states for each of its own results
which of the two relations is in play.

\subsubsection*{The evaluation judgments}

The two judgments are those of \cref{fig:sem}, verbatim.
Two properties of the clauses that
the induction depends on,
and that \cref{sec:system:sem} states in passing,
are worth making explicit.

First, \emph{the clauses are stated for arbitrary terms}.
\Cref{sec:system:sem} presents evaluation on closed terms,
which is the intended reading,
but nothing in the clauses requires closedness.
There is no clause for a variable and no clause for $\bot$,
so a term whose evaluation would have to
look inside a free variable is simply stuck,
and has no value.
Completeness is proved for the relation as literally defined,
and therefore holds at open terms.

Second, \emph{the truth clause for a universal carries a lift}.
The \irl{allT} clause demands a certificate for
$\{\vz \jnat\} \vdash \lft u \cdot \vz$ --
the \emph{lifted} body applied to the schematic variable.
The lift shifts every free variable of $u$ up by one,
making $\vz$ genuinely fresh for the body.
On closed bodies the lift is the identity.
On open bodies the lift is what keeps
evaluation stable under substitution,
and, as \cref{sec:app:complete:allT} shows,
it is also exactly what lets
the completeness proof hand the certificate
to universal introduction unchanged.

\subsection{The Transport Idiom and the Equality Kit}
\label{sec:app:complete:transport}

\rga has no primitive rule of
transitivity of equality,
and no primitive congruence rule for any particular context.
What it has instead is full substitutivity,
the rule \irl{{=}subs} of \cref{fig:rules}:
from $\Gamma \vdash a = b$ and $\Gamma \vdash \ssub{v}{a}{p}$,
with $\wf(p)$, infer $\Gamma \vdash \ssub{v}{b}{p}$,
where the context $p$ is an arbitrary well-formed term.

Every equality-congruence step in the development
is obtained from that one rule
by \emph{choosing a substitution context}.
The recipe, used repeatedly below, runs as follows.
Pick a variable $v$ that does not occur free
in the surrounding material;
write the formula already in hand as $\ssub{v}{a}{p}$,
where $p$ is the target formula
with $v$ planted at the site to be rewritten;
then apply \irl{{=}subs}.
Freshness of $v$ is what makes the two instances
$\ssub{v}{a}{p}$ and $\ssub{v}{b}{p}$
collapse to the intended formulas,
since substitution leaves untouched
every part of $p$ in which $v$ is not free.
Such a $v$ always exists,
a term having only finitely many free variables.
Call this the \emph{transport idiom}.
The idiom explains why so many of the small lemmas below
open by choosing a fresh variable,
and why the only real content in several of them
is the choice of which occurrence to mark.

The following six facts about equality and numerals
are what the completeness induction consumes.
All are derived rules,
proved once in \texttt{RGA\_Proof}
over the primitives of \cref{fig:rules}.

\begin{lem}[numerals are naturals; \textsf{ir\_num\_N}]
\label{lem:app:num-nat}
$\Gamma \vdash \tofn{k} \jnat$ for every $k$ and every $\Gamma$.
\end{lem}

\begin{proof}
Induction on $k$.
The base case is the axiom \irl{\z{}I}.
For the step, \irl{\suc{}TI} takes
$\Gamma \vdash \tofn{k} \jnat$
to $\Gamma \vdash \suc \tofn{k} \jnat$,
and $\suc \tofn{k}$ is $\tofn{k{+}1}$.
\end{proof}

\begin{lem}[transitivity of equality; \textsf{ir\_eq\_trans}]
\label{lem:app:eq-trans}
From $\Gamma \vdash a = b$ and $\Gamma \vdash b = c$
infer $\Gamma \vdash a = c$.
\end{lem}

\begin{proof}
The argument is the transport idiom in its first instance,
and is shown here in full;
later instances are given in a line each.

Choose a variable $v$ not occurring free in $c$,
and take the substitution context $p \ldef (v = c)$.
Because $v$ is not free in $c$,
the two instances of $p$ compute to
$\ssub{v}{b}{p} = (b = c)$ and $\ssub{v}{a}{p} = (a = c)$.
The context $p$ is well-formed,
since $c$ is:
a derivable judgment has a well-formed conclusion,
so the premise $\Gamma \vdash b = c$ supplies $\wf(c)$.

Now apply \irl{{=}subs}
with the equation $b = a$ --
obtained from the first premise by \irl{{=}sym} --
and with the premise $\Gamma \vdash b = c$
in the role of $\Gamma \vdash \ssub{v}{b}{p}$.
The conclusion is $\Gamma \vdash \ssub{v}{a}{p}$,
which is $\Gamma \vdash a = c$.
\end{proof}

\begin{lem}[an equation types both sides;
  \textsf{ir\_eqD\_N1}, \textsf{ir\_eqD\_N2}]
\label{lem:app:eqD}
From $\Gamma \vdash a = b$
infer $\Gamma \vdash a \jnat$, and likewise $\Gamma \vdash b \jnat$.
\end{lem}

\begin{proof}
Recall that $a \jnat$ \emph{is} $a = a$.
Chaining $a = b$ with its symmetric image $b = a$
by \cref{lem:app:eq-trans} gives $a = a$.
The second claim is the first
applied to $b = a$.
\end{proof}

\Cref{lem:app:eqD} is what keeps the induction supplied
with the side condition of the conversion rule \irl{conv{=}}.
Whenever the proof has produced an equation
with a numeral on the right,
the ``left side is a natural'' premise comes for free.

\begin{lem}[successor is a congruence; \textsf{ir\_suc\_eqI}]
\label{lem:app:suc-cong}
From $\Gamma \vdash a = b$ infer $\Gamma \vdash \suc a = \suc b$.
\end{lem}

\begin{proof}
From $\Gamma \vdash a = b$,
\cref{lem:app:eqD} gives $\Gamma \vdash a \jnat$,
so \irl{\suc{}TI} gives $\Gamma \vdash \suc a \jnat$ --
which, unfolding the notation, is
$\Gamma \vdash \suc a = \suc a$.
Transport the right-hand occurrence:
choose $v$ not free in $a$
and take $p \ldef (\suc a = \suc v)$,
so that $\ssub{v}{a}{p}$ is $\suc a = \suc a$
and $\ssub{v}{b}{p}$ is $\suc a = \suc b$.
Applying \irl{{=}subs} with the premise $a = b$
gives the claim.
\end{proof}

The only content in that proof is
the choice of which occurrence to mark.
Marking both would have produced $\suc b = \suc b$,
which is true but useless.

\begin{lem}[distinct numerals are provably unequal;
  \textsf{ir\_num\_ne}]
\label{lem:app:num-ne}
If $j \ne k$ then $\Gamma \vdash \tofn{j} \ne \tofn{k}$.
\end{lem}

\begin{proof}
Recall that $a \ne b$ abbreviates $\neg(a = b)$.
Induct on $j$, with $k$ universally quantified.

If $j = 0$ then $k = k'{+}1$, so $\tofn{k} = \suc \tofn{k'}$.
Rule \irl{\suc{}{\ne}\z}, applied to
$\Gamma \vdash \tofn{k'} \jnat$ (\cref{lem:app:num-nat}),
gives $\Gamma \vdash \suc \tofn{k'} \ne \z$,
and \irl{{\ne}sym} turns that into
$\Gamma \vdash \tofn 0 \ne \tofn{k}$.

If $j = j'{+}1$ and $k = 0$,
the same rule \irl{\suc{}{\ne}\z}
applied to $\Gamma \vdash \tofn{j'} \jnat$
gives the claim directly.
Otherwise $k = k'{+}1$ with $j' \ne k'$;
the induction hypothesis gives
$\Gamma \vdash \tofn{j'} \ne \tofn{k'}$,
and \irl{\suc{}{\ne}I} gives
$\Gamma \vdash \suc \tofn{j'} \ne \suc \tofn{k'}$.
\end{proof}

\begin{lem}[pulling numeral disequality back to the terms;
  \textsf{ir\_ne\_transport}]
\label{lem:app:ne-transport}
If $\Gamma \vdash a = \tofn{j}$ and $\Gamma \vdash b = \tofn{k}$
with $j \ne k$, and $\wf(a)$ and $\wf(b)$,
then $\Gamma \vdash a \ne b$.
\end{lem}

\begin{proof}
Two applications of the transport idiom, one per side.

Start from $\Gamma \vdash \tofn{j} \ne \tofn{k}$
(\cref{lem:app:num-ne}),
and take the substitution context
$p_1 \ldef (\vz \ne \tofn{k})$ at variable $0$.
Then $\ssub{0}{\tofn{j}}{p_1}$ is $\tofn{j} \ne \tofn{k}$
and $\ssub{0}{a}{p_1}$ is $a \ne \tofn{k}$.
No freshness condition is needed at this step:
numerals are closed,
so variable $0$ occurs in $p_1$ exactly at the marked site.
Applying \irl{{=}subs} with the equation $\tofn{j} = a$,
the symmetric image of the first premise,
yields $\Gamma \vdash a \ne \tofn{k}$.

Now transport the right-hand side.
The remaining site sits beside $a$,
which may be open,
so variable $0$ will not serve.
Choose $v$ not free in $a$ and take $p_2 \ldef (a \ne v)$,
so that $\ssub{v}{\tofn{k}}{p_2}$ is $a \ne \tofn{k}$
and $\ssub{v}{b}{p_2}$ is $a \ne b$.
Applying \irl{{=}subs} with $\tofn{k} = b$
yields $\Gamma \vdash a \ne b$.
\end{proof}

The asymmetry between the two halves of that proof
is the open-term generalization at work,
and it is worth recognizing on sight,
because the same shape returns in
\cref{lem:app:spine-transport}.
For closed terms
both transports could substitute at variable $0$,
because a closed context has no variable $0$ to disturb.
Open terms require exactly one change:
the substitution variable is chosen
fresh for the surrounding context.

\subsection{Head Shapes and the Transfer Invariant}
\label{sec:app:complete:xfer}

\begin{defi}[canonical terms; \textsf{rhd\_canon}]
\label{defi:app:canon}
The five head shapes of \cref{fig:sem}
and their \emph{canonical terms} $\canon{w}$ are
\[
\begin{array}{r@{\quad}l}
  w = \lambda b & \canon{w} = \lambda b \\
  w = \cR & \canon{w} = \cR \\
  w = \cR\,g & \canon{w} = \cR \cdot g \\
  w = \cR\,g\,h & \canon{w} = \cR \cdot g \cdot h \\
  w = \cA & \canon{w} = \cA
\end{array}
\]
\end{defi}

Head shapes are a five-way tag
carrying zero, one, or two terms,
so $\canon{\cdot}$ is a total injective function
from tags to terms.
The appendix presents head shapes
through their canonical terms throughout,
with the understanding that
``$t$ head-evaluates to $w$'' asserts the tag,
not any equality of terms.

Now the head motive.

\begin{defi}[the transfer invariant; \textsf{hev\_ok}]
\label{defi:app:xfer}
For a term $t$ and a head shape $w$,
say that \emph{$t$ transfers to $w$},
written $\hok{t}{w}$,
when both of the following hold.
\begin{enumerate}
\item $\wf(\canon w)$.
\item For every spine $\vec u$ of well-formed terms
  and every natural number $m$:
  if $\emptyset \vdash (\canon w \sap \vec u) = \tofn{m}$
  then $\emptyset \vdash (t \sap \vec u) = \tofn{m}$.
\end{enumerate}
\end{defi}

Three features of this definition
are each consumed by the induction,
and a reader re-mechanizing the proof
will not close the induction without them.

\paragraph{Quantifying over all spines is what makes the
motive inductive.}
The naive invariant --
if $\canon w$ provably has the value $m$, so does $t$ --
is the instance $\vec u = \spl{}$ of clause 2,
and it does not compose.
In the head-$\beta$ case
the induction hypothesis has to be applied
at the extended spine $u \cns \vec u$,
with one further argument in front,
and in the recursor cases
at the spine with the scrutinee pushed on.
A one-application invariant simply cannot be chained
through those clauses.
The spine quantification is the one thing the proof sketch of
\cref{sec:metatheory:complete} omits,
and it is the reason the head motive
looks more complicated than the theorem it serves.

\paragraph{The direction of transfer is reduct to redex.}
Clause 2 starts from a provable valuation of
$\canon w$, the more reduced object,
and produces one for $t$, the less reduced.
Reduct to redex is exactly the direction of
the conversion rules \irl{conv} and \irl{conv{=}},
whose premises speak about the reduct
and whose conclusions speak about the redex.
The invariant is shaped to be discharged by those rules.

\paragraph{Clause 1 is the only source of well-formedness
for the hidden arguments.}
Clause 1 is not decoration.
Consider the head-evaluation clause
that dispatches the recursor at a zero scrutinee:
its conclusion is about $t \cdot m$,
while its premises mention $g$ and $h$,
which occur nowhere in $t \cdot m$.
Well-formedness of the conclusion term
therefore says nothing whatever about $g$ and $h$,
and structural decomposition cannot supply it.
What does supply it is
the induction hypothesis for the first premise,
namely $\hok{t}{\cR\,g\,h}$,
whose clause 1 asserts $\wf(\cR \cdot g \cdot h)$
and hence $\wf(g)$ and $\wf(h)$.
Without clause 1 the induction can neither
discharge the side conditions of
the recursor head-step rules
nor apply the induction hypotheses for the branch premises.
The same mechanism supplies $\wf(b)$
in the $\beta$ cases, out of $\hok{t}{\lambda b}$.

Note that $\hok{t}{w}$ does not itself require $\wf(t)$.
Well-formedness of $t$ is carried by the motive,
as \cref{sec:app:complete:plan} described,
not by the invariant.

Everything specific to open terms
is confined to the following lemma.

\begin{lem}[recursor-scrutinee transport under a spine;
  \textsf{ir\_eq\_spine\_R3}]
\label{lem:app:spine-transport}
Suppose $\emptyset \vdash m = x$,
and $\emptyset \vdash
  ((\cR \cdot g \cdot h \cdot x) \sap \vec u) = \tofn{r}$,
with $\wf(g)$, $\wf(h)$,
and every element of $\vec u$ well-formed.
Then $\emptyset \vdash
  ((\cR \cdot g \cdot h \cdot m) \sap \vec u) = \tofn{r}$.
\end{lem}

In words: an evaluated scrutinee may be replaced by
a term it provably equals,
in the recursor's third argument position,
underneath an arbitrary argument spine.
At every use below, $x$ is a syntactic normal form
of the scrutinee -- either $\z$ or $\suc \tofn{k}$ --
and $m$ is the actual, possibly open,
unevaluated scrutinee term.

\begin{proof}
The transport idiom, applied at the recursor's third argument.

\emph{Choose the substitution variable.}
Let $v$ be a variable not occurring free in the term
$((\cR \cdot g \cdot h \cdot \z) \sap \vec u) = \tofn{r}$,
a copy of the target formula
with the constant $\z$ standing in
for the site about to be marked.
Decomposing freshness through spine application
and the constructors,
$v$ is free neither in $g$, nor in $h$,
nor in any element of $\vec u$,
nor in $\tofn{r}$, numerals being closed.

\emph{Form the substitution context.}
Let
\[
  p \;\ldef\;
  ((\cR \cdot g \cdot h \cdot v) \sap \vec u) = \tofn{r},
\]
the same formula with $v$ at the marked site.
The context is well-formed,
by well-formedness of $g$, $h$,
the spine elements and the numeral.

\emph{Compute the two instances.}
For \emph{any} term $y$,
\[
  \ssub{v}{y}{p} \;=\;
  ((\cR \cdot g \cdot h \cdot y) \sap \vec u) = \tofn{r} .
\]
The computation is where the freshness chosen above is spent.
Substitution distributes through spine application
and through the equality and application constructors,
and leaves $g$, $h$, every element of $\vec u$,
and $\tofn{r}$ unchanged,
because $v$ is free in none of them.
The single occurrence of $v$ is the marked one,
which becomes $y$.

\emph{Apply substitutivity.}
From $\emptyset \vdash m = x$ and \irl{{=}sym},
$\emptyset \vdash x = m$.
The lemma's second hypothesis is
exactly $\emptyset \vdash \ssub{v}{x}{p}$,
by the computation above with $y \ldef x$.
Rule \irl{{=}subs} therefore yields
$\emptyset \vdash \ssub{v}{m}{p}$,
which by the same computation with $y \ldef m$
is the desired conclusion.
\end{proof}

\Cref{lem:app:spine-transport} is
the entire cost of admitting open terms.
In the closed development,
all four call sites of this transport
substituted directly at variable $0$,
and the substitution contexts computed away for free,
since a closed $g$, $h$ and spine
cannot contain variable $0$.
Once $g$, $h$ or a spine element may be open,
variable $0$ may occur in the surrounding context,
and substituting at it would rewrite unintended sites.
Choosing $v$ fresh for the context is the repair,
and packaging the repair as one lemma
de-duplicates all four sites.
Nothing else in the completeness proof changed
when it was generalized from closed to open terms.

\subsection{The Mutual Induction}
\label{sec:app:complete:induct}

\begin{prop}[the completeness induction;
  \textsf{hev\_osr\_complete}]
\label{prop:app:induct}
Simultaneously, for all $s$, $t$, $w$, $a$, $r$:
\begin{enumerate}
\item if $t \hev{s} w$ and $\wf(t)$ then $\hok{t}{w}$;
\item if $a \bev{s} r$ and $\wf(a)$ then
  $\emptyset \vdash a = \tofn{r}$.
\end{enumerate}
\end{prop}

\begin{proof}
Mutual rule induction on
the simultaneous inductive definition of $\hev{}$ and $\bev{}$.
For each of the twenty-one clauses,
the motives for the clause's premises are assumed
and the motive for its conclusion is established.
The cases follow, grouped by the shape of the clause:
head cases in
\cref{sec:app:complete:head-base,sec:app:complete:spine-growth,%
sec:app:complete:beta,sec:app:complete:rec},
value cases in
\cref{sec:app:complete:zero,sec:app:complete:eq,%
sec:app:complete:neg,sec:app:complete:disj,sec:app:complete:vapp},
and the two quantifier cases in
\cref{sec:app:complete:allT,sec:app:complete:allF}.
\end{proof}

\subsubsection{Head bases}
\label{sec:app:complete:head-base}

The three clauses
$\lambda b \hev{s+1} \lambda b$,
$\cR \hev{s+1} \cR$
and $\cA \hev{s+1} \cA$
say that each canonical form
head-evaluates to its own shape in one step.
In each, the conclusion term \emph{is} $\canon w$,
so clause 2 of \cref{defi:app:xfer}
is the identity implication
and clause 1 is the assumed well-formedness
of the conclusion term.
There is nothing to prove.

\subsubsection{Spine growth}
\label{sec:app:complete:spine-growth}

Two clauses let the recursor accumulate arguments:
from $t \hev{s} \cR$ conclude $t \cdot g \hev{s+1} \cR\,g$,
and from $t \hev{s} \cR\,g$
conclude $t \cdot h \hev{s+1} \cR\,g\,h$.

For the first, assume $\wf(t \cdot g)$,
hence $\wf(t)$ and $\wf(g)$,
and let the induction hypothesis give $\hok{t}{\cR}$.
Clause 1 of the goal asks for
$\wf(\cR \cdot g)$, which follows.
For clause 2, let $\vec u$ be well-formed and suppose
$\emptyset \vdash ((\cR \cdot g) \sap \vec u) = \tofn{m}$.
Since spine application is left-nested,
$(\cR \cdot g) \sap \vec u$
and $\cR \sap (g \cns \vec u)$ are the same term.
So this is a provable valuation of
$\canon{\cR} \sap (g \cns \vec u)$,
with $g \cns \vec u$ well-formed,
and clause 2 of $\hok{t}{\cR}$ gives
$\emptyset \vdash (t \sap (g \cns \vec u)) = \tofn{m}$,
which is $\emptyset \vdash ((t \cdot g) \sap \vec u) = \tofn{m}$.

The second clause is the same argument
with $\cR\,g$ in place of $\cR$
and $h$ pushed onto the spine,
plus one extra step:
clause 1 of the goal needs $\wf(\cR \cdot g \cdot h)$,
and $\wf(g)$ is not available from $\wf(t \cdot h)$.
The missing premise comes instead from clause 1 of
the induction hypothesis $\hok{t}{\cR\,g}$,
which asserts $\wf(\cR \cdot g)$ --
a first small instance of
the mechanism described in \cref{sec:app:complete:xfer}.

Neither case applies a single rule of the proof system.
The work is done by
the choice of invariant, not by the cases.

\subsubsection{Head \texorpdfstring{$\beta$}{beta}, and the conversion chain}
\label{sec:app:complete:beta}

From $t \hev{s} \lambda b$ and $\bsub{b}{u} \hev{s} w$,
conclude $t \cdot u \hev{s+1} w$.

Assume $\wf(t \cdot u)$, hence $\wf(t)$ and $\wf(u)$.
The first induction hypothesis gives $\hok{t}{\lambda b}$,
whose clause 1 says $\wf(\lambda b)$ and hence $\wf(b)$.
Since $\wf(b)$ and $\wf(u)$,
the $\beta$-instance $\bsub{b}{u}$ is well-formed,
so the second induction hypothesis
gives $\hok{\bsub{b}{u}}{w}$.
Clause 1 of the goal is clause 1 of that invariant.

For clause 2, let $\vec u$ consist of well-formed terms
and suppose
$\emptyset \vdash (\canon w \sap \vec u) = \tofn{m}$.
\begin{enumerate}
\item Clause 2 of $\hok{\bsub{b}{u}}{w}$ gives
  $\emptyset \vdash (\bsub{b}{u} \sap \vec u) = \tofn{m}$.
  Call this $(\dagger)$.
\item The $\beta$ head step
  $(\lambda b) \cdot u \hs \bsub{b}{u}$ holds --
  its side condition is $\wf(u)$ --
  and head steps propagate through spines, so
  $((\lambda b) \cdot u) \sap \vec u \hs \bsub{b}{u} \sap \vec u$.
\item By \cref{lem:app:eqD}, $(\dagger)$ gives
  $\emptyset \vdash (\bsub{b}{u} \sap \vec u) \jnat$.
  Rule \irl{conv{=}}, whose premises are exactly
  a value for the reduct and the step itself,
  therefore yields the equation
  $\emptyset \vdash
    (((\lambda b) \cdot u) \sap \vec u) =
    (\bsub{b}{u} \sap \vec u)$.
\item Chaining that equation with $(\dagger)$
  by \cref{lem:app:eq-trans} gives
  $\emptyset \vdash
    (((\lambda b) \cdot u) \sap \vec u) = \tofn{m}$.
\item The left-hand term is
  $\canon{\lambda b} \sap (u \cns \vec u)$,
  again by left-nesting.
  Clause 2 of $\hok{t}{\lambda b}$,
  applied at the spine $u \cns \vec u$,
  gives $\emptyset \vdash (t \sap (u \cns \vec u)) = \tofn{m}$,
  which is
  $\emptyset \vdash ((t \cdot u) \sap \vec u) = \tofn{m}$.
\end{enumerate}

Steps 1 to 4 are a pattern that recurs
in every case below where the head does something.
Call it the \emph{conversion chain}:
obtain a provable valuation of the reduct
from an induction hypothesis;
turn the head step into a provable equation
$\text{redex} = \text{reduct}$ by \irl{conv{=}},
whose side condition -- the reduct has a value --
is precisely what the first step delivers
through \cref{lem:app:eqD};
then chain the two by transitivity.
Later cases invoke the chain by name.

It is worth noting which conversion rule is used.
The bare rule \irl{conv} transports \emph{truth} backwards
along a head step,
but what is needed here is not truth
but a numeric valuation,
so the rule used is \irl{conv{=}},
which records the step as a provable equation
at the cost of one side condition.
That side condition is where the design decision
to make the conversion rules primitive pays off,
and it is where \cref{sec:system:proof}'s remark that
``the proof system converts along exactly
the head steps that evaluation takes''
is cashed out.

\subsubsection{Head recursor dispatch}
\label{sec:app:complete:rec}

From $t \hev{s} \cR\,g\,h$, $m \bev{s} 0$ and $g \hev{s} w$,
conclude $t \cdot m \hev{s+1} w$.

Assume $\wf(t \cdot m)$, hence $\wf(t)$ and $\wf(m)$.
\begin{enumerate}
\item The first induction hypothesis gives
  $\hok{t}{\cR\,g\,h}$, whose clause 1 says
  $\wf(\cR \cdot g \cdot h)$
  and hence $\wf(g)$ and $\wf(h)$.
  This is the case promised in \cref{sec:app:complete:xfer}:
  $g$ and $h$ are invisible in the conclusion term.
\item The third induction hypothesis,
  now applicable because $\wf(g)$,
  gives $\hok{g}{w}$.
\item The second induction hypothesis,
  applicable because $\wf(m)$,
  is the \emph{value} motive:
  $\emptyset \vdash m = \z$.
\end{enumerate}
Clause 1 of the goal is clause 1 of step 2's invariant.
For clause 2, let $\vec u$ be well-formed and suppose
$\emptyset \vdash (\canon w \sap \vec u) = \tofn{m'}$.
Clause 2 of $\hok{g}{w}$ gives
$\emptyset \vdash (g \sap \vec u) = \tofn{m'}$.
The head step $\cR \cdot g \cdot h \cdot \z \hs g$ holds,
its side condition $\wf(h)$ coming from step 1,
and propagates through the spine,
so the conversion chain yields
$\emptyset \vdash
  ((\cR \cdot g \cdot h \cdot \z) \sap \vec u) = \tofn{m'}$.
Transporting the scrutinee by
\cref{lem:app:spine-transport},
with $x \ldef \z$ and the equation of step 3,
gives
$\emptyset \vdash
  ((\cR \cdot g \cdot h \cdot m) \sap \vec u) = \tofn{m'}$.
The left-hand term is
$\canon{\cR\,g\,h} \sap (m \cns \vec u)$,
so clause 2 of $\hok{t}{\cR\,g\,h}$ at that spine gives
$\emptyset \vdash ((t \cdot m) \sap \vec u) = \tofn{m'}$.

The successor clause is the same proof
with three substitutions.
The scrutinee value is $k{+}1$ rather than $0$,
so step 3 gives $\emptyset \vdash m = \suc \tofn{k}$;
the branch term is
$c \ldef h \cdot \tofn{k} \cdot
  (\cR \cdot g \cdot h \cdot \tofn{k})$
rather than $g$, well-formed by step 1;
and the head step is the recursor-successor step
$\cR \cdot g \cdot h \cdot (\suc \tofn{k}) \hs c$,
whose side condition is $\wf(g)$.
The transport uses $x \ldef \suc \tofn{k}$.

Why is a transport needed at all?
Because the head step for the recursor
fires only at a \emph{syntactic} numeral scrutinee.
That restriction is deliberate:
the head-step relation must stay
primitive-recursively decidable,
so it cannot evaluate.
Evaluation, by contrast,
dispatches on the scrutinee's \emph{value},
which the term $m$ need only compute to.
The gap between ``the term is the numeral''
and ``the term provably equals the numeral''
is bridged exactly once,
by \cref{lem:app:spine-transport},
and that is the only place in the completeness proof
where full substitutivity is needed
for anything beyond the derived equality algebra.

\subsubsection{Values: zero and successor}
\label{sec:app:complete:zero}

For $\z \bev{s+1} 0$,
the goal is $\emptyset \vdash \z = \tofn 0$,
that is $\emptyset \vdash \z = \z$,
that is $\emptyset \vdash \z \jnat$ --
the axiom \irl{\z{}I}.
The passage from $\z \jnat$ to $\z = \z$
is the notational identity noted in
\cref{sec:app:complete:prelim};
the formalization routes it through \irl{{=}refl},
which is a no-op for the same reason.

For the successor clause,
$a \bev{s} r$ implies $\suc a \bev{s+1} r{+}1$:
well-formedness of $\suc a$ gives $\wf(a)$,
so the induction hypothesis gives
$\emptyset \vdash a = \tofn{r}$,
and \cref{lem:app:suc-cong} gives
$\emptyset \vdash \suc a = \suc \tofn{r}$,
whose right-hand side is $\tofn{r{+}1}$.

\subsubsection{Values: grounded equality}
\label{sec:app:complete:eq}

From $a \bev{s} n_a$ and $b \bev{s} n_b$,
conclude $(a = b) \bev{s+1} \tv{n_a = n_b}$:
an equality atom evaluates to the truth value of
the metalevel equation between the two computed numbers.

Assume $\wf(a = b)$, hence $\wf(a)$ and $\wf(b)$.
The two induction hypotheses give
$\emptyset \vdash a = \tofn{n_a}$
and $\emptyset \vdash b = \tofn{n_b}$,
and the goal is
$\emptyset \vdash (a = b) = \tofn{\tv{n_a = n_b}}$.
Split on whether $n_a = n_b$.

If $n_a = n_b$ then $\tofn{n_a}$ and $\tofn{n_b}$
are the same term,
so chaining the second hypothesis' symmetric image
with the first gives $\emptyset \vdash a = b$:
the equality atom is \emph{derivable}.
Rule \irl{T{=}} converts bare derivability
into the numeral valuation,
$\emptyset \vdash (a = b) = \tofn 1$,
and $\tv{n_a = n_b} = \tvT = 1$.

If $n_a \ne n_b$ then \cref{lem:app:ne-transport},
whose well-formedness premises are in hand,
gives $\emptyset \vdash a \ne b$,
which unfolds by definition to
$\emptyset \vdash \neg(a = b)$:
the equality atom is \emph{refutable}.
Rule \irl{F{=}} gives
$\emptyset \vdash (a = b) = \tofn 0$,
and $\tv{n_a = n_b} = \tvF = 0$.

This case is the paradigm for
every remaining value clause of a connective,
and it is where the truth-value bridge earns its keep.
An evaluation clause produces a \emph{number}, $1$ or $0$,
whereas the proof system's natural output is
a derivation of the atom or of the atom's negation.
The rules \irl{T{=}} and \irl{F{=}}
are the two-way dictionary between the two.
Each remaining connective case has the same three beats:
read the induction hypotheses as boolean valuations;
convert them to bare derivability or refutability
by the bridge rules;
apply the connective's own introduction rule
and convert back.

\subsubsection{Values: negation}
\label{sec:app:complete:neg}

From $a \bev{s} \tvF$ conclude $\neg a \bev{s+1} \tvT$.
Assume $\wf(\neg a)$, hence $\wf(a)$.
The induction hypothesis gives
$\emptyset \vdash a = \tofn 0$,
and \irl{F{=}} read right-to-left gives
$\emptyset \vdash \neg a$;
then \irl{T{=}} gives
$\emptyset \vdash (\neg a) = \tofn 1$.

From $a \bev{s} \tvT$ conclude $\neg a \bev{s+1} \tvF$.
The induction hypothesis gives
$\emptyset \vdash a = \tofn 1$,
and \irl{T{=}} read right-to-left gives $\emptyset \vdash a$;
double-negation introduction \irl{{\neg}{\neg}}
gives $\emptyset \vdash \neg\neg a$;
and \irl{F{=}} applied to $\neg a$ gives
$\emptyset \vdash (\neg a) = \tofn 0$.

The second clause is where \rga's boolean negation
surfaces inside the completeness proof:
to give a true formula's negation the value $\tvF$,
one must derive $\neg\neg a$ from $a$.
That step is double-negation \emph{introduction}
as a rule on derivations,
which is unproblematic even intuitionistically.
What makes the case work is that
the semantics' negation clause flips genuine booleans only.
The case uses no elimination of double negation,
though \rga has that rule as well,
and it needs no excluded middle.

\subsubsection{Values: disjunction}
\label{sec:app:complete:disj}

All three clauses follow the three beats of
\cref{sec:app:complete:eq}.
Assume throughout that $a \lor b$ is well-formed,
hence both disjuncts are.

If $a \bev{s} \tvT$ then $(a \lor b) \bev{s+1} \tvT$:
the induction hypothesis and \irl{T{=}}
give $\emptyset \vdash a$;
rule \irl{{\lor}I_1}, whose side condition $\wf(b)$
is available, gives $\emptyset \vdash a \lor b$;
and \irl{T{=}} gives
$\emptyset \vdash (a \lor b) = \tofn 1$.
The clause for a true right disjunct is symmetric,
using \irl{{\lor}I_2} and the side condition $\wf(a)$.

If $a \bev{s} \tvF$ and $b \bev{s} \tvF$
then $(a \lor b) \bev{s+1} \tvF$:
the two induction hypotheses and \irl{F{=}}
give $\emptyset \vdash \neg a$ and $\emptyset \vdash \neg b$;
rule \irl{{\lor}I_3} gives
$\emptyset \vdash \neg(a \lor b)$;
and \irl{F{=}} gives
$\emptyset \vdash (a \lor b) = \tofn 0$.

The negative introduction rule \irl{{\lor}I_3}
is a primitive of \rga.
Classically and intuitionistically,
$\neg(a \lor b)$ follows from $\neg a$ and $\neg b$
by negation introduction,
discharging a hypothesis $a \lor b$.
In \rga, $\neg$ is the boolean flip rather than
implication of a falsum,
and no rule introduces a negation by refuting a hypothesis,
so \irl{{\lor}I_3} must be primitive.
This case is exactly why.
The semantics has a clause making a disjunction false
when both disjuncts are,
and completeness must answer that clause with a rule.
The general point is worth stating here:
\rga's rule set is shaped by its evaluation clauses,
one introduction rule per clause,
and the completeness proof is
the systematic verification of that correspondence.

\subsubsection{Values: application}
\label{sec:app:complete:vapp}

Value evaluation of an application dispatches on
the head shape of the function position,
exactly as head evaluation does,
giving three clauses:
\begin{itemize}
\item from $t \hev{s} \lambda b$ and $\bsub{b}{u} \bev{s} r$,
  conclude $(t \cdot u) \bev{s+1} r$;
\item from $t \hev{s} \cR\,g\,h$, $m \bev{s} 0$ and $g \bev{s} r$,
  conclude $(t \cdot m) \bev{s+1} r$;
\item from $t \hev{s} \cR\,g\,h$, $m \bev{s} k{+}1$
  and $(h \cdot \tofn{k} \cdot
    (\cR \cdot g \cdot h \cdot \tofn{k})) \bev{s} r$,
  conclude $(t \cdot m) \bev{s+1} r$.
\end{itemize}

Each is the corresponding head case of
\cref{sec:app:complete:beta} or \cref{sec:app:complete:rec},
specialized to the empty spine,
with the head motive's conclusion replaced by
the value motive's.
For the $\beta$ clause, concretely:
assume $\wf(t \cdot u)$;
obtain $\hok{t}{\lambda b}$ and hence $\wf(b)$,
then $\wf(\bsub{b}{u})$;
the induction hypothesis for the second premise
is now the value motive,
giving $\emptyset \vdash \bsub{b}{u} = \tofn{r}$;
the conversion chain along
$(\lambda b) \cdot u \hs \bsub{b}{u}$ gives
$\emptyset \vdash ((\lambda b) \cdot u) = \tofn{r}$;
and reading the left side as
$\canon{\lambda b} \sap \spl{u}$
and applying clause 2 of $\hok{t}{\lambda b}$
at that one-element spine
gives $\emptyset \vdash (t \cdot u) = \tofn{r}$.
The two recursor clauses mirror
\cref{sec:app:complete:rec} with $\vec u$ empty,
using \cref{lem:app:spine-transport} at the empty spine
to transport the scrutinee.

These three clauses are not literally instances
of the head cases.
Their hypotheses use the value motive
where the head clauses used the head motive --
$\bsub{b}{u} \bev{s} r$
rather than $\bsub{b}{u} \hev{s} w$ --
and their conclusions are valuations
rather than invariants,
so the two families cannot be merged
in the formalization,
and the duplication there is real.
Mathematically, though, the argument is the same one:
the same conversion chain,
with the spine empty
and the head motive's transfer clause
applied at a one-element spine.

\subsubsection{The reflective quantifier: truth}
\label{sec:app:complete:allT}

From $t \hev{s} \cA$, a proof $P$ with $\quo P \le s$,
and $P \proves \{\vz \jnat\} \vdash \lft u \cdot \vz$,
conclude $(t \cdot u) \bev{s+1} \tvT$.

Assume $\wf(t \cdot u)$, hence $\wf(t)$ and $\wf(u)$,
and let the induction hypothesis for the head premise
give $\hok{t}{\cA}$.
\begin{enumerate}
\item \emph{Read the certificate as a derivation.}
  The clause's third premise says that
  a valid proof $P$ yields the judgment
  $\{\vz \jnat\} \vdash \lft u \cdot \vz$.
  By the definition of derivability,
  $\{\vz \jnat\} \vdash \lft u \cdot \vz$.
  The bound $\quo P \le s$ is discarded.
\item \emph{Apply universal introduction.}
  Rule \irl{{\cA}I} infers $\Gamma \vdash \cA \cdot f$
  from $\vz \jnat,\, \lft\Gamma \vdash \lft f \cdot \vz$.
  Take $\Gamma \ldef \emptyset$ and $f \ldef u$.
  The lifted empty context is the empty context,
  and adjoining $\vz \jnat$ to it
  gives the singleton $\{\vz \jnat\}$,
  so the rule's premise is
  syntactically the judgment produced in step 1 --
  same hypothesis set, same lift, same conclusion,
  with no adjustment whatever.
  Conclude $\emptyset \vdash \cA \cdot u$.
\item \emph{Bridge to the value.}
  Rule \irl{T{=}} gives
  $\emptyset \vdash (\cA \cdot u) = \tofn 1$.
\item \emph{Transfer along the head.}
  $\cA \cdot u$ is $\canon{\cA} \sap \spl{u}$,
  so clause 2 of $\hok{t}{\cA}$
  at the spine $\spl{u}$ gives
  $\emptyset \vdash (t \cdot u) = \tofn 1$,
  which is the goal since $\tvT = 1$.
\end{enumerate}

This case is the reflective quantifier's payoff,
and it repays a moment's attention.
Classically, completeness at the universal quantifier
is the hard case.
Knowing that every numeral instance is true,
one must produce a single finite proof of the universal,
and no rule of a classical system does that --
this is the point at which an $\omega$-rule
would become unavoidable.
In \rga the difficulty never arises,
because the semantics never says
that every instance is true.
The truth clause says instead that
there is a valid \rga proof of the schematic instance,
so completeness consists of
reading that proof off the evaluation derivation
and handing it to universal introduction,
which asks for precisely that judgment.

The alignment in step 2 is exact rather than merely close,
and the exactness is by design:
the \irl{allT} clause was written to be
\irl{{\cA}I}'s premise.
Step 2 goes through verbatim
only because the certificate judgment carries
$\lft u$ rather than $u$.
The lift is needed because
\irl{{\cA}I}'s premise speaks about $\lft u$:
an unlifted certificate about $u$
could be reconciled with that premise only for closed $u$,
so completeness at open terms would fail
(and so would substitution stability of evaluation,
\cref{exa:app:diag}).

\subsubsection{The reflective quantifier: falsity}
\label{sec:app:complete:allF}

From $t \hev{s} \cA$, a proof $P$ with $\quo P \le s$,
and $P \proves \emptyset \vdash \neg(u \cdot \tofn{i})$,
conclude $(t \cdot u) \bev{s+1} \tvF$.

Assume $\wf(t \cdot u)$;
the induction hypothesis gives $\hok{t}{\cA}$.
The certificate yields the derivation
$\emptyset \vdash \neg(u \cdot \tofn{i})$,
exhibiting the numeral $\tofn{i}$ as a refuting instance.
Rule \irl{{\neg\cA}I} infers
$\Gamma \vdash \neg(\cA \cdot f)$
from $\Gamma \vdash \neg(f \cdot a)$ and $\Gamma \vdash a \jnat$;
its first premise is the certificate's judgment
with $f \ldef u$ and $a \ldef \tofn{i}$,
and its second is
$\emptyset \vdash \tofn{i} \jnat$
from \cref{lem:app:num-nat}.
Conclude $\emptyset \vdash \neg(\cA \cdot u)$.
Rule \irl{F{=}} gives
$\emptyset \vdash (\cA \cdot u) = \tofn 0$,
and clause 2 of $\hok{t}{\cA}$ at the spine $\spl{u}$
gives $\emptyset \vdash (t \cdot u) = \tofn 0$,
the goal since $\tvF = 0$.

The \emph{habeas quid} premise of \irl{{\neg\cA}I} --
that the witness is a natural number --
is discharged here by the fact that
the semantics supplies the witness as a numeral
rather than as an arbitrary term.
The numeral form is why the \irl{allF} clause refutes
$u \cdot \tofn{i}$ and not $u \cdot a$ for arbitrary $a$.
The two design choices are locked together.
Had the semantics admitted an arbitrary refuted instance,
completeness would have needed
that instance to be provably a natural number,
which a refutation certificate does not supply.

\subsection{The Completeness Theorems}
\label{sec:app:complete:thms}

\begin{thmc}[value completeness, \cref{thm:complete-val};
  \textsf{rga\_complete\_val}]
If $a \bev{s} r$ and $\wf(a)$,
then $\emptyset \vdash a = \tofn{r}$.
\end{thmc}

\begin{proof}
The value motive of \cref{prop:app:induct},
with the well-formedness implication
discharged by the hypothesis.
\end{proof}

\begin{thmc}[open completeness, \cref{cor:complete-open};
  \textsf{rga\_complete\_open}]
If $\wf(c)$ and $c \bev{s} \tvT$ for some $s$,
then $\emptyset \vdash c$ --
for $c$ open or closed.
\end{thmc}

\begin{proof}
Value completeness gives
$\emptyset \vdash c = \tofn{\tvT}$,
that is $\emptyset \vdash c = \tofn 1$,
and rule \irl{T{=}} read right-to-left
gives $\emptyset \vdash c$.
\end{proof}

One further specialization is worth naming,
because it is the form that
the soundness argument consumes.

\begin{thmc}[closed completeness; \textsf{rga\_complete}]
If $\wf(c)$, $c$ is closed,
and $c$ is grounded-true under
the assignment sending every variable to $\bot$,
then $\emptyset \vdash c$.
\end{thmc}

\begin{proof}
Grounded truth under an assignment $A$ means
that the closure of $c$ under $A$ --
the result of substituting,
for each free variable,
the numeral of its assigned natural or the constant $\bot$ --
evaluates to $\tvT$ at some step index.
Since $c$ is closed,
the closure of $c$ under any assignment is $c$ itself,
so $c \bev{s} \tvT$ for some $s$,
and open completeness applies.
\end{proof}

This last statement is the bridge between
\cref{sec:system:sem}'s official definition of grounded truth,
which quantifies over assignments,
and the raw evaluation judgment
that \cref{prop:app:induct} inducts over.
\Cref{sec:app:sound} uses the bridge
in the \irl{{\cA}I} case of the soundness induction.

The \emph{open} form, however,
is the one that soundness needs most,
and the reason is worth stating here
rather than only where the need arises.
\Cref{sec:app:sound}'s induction proves
the strongest of its four judgment-truth invariants,
the one quantified over
all well-formed substitutions
rather than only over closing ones,
and that strength is what delivers
the open half of \cref{thm:iff}.
With open substitution images
the substituted hypotheses are open terms,
and the universal-introduction case
has to \emph{derive} those hypotheses
in order to discharge them from a certificate.
Closed completeness cannot reach them.
So the open-terms strand of this appendix
is not a generalization pursued for its own sake:
it is a prerequisite for the soundness argument
that \cref{sec:app:sound} runs,
and hence for the truth-provability coincidence
holding at open terms at all.

\subsection{Closing Remarks}
\label{sec:app:complete:remarks}

\paragraph{No model, no measure.}
The semantic side of the completeness statement
is an inductively defined evaluation relation,
with no model, interpretation function
or notion of validity,
and the proof is an induction over that relation.
There is no step-index induction,
no measure, and no descent through certificate codes;
the witnessed pair of \cref{sec:app:sound}
exists to give soundness the same property.

\paragraph{Open terms cost one lemma.}
Everything specific to open terms is confined to
\cref{lem:app:spine-transport}.
A reader who cares only about closed terms
may take \cref{defi:app:xfer}
without its fresh-variable subtlety
and read that lemma as ``substitute at variable $0$'',
giving up the open form of completeness,
which \cref{sec:app:sound}'s open soundness invariant consumes.

\paragraph{Rules and clauses correspond.}
Each evaluation clause is discharged by
one introduction rule of the connective it evaluates,
or by conversion along the head step it takes,
with the scrutinee transport of \cref{lem:app:spine-transport}
at recursor dispatch (\cref{sec:app:complete:rec}),
where evaluation uses the scrutinee's value
and the head step needs a syntactic numeral;
the two quantifier clauses are discharged by
the two positive quantifier rules.

%% file: app-sound.tex
\section{Soundness in Detail}
\label{sec:app:sound}

This appendix gives the soundness argument
(\cref{thm:sound})
in full detail,
together with the determinism and consistency results
(\cref{thm:det,thm:consistent})
that fall out along the way,
transcribed from the machine-checked development
(\texttt{RGA\_Sound}).
As in \cref{sec:app:complete},
\cref{sec:system}'s syntax, rules, and semantics
are assumed and cross-referenced rather than restated.
Unlike \cref{sec:app:complete},
this appendix is not self-contained:
the quantifier case of the main induction
consumes completeness,
in the closed form proved at the end of
\cref{sec:app:complete:thms}.

\subsection*{Why soundness is not the routine half}

In a conventional logic
soundness is the easy direction.
One fixes a semantics
and checks each rule against it,
case by case,
and the checking is local.
Two features of \rga break that pattern.

\emph{The semantics contains the proof system.}
Evaluating a universal runs a proof search,
so a soundness case that \emph{consumes} a universal --
universal elimination,
or elimination of a negated universal --
is handed a certificate:
a proof that is not a sub-proof of
the derivation under induction.
Trusting the certificate means
assuming soundness of an arbitrary proof,
which is the very thing being proved.
\Cref{sec:app:sound:witness} removes the circularity,
and it does so not by descending on certificate size
but by changing the semantics
that the induction runs against.

\emph{The obvious induction motive is provably false.}
The substitution rule can substitute
$\lambda$-terms for variables,
and $\lambda$-terms are not in
the value domain that assignments range over.
\Cref{sec:app:sound:motive} gives
the machine-checked counterexample
and the repair.

The route is therefore longer than the statement suggests.
\Cref{sec:app:sound:motive} fixes the motive.
\Cref{sec:app:sound:witness} fixes the semantics
the induction runs against.
\Cref{sec:app:sound:det} collects
what the new semantics buys immediately:
determinism, uniqueness and exclusivity of evaluation,
unconditionally.
The remaining sections build the transport tools,
run the main rule induction,
collapse the two semantics,
and close the truth-provability loop.

\subsection{The Soundness Invariants}
\label{sec:app:sound:motive}

Four notions of ``the judgment $\Gamma \vdash c$ is true''
appear in this appendix.
Meeting all four at once is easier than
meeting them one at a time,
because the reason there are four
is a single chain of implications:
\[
  \ogtrue{\Gamma}{c}
  \;\Longrightarrow\;
  \wgtrue{\Gamma}{c}
  \;\Longrightarrow\;
  \gtrue{\Gamma}{c}
  \;\Longrightarrow\;
  \agtrue{\Gamma}{c} .
\]
Reading from the right,
$\agtrue{\Gamma}{c}$ is
truth over \emph{assignments},
part~(1) of \cref{thm:sound}.
$\gtrue{\Gamma}{c}$ is truth over
\emph{closing substitutions},
part~(2) of \cref{thm:sound},
which is what makes the induction go through.
$\wgtrue{\Gamma}{c}$ is the same notion
over the \emph{witnessed} evaluation judgment
of \cref{sec:app:sound:witness},
and is the motive the main induction actually runs against.
The middle implication is the one that is not free:
it holds only once the witness collapse is available,
at which point the two predicates become equivalent.
The collapse is downstream of soundness,
which is why the induction runs on the witnessed predicate
and converts at the end rather than at the start.
$\ogtrue{\Gamma}{c}$ quantifies over
\emph{all well-formed} substitutions rather than closing ones,
and is what yields the open characterization
of \cref{thm:iff}.

Everything to the left of $\agtrue{\Gamma}{c}$
implies it,
so each stronger form also yields part~(1).

\subsubsection{Grounded truth, satisfaction, judgment truth}
\label{sec:app:sound:gtruth}

\begin{defi}[grounded truth and grounded consequence;
  \textsf{rtrue}, \textsf{rfalse}, \textsf{rhsat}, \textsf{rjt}]
\label{defi:app:gtruth}
Let $A$ be an assignment
and $\asub A$ the closing substitution it induces,
sending each variable $v$ to the numeral of $A(v)$
when $A(v)$ is a natural number,
and to the constant $\bot$ when $A(v) = \bot$.
\begin{itemize}
\item A term $t$ is \emph{grounded-true} under $A$
  when $\sof{\asub A}{t} \bev{s} \tvT$ for some step index $s$,
  and \emph{grounded-false} when
  $\sof{\asub A}{t} \bev{s} \tvF$ for some $s$.
  A term that reaches neither value at any index
  is \emph{ungrounded}:
  the term has no truth value.
\item $A$ \emph{satisfies} a hypothesis set $\Gamma$
  when every member of $\Gamma$ is grounded-true under $A$.
\item $c$ is a \emph{grounded consequence} of $\Gamma$,
  written $\agtrue{\Gamma}{c}$,
  when every assignment satisfying $\Gamma$
  makes $c$ grounded-true.
\end{itemize}
\end{defi}

Part~(1) of \cref{thm:sound} is exactly the claim that
derivability implies $\agtrue{\Gamma}{c}$.

\begin{rem}[a note about the mechanization,
  not about the mathematics]
\label{rem:app:encoding}
In the Isabelle sources,
evaluation results and assignment values
inhabit the same HOL type --
the naturals -- but are read differently.
An evaluation result $r$ in $t \bev{s} r$
is the natural number $r$ itself,
with the booleans $\tvT = 1$ and $\tvF = 0$.
An assignment value $A(v)$,
by contrast, is \emph{decoded} into
the flat Scott domain $\natbot$
by $0 \mapsto \bot$ and $n{+}1 \mapsto n$.
\Cref{sec:system:sem}'s presentation,
in which an assignment maps into $\natbot$ directly,
is the correct mathematical reading,
and this appendix uses it throughout.
The remark exists only so that
a reader auditing the appendix against the sources
is not misled by the shared type.
One visible consequence:
the assignment written $\lambda v.\,0$ in the sources --
in the closed-completeness statement of
\cref{sec:app:complete:thms}, for instance --
sends every variable to $\bot$, not to zero.
\end{rem}

\subsubsection{The motive failure}
\label{sec:app:sound:fails}

The obvious plan is to prove \cref{thm:sound}
by rule induction on the derivation,
with $\agtrue{\Gamma}{c}$ itself as the motive.
The plan cannot work.

\begin{prop}[the assignment motive is not inductive;
  \textsf{rjt\_vacuous\_varapp},
  \textsf{rjt\_sub1\_motive\_fails}]
\label{prop:app:motive-fails}
The closure property that the substitution rule \irl{sub} would need,
\[
  \text{if } \agtrue{\Gamma}{c}
  \text{ then } \agtrue{\ssub{v}{s}{\Gamma}}{\ssub{v}{s}{c}}
  \text{ for well-formed } s,
\]
fails.
\end{prop}

\begin{proof}
The counterexample comes in two halves.

\emph{Half 1: a hypothesis that applies a variable
is unsatisfiable.}
Let $h \ldef \vz \cdot (\z = \z)$,
the variable $0$ in function position
applied to a true equation.
Then $\agtrue{\{h\}}{c}$ holds for \emph{every} $c$,
vacuously, because no assignment satisfies $\{h\}$.
Fix an assignment $A$.
Closing $h$ under $A$ replaces $\vz$
by either the constant $\bot$ or a numeral,
so the closed hypothesis is $d \cdot (\z = \z)$
with $d$ one of those two closed terms.
Now neither $\bot$ nor any numeral has a head shape:
no evaluation clause head-evaluates
a constant, a zero, or a successor.
And every clause that gives an application a \emph{value} --
the three head-dispatch clauses
and the two quantifier clauses --
carries a head-evaluation premise on the function position.
So the closed hypothesis has no value at all,
in particular not $\tvT$.

\emph{Half 2: substituting a $\lambda$-term makes the hypothesis
satisfiable while the conclusion stays false.}
Take the conclusion $c \ldef (\z = \suc \z)$
and substitute the identity abstraction $\lambda \vz$
for variable $0$.
The substitution is a legal instance of \irl{sub}:
its only side conditions are well-formedness of
the context, the conclusion, and the substituend,
and all three hold.
The substituted judgment is
\[
  \{\, (\lambda \vz) \cdot (\z = \z) \,\}
  \;\vdash\; (\z = \suc \z),
\]
in which both terms are closed,
so the assignment is irrelevant.

The substituted hypothesis is grounded-true.
$\lambda \vz$ head-evaluates to
the head shape $\lambda \vz$ in one step;
its $\beta$-instance at the argument $\z = \z$
is $\z = \z$,
since substituting for the bound variable of
the body $\vz$ returns the argument;
and $\z = \z$ evaluates to $\tv{0 = 0} = \tvT$,
both sides evaluating to the number $0$.
The $\beta$-clause for value evaluation of an application
therefore gives the hypothesis the value $\tvT$.

The substituted conclusion is not grounded-true.
Suppose $(\z = \suc \z) \bev{s} \tvT$.
Inverting the equality clause,
which is the only clause that can conclude
a value for an equality atom,
the value is $\tv{n_a = n_b}$
where $\z \bev{} n_a$ and $\suc \z \bev{} n_b$.
Inverting zero gives $n_a = 0$;
inverting successor gives $n_b = m{+}1$
with $\z \bev{} m$, hence $m = 0$ and $n_b = 1$.
So the value is $\tv{0 = 1} = \tvF \ne \tvT$,
a contradiction.

The substituted judgment therefore fails
the assignment motive,
while the unsubstituted one satisfies it.
\end{proof}

The failure is not an accident of the particular term chosen,
and the general diagnosis is what makes
the repair below look forced rather than clever.
Assignment truth quantifies over assignments,
which range over the value domain $\natbot$ --
that is, over what evaluation can \emph{produce}.
The substitution rule, by contrast,
substitutes arbitrary well-formed \emph{terms},
and in \rga the terms include $\lambda$-abstractions,
which are never values:
functions are head shapes, never results
(\cref{sec:system:sem}).
So \irl{sub} can perform a replacement that
no assignment can simulate,
and the vacuity that assignments create
can be destroyed by it.
First-order predecessors never met the phenomenon.
In \bga and \pga every substituend is
numeral- or bottom-valued,
so substitution there \emph{is} simulable
by a change of assignment,
and the motive is inductive for the same reason
that it fails here.

One thing \cref{prop:app:motive-fails} does not show
is any unsoundness.
The vacuously-true judgment of Half 1 is not derivable in \rga.
The counterexample refutes an induction motive,
not the theorem the induction is trying to prove.

\subsubsection{The repair: substitutional judgment truth}
\label{sec:app:sound:repair}

\begin{defi}[closing substitutions and substitutional truth;
  \textsf{csubf}, \textsf{strue}, \textsf{shsat}, \textsf{sjt}]
\label{defi:app:sjt}
A \emph{closing substitution} is a function $\sigma$
from variables to terms
such that every image $\sigma(v)$
is closed and well-formed.
Write $\sof{\sigma}{t}$ for
the simultaneous, binder-aware substitution
of $\sigma$ into $t$;
when $\sigma$ is closing,
$\sof{\sigma}{t}$ is closed for every $t$.
\begin{itemize}
\item $t$ is \emph{substitutionally true} at $\sigma$,
  written $\gtrue{\sigma}{t}$,
  when $\sof{\sigma}{t} \bev{s} \tvT$ for some $s$.
\item $\sigma$ \emph{satisfies} $\Gamma$,
  written $\gtrue{\sigma}{\Gamma}$,
  when every member of $\Gamma$
  is substitutionally true at $\sigma$.
\item $\gtrue{\Gamma}{c}$ holds when
  every closing substitution satisfying $\Gamma$
  makes $c$ substitutionally true.
\end{itemize}
\end{defi}

The symbol $\models$ is overloaded on both sides,
by a substitution or an assignment on the left
and by a term or a hypothesis set on the right.
Which reading is meant is fixed by the arguments.

What makes this motive work where
the assignment motive did not
is a single closure property.

\begin{lem}[closing substitutions absorb substitution]
\label{lem:app:sub-compose}
Let $\sigma$ be a closing substitution
and $s$ a well-formed term,
and let $\sigma'$ agree with $\sigma$ everywhere except at $v$,
where $\sigma'(v) \ldef \sof{\sigma}{s}$.
Then $\sigma'$ is again a closing substitution,
and for every term $t$,
\[
  \sof{\sigma}{\ssub{v}{s}{t}} \;=\; \sof{\sigma'}{t} .
\]
\end{lem}

\begin{proof}
$\sof{\sigma}{s}$ is closed and well-formed
because $\sigma$ is closing and $s$ is well-formed,
so $\sigma'$ is closing.
The displayed equation is
a pointwise computation on the composite
substitution function,
by structural induction on $t$:
at a variable other than $v$ both sides
are the corresponding image of $\sigma$,
and at $v$ both sides are $\sof{\sigma}{s}$.
\end{proof}

The consequence is worth stating now
rather than rediscovering it inside the induction.
The substitution case of the soundness proof
reduces to instantiating the induction hypothesis
at a \emph{different closing substitution}.
No semantic simulation of substitution is needed:
none of the rule's content
has to be matched by a corresponding move
in the semantics.
The same remark disposes of
every substitution-shaped case in the induction,
and this compositionality is the whole reason for the repair.

Substitutional truth is a strengthening
rather than a change of subject,
which is what makes the repair legitimate.

\begin{lem}[assignments embed;
  \textsf{wft\_rvtrm}, \textsf{csubf\_asn},
  \textsf{strue\_asn}, \textsf{sjt\_rjt}]
\label{lem:app:embed}
For every assignment $A$,
the induced $\asub A$ of \cref{defi:app:gtruth}
is a closing substitution,
and $\gtrue{\asub A}{t}$ holds
exactly when $t$ is grounded-true under $A$.
Consequently $\gtrue{\Gamma}{c}$
implies $\agtrue{\Gamma}{c}$.
\end{lem}

\begin{proof}
The images of $\asub A$ are
the constant $\bot$ and the numerals,
all closed and well-formed,
so $\asub A$ is closing.
The second claim is a matter of unfolding:
the closure of $t$ under $A$
is by definition $\sof{\asub A}{t}$,
so the two definitions of truth
are the same statement.
The third follows,
assignments being a sub-class of the substitutions
that $\gtrue{\Gamma}{c}$ quantifies over.
\end{proof}

\subsubsection{The witnessed and open variants}
\label{sec:app:sound:variants}

Two further members of the same family
complete the chain displayed at
the start of \cref{sec:app:sound:motive}.

\emph{Witnessed truth}
$\wgtrue{\sigma}{t}$ and $\wgtrue{\Gamma}{c}$
are \cref{defi:app:sjt} verbatim,
with the witnessed evaluation judgment
of \cref{sec:app:sound:witness}
in place of the official one.
The witnessed predicate is the motive
that the main induction runs against.

\emph{Open witnessed judgment truth}
$\ogtrue{\Gamma}{c}$ is $\wgtrue{\Gamma}{c}$
quantified over all \emph{well-formed} substitutions --
those whose images are well-formed
but need not be closed --
rather than over closing ones alone.
Closing substitutions are the special case,
so $\ogtrue{\Gamma}{c}$ implies $\wgtrue{\Gamma}{c}$.
The gain is that the well-formed substitutions
include the identity, sending each variable to itself,
which is not closing.
Admitting the identity substitution
is exactly how the open characterization
of \cref{thm:iff} is obtained.

Two weakening facts hold of all four predicates,
immediately from the definitions:
satisfying a larger hypothesis set
implies satisfying any subset,
and a judgment true from fewer hypotheses
is true from more.

\subsection{The Witnessed Evaluation Pair}
\label{sec:app:sound:witness}

\subsubsection{Why not a step-indexed descent}
\label{sec:app:sound:nodescent}

The natural plan for breaking the circularity
of \cref{sec:app:sound}'s opening
is induction on the step index.
The \irl{allT} and \irl{allF} clauses
bound the certificate's G\"odel code by the index,
so a certificate consulted at index $s$
is in a definite sense smaller than $s$,
and one might hope to descend:
soundness at index $s$ appeals to
soundness at smaller indices,
alternating between the evaluation and proof layers.
The plan is coherent.
The plan also fails in practice,
and it fails in two independent ways.
Both are worth stating,
because either failure alone looks fixable.

\emph{A symmetric step-indexed invariant is anti-monotone
at the introduction rules.}
Tie hypothesis truth and conclusion truth to the same index $s$.
Then every rule that consumes an evaluation step
pushes the conclusion's index up
while the hypotheses stay fixed,
and the invariant does not propagate.
In \rga every value clause consumes a step,
so there is no residue of well-behaved rules to retreat to:
the invariant fails at every introduction rule,
not at a few.

\emph{An asymmetric invariant loses the bound.}
Take hypothesis satisfaction at the limit instead --
true at \emph{some} index --
and the certificates extracted by
inverting a universal's truth acquire unbounded codes.
The index that bounds a certificate
is the local evaluation index,
and the index of a conclusion produced by
an induction hypothesis is existentially quantified,
so no measure decreases through
the rules that consume certificates.
Bounding conclusions as a function of hypothesis indices
pushes the problem up one level,
and would additionally demand
a step-accounted version of completeness
for the universal-introduction case.

The resolution changes what the induction runs against,
rather than how the induction descends.
If the difficulty is that a certificate's semantic content
is available only by proving the certificate sound,
then record that content in the semantics itself,
as an extra premise.
The certificate stays where it is,
and alongside it sits the semantic fact it certifies.
An induction over the new relation
then finds the fact it needs
already inside the derivation at hand.

\subsubsection{The definition}
\label{sec:app:sound:wdef}

\begin{defi}[the witnessed evaluation pair;
  \textsf{whev}, \textsf{wosr}]
\label{defi:app:witnessed}
The \emph{witnessed evaluation pair}
is the pair of mutually inductive judgments
$t \whev{s} w$ and $t \wbev{s} r$
defined by exactly the clauses of \cref{fig:sem},
with the two quantifier clauses amended as follows.
\begin{quote}
\emph{Witnessed \irl{allT}.}
If $t \whev{s} \cA$,
and $P$ is a valid proof with $\quo P \le s$
and $P \proves \{\vz \jnat\} \vdash \lft u \cdot \vz$,
\textbf{and for every natural number $n$
there is some $s_n$ with
$(u \cdot \tofn{n}) \wbev{s_n} \tvT$},
then $(t \cdot u) \wbev{s+1} \tvT$.

\emph{Witnessed \irl{allF}.}
If $t \whev{s} \cA$,
and $P$ is a valid proof with $\quo P \le s$
and $P \proves \emptyset \vdash \neg(u \cdot \tofn{i})$,
\textbf{and there is some $s_i$ with
$(u \cdot \tofn{i}) \wbev{s_i} \tvF$},
then $(t \cdot u) \wbev{s+1} \tvF$.
\end{quote}
\end{defi}

The bold text is the entire addition.
All nineteen other clauses are
character-for-character the official definition,
with the witnessed judgments substituted throughout.
Three properties of the addition
each do real work below.

\paragraph{The added premises are positive,
which is why the definition is legitimate.}
Positivity is the point on which
the whole construction stands or falls.
A definition of evaluation
that referred to evaluation \emph{negatively} --
``no proof refutes any instance'',
say -- would not define anything:
the clause would not be monotone in
the relation being defined,
and the least-fixpoint construction
that gives inductive definitions their meaning
would not apply.
The witness premises mention $\wbev{}$
only in premise position and only affirmatively.
The operator whose least fixpoint is being taken
is therefore monotone,
the fixpoint exists,
and the witnessed pair is an ordinary inductive definition,
accepted by the same construction
that defines the official pair of \cref{fig:sem}.
The construction is unusual only in
being infinitely branching,
which the next paragraph takes up.

\paragraph{The witness indices are existential inside
the premise, and are deliberately unbounded.}
The witnessed \irl{allT} premise says that
for each $n$ there is \emph{some} index
at which the $n$-th instance evaluates to true.
The premise does not say that
a single index works for all $n$,
and the non-uniformity is the point.
The instance family need not evaluate
within any common bound,
and the clause compresses an unbounded family of facts
into one derivation step.
Formally, the definition is countably branching,
so its least fixpoint is reached at
a countable ordinal rather than at $\omega$,
and the well-foundedness bookkeeping that
a step-indexed descent would have had to do by hand
is absorbed into the ordinal of the inductive definition.

\paragraph{The outer step index is unaffected.}
The witness premises do not mention $s$,
so all the step-index bookkeeping of the official pair
carries over unchanged.

A caution about what has just been defined.
The witnessed pair is a \emph{proof device}.
The official semantics is not modified:
$\bev{}$ and $\hev{}$ remain exactly as
\cref{fig:sem} defines them,
and the reflective quantifier's story in \cref{sec:system:sem:all}
and the recursive-enumerability results of \cref{sec:power}
concern the official judgments alone.
The relation between the two pairs is a \emph{theorem},
proved in two directions at two different points --
the easy direction immediately below,
the hard direction only after soundness.
The ordering matters:
a reader who takes the collapse for granted
will not see why the architecture is needed at all.

\subsubsection{Projection, and the index ladder}
\label{sec:app:sound:project}

\begin{lem}[projection to the official pair;
  \textsf{whev\_hev}, \textsf{wosr\_osr}]
\label{lem:app:project}
If $t \whev{s} w$ then $t \hev{s} w$,
and if $t \wbev{s} r$ then $t \bev{s} r$.
\end{lem}

\begin{proof}
Mutual rule induction on the witnessed pair.
Each clause maps onto the identically-shaped official clause;
in the two quantifier clauses
the added premise is discarded.
\end{proof}

The converse -- that every official derivation is witnessed --
is the witness collapse,
and is emphatically not available yet.
The converse says that
every certificate really delivers what it promises,
which is soundness.
The collapse arrives in two stages,
and both wait on soundness:
the closed-term form as soon as soundness is proved,
and the form for arbitrary well-formed terms
once soundness has been carried through to
the truth-provability characterization.

\begin{lem}[zero steps and monotonicity;
  \textsf{whev\_step0}, \textsf{wosr\_step0},
  \textsf{whev\_mono}, \textsf{wosr\_mono}]
\label{lem:app:wmono}
Nothing evaluates at index $0$,
and both witnessed judgments are monotone in the index:
$t \whev{s} w$ and $s \le s'$ imply $t \whev{s'} w$,
and likewise for values.
\end{lem}

\begin{proof}
The first claim holds because
every clause concludes at a successor index.
For the second,
a mutual rule induction re-applies each clause
one index higher,
weakening the certificate bound $\quo P \le s$
to $\quo P \le s{+}1$;
arbitrary increments follow by induction on the increment.
The witness premises pass through untouched,
precisely because they do not mention the outer index.
\end{proof}

\subsection{What the Witnesses Buy}
\label{sec:app:sound:det}

This section establishes a single theorem
with an unusually strong statement.
At the witnessed level,
evaluation is deterministic
and the two judgments are mutually exclusive,
\emph{unconditionally} --
with no consistency hypothesis,
no appeal to soundness,
and no induction on step indices.
For the official pair
the same statement is not yet available.
Recovering it takes two stages and is deferred:
the closed-term case follows from the collapse
in \cref{sec:app:sound:collapse},
and the general case needs soundness itself
and waits until \cref{sec:app:sound:official}.

\subsubsection{Vacuity and inversion}
\label{sec:app:sound:inv}

Everything in this subsection follows by
one and the same move:
inspect the finitely many clauses that
could have concluded the judgment in question,
and observe which of them match.
The results are collected rather than proved individually.

\begin{lem}[vacuity]
\label{lem:app:vacuity}
No head shape at any index is had by
the constant $\bot$, zero, a successor term,
any numeral, an equality atom,
a negation, a disjunction, or a variable.
No value at any index is had by
the bare combinator $\cA$,
any abstraction, the bare recursor $\cR$, or a variable.
And if the function position $t$
has no head shape at any index,
then $t \cdot u$ has no value at any index.
\end{lem}

\begin{proof}
Head evaluation concludes only about
abstractions, the two combinator constants,
and applications,
which disposes of the first list;
value evaluation concludes only about
zero, successors, the three connective shapes,
and applications,
which disposes of the second.
For the third claim,
all five clauses that give an application a value
carry a head-evaluation premise
on the function position.
\end{proof}

The third claim is what made
Half 1 of \cref{prop:app:motive-fails} work:
an application whose function is
a variable or a value term is stuck.
All of these facts hold verbatim of
the official pair as well,
and \cref{sec:app:complete} used the official versions.

\begin{lem}[value inversions]
\label{lem:app:val-inv}
Each arithmetic and connective shape
is concluded by exactly one clause,
two for negation and three for disjunction,
so evaluation of such a shape inverts cleanly.
\begin{itemize}
\item $\z \wbev{s} r$ forces $r = 0$.
\item $\suc a \wbev{s} r$ forces $s = s'{+}1$,
  $r = r_a + 1$, and $a \wbev{s'} r_a$.
\item $(a = b) \wbev{s} r$ forces $s = s'{+}1$
  and the existence of $n_a, n_b$ with
  $a \wbev{s'} n_a$, $b \wbev{s'} n_b$
  and $r = \tv{n_a = n_b}$.
\item $(\neg a) \wbev{s} r$ forces $s = s'{+}1$
  and a boolean $\beta$ with
  $a \wbev{s'} \tv{\beta}$ and $r = \tv{\neg\beta}$.
\item $(a \lor b) \wbev{s} r$ forces $s = s'{+}1$
  and one of three alternatives:
  $a$ evaluates to $\tvT$ and $r = \tvT$;
  $b$ evaluates to $\tvT$ and $r = \tvT$;
  or both evaluate to $\tvF$ and $r = \tvF$.
\end{itemize}
Numerals evaluate to themselves and to nothing else:
$\tofn{k} \wbev{s} r$ holds for some $s$ exactly when $r = k$.
\end{lem}

The disjunction inversion is a genuine three-way disjunction
and not a determination.
\rga's disjunction is Kleene-strong,
so a true disjunction need not record
which disjunct made it true,
and \cref{sec:app:sound:wdet} has to work slightly harder
in consequence.

\subsubsection{The application inversions}
\label{sec:app:sound:appinv}

The two application inversions are
the workhorses of everything that follows.
Each is the complete case analysis that
a judgment performs on an application,
and each says:
if an application evaluates,
one of the listed clauses concluded it,
and that clause's premises hold at the predecessor index.
Both are displayed in full,
because the case eliminations of
\cref{sec:app:sound:wdet} are run against them
alternative by alternative.

\begin{lem}[application inversions;
  \textsf{whev\_app\_inv}, \textsf{wosr\_app\_inv}]
\label{lem:app:app-inv}
\emph{Head evaluation} of $t \cdot u$ at index $s$
forces $s = s'{+}1$ and one of:
\begin{enumerate}
\item $t \whev{s'} \cR$, and the shape is $\cR\,u$;
\item $t \whev{s'} \cR\,g$ for some $g$,
  and the shape is $\cR\,g\,u$;
\item $t \whev{s'} \lambda b$ for some $b$,
  and $\bsub{b}{u} \whev{s'} w$;
\item $t \whev{s'} \cR\,g\,h$, $u \wbev{s'} 0$,
  and $g \whev{s'} w$;
\item $t \whev{s'} \cR\,g\,h$, $u \wbev{s'} k{+}1$,
  and $h \cdot \tofn{k} \cdot
    (\cR \cdot g \cdot h \cdot \tofn{k}) \whev{s'} w$.
\end{enumerate}
\emph{Value evaluation} of $t \cdot u$ at index $s$
forces $s = s'{+}1$ and one of:
\begin{enumerate}
\item $t \whev{s'} \lambda b$ and $\bsub{b}{u} \wbev{s'} r$;
\item $t \whev{s'} \cR\,g\,h$, $u \wbev{s'} 0$,
  and $g \wbev{s'} r$;
\item $t \whev{s'} \cR\,g\,h$, $u \wbev{s'} k{+}1$,
  and the recursor unfolding evaluates to $r$;
\item $t \whev{s'} \cA$, $r = \tvT$,
  some certificate $P$ of the schematic instance
  with $\quo P \le s'$,
  \textbf{and every numeral instance
  $u \cdot \tofn{n}$ evaluates to $\tvT$};
\item $t \whev{s'} \cA$, $r = \tvF$,
  some numeral $\tofn{i}$ and certificate $P$
  refuting that instance with $\quo P \le s'$,
  \textbf{and the instance $u \cdot \tofn{i}$
  evaluates to $\tvF$}.
\end{enumerate}
Alongside these,
an abstraction head-evaluates only to its own shape,
and likewise each of the two combinator constants.
\end{lem}

\begin{proof}
One application of the inversion principle
for an inductive definition,
with the non-application clauses discharged
by distinctness of the term constructors.
\end{proof}

Alternatives 4 and 5 of the value inversion
repay a careful reading,
and they are the crux of the section.
The bold clauses are the witnesses.
Inverting a witnessed universal hands back
not merely its certificate
but the \emph{semantic content} of that certificate,
directly, as a fact about evaluation,
with no soundness theorem standing in between.
That single observation is what makes
the next result unconditional.

\subsubsection{Determinism, uniqueness, exclusivity}
\label{sec:app:sound:wdet}

\begin{prop}[witnessed determinism and exclusivity;
  \textsf{whev\_det\_excl}, \textsf{wosr\_unique\_excl}]
\label{prop:app:wdet}
Simultaneously:
\begin{itemize}
\item if $t \whev{s} w$,
  then $t \whev{s'} w'$ implies $w' = w$
  for all $s'$ and $w'$,
  and $t$ has no value at any index;
\item if $t \wbev{s} r$,
  then $t \wbev{s'} r'$ implies $r' = r$
  for all $s'$ and $r'$,
  and $t$ has no head shape at any index.
\end{itemize}
\end{prop}

Note the placement of the quantifiers.
Both results are deterministic
across \emph{different} step indices,
not merely at a fixed one.

\begin{proof}
Mutual rule induction on the witnessed pair.
Each case supplies a derivation ending in some clause,
and requires showing that
any other derivation about the same term agrees,
and that the other judgment is empty on that term.
Every case has the same three beats:
invert the competing derivation
with the appropriate inversion lemma of
\cref{sec:app:sound:inv,sec:app:sound:appinv};
use the induction hypotheses to
kill the alternatives that disagree;
read off the result.

\emph{Base and connective cases.}
For the abstraction and the two combinators,
and for zero, successor, equality,
the two negation clauses
and the three disjunction clauses,
the matching inversion lemma
together with the induction hypotheses
for the immediate subterms gives the result.
Exclusivity is immediate,
none of these term shapes head-evaluating at all
(\cref{lem:app:vacuity}).
The disjunction cases need slight care,
the disjunction inversion being three-way;
but the induction hypotheses give
the value of each disjunct uniquely,
which resolves the three alternatives.
In the both-false case, for instance,
a competing derivation taking
the left-true alternative would give $a$
both the value $\tvT$
and, by induction hypothesis, the value $\tvF$.

\emph{Application cases.}
All five head cases and three value cases read alike.
Take the head recursor-at-successor clause
as representative.
The induction hypotheses give:
the function $t$ head-evaluates uniquely to $\cR\,g\,h$;
the scrutinee $m$ has the unique value $k{+}1$;
and the recursor unfolding head-evaluates uniquely to $w$
and has no value.
Invert a competing derivation of
$t \cdot m \whev{s'} w'$ by \cref{lem:app:app-inv}.
Each of the five alternatives opens with
a claim about the head shape of $t$,
and the first induction hypothesis forces that shape
to be $\cR\,g\,h$,
eliminating alternatives 1 to 3 outright.
Alternative 4 additionally claims
the scrutinee has value $0$,
which the second induction hypothesis refutes.
Only alternative 5 survives,
and its third premise together with
the third induction hypothesis gives $w' = w$.
The same elimination, run against the value inversion,
shows that $t \cdot m$ has no value:
alternatives 1 to 3 die on
the head shape or the scrutinee value,
and alternatives 4 and 5 die on the head shape,
a recursor spine not being $\cA$.

\emph{The certificate collision.}
Consider the witnessed \irl{allT} clause,
concluding $(t \cdot u) \wbev{s+1} \tvT$.
Its induction hypotheses give, first,
that $t$ head-evaluates uniquely to $\cA$,
and second -- from the witness premise,
whose derivations are sub-derivations
and therefore carry induction hypotheses of their own --
that \emph{every} numeral instance
$u \cdot \tofn{n}$ has the unique value $\tvT$.
Invert a competing derivation of
$(t \cdot u) \wbev{s'} r'$.
Alternatives 1 to 3 are eliminated by the first:
$\cA$ is neither an abstraction nor a recursor spine.
Alternative 4 gives $r' = \tvT$, as required.
Alternative 5 would give $r' = \tvF$,
and its own witness premise says that
the instance $u \cdot \tofn{i}$ evaluates to $\tvF$
for the particular numeral $\tofn{i}$ it refutes.
But the second induction hypothesis says
that same instance evaluates uniquely to $\tvT$,
and $\tvT \ne \tvF$.
So alternative 5 cannot occur.
Exclusivity is immediate,
$\cA$ being the head shape
and $\cA \cdot u$ being neither
a spine-growth nor a $\beta$ redex.
The \irl{allF} clause is the mirror image,
with the roles of the two witness premises exchanged.
\end{proof}

The certificate collision is the payoff of the architecture.
For the official pair,
the same collision -- one universal carrying
both a truth certificate and a falsity certificate --
amounts to a proof of $u \cdot \tofn{i}$,
obtained from the schematic certificate
by substituting the numeral,
together with a refutation of the same instance.
Nothing rules that out except the consistency of \rga,
which is downstream of soundness,
which is what is being proved.
At the witnessed level the collision closes semantically,
in one line:
the two witness premises concern the same instance
and assert opposite values,
and determinism at that instance
is exactly the induction hypothesis.
No consistency, no certificate soundness, no step descent.
The corresponding argument at the official level
has to be re-earned later,
and is the one piece of this section
that does not transport for free.

A skeptical reader should press on one point here,
and the answer is worth stating.
How can an induction have an induction hypothesis for
\emph{every} numeral instance at once,
when those instances are evaluated at
unrelated and unbounded step indices?
Nothing about the natural numbers is being inducted on.
The induction is a rule induction over
the witnessed definition,
and the witness premise is
a genuine premise of its clause,
so the derivations witnessing the instances
are genuine sub-derivations of the derivation at hand.
The rule-induction principle
supplies induction hypotheses for all of them uniformly,
and the countable branching is absorbed by
the least-fixpoint construction
that \cref{sec:app:sound:wdef} established the definition has.
The architecture trades
an infinitary but positive definition
for an ordinary induction.

Unbundled, \cref{prop:app:wdet} gives three statements
used constantly below:
witnessed values are unique,
witnessed head shapes are unique,
and no term has both.

\subsection{Substitution, in the Proof System and in Evaluation}
\label{sec:app:sound:subst}

The soundness induction runs against a motive
that quantifies over substitutions
(\cref{defi:app:sjt}),
so two things must be in place before the induction can start.
The proof system must admit
arbitrary well-formed simultaneous substitutions,
and evaluation must be stable under them.
Neither is automatic.
The proof system has only
a single-variable substitution rule,
and with open substituends
the naive iteration of that rule is wrong.
Evaluation's quantifier clauses embed certificates,
which have to be surgically re-derived
at the substituted terms.

\subsubsection{The substitution algebra}
\label{sec:app:sound:algebra}

Substitution and lifting on de Bruijn terms
\cite{debruijn72lambda}
obey a small algebra of interaction laws,
the same laws that the calculi of
explicit substitution make into
a rewriting system \cite{abadi91explicit}.
Each of the four used below
is proved by unfolding substitution and lifting
as instances of a common
``substitute a function of the variable'' operation
and comparing the two composites pointwise,
and each is cited by name from here on.
\begin{enumerate}
\item \emph{Substitution consumes a lift by shifting.}
  $\sof{\sigma}{\lft c} = \sof{\tau}{c}$,
  where $\tau(v) = \sigma(v{+}1)$.
  Lifting $c$ and then substituting
  is the same as substituting the shifted substitution directly,
  the lift having made every variable of $c$ one higher,
  so that each variable reads one entry further along $\sigma$.
\item \emph{The schematic instance.}
  Write $x \cns \sigma$ for $\sigma$ extended at the bottom
  by an arbitrary term $x$:
  the substitution sending variable $0$ to $x$
  and variable $v{+}1$ to $\sigma(v)$.
  The notation is the spine cons of
  \cref{sec:app:complete:prelim},
  a de Bruijn substitution being a stream of images
  read off by variable index.
  Then, whatever $x$ may be,
  \[
    \sof{(x \cns \sigma)}{\lft c} \;=\; \sof{\sigma}{c} .
  \]
  The lift makes variable $0$ unreachable,
  so whatever sits at position $0$ is ignored.
\item \emph{The $k$-fold version of the first law.}
  $\sof{\sigma}{\lftk{k}{c}} = \sof{\tau_k}{c}$,
  where $\tau_k(v) = \sigma(v{+}k)$.
\item \emph{Lifting a substituted term is substituting
  the lifted images.}
  $\lft{(\sof{\sigma}{t})} = \sof{\upsilon}{t}$,
  where $\upsilon(v) = \lft{(\sigma(v))}$.
\end{enumerate}

Law 2 is small,
and it is the mechanism behind
a claim \cref{sec:metatheory:complete} makes in passing,
that the certificate lift
keeps open evaluation stable under substitution.
A certificate's schematic variable
sits at position $0$;
the body it speaks about is lifted;
so a substitution pushed through the certificate
cannot reach the schematic variable,
and cannot be disturbed by whatever
is being substituted at position $0$.
\Cref{sec:app:sound:stability} spends the law
exactly once, to that effect.

One further fact, immediate:
if every image of $\sigma$ is closed
then $\sof{\sigma}{t}$ is closed for every $t$,
since every variable the substitution can reach
is replaced by a closed term
and the images contribute no free variables of their own.

\subsubsection{Simultaneous substitution is admissible}
\label{sec:app:sound:subf}

\begin{prop}[substitution admissibility; \textsf{ir\_subf}]
\label{prop:app:subf}
Let $\Gamma \vdash c$ with every hypothesis in $\Gamma$ well-formed,
and let $\sigma$ send each variable to a well-formed term,
open or closed.
Then $\sof{\sigma}{\Gamma} \vdash \sof{\sigma}{c}$,
where $\sof{\sigma}{\Gamma}$ applies $\sigma$ to every hypothesis.
\end{prop}

The obstacle is that \rga's rule \irl{sub}
replaces one variable at a time.
The obvious route is to iterate the rule
over the finitely many free variables involved,
and with open images that route is simply wrong.
Substitute $\sigma(0)$ for variable $0$ first;
a later step for variable $1$
then reaches into the freshly inserted copy of $\sigma(0)$
and rewrites \emph{its} free variable $1$ as well.
The composite is not the simultaneous substitution.
For closed images the problem cannot arise,
which is why the closed-fragment development
never needed any of the machinery below.

The repair is to substitute in two sweeps,
through a block of variable positions
that nothing else uses.

\begin{proof}
Write $\iter{\sigma}{h}(t)$ for the descending iteration
\[
  \ssub{h{-}1}{\sigma(h{-}1)}{} \circ \cdots \circ
  \ssub{0}{\sigma(0)}{}
\]
applied to $t$;
formally $\iter{\sigma}{0}(t) = t$ and
$\iter{\sigma}{h{+}1}(t) =
  \iter{\sigma}{h}(\ssub{h}{\sigma(h)}{t})$.
Two facts about such iterations are needed,
and both are inductions on $h$
that apply the primitive rule once per step.
Each iteration is \emph{admissible} --
derivability is preserved,
with well-formedness of the hypothesis set
maintained along the way.
And each iteration \emph{is} an ordinary
simultaneous substitution,
its composite substitution function
computed pointwise by unfolding the composition.
So it suffices to exhibit $\sof{\sigma}{\cdot}$
as a composite of two iterations.

Fix $h$ above every free variable of the judgment.
Such an $h$ exists,
the hypothesis set being finite
and each term having finitely many free variables.
Fix $N \ge h$ such that every image $\sigma(v)$ with $v < h$
has all its free variables below $N$,
taking the maximum of finitely many bounds.

\emph{Pass A: rename out of the way.}
Iterate with $\rho(v) \ldef v + N$ for $h$ steps,
renaming each of $0, \dots, h{-}1$
to the corresponding position in the block
$N, \dots, N{+}h{-}1$.
The composite sends $v$ to $v + N$ for $v < h$
and fixes everything else.
No step can disturb an earlier one,
because the iteration descends
and every target position is $\ge N \ge h$,
above every source.

\emph{Pass B: substitute at the block.}
Iterate for $N + h$ steps with
$\tau(v) \ldef v$ for $v < N$
and $\tau(v) \ldef \sigma(v - N)$ for $v \ge N$.
The steps at positions below $N$
are identity substitutions and do nothing.
The steps at positions $N, \dots, N{+}h{-}1$
insert the images.
No inserted image is ever re-substituted,
because every image's free variables lie below $N$,
while every remaining step is either
at a position $\ge N$, which the image does not mention,
or is an identity step.

\emph{Composition.}
For a term $t$ whose free variables all lie below $h$,
a free variable $i$ travels
$i \mapsto v_{N+i} \mapsto \sigma(i)$,
and every other position is untouched.
Two substitutions that agree on
the free variables of a term act identically on it,
so the composite is exactly $\sof{\sigma}{t}$.
Applying this to the conclusion
and to every hypothesis gives the claim.
\end{proof}

\Cref{prop:app:subf} is the proof-theoretic cash value of
\cref{sec:system:syntax}'s variable-range invariant.
\rga substitutes arbitrary well-formed terms
for free variables with no \emph{habeas quid} obligation
on the substituend,
justified semantically by
the substitutional soundness invariant (\cref{defi:app:sjt}),
in which free variables range over
all closed well-formed terms;
the proposition is the corresponding syntactic fact,
that the single-variable rule generates
the whole simultaneous-substitution admissibility,
for open images, with no restriction on the substituends.

One hypothesis of \cref{prop:app:subf}
needs to be remembered.
The underlying single-variable rule
requires $\Gamma$ to consist of well-formed formulas,
and \rga's rules do not enforce that:
hypotheses are unconstrained (\cref{sec:system:proof}).
The guard therefore propagates into
the main soundness statement,
and has to be removed afterwards
by a separate stripping argument.

Two specializations are used
in the certificate surgery below.
Applying \cref{prop:app:subf}
with the substitution that sends
variable $0$ to itself and variable $v{+}1$ to $\sigma(v)$
substitutes everything \emph{above} the schematic variable
while leaving the schematic variable in place.
Such a substitution is exactly what
a certificate for $\{\vz \jnat\} \vdash \lft u \cdot \vz$
must be pushed through,
since a substitution of that form
fixes the hypothesis $\vz \jnat$,
which therefore survives verbatim.
And the de Bruijn lift \emph{is} a substitution,
namely the renaming $v \mapsto v{+}1$,
which is pointwise well-formed;
so \cref{prop:app:subf} gives that
from $\Gamma \vdash c$ with well-formed hypotheses
one may infer $\lft\Gamma \vdash \lft c$.
That admissibility is what makes
the lifted context in \irl{{\cA}I}'s premise reachable at all,
rather than an obstacle.

Alongside these sit the ordinary structural facts --
weakening, cut, and the entailment-level forms of
the rules with several premises --
which are the bookkeeping lift of
\cref{sec:app:complete:prelim} and are used silently.
One of them has content and is named here,
because the soundness proof's
universal-introduction case turns on it.
\emph{Set cut}: if $\Delta \cup \Gamma \vdash c$
and every member of $\Delta$ is derivable from no hypotheses,
then $\Gamma \vdash c$.
The proof is an induction on the finite set $\Delta$,
discharging one hypothesis per step by ordinary cut.
Set cut is what turns
``every hypothesis is derivable''
into ``the certificate has no hypotheses''.

\subsubsection{Substitution stability of evaluation}
\label{sec:app:sound:stability}

Head shapes carry terms,
so a substitution acts on them too:
componentwise on the recursor's accumulated arguments,
as the identity on the two bare combinator shapes,
and on an abstraction's body under one binder,
the images being lifted by one before entering the body.
Write $\sof{\sigma}{w}$ for the result.

\begin{prop}[substitution stability;
  \textsf{hev\_subf\_stable}, \textsf{osr\_subf\_stable}]
\label{prop:app:stability}
Let $\sigma$ be pointwise well-formed,
its images open or closed.
If $t \hev{s} w$ then
$\sof{\sigma}{t} \hev{s_2} \sof{\sigma}{w}$ for some $s_2$,
and if $t \bev{s} r$ then
$\sof{\sigma}{t} \bev{s_2} r$ for some $s_2$.
\end{prop}

Substituting well-formed terms for free variables
can neither destroy a value nor change one.
Note that the resulting index is existential;
\cref{sec:app:sound:whyexistential} returns to that.

\begin{proof}
Mutual rule induction on the official pair.
Nineteen of the twenty-one clauses are mechanical:
substitution commutes with every term constructor,
so applying $\sigma$ to a clause's conclusion
produces exactly the clause instance
whose premises the induction hypotheses supply.
Two points of care are shared across those clauses.

\emph{Index alignment.}
Each induction hypothesis produces its own existential index,
and a clause with several premises needs them at a common index.
Take the maximum of the indices produced
and appeal to monotonicity of evaluation
(\cref{lem:app:wmono} and its official counterpart).

\emph{Commutation at $\beta$.}
In the $\beta$ clauses the induction hypothesis
speaks about $\sof{\sigma}{\bsub{b}{u}}$,
while the clause instance on the substituted side needs
$\bsub{\sof{\sigma'}{b}}{\sof{\sigma}{u}}$,
where $\sigma'$ is $\sigma$ lifted by one.
These are the same term:
substitution commutes with $\beta$-instantiation,
the substitution entering the body lifted by one
so as not to capture the bound variable.
In the recursor-successor clauses
the corresponding fact is that numerals are closed
and therefore substitution-invariant,
so $\sof{\sigma}{\tofn{k}} = \tofn{k}$.

The two quantifier clauses are the content.

\emph{The true certificate.}
The clause supplies a valid proof of
$\{\vz \jnat\} \vdash \lft u \cdot \vz$.
Push it through the schematic-variable-preserving
specialization of \cref{prop:app:subf},
with the images $v \mapsto \lft{(\sigma(v))}$,
which are well-formed because lifting preserves
well-formedness.
The hypothesis $\vz \jnat$ is fixed.
The conclusion becomes the result of substituting
$\vz$ at position $0$ and those lifted images above it
into $\lft u \cdot \vz$,
which computes:
the substitution distributes over the application;
it fixes $\vz$;
and on $\lft u$ it consumes the lift by
law 2 of \cref{sec:app:sound:algebra},
giving the substitution of the lifted images into $u$,
which by law 4 is $\lft{(\sof{\sigma}{u})}$.
So the transported judgment is
\[
  \{\vz \jnat\} \;\vdash\;
  \lft{(\sof{\sigma}{u})} \cdot \vz ,
\]
precisely the certificate that
the substituted universal needs,
on the nose and with no adjustment.
Extract a proof $Q$ of that judgment
and apply the clause at index
$\max(s_1, \quo Q) + 1$,
which meets the code bound by construction.

\emph{The false certificate.}
The clause supplies a valid proof of
$\emptyset \vdash \neg(u \cdot \tofn{i})$.
Here there is no schematic variable to protect,
so the full \cref{prop:app:subf} applies.
The empty hypothesis set stays empty,
the numeral is closed and hence fixed,
and the transported judgment is
$\emptyset \vdash \neg(\sof{\sigma}{u} \cdot \tofn{i})$.
Extract a proof and apply the clause as before.
\end{proof}

The same theorem holds for the witnessed pair,
by the same induction, clause for clause.
The only delta is that
the two quantifier clauses must additionally transport
their witness premises,
and they do so through their own induction hypotheses.
In the true clause the witness premise asserts
that every numeral instance $u \cdot \tofn{n}$
evaluates to $\tvT$;
the induction hypothesis for that premise,
applied at $\sigma$,
gives that $\sof{\sigma}{u \cdot \tofn{n}}$
evaluates to $\tvT$;
and numerals being substitution-invariant,
that term \emph{is} $\sof{\sigma}{u} \cdot \tofn{n}$,
which is the witness the substituted clause requires.
The false clause is the same
with one instance in place of a family.
One corollary of the witnessed form
is used repeatedly in \cref{sec:app:sound:vcr}:
witnessed values survive lifting,
a lift being a renaming substitution
and hence pointwise well-formed.

\subsubsection{Why the index must be existential}
\label{sec:app:sound:whyexistential}

The transported certificate in
\cref{prop:app:stability} is a \emph{different proof}
from the original,
produced by the substitution surgery,
and its G\"odel code bears no relation to the original's.
Since the clause requires the certificate's code
to be bounded by the step index,
the new evaluation must be allowed a larger index.
No index-preserving form of the theorem is possible.

The point is worth connecting back.
\Cref{sec:app:sound:nodescent} argued against
index-based reasoning on two grounds,
and warned that either ground alone might look fixable.
Here is a third, of a different kind:
not a difficulty in making an induction go through,
but a theorem that is inherently existential in the index.
Substitution stability is needed by the soundness proof
in any architecture whatsoever,
and no amount of care about invariants
will make it preserve an index it cannot preserve.

Two smaller facts about free variables
are recorded here for later use:
two substitutions that agree on
a term's free variables act identically on it,
and a variable is free in a finite hypothesis set
exactly when it is free in some member.
The second is consumed by the side condition
$i \notin \fv(\Gamma)$ of the induction rule \irl{ind}.

\subsection{The Value Congruence and the Simulations}
\label{sec:app:sound:vcr}

Several of \rga's rules move a subterm
to a different but equivalued one.
Substitutivity of equality rewrites at an arbitrary context;
schematic recursor unfolding replaces a scrutinee;
universal elimination applies a function
to a term rather than to a numeral.
Soundness of each such rule needs one semantic fact:
replacing a subterm by another with the same value
does not change what the whole term evaluates to.

That fact cannot be routed through provable equality.
\rga's grounded equality holds only between naturals,
so for function-valued subterms
there is no equation to state in the first place.
The replacement relation is therefore defined structurally,
and two things are proved about it --
that evaluation respects it,
and that derivability respects it.
The second is where \cref{sec:app:complete} re-enters.

\subsubsection{The relation}
\label{sec:app:sound:vcrdef}

\begin{defi}[value congruence; \textsf{vcrd}, \textsf{vcr}]
\label{defi:app:vcr}
For a depth $d \ge 0$,
the relation $X \vcd{d} Y$,
read \emph{$X$ and $Y$ are value-congruent at depth $d$},
is generated by three rules.
\begin{enumerate}
\item \emph{Replacement site.}
  If $X_0$ and $Y_0$ are well-formed
  and both evaluate, in the witnessed sense
  and at possibly different indices,
  to the same natural number $n$,
  then $\lftk{d}{X_0} \vcd{d} \lftk{d}{Y_0}$.
\item \emph{Variable.} $v \vcd{d} v$.
\item \emph{Skeleton.}
  If two terms have the same operator
  and the same binder count $nb$,
  and their argument lists are pairwise related
  at depth $d + nb$,
  then the terms are related at depth $d$.
\end{enumerate}
Write $X \vc Y$ for $X \vcd{0} Y$.
\end{defi}

So $X \vc Y$ says that
$X$ and $Y$ are the same term
except at some positions,
where a well-formed subterm has been replaced
by another with the same value.
The depth index tracks
how many binders enclose a position,
and rule 1 insists that a site at depth $d$
be the $d$-fold lift of a term
that makes sense outside all $d$ binders.
A replacement site therefore never mentions
any of the bound variables enclosing it.

That depth discipline is not a technicality.
Two consequences are used constantly below.
A site can be named by a fresh variable
without capture and without renaming
(\cref{sec:app:sound:extract}),
because single-variable substitution's own descent
through $nb$ binders reinserts exactly
the lifts the site already carries.
And $\beta$-instantiation passes sites through untouched,
because a site under the binder being eliminated
does not mention the bound variable --
which is why no well-formedness of the $\beta$ argument
is required anywhere in this section.
Assuming the sites closed
would give the same effects on closed terms only;
the depth index is what removes that assumption.

\begin{rem}[openness, three times]
\label{rem:app:openness}
The last observation completes a pattern
worth naming.
Generalizing this development from closed to open terms
cost, in each of the three places it could have been expensive,
exactly one localized device.
In \cref{sec:app:complete} it cost
\cref{lem:app:spine-transport} alone:
choose the substitution variable
fresh for the surrounding context
rather than substituting at variable $0$.
In \cref{sec:app:sound:subf} it cost
\cref{prop:app:subf}:
substitute in two passes through a fresh block,
rather than iterating the single-variable rule.
Here it costs the depth index on the congruence,
in place of a closedness assumption on sites.
In all three the closed case is the degenerate one,
in which the device does nothing,
and in none of the three does anything else change.

One episode in the development
looks like a fourth instance and is not.
The certificate lift of \cref{fig:sem}
was also introduced when the closed fragment
was generalized to open terms,
but the closed version there was not
merely harder to generalize.
Generalized, it was \emph{false},
as \cref{exa:app:diag} exhibits.
A device that repairs a wrong definition
belongs to a different category
from one that carries a correct argument
across a wider class of terms.
\end{rem}

\subsubsection{Closure properties}
\label{sec:app:sound:vcrclosure}

Reflexivity -- every term is related to itself
at every depth -- is a structural induction
on rules 2 and 3.
The constructor introductions follow from rule 3:
application and negation relate at the same depth,
abstraction with the body one depth deeper.

The property with real content is
stability under lifting:
if $X \vcd{d} Y$ then
$\lftat{v}{k}{X} \vcd{d} \lftat{v}{k}{Y}$
at any position $v$.
The proof is an induction on the relation,
and the site case splits.
If the lift position $v$ is at or below
the site's depth $d$,
the new lift merges into
the site's existing $d$-fold lift,
and the site is still a site:
its inner terms are unchanged,
so their common value is unchanged.
If $v$ is above $d$,
the lift exchanges past the site's lift
and lands \emph{inside} the site's inner terms.
Those terms have changed,
so their common value has to be re-established --
and that is exactly where
substitution stability is consumed,
in its witnessed form,
lifting being a renaming substitution
under which a witnessed value survives.

The dependency is worth stating plainly,
since it is the reason
\cref{sec:app:sound:subst} precedes this section:
the value congruence needs substitution stability
of the \emph{witnessed} pair,
and that is why the witnessed mirror of
the stability induction had to be proved at all.

A companion fact handles the easy direction:
a pair related at depth $a$,
lifted at a position at most $a$,
is related at depth $a + k$.
Here the lifts merely merge and no stability is needed.
The companion fact is what lets the substitution laws below
promote outer-level pairs to their local depth.

Two substitution laws close the section.
If $\sigma(v) \vc \sigma'(v)$ for every $v$,
then $\sof{\sigma}{p} \vc \sof{\sigma'}{p}$
for every $p$:
only the variable leaves can differ,
and there an outer-level related pair
is promoted to the local depth
by the companion fact.
And $\beta$ passes sites through:
if $u \vc u'$ and $b \vcd{1} b'$
then $\bsub{b}{u} \vc \bsub{b'}{u'}$.
The proof is an induction on the body relation.
A site at depth $d{+}1$ does not mention
the bound variable,
so the substitution merely strips
one level of the site's lift,
by law 3 of \cref{sec:app:sound:algebra},
and the site remains a site.
The bound variable itself becomes
$u$ on one side and $u'$ on the other,
related by hypothesis and promoted to the local depth.
Other variables shift down,
and the skeleton case composes.

\subsubsection{Sequentializing a parallel congruence}
\label{sec:app:sound:seq}

The relation $\vc$ relates terms that differ
at many sites at once.
Transporting a derivation across it
requires rewriting one site at a time,
because the rule that does the rewriting --
substitutivity of equality --
rewrites one marked position per application.

Let $X \ost{d} Y$, \emph{one-site replacement at depth $d$},
be generated by the site rule of \cref{defi:app:vcr}
together with the rule that
if $Y$ replaces the $j$-th argument of an operator term
and the replacement is a one-site replacement
at depth $d + nb$,
then the whole terms are related at depth $d$.
So $\ost{d}$ performs exactly one replacement.

\begin{lem}[sequentialization; \textsf{vcr\_imp\_vcrs\_star}]
\label{lem:app:seq}
If $X \vcd{d} Y$ then $X \osts{d} Y$:
$Y$ is reachable from $X$
by finitely many one-site replacements.
\end{lem}

\begin{proof}
Induction on $\vcd{d}$.
The site case is a single step
and the variable case is zero steps.
The skeleton case has the content.
The argument lists $T$ and $T'$ are pairwise related,
and each pair is connected by
a chain of one-site steps
by the induction hypothesis,
so process the components left to right,
walking through the intermediate mixed lists
in which the first $k$ arguments
have been replaced and the rest have not.
Each step of the walk rewrites one component
through that component's own chain,
which lifts to a chain on the whole operator term.
After as many steps as there are arguments,
the list is $T'$.
\end{proof}

\Cref{lem:app:seq} is the only place in the soundness proof
where a finite left-to-right bookkeeping appears.
The bookkeeping is elementary but not vacuous:
the parallel relation is defined by
simultaneous pairwise relatedness
and the sequential one is not,
so the equivalence has to be constructed.

\subsubsection{Naming a single site}
\label{sec:app:sound:extract}

\begin{lem}[site extraction; \textsf{vcrsd\_extract}]
\label{lem:app:extract}
Suppose $A \ost{d} B$,
and let $w$ be a variable fresh for both $A$ and $B$.
Then there are a term $q$, the \emph{skeleton},
well-formed terms $X$ and $Y$,
and a natural number $n$, such that
\[
  A = \ssub{w}{\lftk{d}{X}}{q}, \qquad
  B = \ssub{w}{\lftk{d}{Y}}{q},
\]
with $X$ and $Y$ both evaluating to $n$;
and $q$ is well-formed whenever $A$ is.
\end{lem}

A one-site replacement pair is thus
a common skeleton with a fresh variable
marking the site,
instantiated by the two site terms --
precisely the shape that
substitutivity of equality consumes.

\begin{proof}
Induction on the one-site relation.

\emph{Site.}
Take $q$ to be the fresh variable itself.
Then substituting any term for $w$ in $q$
returns that term,
so both equations hold by definition,
and the value premises are the site rule's own.

\emph{Operator step.}
The replacement happens in the $j$-th argument.
The variable $w + nb$ is fresh
for that argument and for its replacement,
freshness distributing through the constructor
shifted by the binder count.
The induction hypothesis yields a skeleton $q_0$
for the argument pair,
at depth $d + nb$, marked by $w + nb$.
Take $q$ to be the whole term
with $q_0$ in the $j$-th position.
Substituting $\lftk{d}{X}$ for $w$ in $q$ restores $A$:
at the $j$-th position
the substitution descends $nb$ binders,
so it substitutes
$\lftk{nb}{(\lftk{d}{X})} = \lftk{d+nb}{X}$
for $w + nb$,
which is exactly the induction hypothesis' equation,
while at every other position
$w + nb$ is fresh and the substitution is the identity.
The same computation with $Y$ gives $B$.
Well-formedness of $q$ follows from that of $A$,
an operator term's arguments being well-formed
and replacing one argument by a well-formed term
keeping the term well-formed.
\end{proof}

The lift arithmetic in the operator step
is the whole reason the depth discipline
was imposed in \cref{defi:app:vcr}.
Single-variable substitution itself inserts,
on its way down through $nb$ binders,
exactly the lift the site is required to carry.
So the skeleton exists with
no side conditions about closed sites
and no renaming.

\subsubsection{Transporting a derivation across the congruence}
\label{sec:app:sound:transport}

\begin{prop}[congruence transport;
  \textsf{ir\_vcr\_transport}]
\label{prop:app:transport}
If $A \vc B$ and $\Gamma \vdash A$, then $\Gamma \vdash B$.
\end{prop}

\begin{proof}
First for a single site, then by iteration.

Suppose $A \ost{0} B$ and $\Gamma \vdash A$;
then $A$ is well-formed,
derivable judgments having well-formed conclusions.
Choose $w$ fresh for both $A$ and $B$
and extract the skeleton by \cref{lem:app:extract}:
$A = \ssub{w}{X}{q}$ and $B = \ssub{w}{Y}{q}$
with $X$ and $Y$ well-formed
and both evaluating to the same natural $n$.
At depth $0$ the lift is the identity.
\begin{enumerate}
\item $\emptyset \vdash X = \tofn{n}$
  and $\emptyset \vdash Y = \tofn{n}$.
  This is value completeness
  (\cref{sec:app:complete:thms}),
  applied to the projection of the witnessed evaluations
  onto the official ones (\cref{lem:app:project}).
\item Hence $\emptyset \vdash X = Y$,
  by \irl{{=}sym} and \cref{lem:app:eq-trans};
  weaken the equation into the context $\Gamma$.
\item Substitutivity of equality \irl{{=}subs},
  applied to that equation,
  to the derivation $\Gamma \vdash \ssub{w}{X}{q}$,
  and to the well-formed skeleton $q$,
  yields $\Gamma \vdash \ssub{w}{Y}{q}$,
  which is $\Gamma \vdash B$.
\end{enumerate}
For many sites,
\cref{lem:app:seq} makes $B$ reachable from $A$
by a finite chain of one-site replacements,
and the single-site transport applies once per link.
\end{proof}

\Cref{prop:app:transport} is
one of three places at which
the soundness proof consumes completeness,
and the three fall into two kinds.
Two of them are the quantifier introductions
of \cref{sec:app:sound:quartet},
where the semantics demands an actual certificate
and completeness manufactures one
out of a semantic fact the induction has in hand;
\cref{sec:metatheory:sound} describes the first of those.
The need here is of the other kind.
A \emph{semantic} fact -- that two terms have the same value --
has to become a \emph{provable equation},
because the only rule that rewrites at a context
is an equality rule and wants an equation as its premise.
The naturals-only restriction on grounded equality
is what forces the detour.
One cannot write down $X = Y$
for function-valued $X$ and $Y$ at all,
there being no such equation in \rga to write;
but at a replacement \emph{site}
both terms have a common \emph{natural} value,
so the equation exists,
and value completeness is what supplies it.
The restriction that makes the general statement unavailable
and the property that makes the needed instance available
are the same property of grounded equality,
seen from two sides.

\subsubsection{Inverting the relation}
\label{sec:app:sound:vcrinv}

Two families of facts follow immediately
from \cref{defi:app:vcr}.

Two head shapes \emph{correspond}
when they are the same shape
with value-congruent carried terms:
abstractions when their bodies are related at depth $1$;
one- and two-argument recursor spines
when their arguments are related;
and the bare recursor and the bare quantifier
only to themselves.

Inverting the relation at a constructor
gives exactly two alternatives.
If $X \vc Z$ and $X$ has a given shape, then either
$X$ is itself a replacement site --
in which case $X$ has a value,
and so does $Z$, the same one --
or $Z$ has the same shape with related components.
For an abstraction, the bare recursor
and the bare quantifier
the site alternative is impossible,
those having no value (\cref{lem:app:vacuity}),
and the inversion collapses to the skeleton case.

Read the two alternatives as:
either the term under consideration
got replaced wholesale,
or it survived and its parts got replaced.
That shape drives the whole of the next subsection.

\subsubsection{The congruence simulation}
\label{sec:app:sound:sim}

\begin{prop}[congruence simulation;
  \textsf{whev\_vcr}, \textsf{wosr\_vcr}]
\label{prop:app:sim}
Simultaneously, for every $Z$ with $X \vc Z$:
if $X \whev{s} w$ then $Z \whev{s_2} w_2$
for some $s_2$ and some head shape $w_2$ corresponding to $w$;
and if $X \wbev{s} r$ then $Z \wbev{s_2} r$
for some $s_2$ -- the same value.
\end{prop}

Witnessed evaluation is thus
a simulation for value congruence.
Replacing subterms by equivalued ones
preserves the value exactly,
and preserves the head shape
up to the same replacement inside it.

\begin{proof}
Mutual rule induction on the witnessed pair.
Every case begins by inverting $X \vc Z$
at the shape of $X$ (\cref{sec:app:sound:vcrinv}),
giving two alternatives.

\emph{The site alternative,
in which $X$ was replaced wholesale.}
Then $X$ has a value $n$, and so does $Z$.
On the value side,
the case under consideration
also gives $X$ the value $r$,
and uniqueness (\cref{prop:app:wdet}) forces $r = n$,
so $Z$'s value is $r$ and the goal holds at once.
On the head side,
the case gives $X$ a head shape
while the site gives it a value,
which exclusivity (\cref{prop:app:wdet})
makes impossible,
so the alternative cannot occur.
The disposal is one line in every case,
and it is where \cref{sec:app:sound:det} is spent.

\emph{The skeleton alternative,
in which $Z$ has the same shape with related components.}
Apply the induction hypotheses to the components.
Each returns an evaluation of
the corresponding component of $Z$ at its own index;
raise them all to a common index by monotonicity;
re-apply the very same clause on $Z$'s side.
Two clauses need one extra step each.
The $\beta$ clauses need
$\bsub{b}{u} \vc \bsub{b'}{u'}$,
which is the second substitution law of
\cref{sec:app:sound:vcrclosure}.
The recursor-successor clauses need
the unfolded branch terms to be related,
which follows from the constructor introductions
together with reflexivity at the numerals.

\emph{The quantifier cases}
carry the load, and each does two transports.

For the certificate:
the clause holds a proof of
$\{\vz \jnat\} \vdash \lft u \cdot \vz$
and needs one of
$\{\vz \jnat\} \vdash \lft{u'} \cdot \vz$.
The bodies are related,
$u \vc u'$ coming from the skeleton alternative;
lifting preserves the relation
(\cref{sec:app:sound:vcrclosure}),
so $\lft u \vc \lft{u'}$;
applying both to the same variable preserves it;
so the two certificate judgments have
value-congruent conclusions,
and \cref{prop:app:transport} transports the derivation.
Extract a proof of the transported judgment
and use it as the new certificate,
at an index large enough to bound its code.
For the false clause the same runs with
$\emptyset \vdash \neg(u \cdot \tofn{i})$,
related to
$\emptyset \vdash \neg(u' \cdot \tofn{i})$
through the negation and application introductions
and reflexivity at the numeral.

For the witnesses:
the true clause's witness premise asserts
that every instance $u \cdot \tofn{n}$
evaluates to $\tvT$.
Each such instance is value-congruent to
$u' \cdot \tofn{n}$,
and the induction hypothesis attached to
the witness premise --
which exists because
the witness premise is a genuine premise of the clause,
as \cref{sec:app:sound:wdet} discussed --
carries the evaluation across.
The false clause is the same at its single instance.
\end{proof}

\Cref{prop:app:sim} is the meeting point
of everything proved so far,
and a reader who wants one place to see
why the metatheory had to be unwound
in the order it was
should look here.
Determinism and exclusivity
dispose of the site alternative.
Substitution stability is what made
the relation stable under lifting.
Completeness supplies the equations
that transport the certificates.

\subsubsection{Head-step simulation}
\label{sec:app:sound:hstepsim}

Two further simulations concern the head step,
one in each direction.
Together they are the semantic content of
\cref{fig:rules-reflect}'s conversion rules.

\begin{prop}[forward simulation; \textsf{whstep\_sim}]
\label{prop:app:hstep-fwd}
If $c \hs c'$ then every witnessed evaluation
of the reduct $c'$, value or head shape,
yields one of the redex $c$, at some index.
\end{prop}

\begin{proof}
Induction on the head step.
For $\beta$, an abstraction head-evaluates
to its own shape in one step;
feed that and the given evaluation of $\bsub{b}{t}$
to the $\beta$ clause,
in its value or head form as appropriate.
For the recursor at zero,
build the head shape of $\cR \cdot g \cdot h$
by three spine-growth steps from the bare recursor,
note that $\z$ evaluates to $0$,
and feed both plus the given evaluation of $g$
to the zero-dispatch clause.
For the recursor at a successor the same,
with the scrutinee's evaluation obtained from
numerals evaluating to themselves
(\cref{lem:app:val-inv})
together with the successor clause.
For spine congruence,
invert the given evaluation of $t' \cdot u$
by \cref{lem:app:app-inv};
each alternative begins with
a head evaluation of $t'$,
which the induction hypothesis converts into one of $t$,
and every other premise of the alternative is untouched,
so the same clause re-applies with $t$ in place of $t'$.
In the two quantifier alternatives
the certificate and the witnesses transfer verbatim:
both speak only about the argument $u$,
and a head step in function position does not move it.
\end{proof}

Evaluation results thus transfer \emph{backwards}
along a head step,
which is the direction the conversion rules need --
their premise is about the reduct
and their conclusion about the redex.
The forward simulation needs
no determinism and no inversion of the redex.
The simulation only builds,
which is why the conversion rules
can carry so few side conditions.

\begin{prop}[reverse simulation; \textsf{whstep\_rev}]
\label{prop:app:hstep-rev}
If $c \hs c'$ then every witnessed evaluation
of the redex $c$ yields one of the reduct $c'$.
\end{prop}

\begin{proof}
Induction on the head step.
The three redex cases now run by inversion,
and each inverts to exactly one surviving alternative.

For $\beta$, invert the evaluation of $(\lambda b) \cdot t$.
Every alternative starts with a claim about
the head shape of $\lambda b$,
and an abstraction head-evaluates only to its own shape,
so only the $\beta$ alternative survives,
and its second premise \emph{is}
an evaluation of $\bsub{b}{t}$.

For the recursor at zero,
invert the evaluation of
$\cR \cdot g \cdot h \cdot \z$.
The head shape of $\cR \cdot g \cdot h$
is forced to be the two-argument recursor spine,
by a spine-growth inversion applied twice,
which eliminates the $\beta$ and quantifier alternatives;
and $\z$ evaluates only to $0$,
which eliminates the successor alternative.
The surviving alternative's last premise
is an evaluation of $g$.

For the recursor at a successor the same,
with the scrutinee $\suc \tofn{n}$ inverting,
through the successor and numeral inversions,
to the value $n{+}1$,
which eliminates the zero alternative and pins $k = n$.
The surviving premise is
an evaluation of the unfolded branch.

For spine congruence,
symmetrically to the forward case:
invert, replace the head evaluation of $t$
by one of $t'$ from the induction hypothesis,
and re-apply the clause.
\end{proof}

A term that has a value therefore evaluates
\emph{through} its head step.
That behaviour is the semantic content of
the refutation-transport rule \irl{{\neg}convE},
which carries no side condition at all --
a refuted redex evaluates through the head step,
so the redex's falsity certificate
already speaks about the reduct --
and it is also what makes
the backward conversion rule \irl{convE} sound.

The contrast between the two directions
is worth one sentence.
Forward simulation builds, and needs nothing.
Reverse simulation inverts,
and therefore needs the shape inversions of
\cref{sec:app:sound:appinv}.
Determinism proper is needed by neither,
only the fact that each redex shape
admits exactly one clause.

\subsection{The Main Induction}
\label{sec:app:sound:induct}

\subsubsection{The truth-level dictionary}
\label{sec:app:sound:dict}

The cases below argue at the level of
truth-at-a-substitution rather than
at the level of evaluation clauses.
The translation between the two levels
is mechanical,
and the development carries it as
a family of introduction and elimination lemmas,
one per connective.
Rather than enumerate them,
here is the dictionary once;
below it is used silently.

\begin{center}
\begin{tabular}{@{}ll@{}}
\textbf{at the truth level} & \textbf{at the evaluation level} \\
\hline
$\gtrue{\sigma}{a = b}$ &
  $\sof{\sigma}{a}$ and $\sof{\sigma}{b}$
  both evaluate to a common natural \\
$\gtrue{\sigma}{a \jnat}$ &
  $\sof{\sigma}{a}$ evaluates to some natural \\
$\gtrue{\sigma}{\neg a}$ &
  $\sof{\sigma}{a}$ evaluates to $\tvF$ \\
$\gtrue{\sigma}{a \ne b}$ &
  both sides evaluate, to \emph{different} naturals \\
$\gtrue{\sigma}{a \lor b}$ &
  some disjunct is true (introducing);
  one of three alternatives (eliminating) \\
$\gtrue{\sigma}{a \land b}$ &
  both conjuncts are true,
  through the De Morgan unfolding \\
\end{tabular}
\end{center}

The second row is the notational identity again:
$a \jnat$ \emph{is} $a = a$,
so its row is the first row with $b \ldef a$.

Four facts at this level carry content
and are stated rather than tabulated.

\emph{Truth commutes with single substitution.}
$\ssub{v}{t}{c}$ is true at $\sigma$
exactly when $c$ is true at the \emph{rebound} substitution
$\reb{\sigma}{v}{\sof{\sigma}{t}}$,
which agrees with $\sigma$ except at $v$.
The identity is the semantic shadow of
the compositionality of \cref{lem:app:sub-compose},
and it is what the substitution and induction cases
rewrite by.
Its companion says that two substitutions agreeing
on a term's free variables give the same truth.

\emph{Refutation transports forward along a head step.}
If $c \hs c'$ and $\neg c$ is true at $\sigma$,
then $\neg c'$ is true at $\sigma$.
Truth of $\neg c$ means $\sof{\sigma}{c}$ evaluates to $\tvF$;
head steps are stable under substitution,
so $\sof{\sigma}{c} \hs \sof{\sigma}{c'}$;
the reverse simulation (\cref{prop:app:hstep-rev})
carries the $\tvF$ across;
negating gives $\neg c'$ true.
That transport is the semantic content of
the side-condition-free rule \irl{{\neg}convE}.

\emph{Only the quantifier clauses value
an application headed by $\cA$.}
Applying the general application inversion
(\cref{lem:app:app-inv})
and noting that $\cA$ head-evaluates
only to its own shape
kills the other three alternatives.
So a true $\cA \cdot u$ hands over a certificate
\emph{and the truth of every numeral instance},
while a false one hands over a numeral,
a refutation certificate,
\emph{and the falsity of that instance}.
The quantifier eliminations consume this directly.

\emph{A valued swap is a value congruence.}
If $X$ and $Y$ are well-formed
and both evaluate to the same natural,
then $X \vc Y$.
The fact is rule 1 of \cref{defi:app:vcr} at depth $0$,
and it is what lets six of the cases below
move between a term and the numeral of its value.

\subsubsection{Statement and shape of the induction}
\label{sec:app:sound:mainstmt}

\begin{prop}[the main induction; \textsf{prf\_owsjt}]
\label{prop:app:mainind}
Let $P$ be a valid proof containing the judgment $\Gamma \vdash c$,
and suppose every hypothesis in $\Gamma$ is well-formed.
Then $\ogtrue{\Gamma}{c}$:
for every pointwise well-formed substitution $\sigma$,
its images open or closed,
that makes every hypothesis in $\Gamma$ true,
$c$ is true at $\sigma$ --
where truth means \emph{witnessed} evaluation to $\tvT$.
\end{prop}

Note which relation the statement uses.
$P$ is required to \emph{contain} the judgment,
not to conclude with it,
and the difference matters in both directions.
The induction runs \emph{along} the list,
so it needs an induction hypothesis
for every judgment in the tail
and not merely for the tail's own conclusion --
which is why containment is the right relation here,
whereas the certificate clauses of \cref{fig:sem}
need the sharper one,
their code bound being on the whole list.
\Cref{prop:app:strip} below is stated over containment
for the same reason.

The proof is an induction on the proof list $P$
with a case analysis over the rule
that justified the list's head.
Proof lists are written in reverse order of derivation,
so a list's head, its first entry, is its conclusion,
the most-derived judgment,
and its tail is what was derived before that judgment;
``the head judgment'' and ``the conclusion''
below name the same object.
Given $P = j \cns P'$ with $P'$ valid:
if the judgment in question already occurs in $P'$
the induction hypothesis applies directly,
and otherwise it is the head judgment $j$,
which one of the 42 rules derives
from judgments in $P'$,
for which the induction hypothesis is available.
Three features of the setup come first.

\paragraph{No step index appears anywhere.}
Each case builds a witnessed evaluation
at whatever index it needs,
raising the indices supplied by
the induction hypotheses to a common value
by monotonicity (\cref{lem:app:wmono}).
There is no measure, no descent,
and no double induction.
That absence is what \cref{sec:app:sound:witness} was for.

\paragraph{The motive is the open one.}
Only the closing-substitution invariant
is needed for \cref{thm:sound},
but the induction proves the stronger open invariant
at no extra cost,
and the stronger form is what yields
the open characterization of \cref{thm:iff}.

The stronger form is available only because
completeness was proved in its \emph{open} form.
With open images the substituted hypotheses
$\sof{\sigma}{h}$ are open terms,
and the universal-introduction case
below must \emph{derive} them;
closed completeness cannot reach them.
So \cref{sec:app:complete}'s open-terms strand,
which cost that appendix
one lemma and a fresh-variable discipline,
is not a generalization for its own sake.
The strand is what makes the open motive provable here,
and hence what makes \cref{thm:iff} hold at open terms.

\paragraph{The well-formed-hypotheses guard is real.}
\rga's weakening and hypothesis rules
place no condition on hypothesis terms
(\cref{sec:system:proof}),
so a valid proof may carry ill-formed hypotheses.
The guard is forced by
the universal-introduction case,
where an ill-formed but satisfiable hypothesis
could not be discharged by completeness.
\Cref{sec:app:sound:strip} removes the guard afterwards.

\subsubsection{Structural rules}
\label{sec:app:sound:struct}

\emph{Substitution} \irl{sub}
is the case at which the naive motive failed
(\cref{prop:app:motive-fails}),
and under the repaired motive it is three lines.
Given $\Gamma \vdash c'$
and the substituted judgment
$\ssub{v}{t}{\Gamma} \vdash \ssub{v}{t}{c'}$,
take any well-formed $\sigma$
satisfying $\ssub{v}{t}{\Gamma}$
and let $\tau \ldef \reb{\sigma}{v}{\sof{\sigma}{t}}$.
Then $\tau$ is again pointwise well-formed,
and $\sof{\sigma}{\ssub{v}{t}{x}} = \sof{\tau}{x}$
for every term $x$
(\cref{lem:app:sub-compose}).
So $\tau$ satisfies $\Gamma$,
the induction hypothesis gives $c'$ true at $\tau$,
and rewriting back gives
$\ssub{v}{t}{c'}$ true at $\sigma$.

Substitution is thus not \emph{simulated} in the semantics.
The rule is absorbed by re-instantiating
the induction hypothesis at a different substitution,
and that is the entire content of the motive repair.

The other three are immediate.
For \irl{hyp}, $\sigma$ satisfies the context
and the conclusion is a member of it.
For \irl{wk}, a judgment with more hypotheses
is true whenever the one with fewer is.
For \irl{cut}, the first premise's truth at $\sigma$
shows that $\sigma$ satisfies
the extended context of the second premise,
to which the second induction hypothesis applies.

\subsubsection{The truth-value bridge}
\label{sec:app:sound:bridge}

The four halves of the two double rules
\irl{T{=}} and \irl{F{=}}
all rest on the observation that
$\tofn 1$ evaluates to $1 = \tvT$
and $\tofn 0$ evaluates to $0 = \tvF$,
combined with the equality row of the dictionary.
From $a$ true, both $\sof{\sigma}{a}$ and $\tofn 1$
evaluate to $1$, so $a = \tofn 1$ is true.
From $a = \tofn 1$ true,
both sides evaluate to a common $n$,
and numeral inversion pins $n = 1$, so $a$ is true.
From $\neg a$ true, $\sof{\sigma}{a}$ evaluates to $0$,
as does $\tofn 0$, so $a = \tofn 0$ is true.
From $a = \tofn 0$ true,
both sides evaluate to a common $n = 0$,
so $\sof{\sigma}{a}$ evaluates to $\tvF$,
which is $\neg a$ true.

\subsubsection{Negation and disjunction}
\label{sec:app:sound:negdisj}

Double negation, in both directions:
from $a$ true, $\sof{\sigma}{a}$ evaluates to $\tvT$,
so $\neg a$ evaluates to $\tvF$,
so $\neg\neg a$ is true;
conversely, invert the negation clause twice.

The disjunction introductions:
a true disjunct makes the disjunction true
by either of the two or-true clauses,
and two false disjuncts make it false
by the or-false clause,
hence $\neg(a \lor b)$ true.

The disjunction eliminations:
from $\neg(a \lor b)$ true
the disjunction evaluates to $\tvF$,
and or-inversion (\cref{lem:app:val-inv})
forces the both-false alternative,
giving each disjunct's falsity separately.
From $a \lor b$ true,
or-inversion gives one of the two true-alternatives,
so \emph{some} disjunct is true;
in either branch $\sigma$ satisfies
that branch's extended context,
and the branch's induction hypothesis
gives the common conclusion.

Explosion deserves its own treatment,
because it is sound for a reason
none of the others share.

\paragraph{Explosion is sound vacuously,
and its vacuity is determinism.}
The rule \irl{{\neg}E}
concludes an arbitrary well-formed $c$
from $\neg a$ and $a$.
Under any $\sigma$ satisfying the context,
the first premise makes
$\sof{\sigma}{a}$ evaluate to $\tvF$
and the second makes it evaluate to $\tvT$.
Uniqueness of witnessed evaluation
(\cref{prop:app:wdet}) makes that impossible,
so the case never arises
and the conclusion is free.

Soundness of the explosion rule is therefore
\emph{exactly} the determinism theorem,
and \rga's consistency --
which \cref{sec:app:sound:targets} derives
at the level of the whole system --
is already visible here in miniature,
inside a single rule case.

\subsubsection{Equality}
\label{sec:app:sound:eqcases}

Reflexivity in both directions is one line:
$a \jnat$ and $a = a$ are the same formula,
so both directions unfold to the same statement,
that $\sof{\sigma}{a}$ evaluates to some natural.
Symmetry, of equality and of inequality alike,
inverts and re-introduces with the two sides exchanged,
the ``different naturals'' condition being symmetric.

\emph{Boolean typing of equalities} \irl{{=}TI}
unfolds $(a = b) \jbool$
to the grounded excluded middle
$(a = b) \lor \neg(a = b)$.
Both premises give values $n_a$ and $n_b$.
If those are equal
the equality atom evaluates to $\tvT$
and the left disjunct is true;
if not, it evaluates to $\tvF$
and the right disjunct $\neg(a = b)$ is true.
Either way the disjunction is true.
The case is where the slogan
``excluded middle holds for equations
between terms that have values''
is cashed out.
\rga does not have LEM,
but it has this instance of LEM,
and it has the instance because
evaluation of an equality between two valued terms
always terminates with a boolean.
Its converse \irl{{=}TE} runs backwards:
from the disjunction true,
or-inversion gives a true disjunct,
and in \emph{either} case --
the equality or its negation --
the evaluation of the equality atom was performed,
so both sides have values,
which is the De Morgan conjunction
$(a \jnat) \land (b \jnat)$.

\paragraph{Substitutivity.}
From $a = b$ true,
both $\sof{\sigma}{a}$ and $\sof{\sigma}{b}$
evaluate to a common natural,
so they are value-congruent
by the valued-swap fact of \cref{sec:app:sound:dict}.
Hence the two rebound substitutions
$\reb{\sigma}{i}{\sof{\sigma}{a}}$ and
$\reb{\sigma}{i}{\sof{\sigma}{b}}$
are pointwise value-congruent,
so by the substitution law for the congruence
(\cref{sec:app:sound:vcrclosure})
their applications to the common context $p$ are related:
\[
  \sof{\reb{\sigma}{i}{\sof{\sigma}{a}}}{p}
  \;\vc\;
  \sof{\reb{\sigma}{i}{\sof{\sigma}{b}}}{p} .
\]
The second premise says the left-hand term
evaluates to $\tvT$;
the congruence simulation (\cref{prop:app:sim})
carries that value to the right-hand term;
and rewriting by
truth-commutes-with-substitution in both directions
turns this into
``$\ssub{i}{b}{p}$ is true at $\sigma$''.

Substitutivity is full congruence:
the context $p$ is an arbitrary well-formed term,
including positions under $\lambda$
and inside unevaluated arguments.
The cost of that generality
is the entire congruence apparatus of
\cref{sec:app:sound:vcr},
and this case is where the cost is paid.

\subsubsection{Arithmetic}
\label{sec:app:sound:arith}

All seven arithmetic rules follow from
the successor clause and successor inversion,
together with $\z$ evaluating to $0$.
Zero is a natural;
a successor of a valued term has the successor value,
hence is a natural and differs from zero;
successors of different values are different
and of equal values are equal,
and both directions invert.

\paragraph{Bottom typing is vacuous.}
The rule \irl{\bot{}TE}
concludes an arbitrary well-formed $p$
from $\bot \jnat$.
But $\bot \jnat$ is $\bot = \bot$,
whose truth would require $\bot$ to have a value,
and $\bot$ has none (\cref{lem:app:vacuity}).
The premise is unsatisfiable,
the case never arises,
and the conclusion is free.

The vacuity is the semantic statement of
what $\bot$ is for.
\Cref{sec:system:syntax} introduces $\bot$
as a convenient canonical divergent term;
the present case is the proof that
the term deserves the description,
and it is the second of the appendix's
two vacuously sound rules.

\subsubsection{Conversion and schematic recursor unfolding}
\label{sec:app:sound:convcases}

All five cases begin with the same observation:
head steps are stable under substitution,
so from $c \hs c'$ one gets
$\sof{\sigma}{c} \hs \sof{\sigma}{c'}$,
and the simulations of
\cref{sec:app:sound:hstepsim}
apply to the substituted terms.

Forward conversion \irl{conv} has its premise
about the reduct,
whose truth transfers to the redex
by \cref{prop:app:hstep-fwd}.
Backward conversion \irl{convE} has its premise
about the redex,
whose truth transfers to the reduct
by \cref{prop:app:hstep-rev}.
Refutation transport \irl{{\neg}convE}
is the third content-carrying fact of
\cref{sec:app:sound:dict},
which is the reverse simulation
applied to a $\tvF$ value.
Conversion as an equation \irl{conv{=}}
gets a value $n$ for the reduct from its premise,
and the forward simulation gives the redex
the same value $n$;
two terms with a common value
make the equation true.
The rule's side condition --
that the reduct be a natural --
is exactly what supplies $n$,
and without the side condition
there would be no value to share.

\paragraph{Schematic recursor unfolding.}
The rule \irl{\cR\suc{}{=}}
is the soundness-side mirror of
\cref{sec:app:complete:rec},
and has the same gap to bridge.
The rule unfolds
$\cR \cdot g \cdot h \cdot (\suc m)$
for an \emph{arbitrary} $m$ proved to be a natural,
whereas the evaluation clause dispatches on
the scrutinee's \emph{value}
and the syntactic head step fires
only at a literal numeral.

Let $v_m$ be the value of $\sof{\sigma}{m}$
and $v_c$ the value of the unfolded branch
$\sof{\sigma}{h \cdot m \cdot
  (\cR \cdot g \cdot h \cdot m)}$,
both supplied by the premises.
Since $\sof{\sigma}{m}$ and $\tofn{v_m}$
have the same value they are value-congruent,
so the branch term with $\sof{\sigma}{m}$
replaced by $\tofn{v_m}$ throughout
is congruent to the original,
and by the congruence simulation
it too has value $v_c$.
Now build a witnessed evaluation of the left-hand side:
three spine-growth steps give
$\cR \cdot \sof{\sigma}{g} \cdot \sof{\sigma}{h}$
the two-argument recursor head shape;
$\suc \sof{\sigma}{m}$ evaluates to $v_m + 1$;
and the recursor-successor value clause
then gives the whole application the value $v_c$.
Both sides of the equation now have value $v_c$,
so the equation is true.

\subsubsection{Induction}
\label{sec:app:sound:indcase}

The rule \irl{ind} concludes $\ssub{i}{a}{c_0}$
from the base instance $\ssub{i}{\z}{c_0}$,
the step
$v_i \jnat,\, c_0,\, \Gamma \vdash \ssub{i}{\suc v_i}{c_0}$,
and a scrutinee premise $a \jnat$,
with $c_0$ well-formed and $i \notin \fv(\Gamma)$.

Fix a well-formed $\sigma$ satisfying $\Gamma$.
The scrutinee premise gives $\sof{\sigma}{a}$
a value $v_a$.

\emph{The meta-induction.}
Show, by ordinary induction on the natural number $m$,
that $c_0$ is true at
$\reb{\sigma}{i}{\tofn{m}}$.
For the base,
the base premise says $\ssub{i}{\z}{c_0}$
is true at $\sigma$;
rewriting by truth-commutes-with-substitution,
that is $c_0$ true at
$\reb{\sigma}{i}{\sof{\sigma}{\z}}$,
and $\sof{\sigma}{\z}$ is $\z$, which is $\tofn 0$.
For the step,
assume $c_0$ true at
$\tau \ldef \reb{\sigma}{i}{\tofn{m}}$,
and check that $\tau$ satisfies
the step premise's three-part context:
\begin{itemize}
\item $v_i \jnat$, because $\tau(i) = \tofn{m}$,
  which evaluates;
\item $c_0$, which is the meta-induction hypothesis;
\item each member $h$ of $\Gamma$,
  because $i$ is not free in $\Gamma$
  and hence not free in $h$,
  so $\sof{\tau}{h} = \sof{\sigma}{h}$,
  true by assumption.
\end{itemize}
The third item is where the rule's side condition
$i \notin \fv(\Gamma)$ is spent,
and it is spent exactly once.
The step premise's induction hypothesis
then gives $\ssub{i}{\suc v_i}{c_0}$ true at $\tau$;
rewriting once more,
that is $c_0$ true at
$\reb{\tau}{i}{\sof{\tau}{\suc v_i}}
  = \reb{\sigma}{i}{\tofn{m{+}1}}$.

\emph{Closing.}
The meta-induction gives $c_0$ true at
$\reb{\sigma}{i}{\tofn{v_a}}$.
Swap the numeral back for the scrutinee:
$\tofn{v_a}$ and $\sof{\sigma}{a}$ have the same value,
so they are value-congruent,
so the two substitutions are pointwise congruent,
so their applications to $c_0$ are congruent,
and the congruence simulation carries the truth across.
Rewriting a last time gives
$\ssub{i}{a}{c_0}$ true at $\sigma$.

This case is \emph{habeas quid} for induction,
stated semantically.
\rga's induction rule carries a scrutinee premise
$a \jnat$ that classical induction does not have,
and its role is now visible:
the premise supplies an \emph{actual number} $v_a$,
and the meta-induction runs up to that number.
Induction in \rga is not
induction over an abstract domain,
but a finite ascent to a value
that the premise guarantees exists.

Note also what the argument never does.
Nothing anywhere inspects the shape of the motive $c_0$,
which may contain quantifiers, recursion,
even ungrounded subformulas.
\Cref{sec:system:proof} claims that
the arithmetical hierarchy collapses in \rga's semantics,
and that claim is cashed out here:
the meta-induction runs over
the \emph{value of the scrutinee},
not over the complexity of the formula.
That is why one induction rule
serves the whole closed universe,
and it is checkable rather than merely asserted.

\Cref{sec:app:arith:whyack} is where the collapse is used.
Ackermann's totality is proved by induction on a motive
that itself contains a universal quantifier,
which is exactly the kind of motive
a formula-class restriction would forbid,
and the present case is what makes it sound.

\subsubsection{The quantifier quartet}
\label{sec:app:sound:quartet}

The four quantifier cases are the heart of the proof,
and they split cleanly.
The two \emph{introductions produce}
a witnessed universal,
and must manufacture both a certificate
and the semantic witnesses.
The two \emph{eliminations consume} one,
and take what they need
off the derivation they are handed.
That asymmetry is the architecture in one picture,
and \cref{sec:app:sound:whatitbought}
returns to it once all four are done.

\paragraph{Universal introduction.}
The rule \irl{{\cA}I} concludes
$\Gamma \vdash \cA \cdot f$
from $\vz \jnat,\, \lft\Gamma \vdash \lft f \cdot \vz$.
Fix a well-formed $\sigma$ satisfying $\Gamma$.
The witnessed truth clause for
$\cA \cdot \sof{\sigma}{f}$ demands three things:
that $\cA$ head-evaluates to its own shape, immediate;
a valid proof of
$\{\vz \jnat\} \vdash \lft{(\sof{\sigma}{f})} \cdot \vz$
with no other hypotheses;
and the truth of \emph{every} numeral instance
$\sof{\sigma}{f} \cdot \tofn{n}$.

\emph{Producing the certificate.}
The induction hypothesis gives only
semantic truth of the premise judgment,
which is not a proof.
But the premise judgment itself is derivable,
occurring as it does in
the valid proof under induction.
So:
\begin{enumerate}
\item \emph{Push $\sigma$ through the premise's derivation.}
  Apply the schematic-variable-preserving specialization
  of \cref{prop:app:subf}
  with the lifted images $v \mapsto \lft{(\sigma(v))}$.
  The schematic variable is preserved,
  so the hypothesis $\vz \jnat$ survives verbatim;
  each lifted hypothesis $\lft h$
  becomes $\lft{(\sof{\sigma}{h})}$
  and the conclusion becomes
  $\lft{(\sof{\sigma}{f})} \cdot \vz$,
  both by laws 2 and 4 of
  \cref{sec:app:sound:algebra},
  exactly as in \cref{sec:app:sound:stability}.
\item \emph{Discharge the hypotheses by completeness.}
  Each $h \in \Gamma$ is true at $\sigma$,
  so $\sof{\sigma}{h}$ evaluates to $\tvT$;
  and $\sof{\sigma}{h}$ is well-formed,
  the guard of \cref{prop:app:mainind} giving $\wf(h)$
  and $\sigma$ having well-formed images.
  Open completeness (\cref{sec:app:complete:thms})
  gives $\emptyset \vdash \sof{\sigma}{h}$,
  and admissibility of lifting
  gives $\emptyset \vdash \lft{(\sof{\sigma}{h})}$.
\item \emph{Cut them all away.}
  Set cut (\cref{sec:app:sound:subf})
  removes the whole family,
  leaving
  $\{\vz \jnat\} \vdash
    \lft{(\sof{\sigma}{f})} \cdot \vz$.
  Extract a valid proof $Q$ of that judgment.
\end{enumerate}

\emph{Producing the witnesses.}
Fix a natural $n$ and let
$\tofn{n} \cns \sigma$ be
the numeral-extended substitution
of \cref{sec:app:sound:algebra}.
It is pointwise well-formed,
and it satisfies the premise's context:
$\vz \jnat$ because
$\sof{(\tofn{n} \cns \sigma)}{\vz} = \tofn{n}$,
which evaluates;
and each $\lft h$ because
$\sof{(\tofn{n} \cns \sigma)}{\lft h} = \sof{\sigma}{h}$
by law 2 --
the lift makes position $0$ unreachable,
so the extended substitution acts as $\sigma$ --
which is true by assumption.
The premise's induction hypothesis at
$\tofn{n} \cns \sigma$ therefore gives
$\sof{(\tofn{n} \cns \sigma)}{\lft f \cdot \vz}$ true,
and that term \emph{is}
$\sof{\sigma}{f} \cdot \tofn{n}$.

\emph{Assembling.}
Apply the witnessed truth clause
with the head shape,
the certificate $Q$ at any index bounding $\quo Q$,
and the witness family.
So $\cA \cdot f$ is true at $\sigma$.

Two observations about this case.
The case is the reasoning path
\cref{sec:metatheory:sound} singles out,
in which soundness of a single rule
consumes completeness of the whole system,
and it is well-founded only in one direction:
legitimate here because completeness
was established first,
by an induction using no soundness.
Reverse the order and the argument is circular.
And the certificate and the witnesses
come from different places.
The certificate is manufactured by
proof surgery plus completeness;
the witnesses come from
the \emph{same} induction hypothesis
instantiated at numeral-extended substitutions.
Neither is read off the other,
and no certificate's soundness
is assumed anywhere --
which is the property
the whole witnessed architecture exists to secure.

\paragraph{Universal elimination.}
The rule \irl{{\cA}E} concludes $\Gamma \vdash f \cdot a$
from $\Gamma \vdash \cA \cdot f$ and $\Gamma \vdash a \jnat$.
Fix $\sigma$ satisfying $\Gamma$.
The first premise gives
$\cA \cdot \sof{\sigma}{f}$ the value $\tvT$,
the second gives $\sof{\sigma}{a}$ a value $v_a$.
Invert the universal's evaluation
by the third fact of \cref{sec:app:sound:dict}:
the false alternative would give
the value $\tvF \ne \tvT$,
so the true alternative holds,
and it carries the truth of every numeral instance.
Take the instance at $\tofn{v_a}$,
so that $\sof{\sigma}{f} \cdot \tofn{v_a}$
evaluates to $\tvT$.
Since $\tofn{v_a}$ and $\sof{\sigma}{a}$
have the same value they are value-congruent,
so
$\sof{\sigma}{f} \cdot \tofn{v_a}
  \vc \sof{\sigma}{f} \cdot \sof{\sigma}{a}$,
and the congruence simulation
carries the $\tvT$ across.

No certificate is consumed here.
Universal elimination is the case
that defeated the step-indexed plan:
under the official semantics
the only thing a true universal hands over is a \emph{proof},
and using it means assuming the proof's soundness.
Under the witnessed clause
the truth of the instance is handed over directly,
and what remains is the numeral-for-scrutinee swap,
which is the congruence again.

\paragraph{Negated-universal introduction.}
The rule \irl{{\neg\cA}I} concludes
$\Gamma \vdash \neg(\cA \cdot f)$
from $\Gamma \vdash \neg(f \cdot a)$
and $\Gamma \vdash a \jnat$.
Fix $\sigma$ satisfying $\Gamma$.
The premises give that
$\sof{\sigma}{f} \cdot \sof{\sigma}{a}$
evaluates to $\tvF$,
and that $\sof{\sigma}{a}$ has value $v_a$.
Swap the scrutinee for its numeral by the congruence,
so that $\sof{\sigma}{f} \cdot \tofn{v_a}$
evaluates to $\tvF$.
That evaluation is the semantic witness
the false clause needs.
For the certificate, negate:
$\neg(\sof{\sigma}{f} \cdot \tofn{v_a})$
evaluates to $\tvT$ and is well-formed,
so open completeness gives
$\emptyset \vdash
  \neg(\sof{\sigma}{f} \cdot \tofn{v_a})$.
Extract a proof $Q$,
apply the witnessed falsity clause
with the head shape, $Q$, and the witness,
and $\cA \cdot \sof{\sigma}{f}$ evaluates to $\tvF$,
so $\neg(\cA \cdot f)$ is true at $\sigma$.

Completeness is consumed here too.
With universal introduction above
and the congruence transport of
\cref{prop:app:transport},
this is the third of the soundness proof's
three consumption points,
and the second of the two that consume completeness
for the certificate-manufacturing reason.

\paragraph{Negated-universal elimination.}
The rule \irl{{\neg\cA}E} concludes $\Gamma \vdash c$
from $\Gamma \vdash \neg(\cA \cdot f)$
and
$\vz \jnat,\, \neg(\lft f \cdot \vz),\,
  \lft\Gamma \vdash \lft c$:
one reasons under a hypothetical counterexample.
Fix $\sigma$ satisfying $\Gamma$.
The first premise makes
$\cA \cdot \sof{\sigma}{f}$ evaluate to $\tvF$;
invert, the true alternative being excluded
by $\tvT \ne \tvF$.
The false alternative supplies a numeral $\tofn{i}$,
a refutation certificate,
and -- the part that matters --
the semantic falsity of the instance:
$\sof{\sigma}{f} \cdot \tofn{i}$ evaluates to $\tvF$.

Let $\tau \ldef \tofn{i} \cns \sigma$.
It satisfies the second premise's three-part context:
$\vz \jnat$, since $\tofn{i}$ evaluates;
$\neg(\lft f \cdot \vz)$,
which under $\tau$ is
$\neg(\sof{\sigma}{f} \cdot \tofn{i})$,
true precisely because the instance evaluates to $\tvF$,
and this is the witness being spent;
and each $\lft h$,
which under $\tau$ is $\sof{\sigma}{h}$,
true by assumption.
The second premise's induction hypothesis at $\tau$
gives $\lft c$ true at $\tau$,
and $\sof{\tau}{\lft c} = \sof{\sigma}{c}$.
So $c$ is true at $\sigma$.

\Cref{sec:system:proof} says this rule
``is sound precisely because,
as the semantics will guarantee,
a false universal always carries a witness''.
At the witnessed level
the witness is not something the semantics guarantees.
The witness is literally a premise of the clause,
put there for this purpose,
which is the clearest single illustration
of what the witnessed pair buys.

\subsubsection{What the quartet shows}
\label{sec:app:sound:whatitbought}

Put the four cases side by side
and the architecture is visible in one picture.
Each introduction must produce two things,
a certificate and the semantic content,
and it produces them from
two entirely different sources --
proof surgery plus completeness for the one,
the induction hypothesis at
numeral-extended substitutions for the other.
Each elimination consumes only the semantic content
and never touches the certificate it is handed.

That is the asymmetry the step-indexed plan of
\cref{sec:app:sound:nodescent} could not survive.
Under the official semantics
an elimination is handed a certificate and nothing else,
so consuming a universal means
assuming an arbitrary proof sound,
and the only route back is a descent on
the certificate's code.
The witness premises put
the semantic content beside the certificate
at the point of production,
so the eliminations can take it
without ever asking whether the certificate deserves trust.
The price is that the introductions
must now produce both.
Completeness is what lets them pay it.

\subsection{Hypothesis Stripping}
\label{sec:app:sound:strip}

\Cref{prop:app:mainind} carries a guard:
every hypothesis must be well-formed.
\rga's rules do not enforce that,
so the guard has to be discharged separately.

Write $\wfpart{\Gamma}$ for
the subset of well-formed members of $\Gamma$.
Filtering commutes with the three operations
the rules perform on contexts.
With insertion,
$\wfpart{(a, \Gamma)}$
is $a, \wfpart{\Gamma}$
when $a$ is well-formed
and $\wfpart{\Gamma}$ otherwise.
With single substitution by a well-formed substituend,
substitution neither creates nor destroys well-formedness,
so
$\wfpart{(\ssub{v}{s}{\Gamma})}
  = \ssub{v}{s}{\wfpart{\Gamma}}$.
With lifting, likewise,
$\wfpart{(\lft\Gamma)} = \lft{\wfpart{\Gamma}}$.

\begin{prop}[stripping; \textsf{prf\_strip}]
\label{prop:app:strip}
If a valid proof contains $\Gamma \vdash c$,
then $\wfpart{\Gamma} \vdash c$.
\end{prop}

\begin{proof}
Induction on the proof list,
with a case analysis over the last rule.
Every one of the 42 cases re-applies the \emph{same} rule
with the filtered context,
using the induction hypotheses for the premises
and the commutation facts above
to see that the filtered context
is what the rule needs.
Two cases need one remark each.
The substitution case needs
the substituend to be well-formed,
which is that rule's own side condition.
The induction case needs
$i \notin \fv(\wfpart{\Gamma})$,
which follows from $i \notin \fv(\Gamma)$
because filtering only removes members.
\end{proof}

Why does a single shared argument
suffice for all 42 cases?
Because every hypothesis that a rule
\emph{introduces} into a context is well-formed.
Either an explicit side condition says so,
as in the hypothesis rule;
or the hypothesis is a premise's conclusion,
and derivable conclusions are well-formed,
as in cut and disjunction elimination;
or the rule builds the hypothesis itself,
as with the schematic typing hypothesis of
universal introduction,
the hypothetical counterexample of
negated-universal elimination,
and the two hypotheses of induction.
The one rule that can introduce
an ill-formed hypothesis is weakening,
whose whole content is to add an unused formula --
and stripping simply does not add it.

\Cref{prop:app:strip} turns a design choice
into a verified statement.
\rga's rules deliberately constrain
the \emph{conclusions} of judgments to be well-formed
while leaving \emph{hypotheses} unconstrained,
which makes the entailment relation as strong as possible.
The cost is that a valid proof
may carry ill-formed hypotheses,
which the soundness induction
cannot discharge at universal introduction,
step 2 above needing well-formedness
for completeness to apply.
Stripping is the repair,
and what it says is that
ill-formed hypotheses are permitted
but proof-theoretically irrelevant.

\subsection{The Collapse, and the Main Theorems}
\label{sec:app:sound:collapse}

\subsubsection{The witness collapse}
\label{sec:app:sound:fixpoint}

\begin{prop}[the witness collapse;
  \textsf{hev\_whev}, \textsf{osr\_wosr}]
\label{prop:app:collapse}
On \emph{closed} terms,
official evaluation is witnessed evaluation.
If $t \bev{s} r$ and $t$ is closed then $t \wbev{s} r$;
if $t \hev{s} w$ and $t$ is closed
then $t \whev{s} w$ and $w$ is a closed head shape.
\end{prop}

Together with the projection
(\cref{lem:app:project}),
this identifies the two semantics
on the closed fragment,
at the same step index.

\begin{proof}
Mutual rule induction on the official pair.
The nineteen structural clauses
re-apply their witnessed counterparts,
threading closedness through the subterms.
Closedness of a head shape has to be
part of the head motive,
because the recursor-dispatch clauses
need $g$ and $h$ closed
and those terms are invisible in the conclusion --
exactly as $\wf(\canon w)$ carried well-formedness
in \cref{defi:app:xfer},
and for exactly the same reason.

The two quantifier clauses
are where the theorem lives,
and both run the same way.

\emph{The true clause.}
The clause supplies a valid proof $P$ of
$\{\vz \jnat\} \vdash \lft u \cdot \vz$.
Since $u$ is closed the lift is the identity,
so $P$ proves $\{\vz \jnat\} \vdash u \cdot \vz$.
Apply \cref{prop:app:mainind} to $P$:
the hypothesis set $\{\vz \jnat\}$ is well-formed,
so the judgment is true at every
satisfying substitution.
Fix a natural $n$ and instantiate at
the closing substitution
sending variable $0$ to $\tofn{n}$
and every other variable to $\bot$.
That substitution satisfies the hypothesis,
since $\tofn{n}$ evaluates,
so the conclusion is true;
and since $u$ is closed
the substituted conclusion is $u \cdot \tofn{n}$.
The resulting family is what
the witnessed clause needs.

\emph{The false clause.}
The certificate proves
$\emptyset \vdash \neg(u \cdot \tofn{i})$.
\Cref{prop:app:mainind} at the all-$\bot$
closing substitution,
which satisfies the empty context vacuously,
gives that $\neg(u \cdot \tofn{i})$
evaluates to $\tvT$,
hence by negation inversion
that $u \cdot \tofn{i}$ evaluates to $\tvF$.
That evaluation is the witness.
\end{proof}

Here the circle closes,
and it is worth saying so plainly,
because the collapse reads like a formality
once one knows it holds.
The witnessed pair was introduced
because a certificate's semantic content
could not be extracted
without assuming soundness.
\Cref{prop:app:mainind} says
certificates \emph{are} sound.
So the content can be extracted after all,
and the two semantics coincide.

Three things about that coincidence
should not be blurred.
The witnessed pair was scaffolding:
a proof device, introduced for one purpose,
and $\bev{}$ is untouched by any of it.
The coincidence on closed terms is a \emph{theorem}
and not a definition,
proved at the end of a chain
that began with completeness.
And the reflective quantifier's narrative in
\cref{sec:system:sem:all},
together with the recursive-enumerability
of grounded truth in \cref{sec:power},
concerns the official relation alone,
and is unaffected by anything
in \cref{sec:app:sound:witness} through here.

Over \emph{closing} substitutions
the substituted term is closed,
so the collapse and the projection
give both directions:
witnessed truth and official truth coincide,
and so do the hypothesis-satisfaction
and judgment-truth notions built from them.
The main induction's motive is therefore
the substitutional invariant of \cref{defi:app:sjt},
and the proof-device layer
disappears from the statements.

\subsubsection{Soundness}
\label{sec:app:sound:targets}

\begin{thmc}[soundness, \cref{thm:sound}; \textsf{rga\_sound}, \textsf{rga\_sound\_sjt}]
Suppose $\Gamma \vdash c$.
(1) $\agtrue{\Gamma}{c}$:
for every assignment making
every hypothesis in $\Gamma$ grounded-true,
$c$ is grounded-true.
(2) If every hypothesis in $\Gamma$ is well-formed,
then $\gtrue{\Gamma}{c}$:
for every closing substitution
making every hypothesis in $\Gamma$ true,
$c$ is true.
\end{thmc}

\begin{proof}
Take a valid proof containing $\Gamma \vdash c$.
Strip the ill-formed hypotheses
(\cref{prop:app:strip}),
obtaining a derivation from $\wfpart{\Gamma}$.
Apply \cref{prop:app:mainind} to that derivation,
whose hypothesis set is well-formed by construction.
Convert the witnessed invariant into the official one
by the collapse.
Weaken from $\wfpart{\Gamma}$ back to $\Gamma$,
which is legitimate because satisfying
the larger hypothesis set
is the stronger assumption.
\end{proof}

Part~(2) is the same assembly
with the substitutional invariant
in place of the assignment one,
over closing substitutions
into arbitrary closed well-formed terms,
$\lambda$-terms included.
The mechanized statement of part~(2)
carries the well-formedness hypothesis on $\Gamma$;
part~(1) carries none,
and on well-formed hypotheses it is
the special case of value-term substitutions
(\cref{lem:app:embed}).

Note what the assignment form does \emph{not} assume.
There is no well-formedness hypothesis
and no closedness hypothesis:
stripping removed the first,
and derivable conclusions are
automatically well-formed.

\subsubsection{Determinism on closed terms, and consistency}
\label{sec:app:sound:consistency}

\begin{thmc}[closed determinism;
  \textsf{osr\_unique\_closed}]
If $t$ is closed and $t \bev{s} r$
and $t \bev{s'} r'$, then $r = r'$.
\end{thmc}

\begin{proof}
The collapse turns both evaluations into witnessed ones,
and witnessed values are unique unconditionally
(\cref{prop:app:wdet}).
\end{proof}

\begin{thmc}[consistency, \cref{thm:consistent};
  \textsf{rga\_consistent}]
There is no $c$ with
$\emptyset \vdash c$ and $\emptyset \vdash \neg c$.
\end{thmc}

\begin{proof}
Suppose both.
The empty hypothesis set is satisfied
by every assignment,
so soundness makes both $c$ and $\neg c$
grounded-true under the all-$\bot$ assignment.
Let $d$ be the closure of $c$ under that assignment,
which is closed.
Grounded truth of $c$ says
$d$ evaluates to $\tvT$;
grounded truth of $\neg c$ says
$\neg d$ evaluates to $\tvT$,
and inverting the negation clause,
that $d$ evaluates to $\tvF$.
Both evaluations are on a closed term,
so the collapse makes them witnessed,
and uniqueness gives $\tvT = \tvF$.
\end{proof}

\begin{thmc}[\cref{cor:no-0eq1}; \textsf{rga\_no\_proof\_0eq1}]
$\emptyset \nvdash \z = \tofn{1}$.
\end{thmc}

\begin{proof}
\rga derives $\neg(\z = \tofn{1})$,
since distinct numerals are provably distinct;
a derivation of $\z = \tofn{1}$ would contradict
consistency.
\end{proof}

The shape of the consistency statement
is worth remarking on.
The statement carries no well-formedness hypothesis,
no closedness hypothesis,
and quantifies over all statements at once.
The same argument,
run at an assignment satisfying $\Gamma$,
gives consistency of the hypothetical judgment
$\Gamma \vdash c$ for any satisfiable $\Gamma$;
that extension is \cref{thm:sat-consistent}.

For the reader who wants
the dependency chain in one line:
consistency rests on determinism,
which rests on the collapse,
which rests on the main induction,
whose quantifier cases rest on completeness --
and completeness rests on nothing
but the evaluation relation.
That is why \cref{sec:metatheory} runs
completeness, then soundness,
then determinism, then consistency,
rather than the conventional order.

\subsection{Consequences at the Official Level}
\label{sec:app:sound:official}

\Cref{sec:app:sound:collapse} delivered determinism
only on \emph{closed} terms,
because the collapse it rests on
was proved only there.
The general statement needs its own argument,
and the argument is available now that soundness is.
The plan is to replay the determinism induction of
\cref{prop:app:wdet} for the \emph{official} pair
and settle its one genuinely new case --
a universal carrying both a truth certificate
and a falsity certificate --
by soundness of both certificates.
The section then derives the official-level consequences.

\subsubsection{The official-pair mirrors}
\label{sec:app:sound:mirrors}

The vacuity facts of \cref{lem:app:vacuity},
the value inversions of \cref{lem:app:val-inv},
and the application inversions of \cref{lem:app:app-inv}
all hold verbatim for the official pair,
by the identical argument:
inspect the clauses that could conclude
the shape in question.
They have to be restated in the development
only because the two relations
are formally distinct objects.

One difference matters.
The official application inversion's
two quantifier alternatives carry
\emph{only} their certificates,
with no witness premises.
That absence is precisely what
the next subsection has to compensate for.

\subsubsection{The certificate collision}
\label{sec:app:sound:collision}

\begin{prop}[certificate collision;
  \textsf{osr\_cert\_collision}]
\label{prop:app:collision}
There is no term $u$ and index $i$ for which both
$\{\vz \jnat\} \vdash \lft u \cdot \vz$
and
$\emptyset \vdash \neg(u \cdot \tofn{i})$
are derivable --
a schematic proof of the universal's instance
and a refutation of a particular numeral instance.
\end{prop}

\begin{proof}
Let $\sigma$ be the \emph{constant} closing substitution
sending \emph{every} variable to $\tofn{i}$.
It is closing, numerals being closed and well-formed.
\begin{enumerate}
\item \emph{It satisfies the schematic certificate's hypothesis.}
  $\sof{\sigma}{\vz}$ is $\tofn{i}$,
  so $\vz \jnat$ -- which is $\vz = \vz$ --
  closes to $\tofn{i} = \tofn{i}$,
  which evaluates to $\tvT$.
\item \emph{Instantiate the schematic certificate.}
  Substitutional soundness
  (\cref{sec:app:sound:targets}),
  applied to the first judgment,
  gives that $\sof{\sigma}{\lft u \cdot \vz}$
  evaluates to $\tvT$.
  Compute that term.
  By law 1 of \cref{sec:app:sound:algebra},
  $\sof{\sigma}{\lft u}$ is
  the shifted substitution applied to $u$;
  and since $\sigma$ is \emph{constant},
  the shifted substitution is $\sigma$ itself,
  so $\sof{\sigma}{\lft u} = \sof{\sigma}{u}$.
  With $\sof{\sigma}{\vz} = \tofn{i}$,
  the conclusion is that
  $\sof{\sigma}{u} \cdot \tofn{i}$
  evaluates to $\tvT$.
\item \emph{Instantiate the refutation certificate.}
  The empty hypothesis set is satisfied vacuously,
  so soundness applied to the second judgment
  gives that
  $\sof{\sigma}{\neg(u \cdot \tofn{i})}$ --
  that is, $\neg(\sof{\sigma}{u} \cdot \tofn{i})$,
  numerals being fixed by any substitution --
  evaluates to $\tvT$.
  Inverting the negation clause,
  $\sof{\sigma}{u} \cdot \tofn{i}$ evaluates to $\tvF$.
\item \emph{Collide.}
  The term $\sof{\sigma}{u} \cdot \tofn{i}$ is closed,
  $\sigma$ being closing,
  so closed determinism
  (\cref{sec:app:sound:consistency})
  forces $\tvT = \tvF$.
\end{enumerate}
\end{proof}

The choice of substitution is
the whole content of that proof,
and the certificate lift is what allows the choice.
Both certificates have to be landed on
the same \emph{closed} instance.
The refutation names its instance, $\tofn{i}$.
The schematic certificate names none,
speaking as it does about the free variable $\vz$,
so the substitution must send $\vz$ to $\tofn{i}$.
But the certificate's conclusion also mentions
the body $u$,
whose own free variables are shifted up by the lift.
Sending every variable to the same numeral
is what makes the shift invisible,
collapsing $\sof{\sigma}{\lft u}$ to $\sof{\sigma}{u}$.

The manoeuvre is the second dividend of the certificate lift.
The first was substitution stability
(\cref{sec:app:sound:stability}),
where the lift kept the schematic variable
\emph{separate} from the body's own variables
so that a substitution could pass by without disturbing it.
Here the same separation is used in the opposite direction:
because the two are genuinely distinct,
a substitution may deliberately \emph{identify} them,
and a constant substitution is the one that does.
A design decision that looked like de Bruijn bookkeeping
has now paid twice, for two unrelated reasons.

It is worth setting the collision beside
its witnessed-level counterpart in
\cref{prop:app:wdet}.
There, the two clauses' witness premises
were about the same instance
and asserted opposite values,
and the induction hypothesis finished the case in one line,
with no consistency and no soundness.
Here the same collision costs an instantiation,
an appeal to soundness,
and an appeal to closed determinism.
The collision is the single argument
that must be re-earned at the official level,
and the comparison shows
what the witness premises are worth.

\subsubsection{Determinism on open terms}
\label{sec:app:sound:opendet}

\begin{thmc}[determinism, \cref{thm:det};
  \textsf{osr\_unique}, \textsf{hev\_det},
  \textsf{osr\_hev\_excl}]
Evaluation is deterministic in both layers,
across different step indices,
and the two judgments are mutually exclusive,
for \emph{all} terms, open or closed.
\end{thmc}

\begin{proof}
The mutual induction of \cref{prop:app:wdet}, replayed.
Every case is the verbatim mirror,
using the official inversions of
\cref{sec:app:sound:mirrors},
except the two quantifier clauses.
In the true clause,
a competing derivation could only be a false one,
which would supply a refutation certificate
for some numeral instance
while the clause's own premise
supplies a schematic certificate;
\cref{prop:app:collision} refutes that pair,
so the competing derivation cannot exist
and the value is $\tvT$.
The false clause is the mirror image.
Exclusivity in both quantifier clauses is as before:
the function position head-evaluates
to the bare quantifier,
and an application headed by it is
neither a spine-growth nor a $\beta$ redex.
\end{proof}

The order of events here is worth restating,
because the two halves are easy to conflate.
Closed determinism came from the witnessed pair
via the collapse.
\emph{Open} determinism comes from
closed determinism plus soundness,
through \cref{prop:app:collision}.
There is no circularity,
but there is a genuine two-stage structure,
and a reader who expects determinism of an evaluator
to be a syntactic inspection of its rules
should be told plainly what has happened:
in \rga determinism is a theorem \emph{about proofs},
and it is proved twice,
at two different strengths,
for two different fragments.

\subsubsection{The certificate lift, verified}
\label{sec:app:sound:q22}

\paragraph{Why the certificate is lifted.}
An unlifted \irl{allT} clause is not stable under substitution:
\cref{exa:app:diag} below is a counterexample
small enough to check by hand.
The unlifted clause would demand a proof of
$\{\vz \jnat\} \vdash u \cdot \vz$
rather than of
$\{\vz \jnat\} \vdash \lft u \cdot \vz$.
On a closed body the two are the same.
On an \emph{open} body they differ,
and the unlifted version is wrong:
a free variable $0$ of $u$ is conflated
with the schematic variable,
so the certificate asserts only
the \emph{diagonal} instances.

\begin{exa}[the diagonal witness;
  \textsf{lifted\_diag\_excluded}]
\label{exa:app:diag}
Take the open body $\lambda(\vz = v_1)$.
Under the abstraction,
$\vz$ is the bound variable
and $v_1$ is the body's own free variable $0$.

Under the unlifted clause,
the certificate judgment would be
$\{\vz \jnat\} \vdash (\lambda(\vz = v_1)) \cdot \vz$,
whose $\beta$-instance is $\vz = \vz$:
substituting the schematic variable
for the bound variable,
the free $v_1$ shifts down onto the \emph{same} variable.
So the certificate judgment is
the hypothesis itself, trivially derivable,
and $\cA \cdot \lambda(\vz = v_1)$
would evaluate to $\tvT$ --
while its all-zeros closing evaluates to $\tvF$.
So the unlifted clause is not stable under substitution.

With the lift, $\lft{(\lambda(\vz = v_1))}$
is $\lambda(\vz = v_2)$,
so the $\beta$-instance of the certificate judgment
is $\vz = v_1$
and the two variables stay distinct.
The pathology is gone,
and the following claim verifies it:
$\cA \cdot \lambda(\vz = v_1)$
does \emph{not} evaluate to $\tvT$ at any index.
\end{exa}

\begin{proof}
Suppose it did.
Inverting the evaluation
(\cref{sec:app:sound:mirrors}),
the false alternative giving the wrong value,
yields a valid proof of
$\{\vz \jnat\} \vdash (\lambda(\vz = v_2)) \cdot \vz$.
Apply substitutional soundness
at the closing substitution
sending variable $0$ to $\z$
and every other variable to $\tofn 1$.
The hypothesis $\vz \jnat$ closes to $\z = \z$, true,
so the conclusion closes to a true formula.
Compute that closure:
the substitution enters the abstraction lifted by one,
so the bound variable is untouched
and $v_2$ becomes $\tofn 1$,
making the conclusion
$(\lambda(\vz = \tofn 1)) \cdot \z$.
But that term $\beta$-reduces to $\z = \tofn 1$,
which evaluates to $\tvF$,
so by the $\beta$ application clause
the term itself evaluates to $\tvF$.
Open determinism
(\cref{sec:app:sound:opendet})
makes $\tvT = \tvF$.
\end{proof}

\Cref{exa:app:diag} is small enough to check
with a pencil,
and it is the one place where a reader can see
\emph{why} a de Bruijn lift in the semantics
is not bookkeeping.
It also exhibits the shape of the failure:
a semantics that accidentally quantifies over
the diagonal instead of over all instances.

\subsubsection{Grounded truth is provability}
\label{sec:app:sound:capstone}

One consequence of substitution stability
relates the two readings of an open formula.
An open evaluation to $\tvT$ implies \emph{schematic} truth
(\textsf{osr\_open\_sjt}):
if $c \bev{s} \tvT$
then every closing instance of $c$ is grounded-true,
that is, $\gtrue{\emptyset}{c}$,
immediately from \cref{prop:app:stability},
closing substitutions being in particular
pointwise well-formed;
by \cref{lem:app:embed},
$c$ is then also grounded-true under every assignment.
The converse,
from truth under every closing substitution
(or every assignment)
back to an open evaluation to $\tvT$,
is not proved.
Together with open completeness,
an open evaluation to $\tvT$ therefore gives
both schematic truth and schematic derivability.

Assembling exactly as in
\cref{sec:app:sound:targets},
but stopping one step earlier --
before converting the open invariant
into the closing one --
gives \emph{open soundness}:
if $\Gamma \vdash c$ and every hypothesis is well-formed,
then $\ogtrue{\Gamma}{c}$,
so that for every pointwise well-formed substitution
satisfying $\Gamma$, $c$ is true.

\begin{lem}[provability gives open grounded truth;
  \textsf{rga\_provable\_wosr}]
\label{lem:app:prov-true}
If $\emptyset \vdash c$ then $c \bev{s} \tvT$ for some $s$.
\end{lem}

\begin{proof}
Apply open soundness at the \emph{identity} substitution,
the one sending each variable to itself.
It is pointwise well-formed,
a variable being a well-formed term;
the empty hypothesis set is satisfied vacuously;
and the identity substitution applied to $c$
is $c$ itself.
So $c$ evaluates to $\tvT$ in the witnessed sense,
hence by projection (\cref{lem:app:project})
in the official sense.
\end{proof}

That short proof is the one place
in the entire development
where the extra generality of the open motive is spent,
and it is why the induction of
\cref{prop:app:mainind} was run against it.
The identity substitution is well-formed
but is \emph{not} closing:
it sends variables to variables,
and a variable is not a closed term.
Under the closing-substitution motive
the identity is simply not an admissible instance,
and \cref{lem:app:prov-true} is unreachable.
Every other consumer of
\cref{prop:app:mainind} in this appendix
uses only the closing form.
So the whole of \cref{sec:app:sound:induct}
was proved at the stronger motive
for the sake of this single application --
which is the one that yields the capstone.

\begin{thmc}[truth is provability, \cref{thm:iff};
  \textsf{rga\_open\_iff}, \textsf{rga\_val\_iff}]
For well-formed $c$, open or closed:
$\emptyset \vdash c$ if and only if
$c \bev{s} \tvT$ for some $s$;
and $\emptyset \vdash c = \tofn{r}$ if and only if
$c \bev{s} r$ for some $s$.
\end{thmc}

\begin{proof}
For the first, left to right is
\cref{lem:app:prov-true},
and right to left is open completeness
(\cref{sec:app:complete:thms}).
For the second, left to right applies the first
to the equation $c = \tofn{r}$
to obtain a witnessed truth for it,
then inverts:
the equality clause gives both sides a common value,
numeral inversion pins the right side to $r$,
and truth of the equation forces the two values equal,
so $c$ has the witnessed value $r$;
project.
Right to left is value completeness.
\end{proof}

\begin{thmc}[\tnat-soundness, \cref{thm:nsound};
  \textsf{rga\_N\_sound}]
If $\emptyset \vdash t \jnat$
then $t \bev{s} k$ for some step index $s$
and some natural number $k$.
\end{thmc}

\begin{proof}
$t \jnat$ is $t = t$;
provability gives it a witnessed truth
by \cref{lem:app:prov-true};
inverting the equality clause gives $t$ a value.
\end{proof}

\tnat-soundness is \emph{habeas quid} stated as a theorem,
and it is what makes the totality results
of \cref{sec:power} mean anything.
A derivable claim $t \jnat$ is never vacuous:
behind it stands an actual terminating computation
producing an actual number.
Every result of the form
``\rga proves this function total''
converts, through this theorem,
into a termination guarantee
for the underlying computation.

Finally, the collapse extends
from closed terms to all well-formed ones.
For well-formed $c$, open or closed,
$c$ has the value $r$ officially
if and only if it has the value $r$
in the witnessed sense.
Right to left is the projection.
Left to right goes the long way round:
an official value gives a provable equation,
by the second half of the capstone theorem;
the equation gives a witnessed truth,
by \cref{lem:app:prov-true};
and that inverts to a witnessed value.

The route is worth noticing.
The extended collapse and the determinism theorem of
\cref{sec:app:sound:opendet}
are \emph{independent} consequences of soundness,
and neither consumes the other.
The extended collapse runs from open soundness
through \cref{lem:app:prov-true}
to the capstone theorem and back;
open determinism runs from closed determinism
through the certificate collision.
Nothing doubles back.
The order in which the two appear here
is therefore a matter of exposition
rather than of dependency:
stating the collapse while determinism was still owing
would have invited the suspicion
that the owed result had been used to get it.

\subsection{The \texorpdfstring{$\omega$}{omega}-Semantics}
\label{sec:app:sound:omega}

The theories \texttt{RGA\_Omega\_Model} and \texttt{RGA\_Omega\_Sound}
carry the soundness argument over to the $\omega$-semantics
of \cref{sec:metatheory:omega},
and the transfer is a simplification.
$\omega$-evaluation $t \omv r$ is defined inductively,
as the least relation closed under
the clauses of the evaluation pair without step indices,
with the two quantifier clauses replaced:
$\cA \cdot u$ is $\omega$-true
when $u \cdot \tofn{n}$ is $\omega$-true for every $n$,
and $\omega$-false when some $u \cdot \tofn{i}$ is $\omega$-false.
These are exactly the witness premises
of the witnessed clauses of \cref{sec:app:sound:witness},
with the certificates dropped.

\begin{thmc}[\cref{thm:omv-extends};
  \textsf{omv\_unique}, \textsf{osr\_omv}]
$\omega$-evaluation is deterministic,
and if $t$ is closed and $t \bev{s} r$ then $t \omv r$.
\end{thmc}

\begin{proof}
Determinism is the argument of \cref{prop:app:wdet}
with the step index deleted:
the only collision that is not structural
is a universal that is both $\omega$-true and $\omega$-false,
and the true clause carries the truth of
the very instance the false clause refutes.
For the extension,
a closed evaluation is witnessed (\cref{thm:collapse}),
and dropping the certificates from
a witnessed derivation leaves an $\omega$-derivation
(\textsf{wosr\_omv}).
\end{proof}

\begin{thmc}[\cref{prop:omv-liar}; \textsf{omv\_liar\_unvalued}]
The Liar $L$ has no value in the $\omega$-semantics.
\end{thmc}

\begin{proof}
The argument of \cref{sec:app:character:noval},
with $\omega$-evaluation:
a value of $L$ transfers along the Liar's two head steps
to a value of $\neg L$,
so $L$ has a boolean value $\tv{b}$ with $\tv{b} = \tv{\neg b}$
by determinism, which is impossible.
\end{proof}

\begin{thmc}[\cref{thm:omega-sound};
  \textsf{rga\_omega\_sound\_strip}, \textsf{rga\_omega\_true\_consistent}]
If $\Gamma \vdash c$
and every well-formed hypothesis in $\Gamma$ is $\omega$-true,
then $c$ is $\omega$-true.
Consequently no finite set of well-formed $\omega$-true hypotheses
derives $\z = \tofn{1}$.
\end{thmc}

\begin{proof}
The main induction (\cref{prop:app:mainind})
is repeated with $\omega$-truth in place of witnessed truth:
for every pointwise well-formed substitution
that makes the hypotheses $\omega$-true,
the conclusion is $\omega$-true at it
(\textsf{prf\_oomjt}).
The supporting machinery is that of this appendix
with the step bookkeeping deleted:
substitution and lift stability,
head-step simulations in both directions,
and an $\omega$-valued congruence with its transport.
Two places become simpler.
The congruence transports the instance facts of
an $\omega$-true universal directly,
since the $\omega$-clauses carry no certificates,
so the single-site tower of \cref{sec:app:sound:vcr}
is unnecessary.
And the universal-introduction case needs
no certificate manufacture:
its premise judgment, instantiated by substitution
at each numeral, gives the truth of every instance,
which is exactly what the $\omega$-clause asks for.
The theorem as stated is the identity-substitution reading:
for an open $c$, $\omega$-truth is
$\omega$-evaluation of $c$ itself,
with its free variables stuck.
This form, \textsf{rga\_omega\_sound}, assumes every hypothesis
well-formed.
Stripping the ill-formed hypotheses first,
as for \cref{thm:sound} (\cref{prop:app:strip}),
removes that assumption (\textsf{rga\_omega\_sound\_strip}).
For the consequence,
$\z = \tofn{1}$ is $\omega$-false,
and $\omega$-evaluation is deterministic.
\end{proof}

\begin{thmc}[$\omega$-consistency, \cref{thm:omega-consistent};
  \textsf{rga\_omega\_consistent}]
There is no $u$ with
$\emptyset \vdash \neg(\cA \cdot u)$
and $\emptyset \vdash u \cdot \tofn{n}$ for every $n$.
\end{thmc}

\begin{proof}
From $\emptyset \vdash \neg(\cA \cdot u)$
the witness property (\cref{thm:witness})
gives a numeral $\tofn{n}$ with
$\emptyset \vdash \neg(u \cdot \tofn{n})$,
and $\emptyset \vdash u \cdot \tofn{n}$ contradicts
consistency.
The proof uses neither the $\omega$-semantics
nor its soundness.
\end{proof}

\subsection{Closing Remarks}
\label{sec:app:sound:remarks}

\paragraph{The witnessed pair is used twice.}
It is introduced in
\cref{sec:app:sound:witness}
to make certificates' semantic content available
without assuming their soundness,
and taken down in
\cref{sec:app:sound:collapse,sec:app:sound:capstone}
once soundness makes that content recoverable.
The official semantics is never modified.
The pair is used a second time,
outside the proof it was built for.
\Cref{sec:app:represent:onlyif}
needs to read an accepting step
\emph{out} of a refuted universal,
and at the official level a false universal
offers only a certificate,
from which extracting the step
would mean analysing a proof.
At the witnessed level the instance's falsity
is a premise, so inversion hands the step over.

\paragraph{Exactly two of the 42 rules are sound vacuously,
and each vacuity says something.}
Explosion is vacuous by determinism,
so \rga's consistency is visible inside one rule case.
Bottom typing is vacuous because $\bot$ has no value,
which is the semantic statement of
what the constant is for.
No third case joins them.
Each quantifier elimination does have
one never-taken alternative,
excluded because $\tvT \ne \tvF$,
but a dead alternative is not a dead case:
those two rules have live alternatives as well,
and it is the live ones that carry their soundness.

\paragraph{Openness costs one lemma, three times.}
\Cref{rem:app:openness} collects the three devices.
\Cref{sec:app:sound:capstone} is where
the bill is settled:
the open form of completeness makes
the open soundness motive provable,
the open motive admits the identity substitution,
and the identity substitution is what turns
soundness into the open half of \cref{thm:iff}.

%% file: app-compile.tex
\section{The Verified Compiler in Detail}
\label{sec:app:compile}

This appendix transcribes the machine-checked development
behind \cref{sec:power:compile}
(\texttt{RGA\_Compile}),
whose one displayed result is
compiler correctness (\cref{thm:cmp}).
There is a function $\cmp$
taking each index of the two-argument
primitive-recursive calculus
to a closed, well-formed \rga term,
such that the compiled term computes
the indexed function --
on arbitrary closed argument terms that evaluate,
not merely on numerals.

Everything downstream consumes that theorem as a black box.
Recursive enumerability (\cref{sec:app:re}),
representability (\cref{sec:app:represent})
and internal arithmetic (\cref{sec:app:arith})
all need ``\rga can run any primitive recursive function'',
and this is where the claim is earned.

Two features of the appendix are worth knowing in advance.
The appendix touches no derivability, no proof rule
and no certificate:
the whole of it is a statement about
the evaluation relation of \cref{fig:sem},
so it may be read straight after \cref{sec:system}
and needs nothing from
\cref{sec:app:complete,sec:app:sound}.
And it contains no freshness reasoning whatever,
for the reason \cref{sec:app:compile:closed} gives --
which will surprise a reader arriving from
\cref{sec:app:sound}.

A notational caution.
Two kinds of zero appear throughout,
and the distinction is load-bearing:
$0$ is the natural number,
an element of the calculus's value domain,
while $\z$ is \rga's zero \emph{term}.
The compiler's business is precisely
to relate the two.

\subsection{The Two-Argument Calculus}
\label{sec:app:compile:calc}

The source language is
a two-argument primitive-recursive function calculus,
following Robinson's reduction of
general primitive recursion to the two-argument
form~\cite{robinson50essentially}.
Per the scope of these appendices
the calculus is standard background:
its definitions are recalled here
and its laws cited, and none is re-derived.

An \emph{index} is generated by
\[
  f \;::=\; \iZ \mid \iS \mid \iL \mid \iR
    \mid \iSub\,g\,h_1\,h_2
    \mid \iRec\,g\,h\,h' ,
\]
and $\den{f} : \snat^2 \to \snat$ is its denotation.
The four initial indices denote
\[
\begin{array}{r@{\quad}l@{\qquad}r@{\quad}l}
  \den{\iZ}(x,y) = 0 &\text{zero} &
  \den{\iL}(x,y) = x &\text{left projection} \\
  \den{\iS}(x,y) = x+1 &\text{successor} &
  \den{\iR}(x,y) = y &\text{right projection,}
\end{array}
\]
and the two combining forms denote
\begin{align*}
  \den{\iSub\,g\,h_1\,h_2}(x,y)
    &= \den{g}\bigl(\den{h_1}(x,y),\,\den{h_2}(x,y)\bigr) , \\
  \den{\iRec\,g\,h\,h'}(x,0)
    &= \den{g}(x,0) , \\
  \den{\iRec\,g\,h\,h'}(x,k{+}1)
    &= \den{h}\bigl(\den{h'}(x,k),\;
        \den{\iRec\,g\,h\,h'}(x,k)\bigr) .
\end{align*}

One feature of the recursion form
will otherwise stop a reader,
and it is worth disposing of at once.
Why are there \emph{three} sub-indices?
In the usual presentation
the step function of a primitive recursion
takes three arguments:
the parameter $x$, the counter $k$,
and the value computed so far.
A calculus admitting only two-argument functions
cannot supply three,
so two of them are packed into one.
That is the role of $h'$:
it combines the parameter and the counter
into a single number $\den{h'}(x,k)$,
and $h$ then takes that together with
the recursive value.
So $h'$ is a currying workaround
rather than a second piece of recursion,
and it carries no computational content of its own.
It matters below only because
the compiled term has to mirror it.

The same calculus underlies
the G\"odel coding, the proof checker,
and the classical recursion theory
the development uses;
\cref{sec:app:re} leans on that.
Here it is only the source language.

\subsection{The Compiler}
\label{sec:app:compile:cmp}

\begin{defi}[the compiler; \textsf{cmp}]
\label{defi:app:cmp}
$\cmp$ is defined by recursion on the index.
Under two nested abstractions,
$v_1$ is the first argument and $\vz$ the second.
The four initial indices compile to
\[
\begin{array}{r@{\quad\mapsto\quad}l@{\qquad}r@{\quad\mapsto\quad}l}
  \iZ & \llam\, \z &
  \iL & \llam\, v_1 \\
  \iS & \llam\, (\suc v_1) &
  \iR & \llam\, \vz ,
\end{array}
\]
substitution compiles to
\[
  \cmp(\iSub\,g\,h_1\,h_2) \;\ldef\;
  \llam\,\bigl(
    \cmp g \cdot (\cmp h_1 \cdot v_1 \cdot \vz)
           \cdot (\cmp h_2 \cdot v_1 \cdot \vz)\bigr) ,
\]
and primitive recursion compiles to
\[
  \cmp(\iRec\,g\,h\,h') \;\ldef\;
  \llam\,(\cR \cdot G \cdot H \cdot \vz) ,
\]
where
$G \ldef \cmp g \cdot v_1 \cdot \z$ and
$H \ldef \llam\,\bigl(
  \cmp h \cdot (\cmp h' \cdot v_3 \cdot v_1) \cdot \vz\bigr)$.
\end{defi}

Composition of functions is literally
application of terms:
the substitution clause hands the two arguments
to each compiled $h$
and the two results to the compiled $g$.
In the recursion clause,
$\vz$ in the scrutinee position of $\cR$
is the recursion argument,
$G$ is the base branch,
and $H$ is the step branch,
which the recursor applies to the counter
and to the recursive value.

The de Bruijn arithmetic inside $H$
is the one place where an index is not obvious.
$H$ introduces two further abstractions
beyond the compiler's outer two,
so the outer first argument --
the parameter $x$ --
has been shifted twice and appears as $v_3$.
$H$'s own two variables are
$v_1$, bound by its outer $\lambda$
and receiving the counter,
and $\vz$, bound by its inner $\lambda$
and receiving the recursive value.
So $H$ reads: given counter $v_1$
and recursive value $\vz$,
compute $\cmp h'$ at the parameter and the counter,
then feed that and the recursive value to $\cmp h$.
That is exactly the second defining equation of
$\den{\iRec\,g\,h\,h'}$,
with $h'$'s currying workaround
mirrored by the inner application.

\paragraph{Why primitive recursion needs no detour.}
The shape of the recursion clause is not a choice.
\rga's recursor was introduced in \cref{sec:system:syntax}
as a primitive with two dispatch clauses,
and those clauses \emph{are}
the two defining equations of $\den{\iRec\,g\,h\,h'}$:
\begin{center}
\begin{tabular}{@{}ll@{}}
\textbf{the calculus} & \textbf{\rga's recursor} \\
\hline
$\den{\iRec\,g\,h\,h'}(x,0) = \den{g}(x,0)$ &
  $\cR \cdot G \cdot H \cdot \z \hs G$ \\
$\den{\iRec\,g\,h\,h'}(x,k{+}1) = \den{h}(\dots)$ &
  $\cR \cdot G \cdot H \cdot (\suc \tofn{k}) \hs
    H \cdot \tofn{k} \cdot
    (\cR \cdot G \cdot H \cdot \tofn{k})$ \\
\end{tabular}
\end{center}
Primitive recursion therefore maps onto $\cR$ directly.
There is no detour through
G\"odel-$\beta$ coding of finite sequences,
which a first-order arithmetization would need
in order to speak about the course of a recursion.
\rga can simply recurse,
so the compiler simply recurses --
and the table is why
\cref{sec:power:compile}'s claim to that effect
is a matter of inspection rather than of argument.

\paragraph{What the compiled terms are made of.}
They use $\lambda$-abstraction, application, variables,
the numeral constructors $\z$ and $\suc$,
and the recursor $\cR$.
They use no logical operator,
no $\bot$,
and no quantifier combinator $\cA$.
The image of the compiler therefore lies in
the purely computational, quantifier-free fragment of \rga,
which is why compiled functions may be dropped into
any \rga term without disturbing
anything the logic does.

\subsection{Closedness, and the Absence of Freshness Reasoning}
\label{sec:app:compile:closed}

Every compiled term is closed.
The proof is an induction on the index:
each clause binds both its variables
with the two outer abstractions,
and the recursion form's inner $H$ binds its own two,
so no variable escapes.

Three consequences follow immediately,
all of the form
\emph{substitution does nothing to a compiled term}.
For any substitution $\sigma$,
any lift, and any $\beta$-instance,
\[
  \sof{\sigma}{\cmp f} = \cmp f, \qquad
  \lft{(\cmp f)} = \cmp f, \qquad
  \bsub{\cmp f}{t} = \cmp f .
\]
Two further trivia are used silently below:
the identity substitution is the identity on terms,
and every numeral is closed.

Closedness is kept as a hypothesis on purpose,
and it is why this appendix looks so different
from \cref{sec:app:complete,sec:app:sound}.
Those two pay, in three separate places,
for admitting \emph{open} terms --
the theme collected in \cref{rem:app:openness}.
Here the opposite decision is taken.
The compiler only ever produces closed terms,
and \cref{thm:cmp} only ever consumes closed arguments,
so every de Bruijn lift in the proof
collapses to the identity
and no freshness choice is made anywhere.
Nothing is lost by the restriction,
since the compiler's consumers downstream
only ever apply it to closed terms,
and a great deal of bookkeeping is avoided.

\subsection{Two Evaluation Kits}
\label{sec:app:compile:kits}

Every compiled term is a two-argument abstraction,
and every case below evaluates one at two arguments.
The following lemma does that once.

\begin{lem}[the two-argument $\beta$ kit;
  \textsf{osr\_beta2}]
\label{lem:app:beta2}
Suppose $\bsub{(\lambda\,b)}{A} = \lambda\,b'$ --
that is, substituting $A$ into the outer body
leaves an abstraction --
and suppose $\bsub{b'}{B}$ evaluates to $r$.
Then $(\llam\,b) \cdot A \cdot B$ evaluates to $r$.
\end{lem}

\begin{proof}
Three evaluation clauses in sequence.
The term $\llam\,b$ head-evaluates to
the head shape $\lambda\,(\lambda\,b)$ in one step,
by the abstraction clause.
Its $\beta$-instance at $A$ is $\lambda\,b'$ by hypothesis,
which head-evaluates to the head shape $\lambda\,b'$,
again by the abstraction clause;
the head-$\beta$ clause for \emph{head} evaluation
combines the two,
so $(\llam\,b) \cdot A$ head-evaluates to $\lambda\,b'$.
The head-$\beta$ clause for \emph{value} evaluation
then combines that with
the hypothesis about $\bsub{b'}{B}$,
the two raised to a common index by monotonicity.
\end{proof}

The intermediate substitution result $\lambda\,b'$
is a hypothesis of the lemma
rather than something computed inside it.
That is deliberate:
each call site below has
its own de Bruijn computation to perform,
and stating the lemma this way
lets the call site do that once and hand in the answer.
Every use below has the shape
``compute the two $\beta$-instances,
then apply the kit''.

The second kit handles the recursor.
A two-argument recursor spine $\cR \cdot G \cdot H$
head-evaluates to the two-argument recursor shape
in three steps:
the bare $\cR$ head-evaluates to itself,
and two spine-growth clauses accumulate
$G$ and then $H$.
Given that,
if $M$ evaluates to $0$ and $G$ evaluates to $r$,
then $\cR \cdot G \cdot H \cdot M$ evaluates to $r$;
and if $M$ evaluates to $k{+}1$
and the unfolded step term
$H \cdot \tofn{k} \cdot
  (\cR \cdot G \cdot H \cdot \tofn{k})$
evaluates to $r$,
then $\cR \cdot G \cdot H \cdot M$ evaluates to $r$.
Both follow by raising the given indices
to a common value by monotonicity,
supplying the head shape,
and applying the corresponding recursor clause
of value evaluation.

The scrutinee $M$ here is
an arbitrary evaluating term, not a numeral.
The distinction matters elsewhere and not here.
\rga's \emph{syntactic} head step for the recursor
fires only at a literal numeral,
while the \emph{evaluation clause}
dispatches on the scrutinee's value,
and the gap between the two has to be bridged
in \cref{sec:app:complete:rec}
and again in \cref{sec:app:sound:convcases}.
This appendix works only with evaluation,
so the gap never opens.

\subsection{Compiler Correctness}
\label{sec:app:compile:correct}

\begin{thmc}[compiler correctness, \cref{thm:cmp};
  \textsf{cmp\_correct}]
Let $A$ and $B$ be closed terms
with $A$ evaluating to $a$ and $B$ evaluating to $b$.
Then for every index $f$,
$\cmp f \cdot A \cdot B$ evaluates to $\den{f}(a,b)$.
\end{thmc}

The proof is an induction on the index $f$,
with $A$, $B$, $a$, $b$
and the argument evaluation indices all generalized --
so that the induction hypothesis is available
at \emph{any} pair of closed evaluating arguments,
not at the pair fixed in the conclusion.

\paragraph{The generalization is the whole reason
the induction closes.}
The statement a reader expects is
the one about numerals:
$\cmp f \cdot \tofn{x} \cdot \tofn{y}$
evaluates to $\den{f}(x,y)$.
That statement is \emph{not inductive}.
In the substitution case
the induction hypothesis for $g$
must be applied at the arguments
$\cmp h_1 \cdot A \cdot B$ and $\cmp h_2 \cdot A \cdot B$,
which are compiled applications --
not numerals, and not even values.
Restricting to numerals breaks the case immediately.
The numeral form is recovered afterwards
as a corollary (\cref{sec:app:compile:cor}),
and it is a corollary rather than the theorem
for exactly that reason.

The generalization is the direct analogue of
the spine quantification of \cref{defi:app:xfer}:
the same failure mode --
an induction hypothesis needed at objects
more general than the ones the conclusion mentions --
and the same repair.
Both times the statement that looks like the theorem
is the one that cannot be proved,
and the harder-looking statement is the easier one.

Closedness earns its place in the hypothesis twice over:
it keeps every $\beta$-instance computation collapsing
(\cref{sec:app:compile:closed}),
and the induction hypothesis must be re-applied
at newly constructed arguments,
which requires those to be closed as well.
Each case therefore checks closedness of
the sub-terms it builds.
Those checks are routine --
an application of closed terms is closed --
and are not spelled out below.

\subsubsection{The four initial functions}
\label{sec:app:compile:init}

All four are the same three-line argument:
compute the two $\beta$-instances,
apply \cref{lem:app:beta2},
and read off the value.
\begin{center}
\begin{tabular}{@{}llll@{}}
\textbf{index} & \textbf{outer instance} &
  \textbf{inner instance} & \textbf{value} \\
\hline
$\iZ$ & $\lambda\,\z$ & $\z$ & $0$ \\
$\iS$ & $\lambda\,(\suc \lft A)$ & $\suc A$ & $a{+}1$ \\
$\iL$ & $\lambda\,\lft A$ & $A$ & $a$ \\
$\iR$ & $\lambda\,\vz$ & $B$ & $b$ \\
\end{tabular}
\end{center}
Two remarks cover all four.
The lifts $\lft A$ in the middle column
are introduced by the outer substitution
descending under the inner $\lambda$,
and they vanish at once because $A$ is closed.
And the four values are exactly
the four initial denotations of
\cref{sec:app:compile:calc}.

\subsubsection{Substitution}
\label{sec:app:compile:sub}

The goal is that
$\cmp(\iSub\,g\,h_1\,h_2) \cdot A \cdot B$
evaluates to
$\den{g}(\den{h_1}(a,b), \den{h_2}(a,b))$.

Substituting $A$ into the outer body
and then $B$ into the result
turns the compiled term applied to $A$ and $B$
into the body
\[
  \cmp g \cdot (\cmp h_1 \cdot A \cdot B)
         \cdot (\cmp h_2 \cdot A \cdot B) .
\]
Every compiled sub-term is untouched by both substitutions,
by \cref{sec:app:compile:closed};
only the two variables are replaced.
The induction hypotheses for $h_1$ and $h_2$,
applied at the \emph{given} arguments $A$ and $B$,
give that $\cmp h_1 \cdot A \cdot B$
evaluates to $\den{h_1}(a,b)$
and likewise for $h_2$.
The induction hypothesis for $g$ --
and this is the step that needs the generalization --
is then applied at the arguments
$\cmp h_1 \cdot A \cdot B$ and $\cmp h_2 \cdot A \cdot B$.
Those are closed, and by the previous step they evaluate,
so the hypothesis applies
and gives that the displayed body evaluates to
$\den{g}(\den{h_1}(a,b), \den{h_2}(a,b))$,
which is $\den{\iSub\,g\,h_1\,h_2}(a,b)$ by definition.
\Cref{lem:app:beta2} assembles.

Call-by-name is what makes that step
legitimate without further work.
The arguments are substituted unevaluated,
so nothing forces $\cmp h_1 \cdot A \cdot B$
to be reduced before it is passed,
and the induction hypothesis speaks about
exactly the term that ends up in argument position.

\subsubsection{Primitive recursion}
\label{sec:app:compile:rec}

The goal is that
$\cmp(\iRec\,g\,h\,h') \cdot A \cdot B$
evaluates to $\den{\iRec\,g\,h\,h'}(a,b)$.
Write, as in \cref{defi:app:cmp}
but with $A$ already substituted for the parameter,
\[
  G \ldef \cmp g \cdot A \cdot \z , \qquad
  H \ldef \llam\,\bigl(
    \cmp h \cdot (\cmp h' \cdot A \cdot v_1) \cdot \vz\bigr) .
\]
Both are closed, since $A$ is
and the compiled terms are.

The heart of the case is a subsidiary claim,
proved by ordinary induction on the natural number $m$:
\begin{quote}
for every closed $M$ evaluating to $m$,
$\cR \cdot G \cdot H \cdot M$
evaluates to $\den{\iRec\,g\,h\,h'}(a,m)$.
\end{quote}

\emph{Base.}
The scrutinee evaluates to $0$.
The outer induction hypothesis for $g$,
applied at the arguments $A$ and $\z$ --
the latter evaluating to $0$ --
gives that $G$ evaluates to $\den{g}(a,0)$.
Dispatch at zero finishes,
and $\den{g}(a,0)$ is
$\den{\iRec\,g\,h\,h'}(a,0)$
by the first defining equation.

\emph{Step.}
The scrutinee evaluates to $k{+}1$,
so dispatch at a successor reduces the goal to:
the unfolded step term
$H \cdot \tofn{k} \cdot
  (\cR \cdot G \cdot H \cdot \tofn{k})$
evaluates to $\den{\iRec\,g\,h\,h'}(a,k{+}1)$.
Four ingredients assemble it.
\begin{enumerate}
\item \emph{The recursive call.}
  The \emph{inner} induction hypothesis,
  applied at the closed term $\tofn{k}$,
  which evaluates to $k$,
  gives that $\cR \cdot G \cdot H \cdot \tofn{k}$
  evaluates to $\den{\iRec\,g\,h\,h'}(a,k)$.
\item \emph{The currying workaround.}
  The \emph{outer} induction hypothesis for $h'$,
  applied at $A$ and $\tofn{k}$,
  gives that $\cmp h' \cdot A \cdot \tofn{k}$
  evaluates to $\den{h'}(a,k)$.
\item \emph{The step function.}
  The \emph{outer} induction hypothesis for $h$,
  applied at the two terms just produced --
  both closed, both evaluating --
  gives that
  $\cmp h \cdot (\cmp h' \cdot A \cdot \tofn{k})
    \cdot (\cR \cdot G \cdot H \cdot \tofn{k})$
  evaluates to
  $\den{h}(\den{h'}(a,k),
    \den{\iRec\,g\,h\,h'}(a,k))$,
  which is $\den{\iRec\,g\,h\,h'}(a,k{+}1)$
  by the second defining equation.
\item \emph{Assembly.}
  Computing $H$'s two $\beta$-instances --
  substituting $\tofn{k}$ for $H$'s outer variable
  and the recursive-call term for its inner one --
  turns the unfolded step term into
  exactly the term of the previous item,
  and \cref{lem:app:beta2} finishes.
\end{enumerate}

\emph{Closing the outer case.}
The subsidiary claim, applied at $M \ldef B$ --
closed, and evaluating to $b$ --
gives that $\cR \cdot G \cdot H \cdot B$
evaluates to $\den{\iRec\,g\,h\,h'}(a,b)$.
It remains to see that
$\cmp(\iRec\,g\,h\,h') \cdot A \cdot B$
reduces to that.
Computing the two $\beta$-instances
of the compiled term's body,
the outer substitution of $A$
replaces $v_1$ by $A$ and $v_3$ by $A$ --
the latter being the same parameter
seen from under $H$'s two extra binders --
producing $\lambda\,(\cR \cdot G \cdot H \cdot \vz)$,
and the inner substitution of $B$ replaces $\vz$ by $B$.
\Cref{lem:app:beta2} finishes.

\paragraph{Three induction hypotheses are live at once.}
Keeping them apart is the whole difficulty of the case.
The \emph{outer} induction is on the index,
and supplies hypotheses for
the three sub-indices $g$, $h$ and $h'$.
The \emph{inner} induction is on the natural number $m$,
and supplies the recursive call.
The step case uses all four:
the inner hypothesis for item 1,
and the outer ones for $h'$ and $h$
at items 2 and 3.
A reader who loses track of which is being applied
will find the case impenetrable;
a reader who keeps the two inductions apart
will find it mechanical.

\paragraph{Why the inner induction is on the value.}
The subsidiary claim is stated
``for every closed $M$ evaluating to $m$'',
with $m$ the induction variable
and the terms producing it quantified over.
It could not have been an induction on $M$:
the recursive call is at the term $\tofn{k}$,
which is not a sub-term of $M$ in any structural sense.
Inducting on the value
and quantifying over the terms that produce it
is the same device as in \cref{sec:app:sound:indcase},
where soundness of \rga's induction rule
runs a meta-induction on the scrutinee's value
for the same reason.

\subsection{Corollaries}
\label{sec:app:compile:cor}

\emph{At numerals.}
$\cmp f \cdot \tofn{x} \cdot \tofn{y}$
evaluates to $\den{f}(x,y)$.
The corollary is immediate from the theorem
with $A \ldef \tofn{x}$ and $B \ldef \tofn{y}$,
numerals being closed
and evaluating to themselves.
It is the form every downstream consumer actually uses,
and, for the reason given in
\cref{sec:app:compile:correct},
it is a corollary of the general statement
rather than something provable directly.

\emph{Well-formedness.}
Every compiled term is well-formed,
by induction on the index:
each clause builds only from well-formed pieces
at correct arities.

\subsection{Closing Remarks}
\label{sec:app:compile:remarks}

\paragraph{The compiler is a compiler, not an interpreter.}
$\cmp$ is a total, structurally recursive function
from indices to terms, computed once;
correctness then says that
the \emph{output term} computes the right function.
Nothing at run time inspects the index.

\paragraph{Nothing here touches the logic.}
No derivability, no proof rule, no certificate appears.
The whole appendix is a statement about
the evaluation relation,
so it can be read
immediately after \cref{sec:system}.

%% file: app-re.tex
\section{Recursive Enumerability in Detail}
\label{sec:app:re}

This appendix transcribes the machine-checked development
behind the first theorem of \cref{sec:power:re}
(\texttt{RGA\_Primrec}, \texttt{RGA\_RE}):
the set of codes of well-formed grounded-true formulas
is recursively enumerable (\cref{thm:truth-re}).
\Cref{sec:power:re} also states the representation theorem,
which \cref{sec:app:represent} covers;
both appendices back the same subsection.

The argument is three sentences long.
Grounded truth coincides with empty-context derivability
(\cref{sec:app:sound:capstone}).
Derivability is the existence of a valid proof.
Proof checking is primitive recursive --
which is what almost all of the work below establishes --
so the grounded-truth set is the domain of a search over proof codes,
and that is what recursive enumerability means.

\paragraph{Where the weight falls.}
The point is worth stating up front,
because it is not where the section title suggests.
The advertised claim is about \emph{the rule set}.
But the 42 rules are detected by one uniform method
and the per-rule work is short.
What is long, and what carries the mathematics,
is the layer beneath:
\rga's syntactic operations --
well-formedness, $\beta$-instantiation, head reduction --
have to be exhibited as
primitive-recursive functions on G\"odel codes.
Once they are, the rules follow almost formally.
\Cref{sec:app:re:wf,sec:app:re:beta,sec:app:re:hred}
are that layer and are the bulk of the appendix;
\cref{sec:app:re:rules} is the rule set, and is a table.

\subsection{Background}
\label{sec:app:re:bg}

The computability infrastructure is standard background,
shared with the predecessor systems
and long predating \rga.
It is recalled here and cited, not re-derived.

\emph{Two-argument primitive recursion.}
Write $\snat^2 \to \snat$ for the two-argument total functions,
and call such a function primitive recursive
when it is denoted by an index of
the calculus of \cref{sec:app:compile:calc}.
Closure under composition and the recursion schema
is the content of that calculus.

\emph{G\"odel coding.}
A coding layer supplies an injective $\quo{\cdot}$
for each syntactic type --
operators, terms, finite hypothesis sets,
judgments, proof lists --
with a matching $\mathrm{decode}$,
so that $\quo{\mathrm{decode}\,x} = x$ holds
exactly on those $x$ that are valid codes for the type.
Cantor pairing $\cpair{\cdot}{\cdot}$,
with projections,
packs tuples into single numbers.
The guard $\quo{\mathrm{decode}\,x} = x$
is used below as the primitive-recursive test for
``$x$ is the code of a term at all'',
which is why validity of codes
never needs a separate decision procedure.

\emph{Step functions.}
A \emph{step function} is any $f : \snat^2 \to \snat$,
and its domain is
$\sfdom{f} = \{\, n \mid f(n,s) \ne 0 \text{ for some } s \,\}$.
A set is \emph{recursively enumerable}
when it is $\sfdom{f}$ for some primitive recursive $f$.
That definition is the computational formulation
the development uses throughout --
search the second argument,
accept on any non-zero result --
and it is the one \cref{sec:power:re} states.

\emph{Encoded structural recursion.}
A combinator turns a variable case,
a binder-count case and an operator case
into a primitive-recursive function on \emph{term codes},
recursing on the argument list.
Every encoded operation below is built with it.
Its well-foundedness obligation is that
each argument's code is \emph{strictly} smaller
than the node's,
which the pairing coding discharges,
and it is what makes the encoded recursion legitimate.

\subsection{Well-Formedness}
\label{sec:app:re:wf}

The goal is a primitive-recursive function on codes
computing $\quo{\wf(t)}$ from $\quo{t}$.
The encoded function is \textsf{wft\_enc},
and its correctness \textsf{wft\_enc\_correct}.

The first step is to re-express $\wf$
as a check on the node
rather than as an inductive relation.
Well-formedness of an operator node holds
exactly when every argument is well-formed
and the triple
(operator, binder count, argument count)
is one of exactly ten rows:
$\bot/0/0$, $\z/0/0$, $\suc/0/1$, ${=}/0/2$,
$\neg/0/1$, $\lor/0/2$, $\lambda/1/1$,
${\cdot}/0/2$, $\cR/0/0$, $\cA/0/0$.

The encoding is that table transcribed over codes.
The three fields of a node are read off its code
by short compositions of the pairing projections:
writing $t$ for the node's code,
\[
\begin{array}{r@{\;=\;}l}
  \text{operator} & \mathrm{cpx}(\mathrm{cpy}\,t) \\
  \text{binder count} &
    \mathrm{cpx}(\mathrm{cpy}(\mathrm{cpy}\,t)) \\
  \text{argument list} &
    \mathrm{cpy}(\mathrm{cpy}(\mathrm{cpy}\,t)) .
\end{array}
\]
That idiom is worth seeing once,
because every encoded operation below is built from it
and nothing else is needed to read a node.
The ten rows are then compared,
and the recursive results for the arguments
are combined by checking that they are all $1$ --
done as ``their sum equals their length'',
which is a primitive-recursive way to say
``all true'' without an unbounded conjunction.
Correctness is the equation
$\quo{\wf(t)}$ equals the value computed at $\quo{t}$,
proved by the combinator's own induction principle.

\paragraph{The one structural difference from the predecessor,
and it is the binder.}
In \pga every operator carries binder count $0$,
so the arity table is a table of arities.
\rga has $\lambda$, which carries binder count $1$.
The table therefore pairs each operator with
\emph{both} its arity and its binder count,
and the $\lambda/1/1$ row is
the only one with a non-zero middle entry.
That is a small difference,
and it is the reason the whole encoded layer
has to be redone rather than inherited:
every structural walk below --
substitution, lifting, $\beta$, head reduction --
has to carry a binder depth that \pga's did not.

\subsection{\texorpdfstring{$\beta$}{beta}-Instantiation}
\label{sec:app:re:beta}

$\beta$-instantiation $\bsub{b}{t}$
substitutes $t$ for the $\lambda$-bound variable of $b$
and shifts the remaining free variables down.
Encoded directly it would need
its own structural walk with its own capture reasoning.
It needs neither, because
\[
  \bsub{b}{t} \;=\;
  \mathrm{unlift}\bigl(\ssub{0}{\lft t}{b}\bigr) :
\]
substitute the once-lifted argument at index $0$,
then pull every remaining free variable down one slot.
Both halves already have encoded
primitive-recursive forms --
single-variable substitution, and \emph{unlifting},
which is the structural mirror of lifting
with the shift reversed.
So $\bsub{b}{t}$ is a composition of
two encoded operations
and inherits primitive recursiveness from them
(\textsf{trm\_unlift\_enc}, \textsf{n2pr\_bsub}).

The identity itself is proved
in a depth-generalized form,
by induction on $b$ with the binder depth
universally quantified.
The reason is the usual one:
descending under a $\lambda$ increments the depth,
so the statement at depth $0$
is not inductive on its own.

Capture-avoidance disappears here,
and a reader expecting the encoded $\beta$ to be delicate
should be told why.
De Bruijn indices are doing the work.
The argument is lifted \emph{before} insertion,
so its own free variables are not captured
by the binders it passes under,
and the downshift afterwards closes the gap
that the eliminated binder leaves.
Neither step inspects names,
so both are plain structural walks
and both encode without incident.

\subsection{Head Reduction}
\label{sec:app:re:hred}

Head reduction is the longest single construction
in the appendix, and it comes in two halves.
The mathematics is finished in the first.

\subsubsection{The relation is a function}
\label{sec:app:re:hredfun}

\begin{defi}[the head reducer; \textsf{head\_reduce}]
\label{defi:app:hred}
$\hred$ maps a term to either a term or nothing,
by structural recursion.
Anything that is not a two-argument application node
gives $\None$.
At an application node $f \cdot u$ it tries, in order:
\begin{enumerate}
\item if $f = \lambda b$, return $\Some{\bsub{b}{u}}$,
  provided $\wf(u)$ --
  which is the $\beta$ rule's side condition;
  otherwise $\None$;
\item if $f = \cR \cdot g \cdot h$, dispatch on $u$:
  for $u = \z$ return $\Some{g}$ provided $\wf(h)$;
  for $u = \suc w$ with $w$ a \emph{syntactic numeral},
  return
  $\Some{h \cdot w \cdot (\cR \cdot g \cdot h \cdot w)}$
  provided $\wf(g)$;
  otherwise $\None$;
\item otherwise recurse into the function position,
  returning $\Some{f' \cdot u}$ if $f$ itself reduces --
  which is spine congruence.
\end{enumerate}
\end{defi}

Correctness runs in both directions
(\textsf{head\_reduce\_correct}):
$\hred(c) = \Some{c'}$ if and only if $c \hs c'$.
Left to right is an induction on the definition,
matching each branch to its rule;
right to left is an induction on $\hs$,
checking that each of the four rules
lands in the branch that produces it.

\paragraph{Where a \texorpdfstring{\cref{sec:system}}{Section 2}
design decision is cashed in.}
The recursor step fires only at a \emph{syntactic} numeral scrutinee,
never at a term merely provable to be one.
\Cref{sec:system} already gives the reason, twice.
The caption of \cref{fig:rules-reflect} calls $\hs$
``a primitive-recursively decidable syntactic relation,
not a judgment'',
and the prose says the relation is
``deliberately syntactic and primitive-recursively decidable'',
so that proof checking remains effective
and all semantic content stays in the premises.
\Cref{sec:system:proof} also explains why
the provable-numeral case lives in
the separate schematic rule \irl{\cR\suc{}{=}} instead.

What this appendix adds is the concrete form of the payoff.
Decidability is realized here as
a \emph{total function} $\hred$
returning a reduct or nothing,
so the four conversion detectors of
\cref{sec:app:re:rules}
compare function values instead of searching a relation.
That is the cheapest way a conversion rule could be detected,
and it is why the conversion row of
\cref{sec:app:re:table} is unremarkable
when it might have been the hardest thing in the appendix.

Two cautions about how far the counterfactual reaches.
The operative property is \emph{decidability},
not determinism.
A unique reduct would survive a relaxed side condition,
since given $\cR \cdot g \cdot h \cdot (\suc m)$
the reduct is determined by $m$ either way.
What a relaxed condition destroys is that
the relation's domain is decidable and --
more basically --
that $\hs$ is definitionally \emph{prior to} derivability.
A side condition reading ``$m$ is provably a natural''
would make the head-step relation and the proof system
mutually recursive,
which is exactly the entanglement
\cref{sec:system} avoids
by keeping all semantic content in the premises.

And no impossibility is established.
Nothing in the development shows that
no primitive-recursive detector for the conversion rules
could exist under a relaxed rule.
\rga provability is in fact undecidable:
the representation theorem (\cref{sec:app:represent})
represents the halting set $K$ by a closed term,
so a decision procedure for provability would decide $K$.
Undecidability alone does not settle the question.
The supportable claim is the positive one:
the syntactic restriction is what makes
the present construction possible,
and the present construction is what
the recursive-enumerability theorem runs on.
Whether some other construction could survive
the relaxation is open.

Two auxiliaries are total functions
for the same reason:
the test deciding whether a term
is a syntactic numeral,
and the three redex-shape extractors --
an abstraction's body,
a recursor's two accumulated arguments,
and a successor's numeral argument.

\subsubsection{The function is primitive recursive over codes}
\label{sec:app:re:hredenc}

The encoded reducer mirrors \cref{defi:app:hred}
clause for clause at the level of codes,
supported by code-level helpers:
encoded forms of the numeral test,
of the three shape extractors,
and of the constructors used to rebuild the reduct.

One clause is worth exhibiting,
since the others differ from it only in depth.
Take the recursor-at-zero branch,
which exercises all three things a clause does --
a nested shape test, a side condition, and a rebuild.
Abbreviating the two argument reads
by the idiom of \cref{sec:app:re:wf}
as $f$ and $u$, it is
\[
\begin{array}{l}
  \textbf{if } \mathrm{r2args}(f) \ne 0
    \quad\text{--- $f$ is a two-argument recursor spine} \\
  \quad \textbf{if } u = \quo{\z}
    \quad\text{--- the scrutinee is zero} \\
  \qquad \textbf{if } \mathrm{wf}\bigl(
      \mathrm{cpy}(\mathrm{r2args}(f) - 1)\bigr) = 1
    \quad\text{--- the side condition } \wf(h) \\
  \qquad\quad \textbf{then }
    \mathrm{Suc}\,\bigl(
      \mathrm{cpx}(\mathrm{r2args}(f) - 1)\bigr)
    \quad\text{--- return } g \\
  \qquad\quad \textbf{else } 0
    \quad\text{--- return nothing.}
\end{array}
\]
Three conventions are visible there,
and a re-mechanizer has to reproduce all three.
Encoded operations are two-argument functions
on naturals by convention,
with unused slots filled by a dummy
and several inputs packed by pairing.
The projections $\mathrm{cpx}$ and $\mathrm{cpy}$
serve as first and second components
of whatever pair a helper returns.
And ``returned nothing'' is the value $0$,
so a returned value is offset by one --
which is the $\mathrm{Suc}$ above,
and the $-1$ in every read of a helper's result.

That offset is also why the correctness statement
has the shape it does.
Written out, the statement named
\textsf{head\_reduce\_enc\_correct}
reads
\[
  \mathrm{cpx}\bigl(
    \mathrm{hred}_{\mathrm{enc}}\,\quo{0}\,
      \cpair{\quo{0}}{\quo{t}} - 1\bigr)
  \;=\; \quo{\hred(t)} ,
\]
with the dummy first argument and the pairing
as just described,
and the $-1$ undoing the offset before the payload is read.
Prose saying that the encoded reducer
``yields the code of the reduct'' conceals a decision:
$\snat$ has no option type,
so how ``returned a value'' is distinguished from
``returned nothing'' is a choice the encoding must make,
and anyone re-mechanizing this has to make it too.

Primitive recursiveness follows from
the structural-recursion combinator
together with the already-encoded operations of
\cref{sec:app:re:wf,sec:app:re:beta},
since a reduct is built from $\bsub{b}{u}$,
from the recursor's arguments,
and from applications.

This half is long without being deep,
and the reason is worth stating
so that nobody looks for content that is not there.
Every case of $\hred$ inspects a \emph{nested} shape --
an application whose function is an application
whose function is $\cR$ --
and each layer of nesting becomes
another pair of projections on codes.
A reader who has seen \cref{defi:app:hred}
and one encoded clause has seen the method;
the remaining clauses differ only
in how deep they project.

\subsection{Two Smaller Layers}
\label{sec:app:re:small}

\emph{Lifting a hypothesis set.}
Two rules -- universal introduction
and negated-universal elimination --
carry a lifted context in a premise,
so the check needs $\lft\Gamma$ over codes.
Apply the already-encoded term lift
to each member of the encoded finite set.

\emph{The composition kit.}
A toolkit of some three dozen closure lemmas,
all of the uniform shape
\emph{if these code-valued functions are primitive recursive,
then so is this composite},
covers each term constructor,
each destructor and the option projections,
equality of terms and of option-of-term and of finite sets,
the boolean connectives on encoded booleans,
and the operations of the preceding subsections.
They are declared as introduction rules,
which is the point:
a detector's primitive-recursiveness obligation
is discharged by a single chaining step.
No member of the kit differs from the others,
and none is used except through the chaining.

\subsection{Extract and Verify: the 42 Rules}
\label{sec:app:re:rules}

\subsubsection{The method}
\label{sec:app:re:method}

Fix a rule with $n$ premises.
Its detector is a predicate of
a candidate conclusion $c$
together with $n$ candidate premise conclusions
$c_1, \dots, c_n$,
and it does two things in order.

\emph{Extract.}
Apply \emph{total destructors} to $c$ and the $c_i$
to recover the rule's metavariables.
$\dst{1}$ and $\dst{2}$ return
the argument, or the pair of arguments,
of a node carrying a given operator at the right arity,
and nothing otherwise;
the option projections read off the results.
The detector's leading conjuncts are
the ``is not nothing'' checks --
that is, the shape tests.
Nothing is guessed:
each metavariable is \emph{read off}
the term in which it occurs.

\emph{Verify.}
From the recovered metavariables,
rebuild what the rule requires and compare for equality,
adding any $\wf(\cdot)$ side condition,
which \cref{sec:app:re:wf} decides.
Because the rebuild uses the ordinary constructors,
this half is exactly the rule's own statement
read as a computation.

\paragraph{One detector, written out.}
The method is short enough that a single rule
shows the whole of it.
Take equality symmetry,
which has one premise.
Its detector, over a candidate conclusion $c$
and a candidate premise conclusion $c'$, is
\[
  \dst{2}(c') \ne \None
  \;\wedge\;
  c = \mathrm{mk}_{=}\bigl(
    \mathrm{snd}\,\dst{2}(c'),\;
    \mathrm{fst}\,\dst{2}(c')\bigr) ,
\]
where $\dst{2}$ destructs an equality node
into its two sides
and $\mathrm{mk}_{=}$ rebuilds one.
The first conjunct is the shape test:
\emph{the premise is an equality}.
The second is the rebuild-and-compare:
\emph{the conclusion is that equality
with its two sides exchanged}.
Nothing is guessed --
both sides are read off $c'$ --
and there is no numeric range anywhere.

Its correctness runs in two short directions.
For soundness, assume the detector fires:
then $\dst{2}(c') = \Some{(a,b)}$
says exactly that $c'$ is $a = b$,
the second conjunct says $c$ is $b = a$,
and rule \irl{{=}sym} applies.
For completeness, given an application of
\irl{{=}sym},
its premise destructs successfully by construction
and its conclusion is the rebuild,
so both conjuncts hold.

Every one of the 41 detectors has that shape.
They differ in how many premises the combinator ranges over,
in which destructors the shape test uses,
and in how much rebuilding the comparison needs --
and in nothing else.

\emph{The only search is over the proof list.}
An $n$-premise detector is applied
through a combinator ranging over $n$-tuples of judgments
already present in the proof $P$,
and that list is part of the input.
No numeric range is searched,
and no bound is chosen or proved anywhere.

\subsubsection{Why the method is uniform}
\label{sec:app:re:uniform}

The predecessor system's method was \emph{bounded search}.
Given the conclusion $c$,
search the numeric range of codes below $\quo{c}$
for the rule's metavariables,
construct the pattern, and compare --
never destructing anything.
That method needs, per rule, a bound that is
primitive recursive in $c$
and proved to cover every witness,
and both obligations bite.
Two kinds of witness provably escape
any bound derived from $c$:
a pattern variable that does not occur in the conclusion,
since nothing in $c$ constrains it,
and a substituted index,
since the index may be taken enormous
while the conclusion never mentions it.
Those had to be canonicalized rather than bounded.

\rga inverts the method,
and the inversion removes the problem
rather than managing it.
Extract-and-verify \emph{destructs},
and destructing is deterministic:
the metavariables are read off the conclusion,
not guessed from a range.
There is nothing to bound,
hence no bound to get wrong,
hence no class of witnesses that escapes one.
The uniformity of the 42 rules is
a consequence of that determinism
rather than a happy accident,
and it is why this appendix can present
the rule set as a table
where the predecessor's would have needed
a per-rule discussion.

\subsubsection{The rule set}
\label{sec:app:re:table}

All 42 clauses of $\mathit{derives}$ have a detector,
and there are 41 detectors.
The two directions of grounded reflexivity
collapse to one,
because $t \jnat$ \emph{is} $t = t$,
so the two rules are the same check --
the same definitional identity noted in
\cref{sec:app:complete:prelim}.
Of the 41,
ten are obtained by interpreting generic rule locales
shared with the predecessor systems
rather than being written out,
and the other 31 are \rga-specific.

\begin{center}
\begin{tabular}{@{}lp{0.27\textwidth}p{0.38\textwidth}@{}}
\textbf{group} & \textbf{rules} &
  \textbf{what is extracted, and verified against} \\
\hline
structural & \irl{hyp}, \irl{wk}, \irl{cut}, \irl{sub} &
  context membership and insertion;
  for \irl{sub}, the substituted pair \\
truth-value bridge & \irl{T{=}}, \irl{F{=}} (each both ways) &
  the equation's two sides,
  against the numeral $\tofn 1$ or $\tofn 0$ \\
negation, disjunction &
  \irl{{\neg}E}, \irl{{\neg}{\neg}} (both ways),
  \irl{{\lor}I_{1\text{--}3}}, \irl{{\lor}E_{1\text{--}3}} &
  the connective node's arguments \\
equality &
  \irl{{=}refl} (both ways, one detector), \irl{{=}sym},
  \irl{{=}subs}, \irl{{\ne}sym}, \irl{{=}TI}, \irl{{=}TE} &
  the equality node's two sides;
  for \irl{{=}subs}, the substitution context \\
arithmetic &
  \irl{\z{}I}, \irl{\suc{}{\ne}\z}, \irl{\suc{}{\ne}I},
  \irl{\suc{}TI}, \irl{\suc{}{=}E}, \irl{\suc{}{\ne}E},
  \irl{\suc{}TE}, \irl{\bot{}TE} &
  the successor node's argument \\
conversion &
  \irl{conv}, \irl{convE}, \irl{conv{=}}, \irl{{\neg}convE} &
  \textbf{the reduct, computed by $\hred$ and compared} \\
recursor & \irl{\cR\suc{}{=}} &
  the recursor spine's three arguments,
  and the successor's argument \\
quantifier &
  \irl{{\cA}I}, \irl{{\cA}E},
  \irl{{\neg\cA}I}, \irl{{\neg\cA}E} &
  the $\cA$-application's body; the instance;
  the lifted context \\
induction & \irl{ind} &
  the motive, the two substituted instances,
  the scrutinee \\
\end{tabular}
\end{center}

The conversion row is the one that consumes
\cref{sec:app:re:hred}.
Those four detectors do not \emph{search} for a reduct.
They compute it and compare,
which is what the total reducer of
\cref{sec:app:re:hredfun} makes possible.

\subsection{The One-Step Check Is Correct}
\label{sec:app:re:check}

The one-step check (\textsf{primrec\_derives}) is
the 41-way disjunction of the detectors,
encoded as a boolean.
It is primitive recursive,
by the composition kit applied to
the 41 detector obligations.
And it is correct:
the check returns $1$ exactly when
$\mathit{derives}$ holds of the decoded judgment,
with soundness and completeness
proved separately
(\textsf{primrec\_derives\_sound},
\textsf{primrec\_derives\_complete})
and combined.

Soundness is one lemma per rule,
all of the same shape.
Assume the detector fires;
unpack the proof-list search
to obtain the premise judgments and the predicate;
turn each successful destructuring into a term equation,
since ``$\dst{2}$ of $c$ at operator $op$
is the pair $(a,b)$'' says exactly that
$c$ is the node $op\langle a,b\rangle$;
rewrite the conclusion into the rule's shape;
apply the rule.
Completeness is one case per rule,
the same shape run backwards:
given a rule application,
exhibit its conclusion's destructurings as successful
and its rebuild as equal.

\paragraph{Completeness is where a detector is really tested.}
Soundness only ever \emph{consumes}
a detector's existential:
it takes the detector's word
that suitable premises exist.
Completeness has to \emph{produce} them.
In a bounded-search method that asymmetry is dangerous.
A detector with an unmeetable bound
passes soundness
and fails only when completeness is attempted.
Under extract-and-verify the asymmetry is harmless,
there being no bound to be unmeetable,
and both directions are proved.

\subsection{Entailment and Grounded Truth}
\label{sec:app:re:truth}

$\mathit{valid}$ and $\mathit{proves}$
are built from $\mathit{derives}$
by structural recursion over the proof list,
so \cref{sec:app:re:check} makes them primitive recursive.
Packaged as a step function,
one that takes the code of a judgment
and the code of a candidate proof
returns a distinguished non-zero value
exactly when that proof yields that judgment.
Its domain is therefore
the set of codes of derivable judgments, giving:
\emph{entailment is recursively enumerable}
(\textsf{rga\_entails\_re}).

For grounded truth,
the step function $\mathit{truth}$
returns $1$ at $(x,s)$ exactly when
$x$ is a valid term code,
tested by $\quo{\mathrm{decode}\,x} = x$;
and $\mathrm{decode}\,x$ is well-formed,
tested by \cref{sec:app:re:wf};
and $\mathrm{decode}\,s$ is a proof of
$\emptyset \vdash \mathrm{decode}\,x$,
tested by \cref{sec:app:re:check}.
Its domain is
$\{\, \quo{c} \mid \wf(c)
  \text{ and } \emptyset \vdash c \,\}$.

\begin{thmc}[grounded truth is r.e., \cref{thm:truth-re};
  \textsf{rga\_truth\_re}]
The set
$\{\, \quo{c} \mid \wf(c)
  \text{ and } c \bev{s} \tvT \text{ for some } s \,\}$
is recursively enumerable.
\end{thmc}

\begin{proof}
The set in the statement is about \emph{evaluation};
the set just enumerated is about \emph{derivability}.
They are the same set,
by the truth--provability coincidence
(\cref{sec:app:sound:capstone}, \cref{thm:iff}),
which for well-formed $c$, open or closed, gives
$\emptyset \vdash c$ if and only if
$c \bev{s} \tvT$ for some $s$.
Rewriting along that equivalence
turns the statement's set into the step function's domain,
and \cref{sec:app:re:check}
supplies primitive recursiveness.
\end{proof}

\paragraph{The bridge is the proof.}
Nothing in the step function mentions evaluation.
No evaluation derivation is ever enumerated,
no step index is ever searched,
and the reflective quantifier clauses
are never unfolded.
The entire content of the theorem is:
replace grounded truth by derivability,
then enumerate proofs.
The replacement is legitimate only because of
\cref{sec:app:sound}'s capstone,
so \cref{sec:power:re} depends on
\cref{sec:metatheory}'s central theorem,
and a reader is better told that
than left to discover it.

That dependency also sharpens
what the theorem certifies.
The claim covers all of grounded truth --
quantifiers, unconstrained recursion
and reflective certificates included --
and it goes through because
the reflective quantifier's truth clause
was \emph{defined} as the existence of a proof,
and proofs are exactly what a search can enumerate.
The reflective quantifier adds nothing
beyond proof search,
and the route taken here shows why:
grounded truth is recursively enumerable
because it \emph{is} provability,
not because evaluation happens to be searchable.

\subsection{Closing Remarks}
\label{sec:app:re:remarks}

\paragraph{The work is in the syntax, not the rules.}
That \rga's substitution, $\beta$-instantiation
and head reduction are primitive-recursive functions
on codes is the work of this appendix;
that the rule set is primitive recursive
follows from it almost formally.

\paragraph{Where \cref{sec:system}'s decidability principle is spent.}
\Cref{sec:system} keeps $\hs$ syntactic and decidable
so that proof checking stays effective
and the head-step relation stays
definitionally prior to derivability.
This appendix is where that principle is spent:
decidability becomes a total reducer,
and the conversion detectors compare its values
rather than searching a relation.

\paragraph{Determinism replaces bounding.}
The predecessor bounded a search
and had to prove the bound adequate.
\rga destructs deterministically
and has nothing to bound.
A whole class of possible defects is absent
by construction rather than by care.

\paragraph{Dependencies.}
This appendix uses the metatheory once,
at the end:
the truth--provability coincidence
of \cref{sec:app:sound},
in \cref{sec:app:re:truth}.
Everything before that is computability.

%% file: app-represent.tex
\section{The Representation Theorem in Detail}
\label{sec:app:represent}

This appendix transcribes the machine-checked development
behind the second theorem of \cref{sec:power:re},
which is the theory \texttt{RGA\_Represent}.
A unary predicate $p$ on the naturals
is recursively enumerable
if and only if some closed well-formed term $t$
\emph{represents} it, in the sense that
$p(n)$ holds exactly when
$\emptyset \vdash t \cdot \tofn{n}$,
for every $n$ (\cref{thm:represent}).
\Cref{sec:app:re} backs the first theorem of
the same subsection.

The present appendix is the first to consume three earlier ones,
and it is brief precisely because they did the work.
The forward direction runs a compiled step function
through \cref{sec:app:compile};
the backward direction is proof search,
from \cref{sec:app:re:truth};
and the truth--provability coincidence of
\cref{sec:app:sound:capstone}
converts the forward direction's truth into provability.
What remains once those are in hand is
one term, two instance-evaluation lemmas, and an assembly.
The brevity is evidence that the structure is paying off,
not that the result is slight.

\subsection{The Representing Term}
\label{sec:app:represent:term}

Recursive enumerability is formulated computationally
(\cref{sec:app:re:bg}):
$p$ is recursively enumerable when
$p = \sfdom{\varphi}$
for a primitive-recursive step function $\varphi$ --
that is, when $p(n)$ holds exactly if
\emph{some step $i$ accepts},
$\varphi(n,i) \ne 0$.

\begin{defi}[the matrix and the representing term;
  \textsf{rmat}, \textsf{rrep}]
\label{defi:app:rrep}
For an index $f$ of the two-argument calculus
of \cref{sec:app:compile:calc}, put
\begin{align*}
  \rmat_f(n) &\;\ldef\;
    \lambda\,\bigl((\cmp f \cdot \tofn{n} \cdot \vz) = \z\bigr) , \\
  \rrep_f &\;\ldef\;
    \lambda\,\neg\bigl(\cA \cdot
      \lambda\,((\cmp f \cdot v_1 \cdot \vz) = \z)\bigr) .
\end{align*}
Both are well-formed and closed, by inspection.
\end{defi}

Read them as follows.
$\cmp f$ is the compiled step function
of \cref{sec:app:compile},
so $\cmp f \cdot \tofn{n} \cdot \tofn{i}$
computes $\varphi(n,i)$.
The matrix at $n$ is therefore the predicate
\emph{step $i$ rejects},
as a function of $i$:
it says the step function returns $\z$.
And $\rrep_f \cdot \tofn{n}$
$\beta$-reduces to $\neg(\cA \cdot \rmat_f(n))$ --
\emph{not every step rejects $n$}.

The existential is expressed by duality,
and that is the design.
\rga has no primitive existential quantifier;
$\exists$ is \emph{defined} as $\neg\forall\neg$
(\cref{fig:syntax}'s derived forms).
``Some step accepts'' is therefore written as
``not every step rejects'',
which is exactly the shape of $\rrep_f$.

\paragraph{What makes the duality effective.}
The point is the conceptual centre of the appendix,
and it is where \rga differs from a classical setting
rather than merely notating things differently.

Classically, ``not every step rejects'' is a weak statement.
It follows from the failure of a universal
without producing anything:
one learns that not all instances hold,
and is left with no instance in hand.
In \rga it is strong,
because the semantics of a \emph{refuted} universal
is counterexample-carrying.
By the \irl{allF} clause of \cref{fig:sem},
$\cA \cdot u$ is grounded-false
exactly when some numeral instance is refutable.
So the accepting step is not merely implied to exist.
It is what the falsity of the universal \emph{consists in},
and it can be recovered.

\Cref{sec:app:represent:eval} is the two directions
of that recovery,
and neither would go through in a system
whose negated universals carried no witness.

\subsection{Instance Evaluation, in Both Directions}
\label{sec:app:represent:eval}

Two preliminaries, then the two directions.

\emph{The $\beta$-reduction.}
$\rrep_f \cdot \tofn{n}$ evaluates to $r$
exactly when $\neg(\cA \cdot \rmat_f(n))$ does.
Both directions are the $\beta$ clause of evaluation:
$\rrep_f$ head-evaluates to its own abstraction shape,
and the $\beta$-instance of its body at $\tofn{n}$
is exactly $\neg(\cA \cdot \rmat_f(n))$ --
the outer substitution replaces $v_1$ by $\tofn{n}$
and leaves $\vz$ bound, which is the matrix.
Numerals being closed, the lifts vanish.
The computation is the one exhibited in
\cref{sec:app:compile:init};
here it is enough that the two terms evaluate alike.

\emph{The matrix computes the step function.}
$\rmat_f(n) \cdot \tofn{i}$
evaluates to $\tv{\varphi(n,i) = 0}$.
By \cref{sec:app:compile:cor},
$\cmp f \cdot \tofn{n} \cdot \tofn{i}$
evaluates to $\varphi(n,i)$;
$\z$ evaluates to $0$;
and the equality clause compares them
and returns the truth value of the comparison.
So the matrix instance is $\tvT$
when step $i$ \emph{rejects},
and $\tvF$ when it \emph{accepts}.

\subsubsection{If some step accepts, the term is true}
\label{sec:app:represent:if}

\begin{lem}[\textsf{rrep\_true\_if}]
\label{lem:app:rrep-if}
If $\varphi(n,i) \ne 0$ for some $i$,
then $\rrep_f \cdot \tofn{n}$ evaluates to $\tvT$.
\end{lem}

\begin{proof}
Five steps, and the third is the one to pause on.
\begin{enumerate}
\item \emph{The instance is false.}
  The matrix instance $\rmat_f(n) \cdot \tofn{i}$
  evaluates to $\tv{\varphi(n,i) = 0}$,
  which is $\tvF$, since the step accepts.
\item \emph{So its negation is true.}
  The negation clause flips a genuine boolean,
  so $\neg(\rmat_f(n) \cdot \tofn{i})$
  evaluates to $\tvT$.
\item \emph{So its negation is derivable.}
  The term is well-formed,
  so the truth--provability coincidence
  (\cref{sec:app:sound:capstone}) gives
  $\emptyset \vdash \neg(\rmat_f(n) \cdot \tofn{i})$.
  This is where completeness enters,
  and it is what converts a computed falsity
  into a syntactic object.
  Extract a valid proof $P$ of that judgment.
\item \emph{So the universal is false.}
  The \irl{allF} clause takes exactly this.
  The bare quantifier head-evaluates to its own shape,
  $P$ is a refutation certificate
  for the numeral instance $\tofn{i}$,
  and the clause concludes that
  $\cA \cdot \rmat_f(n)$ evaluates to $\tvF$.
\item \emph{So the representing term is true.}
  Negating, $\neg(\cA \cdot \rmat_f(n))$
  evaluates to $\tvT$,
  and so therefore does $\rrep_f \cdot \tofn{n}$.
\end{enumerate}
\end{proof}

The shape of that argument is
\cref{sec:power:re}'s claim made concrete.
The paper says \rga's semantics of refuted universals
``makes the duality effective'',
and that ``the accepting step supplies that counterexample --
which \cref{thm:complete-val} converts into
a refutation certificate''.
Steps 1 to 4 are that sentence, in order:
the accepting step $i$ is the counterexample,
step 3 is the conversion,
and step 4 is the clause that consumes it.

\subsubsection{If the term is true, some step accepts}
\label{sec:app:represent:onlyif}

\begin{lem}[\textsf{rrep\_true\_only\_if}]
\label{lem:app:rrep-onlyif}
If $\rrep_f \cdot \tofn{n}$ evaluates to $\tvT$,
then $\varphi(n,i) \ne 0$ for some $i$.
\end{lem}

\begin{proof}
By the $\beta$-reduction,
$\neg(\cA \cdot \rmat_f(n))$ evaluates to $\tvT$.
The term is closed,
so the witness collapse
(\cref{prop:app:collapse})
makes that evaluation witnessed.
Now invert twice.
Inverting the negation gives that
$\cA \cdot \rmat_f(n)$ evaluates,
in the witnessed sense, to $\tvF$.
Inverting the witnessed application
at a bare-quantifier head
(\cref{sec:app:sound:dict})
leaves only the \irl{allF} alternative,
which supplies a numeral $\tofn{i}$
together with \emph{the witnessed falsity of the instance}
$\rmat_f(n) \cdot \tofn{i}$.
That same instance has the value
$\tv{\varphi(n,i) = 0}$;
collapsing it to the witnessed level as well,
witnessed uniqueness (\cref{prop:app:wdet})
forces $\tv{\varphi(n,i) = 0} = \tvF$,
which is $\varphi(n,i) \ne 0$.
\end{proof}

\paragraph{The witnessed pair earns its keep a second time.}
\Cref{sec:app:sound:witness} introduced
the witnessed evaluation pair as a \emph{proof device}
for soundness:
certificates could not be trusted,
so their semantic content was recorded
as an extra premise.
Here the device is reused for something else entirely.

What this direction needs is to get
the accepting step $i$ \emph{out} of a false universal.
At the official level a false universal
offers only a certificate,
and extracting $i$ from it would mean analysing a proof.
At the witnessed level the instance's falsity
sits in the derivation as a premise,
so inverting hands $i$ over directly,
and uniqueness then identifies its value.
A device introduced to make soundness provable
turns out to be the natural tool for
reading a witness out of a refuted quantifier.

A reader who found the witnessed detour
elaborate in \cref{sec:app:sound}
may take this as the second payment on it.

\subsection{The Two Halves, and the Characterization}
\label{sec:app:represent:thms}

\subsubsection{Every recursively enumerable predicate is represented}
\label{sec:app:represent:fwd}

If $p$ is recursively enumerable,
there is a closed well-formed $t$ with
$p(n)$ holding exactly when
$t \cdot \tofn{n}$ evaluates to $\tvT$,
for all $n$.

Let $\varphi$ be a primitive-recursive step function
with $p = \sfdom{\varphi}$,
let $f$ be an index for $\varphi$
in the two-argument calculus,
and take $t \ldef \rrep_f$.
Then for each $n$:
$p(n)$ holds
iff some step accepts
iff $\rrep_f \cdot \tofn{n}$ evaluates to $\tvT$,
by \cref{lem:app:rrep-if,lem:app:rrep-onlyif}.

One bookkeeping step deserves naming,
because it is where the two calculi are joined,
and because a careful reader will stop at it.
``$\varphi$ is primitive recursive''
and ``$\varphi$ has an index
in the two-argument calculus''
are different statements,
and the development crosses between them
with a choice operator:
$f$ is \emph{chosen} to be some index denoting $\varphi$,
which exists exactly because
$\varphi$ is primitive recursive.
That is legitimate --
the existence of an index is what
primitive-recursiveness means,
and the bridge lemma says
the chosen index denotes $\varphi$ --
but it is a choice,
and nothing here depends on
the choice being canonical.

\subsubsection{Every represented predicate is recursively enumerable}
\label{sec:app:represent:bwd}

If $t$ is well-formed,
then $\{\, n \mid \emptyset \vdash t \cdot \tofn{n} \,\}$
is recursively enumerable.

This half is purely \cref{sec:app:re}.
The step function that,
given $n$,
searches for a proof of
$\emptyset \vdash t \cdot \tofn{n}$
is primitive recursive --
its constructor is
$n \mapsto \quo{t \cdot \tofn{n}}$
composed with the proof-checking step function
of \cref{sec:app:re:truth} --
and its domain is exactly that set.
The argument is the observation that derivability
is recursively enumerable,
specialized to a family of judgments
indexed by a numeral,
and it costs nothing beyond checking that
the family's codes are computed
primitive-recursively.

\subsubsection{The characterization}
\label{sec:app:represent:char}

\begin{thmc}[representation, \cref{thm:represent};
  \textsf{rga\_representable\_iff\_renp}]
A predicate $p$ is recursively enumerable
if and only if there is
a closed well-formed term $t$ with
$p(n) \iff \emptyset \vdash t \cdot \tofn{n}$
for all $n$.
\end{thmc}

\begin{proof}
Left to right:
\cref{sec:app:represent:fwd} gives a $t$
with the \emph{evaluation} form of the equivalence,
and \cref{sec:app:sound:capstone}
converts it to the \emph{derivability} form.
Right to left is \cref{sec:app:represent:bwd}.
\end{proof}

The coincidence is used on one side only.
The forward half produces a term
whose instances are grounded-\emph{true},
where the statement wants them \emph{derivable},
and \cref{thm:iff} makes that crossing.
The backward half starts from \emph{derivable}
and enumerates proofs,
with no appeal to truth at all.
This is the appendix in which
the truth--provability coincidence stops being
a headline result and becomes a working tool,
used without comment --
which is worth remarking,
because it is what makes the statement of the theorem
simple enough to display.

\subsection{What the Theorem Says, and What It Does Not}
\label{sec:app:represent:scope}

\paragraph{\rga captures every recursively enumerable predicate,
with room to spare in the form of the capture.}
The representing term is applied to a numeral,
$t \cdot \tofn{n}$:
a single application, positive occurrences only.
There is no coding of the predicate as a formula schema,
no substitution into a template,
and no side condition on $n$.
Representation is function application.

\paragraph{And \rga captures nothing more.}
Provability is recursively enumerable
(\cref{sec:app:represent:bwd}, from \cref{sec:app:re}),
and by the coincidence so is grounded truth.
\Cref{sec:omega} draws the consequence:
since instance-wise truth of a $\Pi_1$ statement
is not recursively enumerable,
the reflective universal cannot coincide with
the $\omega$-rule.
\Cref{sec:app:omega} develops that.
Here it is enough to record that
the boundary established in this appendix
is what makes the separation possible to state.

%% file: app-arith.tex
\section{Internal Arithmetic in Detail}
\label{sec:app:arith}

This appendix transcribes the machine-checked development
behind \cref{sec:power:arith},
which is the theories \texttt{RGA\_Demo},
\texttt{RGA\_Total} and \texttt{RGA\_Ackermann}.
\rga states and proves quantified totality theorems
about its own terms, by induction, at the empty context.
\begin{gather*}
  \emptyset \vdash \forall.\forall.\;
    (\tadd \cdot v_1 \cdot \vz) \jnat
  \qquad
  \emptyset \vdash \forall.\forall.\;
    (\tmult \cdot v_1 \cdot \vz) \jnat \\
  \emptyset \vdash \forall.\forall.\;
    (\tack \cdot v_1 \cdot \vz) \jnat
\end{gather*}
These are \textsf{rga\_add\_total},
\textsf{rga\_mult\_total}
and \textsf{rga\_ack\_total},
and they are \cref{thm:addmult,thm:ack}.

\paragraph{A change of altitude.}
\Cref{sec:app:complete,sec:app:sound,sec:app:compile,%
sec:app:re,sec:app:represent}
are metatheory: statements \emph{about} \rga,
proved in the metalanguage.
This appendix is the first whose subject is
\emph{derivations inside \rga}.
Every proof below is a sequence of \rga rule applications,
and the metalanguage does nothing but assemble them.
The reader's task changes accordingly.
It is no longer ``is this argument valid?''
but ``can I see which rule fires,
and why its side conditions are met?''

There is one \emph{shape} of totality derivation,
exhibited in full for addition in
\cref{sec:app:arith:shape}.
Multiplication reuses that shape verbatim,
so \cref{sec:app:arith:mult} says so
and does not repeat it --
but its step case is a genuinely new technique,
and that is exhibited.
\Cref{sec:app:arith:class} then proves the same thing
for each primitive-recursive index, uniformly in the index,
which is what turns the two worked examples
into instances rather than demonstrations;
and Ackermann, being the first function
outside that class, gets the room.
\Cref{sec:app:arith:systemT} closes with
the System~T theorem,
a fact of HOL rather than an internal derivation.

\subsection{The Terms}
\label{sec:app:arith:terms}

Written with named variables for readability,
as \cref{sec:power:arith} does;
the development uses de Bruijn indices throughout,
and where the index arithmetic matters
it is given below.
\begin{align*}
  \tadd &\;\ldef\; \lambda x\, y.\;
    \cR \cdot x \cdot (\lambda k\, r.\; \suc r) \cdot y \\
  \tmult &\;\ldef\; \lambda x\, y.\;
    \cR \cdot \z \cdot
      (\lambda k\, r.\; \tadd \cdot x \cdot r) \cdot y \\
  \tack &\;\ldef\; \mathit{fix}(F) ,
    \quad\text{where } F \ldef
    \lambda a\, m\, n.\;
    \cR \cdot (\suc n) \cdot
      \bigl(\lambda m'\, \_.\;
        \cR \cdot (a \cdot m' \cdot \tofn 1) \cdot
          (\lambda \_\, r.\; a \cdot m' \cdot r) \cdot n\bigr)
      \cdot m
\end{align*}
Here $\mathit{fix}$ is the \emph{term-level}
fixpoint construction of \cref{sec:app:character:fix} --
$\mathit{self}(f) \ldef
  \lambda(\lft f \cdot (\vz \cdot \vz))$
and $\mathit{fix}(f) \ldef
  \mathit{self}(f) \cdot \mathit{self}(f)$ --
so that $\tack$ head-steps in \emph{one} step
to $F \cdot \tack$.
All are well-formed, by inspection, and closed.

\paragraph{Write $\mathit{fix}(F)$, not $Y \cdot F$.}
The difference is not pedantry.
\rga can define a $Y$ combinator,
and $Y \cdot F$ is one conversion away from
$\mathit{fix}(F)$.
But $\tack$ is \emph{not} defined as
an application of $Y$.
It is the self-application
$\mathit{self}(F) \cdot \mathit{self}(F)$,
and the first step of the unfolding
(\cref{sec:app:arith:recur}) is
$\tack \cdot m \cdot n \hs F \cdot \tack \cdot m \cdot n$
by the fixpoint step,
with no $Y$ in the term at any point.
Writing $\tack \ldef Y \cdot F$ would put a term
in the definition that the derivations never touch,
and a reader following this appendix
would then be unable to reproduce the first conversion.

It also explains $F$'s leading parameter.
$F$ receives \emph{$\tack$ itself} as its first argument,
which is exactly what $\mathit{fix}$ supplies --
so the parameter $a$ is not arbitrary
but the recursive occurrence.

\paragraph{The step functions ignore their first argument.}
\rga's recursor hands its step function two arguments,
the counter and the accumulated result,
because the counter is needed in general.
Addition does not need it --
$\lambda k\, r.\; \suc r$ drops $k$ --
and neither does multiplication's step.
Call-by-name is what makes the dropped argument free:
nothing forces $k$,
so discarding it costs nothing
and imposes no \emph{habeas quid} obligation.
In Ackermann both dropped-argument positions recur,
and there the discipline is load-bearing
rather than merely convenient.

\paragraph{Why Ackermann is defined by general recursion.}
A reader will ask why the classical definition --
primitive recursion returning a \emph{function} --
is not used.
It evaluates perfectly well under
\rga's call-by-name semantics.
What it cannot do is be \emph{reasoned about}
with the schematic recursor rule,
because that rule concludes a \emph{grounded equality},
and grounded equality holds only between terms
with natural-number values.
A function-valued recursor spine has no value,
so the equation would be ungrounded by design.
The development therefore defines $\tack$
through the fixpoint combinator
with every recursor at base type:
the outer recursor dispatches on $m$,
the inner one iterates on $n$,
and every unfolding is either pure $\beta$,
which is always available,
or a natural-valued dispatch.
Partial by default; proved total.

\subsection{The Shape of a Totality Derivation}
\label{sec:app:arith:shape}

Take the goal
$\emptyset \vdash \forall.\forall.\;
  (\tadd \cdot v_1 \cdot \vz) \jnat$.
Because the quantifier is
the $\cA$ combinator applied to an abstraction,
the goal is reached by two schematic universal introductions
from a judgment with two typing hypotheses in context.
So the real work is
\[
  \vz \jnat,\; v_1 \jnat \;\vdash\;
    (\tadd \cdot v_1 \cdot \vz) \jnat ,
\]
and the derivation has five movements.
\Cref{sec:app:arith:mult,sec:app:arith:ack}
both instantiate this shape.

\begin{enumerate}
\item \emph{Set up the schematic context.}
  Two nested universal introductions mean
  the context is built by lifting:
  after the outer introduction
  the parameter $x$ sits at index $1$
  and the recursion argument $y$ at index $0$.
  Both typing hypotheses are available,
  the first directly and the second
  after computing what the lift did to the context.
\item \emph{Reduce the application to its recursor spine.}
  $\tadd \cdot x \cdot y$ $\beta$-reduces in two steps to
  $\cR \cdot x \cdot \tsucf \cdot y$.
  The conversion rule \irl{conv{=}}
  records each step as a provable equation,
  so the goal reduces to the same claim about the spine.
  This is the \emph{unfolding} lemma,
  and every function below has one.
\item \emph{The base case.}
  At scrutinee $\z$ the recursor's \emph{syntactic} step fires:
  $\cR \cdot x \cdot \tsucf \cdot \z \hs x$,
  so \irl{conv{=}} gives
  $\cR \cdot x \cdot \tsucf \cdot \z = x$,
  and the hypothesis $x \jnat$ supplies the value.
\item \emph{The step case.}
  Assume the induction hypothesis $X \jnat$,
  where $X \ldef \cR \cdot x \cdot \tsucf \cdot k$.
  Two $\beta$-conversions compute the applied step function,
  $\tsucf \cdot k \cdot X = \suc X$;
  rule \irl{\suc{}TI} on the induction hypothesis
  gives $(\suc X) \jnat$,
  which is the side condition the second conversion needs;
  the \emph{schematic} recursor rule \irl{\cR\suc{}{=}} gives
  $\cR \cdot x \cdot \tsucf \cdot (\suc k)
    = \tsucf \cdot k \cdot X$;
  and chaining these,
  $\cR \cdot x \cdot \tsucf \cdot (\suc k)$ is a natural.
\item \emph{Close.}
  Apply \irl{ind} with the base, the step,
  the scrutinee premise $y \jnat$,
  and the side condition that the induction variable
  is not free in the context;
  then two universal introductions.
\end{enumerate}

\paragraph{Both of \cref{sec:system}'s recursor rules
earn their place here.}
The base case uses the \emph{syntactic} head step,
its scrutinee being a literal $\z$.
The step case uses the \emph{schematic} rule
\irl{\cR\suc{}{=}},
which dispatches at $\suc k$ for an arbitrary $k$
proved to be a natural.
\Cref{sec:system:proof} introduces the schematic rule
as ``the form needed under induction,
where no syntactic numeral is in hand'',
and this derivation is that sentence's cash value:
under an induction the scrutinee is a \emph{variable},
so the syntactic step cannot fire,
and without the schematic rule
the step case would be unreachable.
Neither rule would do alone,
and the five movements above show both.

\paragraph{The not-free side condition is doing real work.}
The requirement of \irl{ind} that
the induction variable be absent from the context
is here a small computation over
the two typing hypotheses and the lifted empty set.
It is the same side condition whose \emph{semantic} role
appeared in \cref{sec:app:sound:indcase} --
there it kept the outer hypotheses invariant
under the meta-induction,
and here it is a syntactic proof obligation.

\subsection{Multiplication}
\label{sec:app:arith:mult}

\emph{The derivation is \cref{sec:app:arith:shape}'s, verbatim}:
same context setup, same two typing hypotheses,
same unfolding to a recursor spine,
same closing by \irl{ind}
and two universal introductions.
The base case differs only in
which term supplies the value --
$\cR \cdot \z \cdot \tmulf \cdot \z = \z$,
discharged by the axiom \irl{\z{}I}
rather than by a hypothesis.
The shape is not repeated here.

The step case is a new technique,
and it is why the example is in the paper.
It needs $(\tmulf \cdot k \cdot X) \jnat$,
where $\tmulf \cdot k \cdot X$ $\beta$-reduces to
$\tadd \cdot x \cdot X$.
So it needs \emph{totality of addition
at the schematic $x$ and the induction-hypothesis term $X$} --
not at numerals,
but at the variables and terms currently in hand.
Four moves obtain it.
\begin{enumerate}
\item \emph{Weaken the already-proved theorem in.}
  Totality of addition is a theorem
  at the \emph{empty} context,
  so it holds in any context:
  $\Gamma \vdash \forall.\forall.\;
    (\tadd \cdot v_1 \cdot \vz) \jnat$.
\item \emph{Instantiate it twice.}
  Rule \irl{{\cA}E} takes
  $\Gamma \vdash \cA \cdot f$ and $\Gamma \vdash a \jnat$
  to $\Gamma \vdash f \cdot a$.
  Applying it at $x$ and then at $X$ --
  each \emph{habeas quid} premise discharged by
  a typing fact already available,
  the second being the induction hypothesis itself --
  yields the instance.
\item \emph{$\beta$-convert the instances.}
  Each elimination leaves an application of an abstraction
  to the instantiating term,
  and \irl{conv} reduces it to the instance formula.
\item The result is $(\tadd \cdot x \cdot X) \jnat$,
  which is what the step needed.
\end{enumerate}

\paragraph{A quantified theorem is an object you can apply.}
A quantified theorem in \rga is
a \emph{reusable internal lemma}:
it can be weakened into any context
and \emph{applied}, under another induction,
to whatever terms are locally in scope.
For multiplication the quantifier is convenient
rather than indispensable.
An open lemma
$x \jnat, y \jnat \vdash (\tadd \cdot x \cdot y) \jnat$
would serve as well,
instantiated at the step's terms through \irl{sub}
and discharged through \irl{cut},
as in primitive recursive arithmetic.
What an open lemma cannot supply is
an induction \emph{hypothesis} that is itself a universal,
instantiated at values computed during the induction;
Ackermann's inner induction needs exactly that
(\cref{sec:app:arith:ack}).

One index fact is worth recording,
because a reader will otherwise trip on it.
Inside the step function's two binders
the outer parameter $x$ sits at index $3$,
two further abstractions having intervened.
The phenomenon is the same as
\cref{sec:app:compile:cmp}'s $v_3$
in the compiled recursor's step branch,
from the same cause and at the same depth.

\subsection{The Whole Primitive-Recursive Class}
\label{sec:app:arith:class}

Addition and multiplication are not
two isolated demonstrations.
They are two instances of something \rga proves
for the \emph{entire} primitive-recursive class,
uniformly;
and Ackermann is then not a third example
but the first function \emph{outside}
the class that has just been handled wholesale.

\subsubsection{What the theorem says}
\label{sec:app:arith:classstmt}

For every two-argument primitive-recursive index $f$,
\[
  \emptyset \vdash \forall.\forall.\;
    (\cmp f \cdot v_1 \cdot \vz) \jnat ,
\]
where $\cmp$ is the compiler of \cref{sec:app:compile}
and the two universals are,
as always in \rga,
the combinator $\cA$ applied to an abstraction
(\textsf{cmp\_total}).

\paragraph{Read the quantifier over $f$ in the right place.}
The two readings differ exactly where
this appendix has spent its argument caring,
so the distinction has to be made explicitly.
The index $f$ is quantified in the \emph{metalanguage},
not inside \rga.
The result is therefore a \emph{schema}:
one \rga theorem for each primitive-recursive index,
not a single \rga sentence quantifying over indices.
\rga is not being asked to reason about its own compiler.
It is handed, for each $f$ separately,
a closed sentence about one concrete term,
and it proves it.

That matters here more than it would elsewhere,
because \cref{sec:app:arith:mult}'s central observation
is precisely the difference between
a quantified theorem and a schema.
An appendix that draws that distinction
and then described this result as
a quantified theorem
would be undercutting itself on its own subject.

\paragraph{It is nonetheless a \emph{uniform} schema.}
The metalevel proof is a structural recursion
over the index (\cref{sec:app:arith:indind}),
with no choices and no search:
read as a construction,
it builds the derivation from $f$'s syntax tree
by a fixed procedure,
so the map from $f$ to
an \rga derivation of its totality is effective.
That effectiveness describes the construction;
it is not itself a mechanized theorem,
since the development proves, for each $f$,
that a derivation exists.

\subsubsection{Why this is not a corollary of the compiler}
\label{sec:app:arith:notcor}

A reader arriving from \cref{sec:app:compile}
will expect a corollary,
and it is worth saying why there is not one.
The tempting route runs:
compiler correctness gives that
$\cmp f \cdot \tofn{a} \cdot \tofn{b}$
\emph{evaluates} to $\tofn{\den{f}(a,b)}$;
completeness gives that a closed term with a value
has its value provably;
hence $\emptyset \vdash
  (\cmp f \cdot \tofn{a} \cdot \tofn{b}) \jnat$
for every pair of numerals.

That conclusion is true, and it is not the theorem.
What it gives is every \emph{numeral instance}.
What the theorem gives is
the universally quantified sentence,
and \cref{sec:app:omega} is the proof that
the gap between those two is real:
there is a closed well-formed $u$
with every numeral instance provable
whose universal closure \rga does not prove.
Instance-wise provability does not in general
yield the universal.
The quantified statement has to be earned by
an \emph{open} derivation --
one that works at a variable,
not at each numeral --
and that is what
\cref{sec:app:arith:indind} constructs.

\paragraph{The two theorems about the compiler
are independent.}
Neither is a weakening of the other.
\Cref{sec:app:compile}'s correctness theorem
is external, and says the compiled term
computes the \emph{right} value.
The present theorem is internal,
and says \rga can \emph{prove} it computes
\emph{a} value.
The development bears the independence out literally:
the totality proof cites no evaluation lemma
at any point,
using only the compiler's definition
and its closedness.
Taken together the two say that
the compiled term is provably total inside \rga,
and that the value it is total at is the intended one --
which is the pair a reader wants,
and which neither half supplies alone.

\subsubsection{The statement a reader expects
is not the one that is provable}
\label{sec:app:arith:general}

This is the appendix's recurring pattern once more,
and here it is at its sharpest,
because it is \emph{the same theorem's other half}.
\Cref{sec:app:compile:correct} had to generalize
the \emph{evaluation} statement
from numerals to arbitrary closed arguments;
the present section has to generalize
the \emph{derivability} statement
from a closed quantified sentence
to an open judgment at
an arbitrary context and arbitrary arguments.
Same compiler, same three cases,
and the same case forcing it.

The induction runs instead on
(\textsf{cmp\_total\_app}):
for every index $f$, every context $\Gamma$,
and all terms $a$ and $b$,
if $\Gamma \vdash a \jnat$ and $\Gamma \vdash b \jnat$
then $\Gamma \vdash (\cmp f \cdot a \cdot b) \jnat$.

Three generalizations are made at once,
each forced by a specific case,
and a re-mechanizer who weakens any of them
will fail at the named place.
\emph{Quantifier-free}:
the conclusion is an application judgment,
not a quantified sentence,
so induction hypotheses can be used directly
rather than instantiated.
\emph{General in the arguments}:
forced by the composition case,
which applies $g$'s hypothesis to
the two compound terms
$\cmp h_1 \cdot a \cdot b$ and
$\cmp h_2 \cdot a \cdot b$,
and again by the recursion case;
none of these is a numeral and none is closed.
\emph{General in the context}:
forced by the recursion case alone,
which must use its hypotheses in
a different, larger context than the one
it was handed.

\paragraph{Context-generality and the induction rule
are in tension.}
Resolving that tension is the one genuinely new move
in this section.
Context-generality is what makes the hypotheses usable.
But \rga's induction rule demands that
the induction variable not be free in the context,
and for an \emph{arbitrary} $\Gamma$
that side condition cannot be discharged --
indeed for a $\Gamma$ whose hypotheses
mention the induction variable it is false.
So the recursion case cannot run its induction
in the $\Gamma$ it was given.
What it does instead is
\cref{sec:app:arith:indind}'s third move.

\subsubsection{One transport lemma covers every case}
\label{sec:app:arith:transport}

Every clause of the compiler produces
a two-argument abstraction
(\cref{defi:app:cmp}),
so one lemma serves the whole induction.
Suppose $\bsub{(\lambda\,b_1)}{a} = \lambda\,b_2$
and $\bsub{b_2}{b} = r$,
that $a$ and $b$ are well-formed,
and that $\Gamma \vdash r \jnat$.
Then $\Gamma \vdash ((\llam\,b_1) \cdot a \cdot b) \jnat$
(\textsf{ir\_app2\_beta\_N}).
The proof is two head $\beta$-steps
and a backward transport of $\jnat$ along each,
the transport being \irl{conv{=}}
giving the equation
and grounded equality handing back
$\jnat$ for its left side.

This is the derivability-level twin of
\cref{lem:app:beta2}.
That lemma says a two-argument abstraction
applied to two arguments \emph{evaluates}
to whatever its doubly-instantiated body evaluates to;
this one says the same redex \emph{is a natural}
whenever its doubly-instantiated body is.
Both hand the two $\beta$-instances in as hypotheses
rather than computing them inside,
for the reason \cref{sec:app:compile:kits} gives.
The $\beta$-instance computations themselves
are the closedness kit of
\cref{sec:app:compile:closed} again:
$\cmp f$ is closed,
so every lift and substitution passes through it
unchanged.

\subsubsection{The induction on the index}
\label{sec:app:arith:indind}

The metalevel induction is on the structure of $f$,
with the context and both arguments generalized.
Three cases.

\emph{The four initial functions.}
Each compiles to a two-argument abstraction
whose doubly-instantiated body is
$\z$, $\suc a$, $a$, or $b$,
and the typing is immediate:
the axiom \irl{\z{}I},
\irl{\suc{}TI} on the hypothesis $a \jnat$,
or one of the two hypotheses outright.

\emph{Composition.}
$\cmp(\iSub\,g\,h_1\,h_2) \cdot a \cdot b$
$\beta$-reduces in two steps to
$\cmp g \cdot (\cmp h_1 \cdot a \cdot b)
  \cdot (\cmp h_2 \cdot a \cdot b)$.
The hypotheses for $h_1$ and $h_2$
give the typings of the two compound terms
at the very same $\Gamma$, $a$, $b$;
the hypothesis for $g$,
applied at those two compound terms,
gives the typing of the whole.
This is where argument-generality is spent,
and the case is three lines long.

\emph{Primitive recursion.}
This case does the work.
Its shape is
\cref{sec:app:arith:shape}'s five movements
with the compiled step branch in place of $\tsucf$,
and the five movements are not repeated here.
What is new is the context management,
in three moves.

\begin{enumerate}
\item \emph{Move to a canonical context.}
  Fix the two-hypothesis schematic context
  of \cref{sec:app:arith:shape},
  built by lifting twice from $\emptyset$.
  Everything below happens there,
  and the $\Gamma$ the case was handed
  is set aside until the third move.
\item \emph{Run the derivation.}
  The motive is the typing of
  the recursor spine at the induction variable.
  The base case fires the \emph{syntactic} head step;
  the step case computes the applied step branch
  by two $\beta$-conversions,
  types the inner call by $h'$'s hypothesis
  and the whole by $h$'s --
  applied at that compound term
  and at the induction hypothesis itself --
  and then the \emph{schematic} recursor rule
  transports the typing to the successor scrutinee.
  Discharge the freshness side condition,
  which is a finite check over
  two hypotheses and a lifted empty set,
  and apply \irl{ind}.
\item \emph{Generalize, then re-instantiate.}
  Two schematic universal introductions
  turn the judgment into
  the closed quantified sentence
  at the empty context;
  weakening carries it into the given $\Gamma$;
  and two eliminations instantiate it
  at $a$ and $b$,
  each \emph{habeas quid} premise discharged by
  the case's own hypotheses.
\end{enumerate}

\paragraph{The quantified form is used \emph{inside}
the proof of the quantifier-free form.}
It may look circular.
The closed quantified theorem is a corollary of
the open one
(\cref{sec:app:arith:quantform}),
yet the open one's hardest case
builds an instance of exactly that quantified shape
and consumes it.
There is no circularity --
the instance built in the third move
is at the index already under construction,
and is proved before it is used --
but the reason it happens is worth stating.

It is the tension of
\cref{sec:app:arith:general}:
\emph{the internal induction rule can only be run
in a context the prover controls,
and the internal quantifier is the device
that carries a result out of that controlled context
into an arbitrary one}.
Universal introduction discharges
the schematic context;
universal elimination re-supplies
whatever the caller actually has.
That is quantification being used for
precisely what quantification is for,
one level down from where
\cref{sec:app:arith:mult} was using it --
and it adds to that section's observation.
A quantified theorem is not only
an object one can \emph{apply};
it is also how a result \emph{leaves}
a context the prover had to control.

The third move is, in detail,
\cref{sec:app:arith:mult}'s
weaken-instantiate-convert sequence verbatim.
The only difference is what it is applied to:
there, a theorem proved in an earlier section;
here, a sentence proved moments earlier
in the same case.

\subsubsection{The quantified form and its consumer}
\label{sec:app:arith:quantform}

The closed theorem follows by instantiating
the open one at the canonical context
and at the two schematic variables,
whose typings are the two hypotheses,
and then applying schematic universal introduction twice;
the canonical context having been built
by lifting twice from $\emptyset$,
the result sits at the empty context.

Its consumer (\textsf{ir\_cmp\_total\_inst})
is the instantiation lemma:
from the quantified sentence in a context $\Gamma$,
together with $\Gamma \vdash x \jnat$
and $\Gamma \vdash y \jnat$,
conclude $\Gamma \vdash (\cmp f \cdot x \cdot y) \jnat$
by two universal eliminations,
each followed by a $\beta$ conversion
of the instance redex.

One intermediate form looks wrong and is not.
Substituting the first argument
under the inner abstraction shifts it by one,
so after the first elimination the formula reads
$(\cmp f \cdot \lft x \cdot \vz) \jnat$
rather than the $x$ a reader expects;
the second substitution cancels the shift.
This is ordinary de Bruijn bookkeeping.
It is \emph{not} an instance of
\cref{rem:app:openness}'s theme,
which is about generalizing
closed-fragment arguments,
and it should not be read as one.

\subsubsection{What the arc says}
\label{sec:app:arith:arc}

Compare the two inductions by their motives:
\[
\begin{array}{l@{\quad}l}
  \text{the whole class} &
    (\cR \cdot G \cdot H \cdot v_2) \jnat \\
  \text{Ackermann} &
    \forall n.\; (\tack \cdot m \cdot n) \jnat
\end{array}
\]
The first is \emph{quantifier-free}:
a typing judgment about a recursor spine,
and structurally the same motive as
\cref{sec:app:arith:shape}'s.
The second is not --
it quantifies over the second argument,
and \cref{sec:app:arith:whyack} explains why it must.

So the arc is not a rhetorical escalation
from a small example to a big one.
It is a visible change in what
the induction rule is being asked to carry:
\rga's single induction rule
discharges the whole primitive-recursive class
with quantifier-free motives,
and then keeps going --
because the rule places no condition
on the motive at all.

\paragraph{The analogy this invites,
and where it stops being one.}
The pairing invites a classical gloss:
quantifier-free induction is the strength of \pra,
and Ackermann is the standard example of
what lies beyond it.
The gloss is worth making,
because it tells a reader what they are looking at.
It is \emph{not} a theorem about \rga,
and must not be presented as one.

Nothing here shows that Ackermann's totality
is \emph{unprovable} in \rga
from a quantifier-free motive,
and the obvious argument does not transfer:
an \rga motive containing no quantifier
may still contain arbitrary recursion --
$\tack$ itself, for one --
so \rga's quantifier-free-motive fragment
is not \pra,
and the classical non-provability result
says nothing about it.
What is established is exactly this:
the class is handled with quantifier-free motives,
the hard example is not,
and \rga's induction rule accommodates both
without a change of formula class,
because it has none.
Anything stronger would need a separate proof,
which the development does not contain.

\subsection{Ackermann}
\label{sec:app:arith:ack}

\subsubsection{Why this one is different}
\label{sec:app:arith:whyack}

Ackermann's function lies outside
the class \cref{sec:app:arith:class} has just handled
wholesale: it is not primitive recursive,
and its totality is not provable in
primitive recursive arithmetic,
nor by \pra's quantifier-free induction principle.
Both of those are classical facts about \pra,
and neither is a claim about \rga{} --
\cref{sec:app:arith:arc} says why the analogy
stops short of one.
The reason is exact,
and it is what makes the example a showpiece:
\emph{the termination argument is a nested induction
whose outer motive contains a universal quantifier}.
One must induct on $m$ with the motive
``for all $n$, $\tack \cdot m \cdot n$ is a natural'',
and that motive is not a formula \pra can state.

\rga states it,
because \irl{ind} places no restriction on the motive
beyond well-formedness --
\cref{sec:system:proof}: the motive
``may contain quantifiers, recursion,
even ungrounded subformulas''.
The motive is just another well-formed term,
so the informal termination argument
transcribes directly.

Here \cref{sec:system}'s no-formula-classes design
and \cref{sec:power:arith}'s showpiece meet.
\Cref{sec:app:sound:indcase} established
the \emph{semantic} half:
soundness of \irl{ind} inspects
the motive's \emph{value}, never its shape,
which is why the arithmetical hierarchy
collapses in \rga's semantics.
The present subsection is the \emph{use} --
a motive containing a quantifier,
inducted on, at the empty context.

\subsubsection{The unfolding and the three recurrences}
\label{sec:app:arith:recur}

Before either induction,
the derivation establishes that $\tack$
satisfies its defining equations
\emph{as provable equalities}.

The unfolding takes $\tack \cdot m \cdot n$
through four conversions to the dispatch spine.
The first is the fixpoint step
$\tack \hs F \cdot \tack$,
which is itself a $\beta$ step,
$\tack$ being a self-application
(\cref{sec:app:arith:terms});
then three $\beta$ steps peel
$F$'s three abstractions,
the first of which lands $\tack$
in $F$'s parameter $a$.
Every step is pure $\beta$,
hence always available with no side condition,
which is exactly why the definition was arranged
as \cref{sec:app:arith:terms} describes.

The three recurrences are then as follows.
$\tack \cdot \z \cdot n = \suc n$,
given $n \jnat$,
by the \emph{syntactic} step,
the scrutinee being a literal $\z$.
$\tack \cdot (\suc m) \cdot n$
equals the inner iteration spine
$\cR \cdot (\tack \cdot m \cdot \tofn 1) \cdot
  (\lambda \_\, r.\; \tack \cdot m \cdot r) \cdot n$,
by two $\beta$ steps taking
the $m$-step function to the spine.
And the inner spine at $\z$ and at $\suc k$
behave as the base and step of
\cref{sec:app:arith:shape}.

These are \emph{equations}, not evaluations.
Each is a grounded equality,
so each carries the information that
both sides have the same natural value,
which is what the inductions consume.

\subsubsection{The nested induction}
\label{sec:app:arith:nested}

\emph{Outer.}
Induct on $m$ with motive
$\Phi(m) \ldef \forall n.\; (\tack \cdot m \cdot n) \jnat$,
the universal being, as always in \rga,
the $\cA$ combinator applied to an abstraction.

For the base $\Phi(\z)$:
for a schematic $n$ with $n \jnat$,
the first recurrence gives
$\tack \cdot \z \cdot n = \suc n$,
whose left side is therefore a natural,
and universal introduction over $n$ gives $\Phi(\z)$.
The step from $\Phi(m)$ to $\Phi(\suc m)$
is where the second induction lives.

\emph{Inner.}
Under the outer hypothesis $\Phi(m)$,
available in the context as a quantified formula,
induct on $n$ with motive
``the inner iteration spine at scrutinee $n$
is a natural''.

At the inner base,
the spine at $\z$ reduces to its seed
$\tack \cdot m \cdot \tofn 1$.
Its totality comes from
the outer induction hypothesis instantiated at $\tofn 1$:
\irl{{\cA}E} applied to $\Phi(m)$,
with the \emph{habeas quid} premise
$\tofn 1 \jnat$ supplied by
\irl{\suc{}TI} on \irl{\z{}I},
followed by a $\beta$ conversion of the instance redex.

At the inner step,
the spine at $\suc k$ reduces,
by the schematic recursor rule,
to $\tack \cdot m \cdot X$
where $X$ is the spine at $k$.
Totality of $X$ is the \emph{inner} induction hypothesis.
Totality of $\tack \cdot m \cdot X$
then comes from the \emph{outer} hypothesis $\Phi(m)$
instantiated at $X$ --
and the \emph{habeas quid} premise of that elimination,
$X \jnat$, is the inner induction hypothesis itself.

\paragraph{The interlock is the whole proof.}
It is worth stating in one sentence:
\begin{quote}
the inner induction hypothesis is precisely
the \emph{habeas quid} premise that
the outer induction hypothesis needs
in order to be applied.
\end{quote}
The two inductions are not merely nested.
They feed each other,
exactly as in the informal termination argument
for Ackermann,
where one argues that $\tack(m,\cdot)$ is total
by knowing $\tack(m{-}1,\cdot)$ is total
\emph{and} that the inner iteration
has produced a value to feed it.
\rga reproduces that interlock rule for rule,
because \irl{{\cA}E} demands a witness
and the inner induction supplies one.

Universal introduction over $n$ gives $\Phi(\suc m)$,
completing the outer step;
\irl{ind} then gives $\Phi(m)$ for a schematic $m$,
and a final universal introduction
gives the theorem at the empty context.

\subsection{System~T}
\label{sec:app:arith:systemT}

The theories \texttt{RGA\_SystemT} and \texttt{RGA\_SystemT\_PR}
treat \gdl's System~T as a fragment of \rga's own terms.
Simple types are built from the ground type $\iota$ by $\to$.
A typing relation $\Delta \vdash_T t : \sigma$
over \rga terms, with contexts as type lists indexed by de Bruijn variables,
has the usual rules for variables, $\lambda$ and application;
$\z : \iota$ and $\suc a : \iota$ for $a : \iota$;
and one rule for the recursor,
typing a fully applied spine:
if $g : \sigma$, $h : \iota \to \sigma \to \sigma$ and $m : \iota$,
then $\cR \cdot g \cdot h \cdot m : \sigma$.
Partial recursor applications lie outside the fragment;
$\eta$-expansion recovers them.
Typed terms are well-formed,
and every compiled function $\cmp f$ has type
$\iota \to \iota \to \iota$
(\textsf{cmp\_hastype}, by induction on the index).

\begin{thmc}[\cref{thm:systemT};
  \textsf{fragment\_habeas\_quid}, \textsf{cmp\_hastype}]
If $t$ is a closed System~T term of type $\iota$,
then $\emptyset \vdash t \jnat$.
\end{thmc}

\begin{proof}
Tait's reducibility argument, in HOL.
Reducibility $\mathrm{Red}_\sigma(t)$ is defined by recursion on $\sigma$:
at $\iota$, $t$ has a value;
at $\sigma \to \tau$, $t \cdot u$ is reducible at $\tau$
for every closed $u : \sigma$ reducible at $\sigma$.
Reducibility is closed under behavioral equivalence,
the relation of having the same values and head shapes,
and in particular under head expansion.
The fundamental lemma (\textsf{fundamental})
says that a typed term closed by a substitution of
closed, typed, reducible terms is reducible,
by induction on the typing derivation.
The variable, $\z$, successor, $\lambda$ and application cases
are routine.
In the recursor case the scrutinee $m$ is reducible at $\iota$,
so it has a value $k$,
and by determinism on closed terms
$\cR \cdot g \cdot h \cdot m$ is behaviorally equivalent to
$\cR \cdot g \cdot h \cdot \tofn{k}$.
An inner induction on $k$ then shows the numeral case reducible:
at $\z$ the head step to $g$,
at $\suc \tofn{j}$ the head step to
$h \cdot \tofn{j} \cdot (\cR \cdot g \cdot h \cdot \tofn{j})$,
each followed by head expansion.
At the empty environment,
every closed term of type $\iota$ therefore has a value
(\textsf{fragment\_totality}),
and value completeness (\cref{thm:complete-val})
turns the value into a derivation of $t \jnat$.
\end{proof}

The theorem is a fact of HOL that yields
one \rga derivation for each closed ground-type term.
It does not internalize the reducibility argument,
and it gives no quantified \rga theorem per function:
for that, compare \cref{sec:app:arith:class},
where the quantified sentence is built by an open derivation.

\subsection{Closing Remarks}
\label{sec:app:arith:remarks}

\paragraph{Both recursor rules earn their place
in one derivation.}
The syntactic head step fires at literal $\z$
in every base case;
the schematic rule fires at $\suc k$ in every step case,
where no numeral is in hand.
\Cref{sec:app:arith:shape} shows both
in five movements.

\paragraph{Where the quantifier is indispensable.}
Multiplication uses addition's theorem,
which an open lemma could replace;
Ackermann's inner induction uses
the outer hypothesis, a universal, twice,
which no open lemma or schema can replace.

\paragraph{\emph{Habeas quid} appears as
a proof obligation to be discharged.}
Every universal elimination above needed a witness
that its instantiating term has a value.
In \cref{sec:app:arith:mult} it was a typing hypothesis;
in Ackermann's inner base an arithmetic triviality;
and in Ackermann's inner step
the other induction's hypothesis.
The discipline that \cref{sec:system} describes
as a design principle is, in the derivations,
a premise to be proved.

%% file: app-omega.tex
\section{\texorpdfstring{$\omega$}{omega}-Incompleteness in Detail}
\label{sec:app:omega}

This appendix transcribes the machine-checked development
behind \cref{sec:omega}
(\texttt{RGA\_Omega}, with the diagonal anchor
in the shared computability layer):
there is a closed well-formed term $u$
every numeral instance of which \rga proves,
whose universal closure \rga does \emph{not} prove
(\cref{thm:omega}) --
and, sharpening,
whose universal closure is not grounded-true
at any step index (\cref{cor:ungrounded}).

This appendix is dense.
Little in it batches:
almost nothing here is a family of similar lemmas,
but one construction, one diagonal argument,
and two ways of reading the result.
Nothing below is compressed
on grounds of routineness.

\subsection{What Is Cited, and What Is Exhibited}
\label{sec:app:omega:scope}

The diagonal machinery lives in
the shared computability layer, not in \rga,
so the line has to be drawn explicitly.
Cited and not re-derived:
the universal interpreter $\sfapp$
as a primitive-recursive step function,
and the recursive-enumerability apparatus
of \cref{sec:app:re:bg}.
Exhibited in full:
the diagonal set, the theorem that
its complement is not recursively enumerable,
and above all the \emph{gap lemma}.

The gap lemma is the one place
where a citation would not do,
and the reason is worth stating,
because it is not that the lemma is important.
The familiar textbook fact is
``the complement of the halting set
is not recursively enumerable''.
What \cref{sec:app:omega:incomplete} needs is
a \emph{refinement} of that fact:
given any recursively enumerable set
contained in the complement of the diagonal,
there is a point of the complement that it misses --
and, in the form the development proves,
that point is \emph{computed from the set's own index}.
That is a different statement.
It turns a counting fact into a statement about
a gap between two sets,
and no citation to the classical result supplies it.

Its proof, moreover,
is the same three lines as the classical theorem's.
So exhibiting the lemma costs nothing
beyond exhibiting what would have been cited anyway,
which is the whole of the argument
for exhibiting it here.

\subsection{The Diagonal, and the Gap Lemma}
\label{sec:app:omega:gap}

$\sfapp\,\cpair{x}{x}\,n$ runs step $n$ of
computation $x$ on input $x$,
returning $0$ when that step has not halted it.
The \emph{halting diagonal} is
\[
  K \;\ldef\;
  \{\, x \mid \sfapp\,\cpair{x}{x}
    \text{ halts at some step} \,\} ,
\]
which the development calls $\mathit{diagK}$,
so that $x \notin K$ exactly when
$\sfapp\,\cpair{x}{x}\,n = 0$ for every $n$.

$K$ is recursively enumerable:
its own step function is
$(x,n) \mapsto \sfapp\,\cpair{x}{x}\,n$,
primitive recursive by composition,
and $K$ is exactly that function's domain.

Its complement is not
(\textsf{not\_rens\_compl\_diagK}).
Suppose the complement of $K$ were
the domain of a primitive-recursive step function.
Let $i$ be an index for that function
and put $x_0 \ldef \quo{i}$,
the code of the index.
The interpreter is universal,
so running computation $x_0$ on input $x_0$
\emph{is} running the function on $x_0$, whence
\[
  x_0 \in K
  \iff \sfapp\,\cpair{x_0}{x_0} \text{ halts}
  \iff \text{the function halts on } x_0
  \iff x_0 \notin K ,
\]
and a set and its complement cannot share a member.

\begin{lem}[the gap lemma; \textsf{diag\_gap}]
\label{lem:app:gap}
Let $p$ be the domain of the step function of index $i$,
and suppose $p$ is \emph{sound for divergence}:
every $x$ with $p(x)$ really lies outside $K$.
Then $\quo{i} \notin K$ and $\neg p(\quo{i})$.
\end{lem}

\begin{proof}
Put $x_0 \ldef \quo{i}$
and reuse the universality equivalence:
$p(x_0)$ holds exactly when
the function of index $i$ halts on $x_0$,
which is exactly when $x_0 \in K$.
If $p(x_0)$ held then $x_0 \in K$,
while soundness would put $x_0$ outside $K$.
So $\neg p(x_0)$;
and the same equivalence read backwards
gives $x_0 \notin K$.
\end{proof}

Three things about that lemma deserve saying.

\emph{It is constructive in the index.}
The witness is $\quo{i}$,
the code of the enumerating index itself.
No choice operator appears and none is needed:
this is the \emph{primary} form
the development proves.
There is a derived form for callers holding only
``$p$ is recursively enumerable'' rather than an index,
and that one localizes a Hilbert choice
to a single specialization --
so a caller able to supply an explicit index
incurs no choice at all.

\emph{Soundness of $p$ is a hypothesis, not a conclusion.}
The lemma says nothing about
arbitrary recursively enumerable sets.
It applies to those that only ever
assert genuine divergence,
and \cref{sec:app:omega:incomplete}
must supply that hypothesis for \rga's provable universals,
which is a \emph{soundness} obligation about \rga.

\emph{It is the classical argument reused, not re-run.}
The proof just given and the proof that
the complement is not recursively enumerable
are the same three lines.
Only the conclusion drawn from them differs.

\subsection{The Diagonal Family Inside \texorpdfstring{\rga}{RGA}}
\label{sec:app:omega:family}

Compiling the interpreter with
\cref{sec:app:compile}'s compiler
turns the diagonal into \rga formulas.
For each $x$,
\[
  u_x \;\ldef\;
    \lambda\bigl((\cmp \sfapp \cdot
      \tofn{\cpair{x}{x}} \cdot \vz) = \z\bigr),
  \qquad
  U_x \;\ldef\; \cA \cdot u_x .
\]
Both are well-formed and closed.
The instance $u_x \cdot \tofn{n}$ says
\emph{step $n$ does not halt computation $x$ on $x$};
the universal $U_x$ says
\emph{computation $x$ never halts on $x$},
a $\Pi_1$ statement.

\emph{The instances evaluate.}
By compiler correctness at numerals
(\cref{sec:app:compile:cor}),
$\cmp \sfapp \cdot \tofn{\cpair{x}{x}} \cdot \tofn{n}$
evaluates to $\sfapp\,\cpair{x}{x}\,n$;
comparing with $\z$ by the equality clause,
the instance evaluates to
$\tv{\sfapp\,\cpair{x}{x}\,n = 0}$.
So every instance terminates:
it is a primitive-recursive computation
with a boolean value.

\emph{The instances are settled, provably.}
Hence by value completeness
(\cref{sec:app:complete:thms}) in the true case
and consistency (\cref{sec:app:sound:consistency})
in the false,
for every $x$ and $n$,
\[
  \emptyset \vdash u_x \cdot \tofn{n}
  \quad\text{iff}\quad
  \sfapp\,\cpair{x}{x}\,n = 0 .
\]

\paragraph{The family is chosen to make
the instance level trivial.}
Every instance is a terminating computation,
so \rga decides every instance, provably,
with no quantifier involved.
Whatever difficulty arises
can therefore only be at the universal,
which is exactly where
the $\omega$-rule question lives.
A family with undecidable instances
would prove nothing about quantification.

\emph{The provable universals are recursively enumerable.}
The map $x \mapsto \quo{U_x}$
is primitive recursive --
the code of $U_x$ is built from $x$
by numeral, abstraction and application
code constructors, each primitive recursive --
and derivability is recursively enumerable
by \cref{sec:app:re:truth}.
Composing, $\{\, x \mid \emptyset \vdash U_x \,\}$
is recursively enumerable.

\subsection{\texorpdfstring{$\omega$}{omega}-Incompleteness}
\label{sec:app:omega:incomplete}

\begin{thmc}[$\omega$-incompleteness, \cref{thm:omega};
  \textsf{rga\_omega\_incomplete}]
There is a closed well-formed $u$ with
$\emptyset \vdash u \cdot \tofn{n}$ for every $n$,
while $\emptyset \nvdash \cA \cdot u$.
\end{thmc}

\begin{proof}
Suppose not.
Then the closed $\omega$-rule is admissible in \rga:
for every closed well-formed $u$,
if $\emptyset \vdash u \cdot \tofn{n}$ for all $n$,
then $\emptyset \vdash \cA \cdot u$.
Under that assumption,
membership in the complement of $K$
coincides with provability of the universal.

If $x \notin K$ then
$\sfapp\,\cpair{x}{x}\,n = 0$ for every $n$,
so every instance $u_x \cdot \tofn{n}$ is provable;
$u_x$ is closed and well-formed;
so the principle applies and delivers
$\emptyset \vdash U_x$.
Conversely, if $\emptyset \vdash U_x$,
then universal elimination \irl{{\cA}E},
with its \emph{habeas quid} premise supplied by
numerals being naturals,
gives $\emptyset \vdash u_x \cdot \tofn{n}$
for each $n$;
by soundness each such instance evaluates to $\tvT$,
and by determinism the compiled step function's value
is then $0$,
so $\sfapp\,\cpair{x}{x}\,n = 0$ for every $n$,
which is $x \notin K$.

So the complement of $K$ equals
$\{\, x \mid \emptyset \vdash U_x \,\}$,
which is recursively enumerable
by \cref{sec:app:omega:family} --
contradicting \cref{sec:app:omega:gap}.
\end{proof}

\paragraph{What is refuted, and hence how strong
the theorem is.}
The assumption discharged is quantified over
\emph{every} closed well-formed $u$,
not merely over the diagonal family.
So the theorem does not say
that this particular family defeats \rga.
It says that no sound effective system
expressing the diagonal family
admits the closed $\omega$-rule,
and that \rga, being sound and effective, does not.
(An inconsistent effective system admits it trivially.)
The existential form of the conclusion
is the exact negation of that universal assumption,
which is why \cref{sec:omega} can say
the family ``collectively outruns
every uniform proof method''
without the theorem pointing at
the particular $x$ where provability fails.

\begin{thmc}[the gap inside the family, \cref{cor:gap};
  \textsf{rga\_ouniv\_gap}]
Some $x \notin K$ has $\emptyset \nvdash U_x$.
\end{thmc}

\begin{proof}
The set $\{\, x \mid \emptyset \vdash U_x \,\}$
is recursively enumerable (\cref{sec:app:omega:family}),
and by the backward direction above it contains
only points of the complement of $K$.
So the set is sound for divergence,
and the gap lemma (\cref{lem:app:gap})
applied to it yields a point of the complement outside it.
Given an explicit index of a step function
for that set,
the development computes the point from the index
(\textsf{rga\_ouniv\_gap\_witness\_idx}),
with no choice operator.
\end{proof}

\paragraph{Which half of the metatheory each direction uses.}
The forward direction uses completeness
(\cref{sec:app:complete})
to settle the instances,
and then the assumed rule.
The backward direction uses
the quantifier elimination rule,
then soundness with determinism
(\cref{sec:app:sound})
to read each provable instance back
as a fact about the step function.
The contradiction uses \cref{sec:app:re}:
derivability is recursively enumerable.
So the theorem rests on
completeness, soundness with determinism,
and enumerability,
as \cref{sec:omega}'s proof sketch says.

\subsection{The Semantic Sharpening}
\label{sec:app:omega:sharp}

\begin{thmc}[$\omega$-ungroundedness, \cref{cor:ungrounded};
  \textsf{rga\_omega\_ungrounded}]
There is a closed well-formed $u$ with
$\emptyset \vdash u \cdot \tofn{n}$ for every $n$,
while $\cA \cdot u$ evaluates to $\tvT$
at no step index.
\end{thmc}

\begin{proof}
Take the $u$ of the previous theorem.
By the truth--provability coincidence
(\cref{sec:app:sound:capstone})
applied to the well-formed term $\cA \cdot u$,
unprovability \emph{is} not-being-grounded-true.
\end{proof}

\begin{thmc}[\cref{cor:omega-true};
  \textsf{omv\_strictly\_extends\_grounded}]
Some closed well-formed $c$ is $\omega$-true
while $\emptyset \nvdash c$
and $c$ evaluates to $\tvT$ at no step index.
\end{thmc}

\begin{proof}
Take $c \ldef \cA \cdot u$ for the $u$ of \cref{thm:omega}.
Every instance $u \cdot \tofn{n}$ is provable,
hence $\omega$-true by $\omega$-soundness
(\cref{sec:app:sound:omega}),
so $\cA \cdot u$ is $\omega$-true by the $\omega$-clause.
It is unprovable by \cref{thm:omega},
and not grounded-true by the coincidence.
\end{proof}

Classically, a $\Pi_1$ statement
all of whose instances are true
but which is unprovable
is \emph{the} true-but-unprovable sentence.
\rga relocates the phenomenon
rather than exhibiting it,
and the relocation takes three steps.

The witness universal is not grounded-true;
that is the corollary.
It is not grounded-false either
(\cref{thm:not-false}):
a falsity certificate would have to
refute some numeral instance,
but every instance is \emph{provable},
so a refutation would contradict consistency
(\cref{sec:app:sound:consistency}).
Therefore it is ungrounded outright --
the reflective search finds neither
a schematic proof nor a counterexample,
ever, at no step index and not in the limit.

So where a classical reader expects
a true-but-unprovable sentence,
\rga presents an \emph{undecided} one,
true only in the instance-wise $\omega$-semantics
(\cref{cor:omega-true}).
Grounded truth and provability move together,
being the same thing,
and what gives instead is the classical conviction
that instance-wise truth must sum to universal truth.
Summing infinitely many instance proofs
\emph{is} the $\omega$-rule, an infinitary ground.
The reflective quantifier accepts only
an effective ground:
one schematic proof, uniform in the variable.
The gap between the two is not an artifact
of \rga's design
but the boundary of recursive enumerability itself
(\cref{rem:no-effective}).
\rga's contribution is to put the gap
\emph{inside} the semantics,
rather than between the semantics
and the proof system.

\subsection{Closing Remarks}
\label{sec:app:omega:remarks}

\paragraph{Dependencies.}
This appendix consumes four earlier ones.
\Cref{sec:app:compile} compiles the interpreter,
\cref{sec:app:complete} settles the instances,
\cref{sec:app:re} makes derivability
recursively enumerable,
and \cref{sec:app:sound:capstone}
turns unprovability into ungroundedness.

%% file: app-character.tex
\section{The Logic's Character in Detail}
\label{sec:app:character}

This appendix transcribes the machine-checked development
behind \cref{sec:character}
(\texttt{RGA\_Paradox},
with the disjunction and witness properties,
the decided-hypothesis rules
and the exact deduction theorem in \texttt{RGA\_Sound}):
four facts, each locating \rga
against a familiar system.
\rga is not classical --
excluded middle at the Liar is \emph{explosive},
not merely unavailable.
It is not intuitionistic --
double negation is eliminable as an entailment;
the disjunction and witness properties,
which hold too, \rga shares with Heyting arithmetic.
Its conditional is material --
the deduction theorem holds
exactly at provably decided hypotheses.
And its self-referential sentences are
\emph{independent} rather than paradoxical.

The four results share little machinery:
the fixpoint kit,
and the independence lemma of
\cref{sec:app:character:indep},
which rests on the truth--provability coincidence
of \cref{sec:app:sound}.
So this appendix is a sequence of
small complete arguments
with little to compress.
They triangulate rather than accumulate:
each is a boundary with a different neighbour.

\subsection{The Fixpoint Kit}
\label{sec:app:character:fix}

For any term $f$, put
\[
  \mathit{self}(f) \;\ldef\;
    \lambda\bigl(\lft f \cdot (\vz \cdot \vz)\bigr),
  \qquad
  \mathit{fix}(f) \;\ldef\;
    \mathit{self}(f) \cdot \mathit{self}(f) ,
\]
so that $\mathit{fix}(f)$ is
\emph{one head step} from $f \cdot \mathit{fix}(f)$:
the $\beta$-instance of $\mathit{self}(f)$'s body
at $\mathit{self}(f)$
is exactly $f \cdot (\mathit{self}(f) \cdot \mathit{self}(f))$,
and the lift on $f$ vanishes
because the substituted argument does not reach it.
$\mathit{fix}(f)$ is well-formed exactly when $f$ is,
and closed when $f$ is.

\paragraph{No G\"odel machinery appears in this appendix.}
A reader expecting arithmetization
should be told at once that none is coming.
Self-reference here is \emph{term-level}:
a term that reduces to a context applied to itself,
obtained by ordinary call-by-name self-application.
There is no arithmetization,
no diagonal lemma, and no provability predicate.
\Cref{sec:system:syntax}'s observation that
unconstrained recursion needs no primitive
is what makes this available,
and the present appendix is where
the consequences are collected.

\subsection{Ungroundedness Is Independence}
\label{sec:app:character:indep}

\begin{lem}[\textsf{no\_value\_independent}]
\label{lem:app:noval}
If $c$ is well-formed and has no value at any step index,
then $\emptyset \nvdash c$ and $\emptyset \nvdash \neg c$.
\end{lem}

\begin{proof}
Both halves by the truth--provability coincidence.
If $\emptyset \vdash c$ then $c$ evaluates to $\tvT$
(\cref{sec:app:sound:capstone}),
and since $c$ is well-formed
that evaluation is witnessed
by the extended collapse,
contradicting the hypothesis.
If $\emptyset \vdash \neg c$ then $\neg c$
evaluates to $\tvT$;
inverting the negation clause
gives $c$ a boolean value,
again contradicting the hypothesis.
\end{proof}

\paragraph{This two-line lemma is the engine
of the whole appendix.}
It is worth putting plainly:
\emph{independence proofs in \rga
are termination arguments}.

In a classical setting,
showing a sentence independent
means constructing two models.
Here it means showing that
a \emph{computation does not terminate},
and everything else follows,
because \rga's provability and its semantics coincide.
So each self-referential sentence below
is dispatched by a single evaluation argument,
and the proof-theoretic conclusions come free.
Two lines of \cref{sec:app:sound} do that work,
and every result in this appendix runs through them.

\subsection{The Liar, and Why Excluded Middle Is Explosive}
\label{sec:app:character:liar}

\subsubsection{The Liar has no value}
\label{sec:app:character:noval}

Put $L \ldef \mathit{fix}(\lambda\,\neg\vz)$,
which is two head steps from its own negation:
$L \hs (\lambda\,\neg\vz) \cdot L \hs \neg L$.

\begin{thmc}[\textsf{rga\_liar\_unprovable},
  \textsf{rga\_liar\_irrefutable}]
$L$ has no value at any step index,
and is therefore independent --
neither provable nor refutable.
\end{thmc}

\begin{proof}
Suppose $L$ evaluates to $r$.
Head steps transport values forward
(\cref{prop:app:hstep-rev}),
so $\neg L$ also evaluates to $r$.
Inverting the negation clause,
$L$ evaluates to some boolean $\tv{b}$
with $r = \tv{\neg b}$.
But $L$'s value is unique
(\cref{prop:app:wdet}),
so $\tv{b} = \tv{\neg b}$,
and $b = \neg b$ is false.
Independence then follows by
\cref{lem:app:noval}.
\end{proof}

The paradox is defused by non-termination,
and the mechanism deserves naming.
The classical Liar is a contradiction.
Here it is a term whose evaluation runs forever,
flipping a boolean it never has.
Nothing is inconsistent
because nothing is ever asserted:
the reflective search finds
neither a proof nor a counterexample.
The phenomenon is the paracompleteness of \cref{sec:intro}
in its smallest instance.

\subsubsection{Excluded middle at the Liar is explosive}
\label{sec:app:character:lem}

\begin{thmc}[\cref{thm:lem};
  \textsf{liar\_hyp\_explosive},
  \textsf{rga\_liar\_lem\_explosive}]
For every well-formed $c$,
$\;\Gamma,\, L \vdash c$
and
$\;\Gamma,\, (L \lor \neg L) \vdash c$.
\end{thmc}

\begin{proof}
For the first, convert the hypothesis $L$
\emph{forward} along the Liar's two head steps
(rule \irl{convE}) to obtain $\neg L$;
with $L$ itself, explosion \irl{{\neg}E} gives $c$.
For the second,
from the hypothesis $L \lor \neg L$,
reason by disjunction elimination.
In the left branch, assuming $L$,
convert \emph{forward} along the Liar's two head steps --
rule \irl{convE},
whose premise is about the redex
and whose conclusion is about the reduct --
obtaining $\neg L$.
Both $L$ and $\neg L$ are then derivable in that branch,
so explosion \irl{{\neg}E} gives $c$.
In the right branch, assuming $\neg L$,
convert \emph{backward} along the same two steps,
by \irl{conv}, obtaining $L$;
again both, again explosion, again $c$.
\end{proof}

\paragraph{The strength of this claim, stated carefully.}
It is not that \rga merely \emph{lacks} excluded middle,
as an intuitionist would say.
It is that no consistent extension of \rga
can validate it as a scheme
(\cref{rem:lem-scheme}):
any extension closed under
\irl{hyp}, \irl{conv}, \irl{convE}, \irl{{\neg}E},
\irl{{\lor}E_1} and \irl{cut}
that derives $\Gamma \vdash a \lor \neg a$
for every well-formed $a$
derives $\Gamma \vdash \z = \tofn{1}$,
by instantiating the scheme at $L$
and cutting against the theorem.
The development records this step as
\textsf{rga\_lem\_scheme\_explosive},
stated for \rga's own derivability;
there its premise is refuted by the theorem itself,
so the formal statement is vacuous,
and the extension claim is the unmechanized remark.
The derivation uses no arithmetic,
no coding, and no quantifier --
only the fixpoint,
the two conversion rules, and explosion.

\rga's exchange of
unrestricted excluded middle
for unrestricted recursion
is thereby a \emph{forced} trade
rather than a design preference.
With unrestricted recursion in hand,
excluded middle is not merely unproved
but unavailable on pain of triviality.

Both directions of conversion are used,
one per branch:
the left branch runs the head steps forwards
and the right backwards.
That derivation is the only one in the appendices
that \emph{fires} both conversion rules,
and it is a good illustration of
why \rga has both.

\subsubsection{Decidedness propagates}
\label{sec:app:character:decided}

\begin{thmc}[\cref{thm:decided};
  \textsf{ir\_negTI}, \textsf{ir\_orTI}]
If $\Gamma \vdash a \jbool$ and $\Gamma \vdash b \jbool$,
then $\Gamma \vdash (\neg a) \jbool$
and $\Gamma \vdash (a \lor b) \jbool$,
hence also $\Gamma \vdash (a \limp b) \jbool$
(\textsf{ir\_impTI}).
\end{thmc}

\begin{proof}
Each is a case analysis on the decided premises.
For $\neg a$: in the $a$ branch,
double-negation introduction gives $\neg\neg a$,
hence $\neg a \lor \neg\neg a$ by \irl{{\lor}I_2};
in the $\neg a$ branch, \irl{{\lor}I_1} gives it directly.
For $a \lor b$: in the $a$ branch, $a \lor b$ holds;
in the $\neg a$ branch, split on $b$:
if $b$, then $a \lor b$;
if $\neg b$, then \irl{{\lor}I_3} gives $\neg(a \lor b)$.
The implication case is the disjunction case
at $\neg a$ and $b$.
\end{proof}

\subsection{The Disjunction and Witness Properties}
\label{sec:app:character:markov}

\begin{thmc}[disjunction property, \cref{thm:disj};
  \textsf{ir\_disj\_prop}]
If $\emptyset \vdash a \lor b$,
then $\emptyset \vdash a$ or $\emptyset \vdash b$.
\end{thmc}

\begin{proof}
A provable disjunction is grounded-true,
so some witnessed evaluation gives $a \lor b$ the value $\tvT$.
Inverting the disjunction clause,
$a$ or $b$ evaluates to $\tvT$,
and both are well-formed,
since the disjunction is.
By the coincidence
(\cref{sec:app:sound:capstone}),
the true disjunct is provable.
\end{proof}

\begin{thmc}[witness property, \cref{thm:witness};
  \textsf{rga\_not\_all\_witness}]
If $\emptyset \vdash \neg(\cA \cdot u)$,
then $\emptyset \vdash \neg(u \cdot \tofn{n})$
for some $n$.
\end{thmc}

\begin{proof}
By the coincidence,
$\neg(\cA \cdot u)$ is grounded-true,
so inverting the negation clause
$\cA \cdot u$ evaluates to $\tvF$.
Only one clause gives
an $\cA$-application that value --
the \irl{allF} clause --
and its premises carry a numeral $\tofn{i}$
together with a certificate refuting the instance.
That certificate is a valid proof of
$\emptyset \vdash \neg(u \cdot \tofn{i})$.
\end{proof}

\rga's existential is the dual
$\exists.\,b \ldef \neg(\cA \cdot \lambda\neg b)$,
so the theorem says that
every provable existential is witnessed
by a provable numeral instance.
Markov's rule --
from $\vdash \neg\neg\exists$ conclude $\vdash \exists$ --
then follows by double-negation elimination,
and the existential so obtained carries a witness.
Markov's principle as an implication
holds exactly where \rga decides the existential
(\cref{thm:dne-imp}).

\paragraph{\rga does not postulate the rule,
because its semantics already is the search.}
The Russian constructivist school
\emph{postulates} Markov's principle:
an unbounded search guaranteed to terminate,
justified by an appeal outside the formal system.
\rga needs no such appeal,
because falsity of a universal was \emph{defined}
as a found counterexample
(the \irl{allF} clause of \cref{fig:sem}).
So the witness property is not an addition to \rga
but a reading-off of how universal falsity
was grounded in the first place.
Where the Russian school adds an axiom,
\rga reads off a rule.
Heyting arithmetic also has the disjunction property
and closure under Markov's rule,
so neither property separates \rga from intuitionism;
double-negation elimination as an entailment does.

That is the third time the \irl{allF} clause
has done work outside the place
it was designed for.
It supplied the semantic witness that made
the soundness of negated-universal elimination
go through (\cref{sec:app:sound:quartet});
it supplied the accepting step of a representing term
(\cref{sec:app:represent:onlyif});
and it supplies Markov's witness here.
A clause introduced to say
what a false universal \emph{means}
answers all three questions.

\subsection{The Deduction Theorem, Exactly at Decided Hypotheses}
\label{sec:app:character:dt}

\rga has no implication primitive.
Hypothesis sets carry that weight,
and the definable candidate for $a \limp c$
is $\neg a \lor c$.

\emph{The modus-ponens direction holds}
(\textsf{ir\_dt\_use}, part~1 of \cref{thm:dt-exact}):
if $\Gamma \vdash \neg a \lor c$
then $\Gamma,\, a \vdash c$.
Weaken the premise into the extended context
and eliminate the disjunction.
In the $\neg a$ branch,
both $\neg a$ and $a$ are available --
the latter as the new hypothesis --
so explosion gives $c$.
In the $c$ branch, $c$ is the hypothesis.

\emph{Discharge holds exactly at decided hypotheses}
(\textsf{deduction\_theorem\_exact}, part~2 of \cref{thm:dt-exact}).
Suppose first $\Gamma \vdash a \lor \neg a$
and $\Gamma,\, a \vdash c$.
Eliminate the disjunction:
in the $a$ branch, $\neg a \lor c$ follows from $c$
by \irl{{\lor}I_2};
in the $\neg a$ branch, from the hypothesis $\neg a$
by \irl{{\lor}I_1}
(the trunk rule \textsf{ir\_dt\_abs\_bool}).
Conversely, suppose discharge works for every $c$.
At $c \ldef a$ the hypothesis rule gives $\Gamma,\, a \vdash a$,
so $\Gamma \vdash \neg a \lor a$,
and one more disjunction elimination
reorders it to $a \lor \neg a$.

\begin{thmc}[\cref{thm:dt};
  \textsf{rga\_deduction\_theorem\_fails}]
$L \vdash L$, but $\emptyset \nvdash \neg L \lor L$.
\end{thmc}

\begin{proof}
The first is the hypothesis rule.
For the second, suppose
$\emptyset \vdash \neg L \lor L$.
Then that disjunction evaluates to $\tvT$,
and witnessed by the collapse.
Inverting the disjunction clause,
one of the two disjuncts evaluates to $\tvT$ --
so either $L$ has a value,
or $\neg L$ does and hence $L$ does.
Either way $L$ has a value,
contradicting \cref{sec:app:character:noval}.
\end{proof}

\begin{thmc}[\cref{thm:sigma-dt};
  \textsf{sigma\_deduction\_theorem}]
If $\vz \jnat, \lft\Gamma \vdash D \jbool$, then
$\exists.\,D, \Gamma \vdash c$
if and only if
$\Gamma \vdash \forall.\,(D \limp \lft c)$.
\end{thmc}

\begin{proof}
Left to right:
by universal introduction it suffices to derive
$D \limp \lft c$ under $\vz \jnat, \lft\Gamma$,
and since $D$ is decided there,
by part~2 of \cref{thm:dt-exact}
it suffices to derive $\lft c$ under the further hypothesis $D$.
From $D$ and $\vz \jnat$,
existential introduction at the eigenvariable gives $\lft(\exists.\,D)$,
and a cut against the lifted entailment
$\lft(\exists.\,D), \lft\Gamma \vdash \lft c$
gives $\lft c$.
Right to left, which needs no decidedness:
eliminate the hypothesis $\exists.\,D$
by \irl{{\neg\cA}E},
instantiate the universal at the eigenvariable
by \irl{{\cA}E},
and apply part~1 of \cref{thm:dt-exact}
to the instance $D \limp \lft c$.
\end{proof}

\paragraph{Where the failure comes from.}
The judgment $L \vdash L$ is a \emph{hypothetical} claim.
It says nothing about whether $L$ is grounded,
only that \emph{if} one assumes it,
one may conclude it.
The disjunction $\neg L \lor L$,
by contrast, is a \emph{grounded} claim:
to be true it must evaluate,
and a disjunction evaluates
only when one of its disjuncts does.
So discharging the hypothesis would
manufacture a groundedness commitment
that the hypothetical judgment never made --
it would turn ``on the assumption that $L$''
into ``$L$ is decided''.
Discharge fails exactly where the hypothesis is not provably decided,
by part~2 of \cref{thm:dt-exact}.
At the empty context and a closed $a$,
``provably decided'' coincides with
``grounded-true or grounded-false''.
In a non-empty context the two can differ.
At a term whose value is not boolean,
such as the numeral $\tofn{2}$,
which has a value but is neither true nor false,
they agree:
$\tofn{2}$ is neither provably decided
nor grounded-true or grounded-false.
Decidedness is the boundary.

\paragraph{The explosion result and this one
are the same fact, seen from two sides.}
$\neg L \lor L$ is the Liar's excluded-middle instance.
\Cref{sec:app:character:lem} says that
assuming it is explosive;
the present theorem says that it is not provable.
Put together:
\begin{quote}
the formula that hypothesis-discharge
would have to produce
is precisely the one that would make \rga trivial.
\end{quote}
Separately the two read as two curiosities.
Connected, they are one structural claim
about what grounding costs.
Where discharge fails,
it fails \emph{for a reason},
and not for want of a rule.

\subsection{Closing Remarks}
\label{sec:app:character:remarks}

\paragraph{Independence proofs here
are termination arguments.}
\Cref{lem:app:noval} converts
``this computation has no value''
into ``this sentence is neither provable
nor refutable'',
and every other result above goes through it.
That is a dividend of
the truth--provability coincidence,
and it is what lets a four-line evaluation argument
settle a proof-theoretic question.

%% file: app-notation.tex
\clearpage
\section{Notation}
\label{sec:app:notation}

The table collects the paper's notation,
with a reading and the place where each symbol is introduced.
Symbols used only inside one appendix are listed
with that appendix.

\begin{center}
\small
\begin{tabular}{@{}>{\raggedright\arraybackslash}p{0.2\textwidth}>{\raggedright\arraybackslash}p{0.5\textwidth}>{\raggedright\arraybackslash}p{0.24\textwidth}@{}}
\textbf{Symbol} & \textbf{Reading} & \textbf{Introduced} \\
\hline
\multicolumn{3}{@{}l}{\emph{Terms}} \\
$\vz, v_1, \dots$ & de Bruijn variables; $\vz$ is the innermost & \cref{fig:syntax} \\
$\bot$ & the divergent constant, a term with no value & \cref{fig:syntax} \\
$\cR$, $\cA$ & the primitive recursor; the universal-quantifier combinator & \cref{fig:syntax} \\
$\forall.\,b$, $\exists.\,b$ & $\cA \cdot (\lambda b)$; $\neg(\cA \cdot (\lambda \neg b))$ & \cref{fig:syntax} \\
$t \jnat$ & $t = t$: $t$ denotes a natural number & \cref{fig:syntax} \\
$t \jbool$ & $t \lor \neg t$: $t$ denotes a boolean & \cref{fig:syntax} \\
$\tofn{n}$ & the numeral of $n$, $\suc^n \z$ & \cref{fig:rules-reflect} \\
$\tvT$, $\tvF$ & the boolean values, the numerals $\tofn 1$ and $\tofn 0$ & \cref{fig:sem} \\
$\mathit{self}(f)$, $\mathit{fix}(f)$ & self-application; the fixpoint, one head step from $f \cdot \mathit{fix}(f)$ & \cref{sec:system:syntax,sec:app:character:fix} \\
$L$ & the Liar, $\mathit{fix}(\lambda\,\neg\vz)$ & \cref{sec:system:syntax} \\
$\wf(t)$ & $t$ is well-formed & \cref{sec:system:syntax} \\
$\natbot$ & the naturals with $\bot$ added: base-type values, and the range of assignments & \cref{sec:system:syntax} \\
$\cpair{x}{y}$ & the Cantor pair & \cref{sec:omega:family} \\
\multicolumn{3}{@{}l}{\emph{Proofs and syntax operations}} \\
$\Gamma \vdash c$ & judgment: hypothesis set $\Gamma$, conclusion $c$; $a, \Gamma$ inserts $a$ into $\Gamma$ & \cref{sec:system:proof} \\
$P \proves \Gamma \vdash c$ & $P$ is a valid proof whose conclusion is $\Gamma \vdash c$ & \cref{sec:system:proof} \\
$\quo{P}$ & the \gdl code of $P$ & \cref{fig:sem} \\
$\ssub{v}{s}{t}$ & substitution of $s$ for the free variable $v$ in $t$ & \cref{fig:rules} \\
$\lft t$ & lift: every free de Bruijn index of $t$ shifted up by one & \cref{fig:rules-reflect} \\
$b[t]$ & $\beta$-instance: $t$ for the outermost bound variable of $b$ & \cref{fig:rules-reflect} \\
$t \hs t'$ & head step (syntactic, primitive-recursively decidable) & \cref{fig:rules-reflect} \\
$\sof{\sigma}{t}$ & simultaneous substitution $\sigma$ applied to $t$ & \cref{defi:app:sjt} \\
$\spl{u_1, \dots, u_k}$, $t \sap \vec u$ & an argument spine; $t$ applied to a spine & \cref{sec:app:complete:prelim} \\
$\wfpart{\Gamma}$ & the well-formed members of $\Gamma$ & \cref{sec:app:sound:strip} \\
\multicolumn{3}{@{}l}{\emph{Evaluation and truth}} \\
$t \bev{s} r$ & $t$ evaluates to the number $r$ within step index (fuel) $s$ & \cref{fig:sem} \\
$t \hev{s} w$ & $t$ head-evaluates to the head shape $w$ & \cref{fig:sem} \\
$t \wbev{s} r$, $t \whev{s} w$ & the witnessed evaluation pair & \cref{sec:metatheory:arch,defi:app:witnessed} \\
$t \omv r$ & $\omega$-evaluation, the universal read instance by instance & \cref{sec:metatheory:omega} \\
$\agtrue{\Gamma}{c}$ & grounded consequence: every assignment making $\Gamma$ grounded-true makes $c$ grounded-true & \cref{sec:system:sem,defi:app:gtruth} \\
$\gtrue{\Gamma}{c}$ & substitutional consequence, over closing substitutions & \cref{defi:app:sjt} \\
$X \vc Y$ & value congruence & \cref{defi:app:vcr} \\
$X \ost{d} Y$ & one-site replacement at depth $d$ & \cref{sec:app:sound:seq} \\
\end{tabular}
\end{center}

%% file: app-endpoints.tex
\clearpage
\section{Endpoint Table}
\label{sec:app:endpoints}


The table maps every numbered theorem, corollary and proposition
of the main text to the Isabelle facts that state it,
with the theory and session that declare each fact.
A name declared in more than one theory visible to the session
is qualified by its theory, and a fact proved inside a locale
by the locale's name.
Remarks are not mechanized and do not appear.

\begingroup\small
\begin{longtable}{@{}l>{\raggedright\arraybackslash}p{0.43\textwidth}ll@{}}
\textbf{Result} & \textbf{Isabelle fact(s)} & \textbf{Theory} & \textbf{Session} \\
\hline
\endfirsthead
\textbf{Result} & \textbf{Isabelle fact(s)} & \textbf{Theory} & \textbf{Session} \\
\hline
\endhead
\Cref{thm:complete-val} & \textsf{rga\_complete\_val} & \texttt{RGA\_Complete} & \texttt{RGA} \\
\Cref{cor:complete-open} & \textsf{rga\_complete\_open} & \texttt{RGA\_Complete} & \texttt{RGA} \\
\Cref{thm:sound} & \textsf{rga\_sound}, \textsf{rga\_sound\_sjt} & \texttt{RGA\_Sound} & \texttt{RGA} \\
\Cref{thm:det} & \textsf{osr\_unique}, \textsf{osr\_hev\_excl}, \textsf{hev\_det} & \texttt{RGA\_Sound} & \texttt{RGA} \\
\Cref{thm:consistent} & \textsf{rga\_consistent} & \texttt{RGA\_Sound} & \texttt{RGA} \\
\Cref{cor:no-0eq1} & \textsf{rga\_no\_proof\_0eq1} & \texttt{RGA\_Sound} & \texttt{RGA} \\
\Cref{thm:sat-consistent} & \textsf{rga\_consistent\_sat} & \texttt{RGA\_Sound} & \texttt{RGA} \\
\Cref{thm:iff} & \textsf{rga\_open\_iff}, \textsf{rga\_val\_iff} & \texttt{RGA\_Sound} & \texttt{RGA} \\
\Cref{thm:nsound} & \textsf{rga\_N\_sound} & \texttt{RGA\_Sound} & \texttt{RGA} \\
\Cref{thm:disj} & \textsf{RGA\_Sound.ir\_disj\_prop} & \texttt{RGA\_Sound} & \texttt{RGA} \\
\Cref{thm:witness} & \textsf{rga\_not\_all\_witness} & \texttt{RGA\_Sound} & \texttt{RGA} \\
\Cref{thm:collapse} & \textsf{osr\_wosr}, \textsf{wosr\_osr} & \texttt{RGA\_Sound} & \texttt{RGA} \\
\Cref{thm:cmp} & \textsf{cmp\_correct} & \texttt{RGA\_Compile} & \texttt{RGA} \\
\Cref{thm:truth-re} & \textsf{rga\_truth\_re} & \texttt{RGA\_RE} & \texttt{RGA} \\
\Cref{thm:represent} & \textsf{rga\_representable\_iff\_renp} & \texttt{RGA\_Represent} & \texttt{RGA} \\
\Cref{thm:addmult} & \textsf{rga\_add\_total}, \textsf{rga\_mult\_total} & \texttt{RGA\_Demo} & \texttt{RGA} \\
\Cref{thm:cmp-total} & \textsf{cmp\_total} & \texttt{RGA\_Total} & \texttt{RGA} \\
\Cref{thm:ack} & \textsf{rga\_ack\_total} & \texttt{RGA\_Ackermann} & \texttt{RGA} \\
\Cref{thm:systemT} & \textsf{fragment\_habeas\_quid} & \texttt{RGA\_SystemT\_PR} & \texttt{RGA} \\
 & \textsf{cmp\_hastype} (compiled functions are System~T terms) & \texttt{RGA\_SystemT} & \texttt{RGA} \\
\Cref{thm:omega} & \textsf{rga\_omega\_incomplete} & \texttt{RGA\_Omega} & \texttt{RGA} \\
\Cref{cor:gap} & \textsf{rga\_ouniv\_gap} & \texttt{RGA\_Omega} & \texttt{RGA} \\
\Cref{thm:omv-not-re} & \textsf{omv\_ouniv\_iff}, \textsf{omega\_truth\_not\_rens} & \texttt{RGA\_Omega\_Model} & \texttt{RGA} \\
\Cref{prop:omv-liar} & \textsf{omv\_liar\_unvalued} & \texttt{RGA\_Omega\_Gradient} & \texttt{RGA} \\
\Cref{thm:omv-extends} & \textsf{omv\_unique}, \textsf{osr\_omv} & \texttt{RGA\_Omega\_Model} & \texttt{RGA} \\
\Cref{thm:omega-sound} & \textsf{rga\_omega\_sound\_strip}, \textsf{rga\_omega\_true\_consistent}, \textsf{rga\_omega\_sound} & \texttt{RGA\_Omega\_Sound} & \texttt{RGA} \\
\Cref{thm:omega-consistent} & \textsf{rga\_omega\_consistent} & \texttt{RGA\_Sound} & \texttt{RGA} \\
\Cref{cor:ungrounded} & \textsf{rga\_omega\_ungrounded} & \texttt{RGA\_Omega} & \texttt{RGA} \\
\Cref{thm:not-false} & \textsf{rga\_all\_instances\_not\_false} & \texttt{RGA\_Sound} & \texttt{RGA} \\
 & \textsf{rga\_omega\_witness\_ungrounded} & \texttt{RGA\_Omega} & \texttt{RGA} \\
\Cref{cor:omega-true} & \textsf{omv\_strictly\_extends\_grounded} & \texttt{RGA\_Omega\_Model} & \texttt{RGA} \\
\Cref{thm:omega-ungrounded} & \textsf{omega\_true\_unprovable\_ungrounded} & \texttt{RGA\_Omega\_Model} & \texttt{RGA} \\
\Cref{thm:weak-strong} & \textsf{rga\_strong\_complete\_fails\_open} (part 1) & \texttt{RGA\_Omega} & \texttt{RGA} \\
 & \textsf{rga\_strong\_complete\_fails\_con} (part 2) & \texttt{Entailment\_Goedel} & \texttt{RGA\_Goedel} \\
 & \textsf{rga\_strong\_complete\_closed\_iff} (part 3) & \texttt{RGA\_Sound} & \texttt{RGA} \\
\Cref{thm:value-explosive} & \textsf{value\_ne1\_explosive} & \texttt{RGA\_Sound} & \texttt{RGA} \\
\Cref{thm:liar} & \textsf{rga\_liar\_unprovable}, \textsf{rga\_liar\_irrefutable} & \texttt{RGA\_Paradox} & \texttt{RGA} \\
\Cref{thm:lem} & \textsf{liar\_hyp\_explosive}, \textsf{rga\_liar\_lem\_explosive} & \texttt{RGA\_Paradox} & \texttt{RGA} \\
\Cref{thm:lem-rule} & \textsf{rga\_liar\_lem\_explosive\_rule} & \texttt{RGA\_Paradox} & \texttt{RGA} \\
\Cref{thm:decided} & \textsf{RGA\_Sound.ir\_negTI}, \textsf{ir\_orTI}, \textsf{RGA\_Sound.ir\_impTI} & \texttt{RGA\_Sound} & \texttt{RGA} \\
\Cref{thm:dne-imp} & \textsf{dne\_imp\_iff\_decided}, \textsf{markov\_imp\_iff\_decided} & \texttt{RGA\_Sound} & \texttt{RGA} \\
\Cref{thm:markov-fails} & \textsf{markov\_imp\_fails} & \texttt{RGA\_Omega} & \texttt{RGA} \\
\Cref{thm:dt-exact} & \textsf{ir\_dt\_use} (part 1) & \texttt{RGA\_Paradox} & \texttt{RGA} \\
 & \textsf{deduction\_theorem\_exact} (part 2) & \texttt{RGA\_Sound} & \texttt{RGA} \\
\Cref{thm:dt} & \textsf{rga\_deduction\_theorem\_fails} & \texttt{RGA\_Paradox} & \texttt{RGA} \\
\Cref{thm:curry} & \textsf{curry\_ext.}\allowbreak\textsf{ext\_curry\_explosion} (part 1) & \texttt{IGA\_Curry} & \texttt{IGA} \\
 & \textsf{curry\_cond.}\allowbreak\textsf{curry\_explosion}, \textsf{curry\_cond.}\allowbreak\textsf{curry\_inconsistent} (part 2) & \texttt{IGA\_Curry} & \texttt{IGA} \\
 & \textsf{rga.}\allowbreak\textsf{liar\_explosion\_cp} (the material conditional) & \texttt{IGA\_Curry} & \texttt{IGA} \\
\Cref{thm:sigma-dt} & \textsf{sigma\_deduction\_theorem} & \texttt{RGA\_Sound} & \texttt{RGA} \\
\end{longtable}
\endgroup

The session \texttt{IGA} builds on session \texttt{RGA},
and its theory \texttt{IGA\_Curry} imports only
a theory of session \texttt{RGA}.
The session \texttt{RGA\_Goedel} belongs to the companion paper's
development~\cite{ford-conrga},
where \rga's consistency sentence is defined.